\documentclass[10pt]{amsart}
\usepackage{}
\usepackage{amssymb}
\usepackage{xcolor}
\usepackage{mathabx}

\usepackage{amsmath}

\usepackage[top=0.8in, bottom=0.8in, left=1.25in, right=1.25in]{geometry}

\input xy
\xyoption{all}
\usepackage{amsfonts}
\usepackage{mathrsfs}
\usepackage{tikz-cd}

\usepackage{graphicx}
\usepackage{amscd,amsbsy,amsthm}
\usepackage[all]{xy}
\usepackage[colorlinks,plainpages,urlcolor=blue]{hyperref}
\usepackage{verbatim}

\usepackage{enumitem}

\usepackage[normalem]{ulem}

\usepackage{tikz}
\usetikzlibrary{arrows,calc}
\tikzset{
>=stealth',
help lines/.style={dashed, thick},
axis/.style={<->},
important line/.style={thick},
connection/.style={thick, dotted},
}

\newcommand {\Omit}[1]{}

\usepackage{stackengine,scalerel}

\newcommand{\nc}{\newcommand}
\nc{\rnc}{\renewcommand}
\nc{\bb}[1]{{\mathbb #1}}
\nc{\bbA}{\bb{A}}\nc{\bbB}{\bb{B}}\nc{\bbC}{\bb{C}}\nc{\bbD}{\bb{D}}
\nc{\bbE}{\bb{E}}\nc{\bbF}{\bb{F}}\nc{\bbG}{\bb{G}}\nc{\bbH}{\bb{H}}
\nc{\bbI}{\bb{I}}\nc{\bbJ}{\bb{J}}\nc{\bbK}{\bb{K}}\nc{\bbL}{\bb{L}}
\nc{\bbM}{\bb{M}}\nc{\bbN}{\bb{N}}\nc{\bbO}{\bb{O}}\nc{\bbP}{\bb{P}}
\nc{\bbQ}{\bb{Q}}\nc{\bbR}{\bb{R}}\nc{\bbS}{\bb{S}}\nc{\bbT}{\bb{T}}
\nc{\bbU}{\bb{U}}\nc{\bbV}{\bb{V}}\nc{\bbW}{\bb{W}}\nc{\bbX}{\bb{X}}
\nc{\bbY}{\bb{Y}}\nc{\bbZ}{\bb{Z}}
\nc{\mbf}[1]{{\mathbf #1}}
\nc{\bfA}{\mbf{A}}\nc{\bfB}{\mbf{B}}\nc{\bfC}{\mbf{C}}\nc{\bfD}{\mbf{D}}
\nc{\bfE}{\mbf{E}}\nc{\bfF}{\mbf{F}}\nc{\bfG}{\mbf{G}}\nc{\bfH}{\mbf{H}}
\nc{\bfI}{\mbf{I}}\nc{\bfJ}{\mbf{J}}\nc{\bfK}{\mbf{K}}\nc{\bfL}{\mbf{L}}
\nc{\bfM}{\mbf{M}}\nc{\bfN}{\mbf{N}}\nc{\bfO}{\mbf{O}}\nc{\bfP}{\mbf{P}}
\nc{\bfQ}{\mbf{Q}}\nc{\bfR}{\mbf{R}}\nc{\bfS}{\mbf{S}}\nc{\bfT}{\mbf{T}}
\nc{\bfU}{\mbf{U}}\nc{\bfV}{\mbf{V}}\nc{\bfW}{\mbf{W}}\nc{\bfX}{\mbf{X}}
\nc{\bfY}{\mbf{Y}}\nc{\bfZ}{\mbf{Z}}
\nc{\bfa}{\mbf{a}}\nc{\bfb}{\mbf{b}}\nc{\bfc}{\mbf{c}}\nc{\bfd}{\mbf{d}}
\nc{\bfe}{\mbf{e}}\nc{\bff}{\mbf{f}}\nc{\bfg}{\mbf{g}}\nc{\bfh}{\mbf{h}}
\nc{\bfi}{\mbf{i}}\nc{\bfj}{\mbf{j}}\nc{\bfk}{\mbf{k}}\nc{\bfl}{\mbf{l}}
\nc{\bfm}{\mbf{m}}\nc{\bfn}{\mbf{n}}\nc{\bfo}{\mbf{o}}\nc{\bfp}{\mbf{p}}
\nc{\bfq}{\mbf{q}}\nc{\bfr}{\mbf{r}}\nc{\bfs}{\mbf{s}}\nc{\bft}{\mbf{t}}
\nc{\bfu}{\mbf{u}}\nc{\bfv}{\mbf{v}}\nc{\bfw}{\mbf{w}}\nc{\bfx}{\mbf{x}}
\nc{\bfy}{\mbf{y}}\nc{\bfz}{\mbf{z}}

\nc{\mcal}[1]{{\mathcal #1}}
\nc{\calA}{\mcal{A}}\nc{\calB}{\mcal{B}}\nc{\calC}{\mcal{C}}\nc{\calD}{\mcal{D}}
\nc{\calE}{\mcal{E}} \nc{\calF}{\mcal{F}}\nc{\calG}{\mcal{G}}\nc{\calH}{\mcal{H}}
\nc{\calI}{\mcal{I}}\nc{\calJ}{\mcal{J}}\nc{\calK}{\mcal{K}}\nc{\calL}{\mcal{L}}
\nc{\calM}{\mcal{M}}\nc{\calN}{\mcal{N}}\nc{\calO}{\mcal{O}}\nc{\calP}{\mcal{P}}
\nc{\calQ}{\mcal{Q}}\nc{\calR}{\mcal{R}}\nc{\calS}{\mcal{S}}\nc{\calT}{\mcal{T}}
\nc{\calU}{\mcal{U}}\nc{\calV}{\mcal{V}}\nc{\calW}{\mcal{W}}\nc{\calX}{\mcal{X}}
\nc{\calY}{\mcal{Y}}\nc{\calZ}{\mcal{Z}}
\nc{\fA}{\frak{A}}\nc{\fB}{\frak{B}}\nc{\fC}{\frak{C}} \nc{\fD}{\frak{D}}
\nc{\fE}{\frak{E}}\nc{\fF}{\frak{F}}\nc{\fG}{\frak{G}}\nc{\fH}{\frak{H}}
\nc{\fI}{\frak{I}}\nc{\fJ}{\frak{J}}\nc{\fK}{\frak{K}}\nc{\fL}{\frak{L}}
\nc{\fM}{\frak{M}}\nc{\fN}{\frak{N}}\nc{\fO}{\frak{O}}\nc{\fP}{\frak{P}}
\nc{\fQ}{\frak{Q}}\nc{\fR}{\frak{R}}\nc{\fS}{\frak{S}}\nc{\fT}{\frak{T}}
\nc{\fU}{\frak{U}}\nc{\fV}{\frak{V}}\nc{\fW}{\frak{W}}\nc{\fX}{\frak{X}}
\nc{\fY}{\frak{Y}}\nc{\fZ}{\frak{Z}}
\nc{\fa}{\frak{a}}\nc{\fb}{\frak{b}}\nc{\fc}{\frak{c}} \nc{\fd}{\frak{d}}
\nc{\fe}{\frak{e}}\nc{\fFf}{\frak{f}}\nc{\fg}{\frak{g}}\nc{\fh}{\frak{h}}
\nc{\fri}{\frak{i}}\nc{\fj}{\frak{j}}\nc{\fk}{\frak{k}}\nc{\fl}{\frak{l}}
\nc{\fm}{\frak{m}}\nc{\fn}{\frak{n}}\nc{\fo}{\frak{o}}\nc{\fp}{\frak{p}}
\nc{\fq}{\frak{q}}\nc{\fr}{\frak{r}}\nc{\fs}{\frak{s}}\nc{\ft}{\frak{t}}
\nc{\fu}{\frak{u}}\nc{\fv}{\frak{v}}\nc{\fw}{\frak{w}}\nc{\fx}{\frak{x}}
\nc{\fy}{\frak{y}}\nc{\fz}{\frak{z}}

\newtheorem{theorem}{Theorem}[section]
\newtheorem{lemma}[theorem]{Lemma}
\newtheorem{corollary}[theorem]{Corollary}
\newtheorem{prop}[theorem]{Proposition}

\newtheorem{setting}[theorem]{Setting}

\theoremstyle{definition}
\newtheorem{definition}[theorem]{Definition}
\newtheorem{example}[theorem]{Example}
\newtheorem{remark}[theorem]{Remark}

\newtheorem{notation}[theorem]{Notation}
\newtheorem{ansatz}[theorem]{Ansatz}

\newtheorem{thm}{Theorem}[section]

\newtheorem{prop-defi}[thm]{Proposition-Definition}
\newtheorem{thm-defi}[thm]{Theorem-Definition}
\newtheorem{lem-defi}[thm]{Lemma-Definition}
\newtheorem{ass}[thm]{Assumption}

 \DeclareMathOperator{\id}{id}
 \DeclareMathOperator{\Sym}{Sym}
 
 \DeclareMathOperator{\GL}{GL}
\DeclareMathOperator{\Hom}{{Hom}}

\DeclareMathOperator{\Hilb}{{Hilb}}

\DeclareMathOperator{\Spec}{{Spec}} \DeclareMathOperator{\tr}{tr}

 \DeclareMathOperator{\End}{End}

\DeclareMathOperator{\Gr}{Gr}

\DeclareMathOperator{\Rep}{Rep}
\DeclareMathOperator{\Res}{Res}

\newcommand{\inj}{\hookrightarrow}

\newcommand{\pt}{\text{pt}}

\newcommand{\C}{\bbC}

\newcommand{\Q}{\bbQ}

\DeclareMathOperator{\Crit}{Crit}

\DeclareMathOperator{\fBun}{\fB un}
\DeclareMathOperator{\bCrit}{\bfC rit}
\DeclareMathOperator{\bMap}{\pmb{\mathfrak{M}}ap}

\DeclareMathOperator{\bSpec}{\bfS pec}

\newcount\cols
{\catcode`,=\active\catcode`|=\active
 \gdef\Young(#1){\hbox{$\vcenter
 {\mathcode`,="8000\mathcode`|="8000
  \def,{\global\advance\cols by 1 &}%\fBun
  \def|{\cr
        \multispan{\the\cols}\hrulefill\cr
        &\global\cols=2 }%
  \offinterlineskip\everycr{}\tabskip=0pt
  \dimen0=\ht\strutbox \advance\dimen0 by \dp\strutbox
  \halign
   {\vrule height \ht\strutbox depth \dp\strutbox##
    &&\hbox to \dimen0{\hss$##$\hss}\vrule\cr
    \noalign{\hrule}&\global\cols=2 #1\crcr
    \multispan{\the\cols}\hrulefill\cr%
   }}$}}}

\newcommand{\gufang}[1]{\textcolor{red}{$[$ Gufang: #1 $]$}}

\newcommand{\yl}[1]{\textcolor{blue}{$[$ Yalong: #1 $]$}}

\DeclareFontFamily{U}{rsfs}{%
\skewchar\font127}
\DeclareFontShape{U}{rsfs}{m}{n}{%
<-6>rsfs5<6-8.5>rsfs7<8.5->rsfs10}{}
\DeclareSymbolFont{rsfs}{U}{rsfs}{m}{n}
\DeclareSymbolFontAlphabet
{\mathrsfs}{rsfs}
\DeclareRobustCommand*\rsfs{%
\@fontswitch\relax\mathrsfs}

\newdimen\argwidth
\def\db[#1\db]{
 \setbox0=\hbox{$#1$}\argwidth=\wd0
 \setbox0=\hbox{$\left[\box0\right]$}
  \advance\argwidth by -\wd0
 \left[\kern.3\argwidth\box0 \kern.3\argwidth\right]}

\newcommand{\Eff}{\mathrm{Eff}}

\newcommand{\cC}{\mathcal{C}}

\newcommand{\oO}{\mathcal{O}}

\newcommand{\dR}{\mathbf{R}}

\newcommand{\Pic}{\mathop{\rm Pic}\nolimits}

\newcommand{\rk}{\mathop{\rm rk}\nolimits}

\newcommand{\cneq}{\mathrel{\raise.095ex\hbox{:}\mkern-4.2mu=}}
\newcommand{\eqcn}{\mathrel{=\mkern-4.5mu\raise.095ex\hbox{:}}}

\newcommand{\DT}{\mathop{\rm DT}\nolimits}
\newcommand{\PT}{\mathop{\rm PT}\nolimits}

\newcommand{\Ker}{\mathop{\rm Ker}\nolimits}

\newcommand{\RHom}{\mathop{\dR\mathrm{Hom}}\nolimits}

\makeatletter
 
  \@addtoreset{equation}{section}
\makeatother

\title[Quasimaps to quivers with potentials]
{Quasimaps to quivers with potentials}

\author{Yalong Cao}
\address{Morningside Center of Mathematics, Institute of Mathematics \& State Key Laboratory of Mathematical Sciences, Academy of Mathematics and Systems Sciences, Chinese Academy of Sciences, 55 Zhongguancun East Road, Beijing, China}
\email{yalongcao@amss.ac.cn}
\author{Gufang Zhao}
\address{School of Mathematics and Statistics, University of Melbourne, Parkville VIC 3010, Australia}
\email{gufangz@unimelb.edu.au}

 \date{\today}

\begin{document}
\maketitle
\begin{abstract}

This paper is concerned with a non-compact  GIT quotient of a vector space, in the presence of an abelian group action and an equivariant regular function (potential)  on the quotient.
We define virtual counts of quasimaps from prestable curves to the critical locus of the potential, and prove a gluing formula in the formalism of cohomological field theories. 

The main examples studied in this paper are when the above setting arises from quivers with potentials, where the above construction 
gives quantum correction to the equivariant Chow homology of the zero locus. 
Following similar ideas 
as in quasimaps to Nakajima quiver varieties studied by the Okounkov school, we analyse vertex functions in
several examples, including Hilbert schemes of points on $\mathbb{C}^3$, moduli spaces of perverse coherent systems 
on the resolved conifold, and a quiver which defines higher $\mathfrak{sl}_2$-spin chains. Bethe equations are calculated in these cases. 

The construction in the present paper is based on the theory of gauged linear sigma models as well as shifted symplectic geometry of Pantev, To\"en, Vaquie and Vezzosi, and 
uses the virtual pullback formalism of symmetric obstruction theory of Park, which arises from 
the recent development of Donaldson-Thomas theory of Calabi-Yau 4-folds.

\iffalse 
Let $W$ be a complex vector space with an action of the product $H=G\times F$ of two linear reductive groups with $F$ being commutative, 
let $\phi$ be a $H$-equivariant regular function on $W$ with $H$-action on $\mathbb{C}$ given by a character $\chi$ of $F$. 
Fixing a character of $G$, by descending $\phi$ to a regular function on the GIT quotient $W/\!\!\,/G$, one can associate the cohomology of 
corresponding sheaves of vanishing cycles which is equivariant under the action of $F_0:=\Ker\chi$ and has a natural product structure. 
We define quantum corrections to the product structure by using virtual invariants for maps from curves to the critical locus. 
\fi
\end{abstract}

\hypersetup{linkcolor=black}
\tableofcontents

\section{Introduction}

\subsection{Motivation}

In this paper, we study ``counting maps" from Riemann surfaces (or algebraic curves) to GIT quotient targets, 
continuing from the theory of gauged linear sigma model (GLSM) \cite{Witt, FJR2, CFGKS, KL, TX, FK} as well as its predecessor, the theory of quasimaps \cite{CiK1, CiKM} and Gromov-Witten theory \cite{KM, RT}. 
More specifically, we are interested in maps whose target is 
the critical locus $\Crit(\phi)$ of a regular function $\phi$ on a GIT quotient $X$ of a vector space.  The precise setup is reviewed below. 
We focus on examples when the GIT quotient $X$ comes from a quiver and $\phi$ comes from a potential of the quiver \cite{Gin,DWZ, King}.

From representation theoretic point of view, moduli spaces of framed representations of
quivers with potentials  have been considered to be analogues of Nakajima quiver varieties \cite{Nak0}. This setting is
 flexible enough to include non-fundamental representations of simply-laced quantum groups \cite{BZ,VV} 
as well as possibly non-simply-laced quantum groups \cite{YZ}, but is also  structured enough to afford  explicit descriptions of the quantum groups  \cite{Ne}. 
From enumerative geometry point of view, Nakajima quiver varieties are local models of moduli of sheaves on Calabi-Yau surfaces. The virtual count of maps from an algebraic curve to a Nakajima quiver variety is an analogue of Donaldson-Thomas theory of a 3-fold \cite{O,OP}. Taking this analogy one step further, quivers with potentials are local models of moduli of sheaves on Calabi-Yau threefolds. The virtual count of maps from an algebraic curve to a quiver with potential is an analogue of Donaldson-Thomas theory of a 4-fold. 
The present paper wishes to take the latter perspective, and uses the recent progress in Donaldson-Thomas theory of Calabi-Yau 4-folds \cite{OT,Park}, which is based on  
the shifted symplectic geometry in the sense of Pantev, To\"en, Vaquie, and Vezzosi \cite{PTVV} and local Darboux theorem of Bouaziz and Grojnowski \cite{BG}, Brav, Bussi and Joyce \cite{BBJ}.

As will be elaborated in \S \ref{subsec:intro_other_works}, the problem of virtual counting in the present paper, although coming from a completely different physical background, fits into similar mathematical framework as the theory of gauged linear sigma model (GLSM). 
The methods used in the literature studying GLSM (e.g., \cite{KL, CFGKS, FK}) involve
difficult constructions of various ambient spaces, which are interesting and beautiful on their own. 
The approach in the present paper, however, takes a different perspective, and is based on derived algebraic geometry \cite{Lur, TV} and shifted symplectic geometry \cite{PTVV}.  
Because of this more intrinsic point of view, we expect the method developed in this paper to be useful to establish cohomological field theory (CohFT) for more general targets which have $(-1)$-shifted symplectic derived structures
%The approach in the present paper, however, takes a more intrinsic  perspective, based on derived algebraic geometry \cite{Lur, TV} and shifted symplectic geometry \cite{PTVV}.  
%Because of this different perspective, we expect the method developed in this paper to be useful to establish cohomological field theory (CohFT) for more general targets which have $(-1)$-shifted symplectic derived structures 
(see \S \ref{sect on to qh of -1} for more details).

We point out that from the point of view of Donaldson-Thomas type theory of 4-folds, it is interesting to consider a non-Calabi-Yau 4-fold with an anticanonical divisor, which typically appears in a family of degeneration of Calabi-Yau 4-folds.
%where Donaldson-Thomas type invariants are to be dealt with relative to the divisor. 
Counting  maps from a curve with marked points to a moduli of sheaves on a Calabi-Yau 3-fold gives an example of (relative) Donaldson-Thomas 4-fold invariants, which had not been previously defined.  Because of this connection, we expect the construction of this paper to be helpful in establishing a gluing formula for such invariants, which 
we wish to elaborate in future investigations, including \cite{CZZ}.

 In the rest of the introduction, we give a more detailed summary of the results, and briefly explain the methods.  

\subsection{Moduli of quasimaps}

Let $W$ be a complex vector space endowed with an action of the product $H=G\times F$ of two complex reductive groups with $F$ being commutative. Let 
$$\phi: W\to \mathbb{C}$$ be a $H$-equivariant regular function on $W$ with $H$-action on $\mathbb{C}$ given by a nontrivial character $\chi: F\to \mathbb{C}^*$. 
Fixing a character $\theta$ of $G$, by abuse of notation, we denote 
$$\phi: X:=W/\!\!/_\theta G\to \mathbb{C}$$
to be the descent regular function on the smooth GIT quotient, which is invariant under the action of Calabi-Yau subtorus $F_0:=(\Ker\chi)$.

Let $R:\bbC^*\to F$ be a group morphism, called $R$-charge (Definition~\ref{defi of R-charge}) with  
$R_\chi:=\chi \circ R$. 
For simplicity of exposition, we assume $\Ker R_\chi=\{1\}$ to avoid dealing with the general theory of stable maps from orbicurves by Abramovich and Vistoli \cite{AbV}.
%Let $\bbX(G)$ be the character group of $G$ and fix $\beta\in \Hom_{\bbZ}(\bbX(G), \bbZ)$. 

A genus $g$, $n$-pointed \textit{$R$-twisted quasimap} to $X$ is 
a quadruple $\big((C,p_1,\ldots,p_n),P,u,\varkappa\big)$, where 
$(C,p_1,\ldots,p_n)$ is a prestable genus $g$, $n$-pointed curve, $P$ is a principal $(G\times \mathbb{C}^*)$-bundle on $C$ with an isomorphism 
$\varkappa:P/G\times_{ \bbC^*}R_\chi\cong\omega_{\mathrm{log}}$, and $u: P\times_{(G\times  \bbC^*)}(G\times R)\to W$ is a $(G\times F)$-equivariant map. 
Here $$\omega_{\mathrm{log}}:=\omega_{C}\left(\sum_{i=1}^n p_i\right)$$ is the \textit{log-canonical} bundle of $C$. 
The \textit{class} of such a quadruple is an integer-valued function 
$$\beta\in \Hom_{\mathbb{Z}}(\bbX(G), \bbZ), \quad \beta(\xi):=\deg_C(P_G\times_G \bbC_{\xi}), $$
on the character group $\bbX(G)$, given by the degree of the principal $G$-bundle $P_G:=P/\mathbb{C}^*$.

The above quadruple $\big((C,p_1,\ldots,p_n),P,u,\varkappa\big)$ is called \textit{stable}\footnote{One can consider more general $\epsilon$-stability as in \cite[\S 7.1]{CiKM}, where in the present paper we restrict to $0^+$-stability. Most results in this paper extend directly to the general stability.} if 
\begin{enumerate}
\item the image of $u$ lies in the (open) stable locus $W^s$ on the entire $C$ but  finitely many (possibly none) points, the so-called {\it base points}; 
\item base points are away from the markings and nodes;
\item $\omega_{\mathrm{log}}\otimes L_\theta^\epsilon$ is ample for every rational number $\epsilon>0$, where
$L_\theta =P/\bbC^*\times_G \mathbb{C}_\theta$.  
\end{enumerate}
We denote $QM_{g,n}^{R_{\chi}=\omega_{\mathrm{log}}}(X,\beta)$
to be the moduli stack of genus $g$, $n$-pointed stable $R$-twisted quasimaps to $X$ of class $\beta$. This is a separated Deligne-Mumford stack of finite type (Theorem~\ref{prop:geom_properties}). 

Via the embedding $\Crit(\phi)\hookrightarrow X$, we have a closed substack 
\begin{equation*}QM_{g,n}^{R_{\chi}=\omega_{\mathrm{log}}}(\Crit(\phi),\beta)\hookrightarrow QM_{g,n}^{R_{\chi}=\omega_{\mathrm{log}}}(X,\beta) \end{equation*}
of corresponding quasimaps to $\Crit(\phi)$. 
This moduli stack is the main focus of the present paper. When the torus fixed locus $\Crit(\phi)^{F_0}$ is proper, $QM_{g,n}^{R_{\chi}=\omega_{\mathrm{log}}}(\Crit(\phi),\beta)^{F_0}$ is also proper, which follows directly by combining works of Fan, Jarvis, Ruan \cite{FJR2} and Kim \cite{Kim} (ref.~Theorem~\ref{thm on properness}).

\subsection{Virtual structures}
By forgetting the map $u$ in the quasimap data, we obtain a map
\begin{equation}\label{intro map1}QM_{g,n}^{R_{\chi}=\omega_{\mathrm{log}}}(\Crit(\phi),\beta)\to \fBun_{H_R,g,n}^{R_{\chi}=\omega_{\mathrm{log}}}   \end{equation}
to the smooth Artin stack $\fBun_{H_R,g,n}^{R_{\chi}=\omega_{\mathrm{log}}}$ of principal $(H_R:=G\times \mathbb{C}^*)$-bundles $P$ on genus $g$, $n$-pointed prestable curves $C$ together with an isomorphism $\varkappa: P\times_{H_R}R_\chi\cong \omega_{\mathrm{log}}$.

We also have the product of evaluation maps at the marked points 
\begin{equation}\label{intro map2}ev^n:=ev_1\times\cdots\times ev_n: QM_{g,n}^{R_{\chi}=\omega_{\mathrm{log}}}(\Crit(\phi),\beta)\to \Crit(\phi)^n \hookrightarrow X^n. \end{equation}
Combining maps \eqref{intro map1}, \eqref{intro map2}, we obtain a map (Eqn.~\eqref{equ on f qm}):
$$f: QM_{g,n}^{R_{\chi}=\omega_{\mathrm{log}}}(\Crit(\phi),\beta)\to \fBun_{H_R,g,n}^{R_{\chi}=\omega_{\mathrm{log}}}\times_{[\pt/G]^n}X^n. $$
Denote $Z(\boxplus^{n}\phi)$ to be the zero locus of the function:
$$\boxplus^{n}\phi: X^n\to \mathbb{C}, \quad (\boxplus^{n}\phi)(x_1,\ldots,x_n):=\sum_{i=1}^n \phi(x_i). $$
We construct a pullback morphism using the map $f$.  
\begin{theorem}\emph{(Theorem \ref{prop:symm_ob}, Definition \ref{defi of qm vir class})}\label{intro prop:symm_ob}
The map $f$ has a canonical  symmetric obstruction theory in the sense of Park \cite{Park}, which is isotropic after base change via $Z(\boxplus^{n}\phi) \hookrightarrow X^n$.

In particular,  there is a square root virtual pullback
$$\sqrt{f^!}: A^{F_0}_*\left(\fBun_{H_R,g,n}^{R_{\chi}=\omega_{\mathrm{log}}}\times_{[\emph{\pt}/G]^n}Z(\boxplus^{n}\phi)\right)\to A^{F_0}_*\left(QM_{g,n}^{R_{\chi}=\omega_{\mathrm{log}}}(\Crit(\phi),\beta)\right). $$
\end{theorem}
%(see Definitions \ref{def on sym ob} and \ref{def on iso sym ob} for the precise definitions). 
We recall basics of isotropic symmetric obstruction theory and square root virtual pullback 
in \S \ref{sect on vir pullback} and apply it to our setting in \S \ref{subsec:obst}. 
We prove several properties of this pullback in \S \ref{sect on glu}, which 
arises from the context of CohFT axioms. 
The key idea behind this theorem is to consider the \textit{derived mapping stack} as a derived enhancement of the moduli stack $QM_{g,n}^{R_{\chi}=\omega_{\mathrm{log}}}(\Crit(\phi),\beta)$. 
Then the restriction of the (relative) derived cotangent complex to the classical truncation gives the desired obstruction theory. 
To check the isotropic condition, it is enough to work on the base change of $f$ under any chart 
$\Spec K\to \fBun_{H_R,g,n}^{R_{\chi}=\omega_{\mathrm{log}}}\times_{[\pt/G]^n}X^n$. On the derived enhancement of any such base change, we construct a (canonical) \textit{$(-2)$-shifted symplectic structure} in the sense of Pantev, To\"en, Vaquie, and Vezzosi \cite{PTVV} (Theorems \ref{thm:symp_no_mark}, \ref{thm:sympl_marked}), and then use a (relative) local Darboux theorem \cite{Park2} of Bouaziz and Grojnowski \cite{BG}, Brav, Bussi and Joyce \cite{BBJ} to conclude the isotropic property after the specified base-change (Proposition~\ref{lem:HN_HP}). 

The construction of shifted symplectic structures is explained in \S \ref{sect on sss I} and \S \ref{subsec:marked}, where we give a Alexandrov-Kontsevich-Schwarz-Zaboronsky (AKSZ)-type construction for twisted mapping stacks from domains of arbitrary dimension (Theorem \ref{thm:symp_no_mark}) as well as twisted mapping stacks from curves with marked points (Theorem \ref{thm:sympl_marked}). These are two variants of \cite[Thm.~2.5]{PTVV} which may be of independent interest. 

\subsection{Quasimap invariants and gluing formula}
When $2g-2+n>0$, we consider the composition 
\[\fBun_{ H_R,g,n}^{R_\chi=\omega_{\mathrm{log}}}   \to \mathfrak{M}_{g,n}   \xrightarrow{st} \overline{M}_{g,n} \]
of the forgetful map and the stablization map to the Deligne-Mumford moduli stack $\overline{M}_{g,n}$  of stable curves, which is flat. Composing it with the smooth map 
$$\fBun_{ H_R,g,n}^{R_\chi=\omega_{\mathrm{log}}}\times_{[\pt/G]^n} Z(\boxplus^{n}\phi)\to \fBun_{ H_R,g,n}^{R_\chi=\omega_{\mathrm{log}}}\times Z(\boxplus^{n}\phi), $$
we obtain a flat map
\begin{equation*}\nu: \fBun_{ H_R,g,n}^{R_\chi=\omega_{\mathrm{log}}}\times_{[\pt/G]^n}Z(\boxplus^{n}\phi)\to \overline{M}_{g,n}\times Z(\boxplus^{n}\phi), \end{equation*}
and its pullback $\nu^*$. We then define a group homomorphism (Definition~\ref{defi of Phi map}):
\begin{equation}\label{intro cohft map}\Phi_{g,n,\beta}:=p_*\circ \sqrt{f^!}\circ  \nu^*\circ \boxtimes: 
A_{*}(\overline{M}_{g,n})\otimes A_{*}^{F_0}(Z(\boxplus^{n}\phi))\to  A_*^{F_0}(\pt)_{loc},   \end{equation}
where $p_{*}:A_{*}^{F_0}(QM^{R_\chi=\omega_{\mathrm{log}}}_{g,n}(\Crit(\phi),\beta))\to  A_*^{F_0}(\pt)_{loc}$
is the equivariantly localized pushforward map for the projection $p$, defined using Eqn.~\eqref{pf for nonproper} and 
the equivariant properness (Theorem \ref{thm on properness}). 

%Using proper pushforward via $$\Crit(\phi)^n\hookrightarrow Z(\boxplus^{n}\phi), $$ 
%we then obtain the upper horizontal map in diagram \ref{intro comm diag}. 
Further plug-in the fundamental class $[\overline{M}_{g,n}]$, we can define \textit{Gromov-Witten type}  invariants for $\Crit(\phi)$ (Definition~\ref{def of QC}).

The map \eqref{intro cohft map} has several nice properties. We state the following \textit{gluing formula} in the formalism of 
CohFT (e.g.~\cite{KM, RT, P})\footnote{The more accurate terminology in the present setting is Chow field theory (ChowFT) following \cite[\S 1.4]{P}.}. 
Fix $n=n_1+n_2$, $g=g_1+g_2$, there is a gluing morphism
$$\iota: \overline{M}_{g_1,n_1+1}  \times \overline{M}_{g_2,n_2+1} \to \overline{M}_{g,n}. $$ 
Note also that any class in $A_*^{F_0}\left(Z(\boxplus^{n_1}\phi)\times Z(\boxplus^{n_2}\phi)\right)$ can be considered 
as an element in $A_{*}^{F_0}\left(Z(\boxplus^{n}\phi)\right)$ with $n=n_1+n_2$ by the pushforward of the obvious inclusion.
\begin{theorem}\emph{(Theorem~\ref{thm:CohFT_glue})}\label{intro thm:CohFT_glue}
Let $\gamma \in \mathrm{Im}\left(A_*^{F_0}\left(Z(\boxplus^{n_1}\phi)\times Z(\boxplus^{n_2}\phi)\right)\to A^{F_0}_*\left(Z(\boxplus^{n}\phi)\right)\right)$ be in the image, 
$\alpha\in A_*(\overline{M}_{g_1,n_1+1} \times \overline{M}_{g_2,n_2+1})$ and $\eta\in A_{*}^{F_0}\left(Z(\boxplus^{2}\phi)\right)$ be the anti-diagonal class \eqref{diagon class}. Then  
\begin{equation}\label{glue formula new form} 
\Phi_{g,n,\beta}((\iota_*\alpha)\boxtimes\gamma)=\sum_{\beta_1+\beta_2=\beta}\left(\Phi_{g_1,n_1+1,\beta_1}\otimes \Phi_{g_2,n_2+1,\beta_2} \right)(\alpha\boxtimes (\gamma\boxtimes\eta)), \end{equation}
where $\Phi_{g_1,n_1+1,\beta_1}\otimes \Phi_{g_2,n_2+1,\beta_2}$ is defined similarly as \eqref{intro cohft map} in ~Eqn.~\eqref{equ on prod phi}. 
\end{theorem}
The proof of the above gluing formula makes use of the functoriality of square root virtual pullbacks which is explained in \S \ref{subsec:CohFT_glue}. 
In \S \ref{sect on WDVV}, we use this formula to prove a \textit{Witten-Dijkgraaf-Verlinde-Verlinde} (WDVV)-type equation (Theorem~\ref{thm:WDVV_invariants_form}).

\subsection{Quantum critical cohomology}\label{sect on qch intr}
Consider the perverse sheaf $\varphi_{{\phi}}$ of vanishing cycles of ${\phi}$ and $F_0$-equivariant cohomology of $X$ valued in $\varphi_{{\phi}}$ (ref.~\S \ref{sect on bm ho}, \S \ref{sect on equi}):
\begin{equation}\label{intro equ on crh}H_{F_0}(X,\varphi_{{\phi}}), \end{equation}
also called \textit{critical cohomology}. 
This cohomology emerges naturally from several different sources:~(1) it is the fundamental object in the theory of perverse sheaves and 
singularities \cite{BBD, KaSc}; (2) when $F_0=\{1\}$, it is the state space\footnote{Here for simplicity, we restrict ourselves to the case when the GIT quotient 
is a variety so one does not need inertia stack construction in Chen-Ruan's theory of orbifold cohomology \cite{CR, AGV}.}\,of the so-called 
gauged linear sigma models (GLSM); (3) it is related to the cohomological Donaldson-Thomas theory of 
Calabi-Yau 3-categories \cite{KS, BBBJ, KL2}, and  Kontsevich-Soibelman critical cohomology \cite{KS} of  quivers with potentials. 

Viewing \eqref{intro equ on crh} as the state space of some physical system, it is interesting and important to introduce \textit{quantum corrections} to it. 
%which is one main purpose of this paper. 
To be precise, we expect to have a commutative diagram:  
\begin{equation}\label{intro comm diag}\xymatrix{
A_*(\overline{M}_{g,n})\otimes A_*^{F_0}(Z(\boxplus^n\phi)) \ar[r]^{\quad \quad \quad \,\,\, \,\, \Phi_{g,n,\beta}^{\mathrm{alg}}} \ar[d]_{(\id\times can)\circ cl}& A_*^{F_0}(\pt)_{loc} \ar[d]_{\cong}^{cl} \\
H^{BM}(\overline{M}_{g,n})\otimes H_{F_0}(X,\varphi_{{\phi}})^{\otimes n} \ar[r]^{\quad \quad\quad\,\,\,\, \Phi_{g,n,\beta}^{\mathrm{top}}} & H^{BM}_{F_0}(\pt)_{loc}.}\end{equation}
Here $cl$ is the cycle map from Chow homology to Borel-Moore (BM) homology, $can$ is the canonical map (ref.~Eqn.~\eqref{map can})
from BM homology to the critical cohomology and we have used the Thom-Sebastiani isomorphism \eqref{TS iso}:
$$H_{F_0}(X,\varphi_{{\boxplus^n\phi}}) \cong H_{F_0}(X,\varphi_{{\phi}})^{\otimes n}, $$
and its compatibility \eqref{compatible of TS and Milnor} 
with the canonical map, 
$(-)_{loc}$ denotes the localization with respect to 
the field $A_*^{F_0}(\pt)_{loc}$ of fractions of $A_*^{F_0}(\pt)$. Horizontal maps in the above diagram encode the information of quantum corrections and are expected to satisfy properties in the so-called cohomological field theory (CohFT). 
The map $\Phi_{g,n,\beta}^{\mathrm{alg}}$
is obtained from \eqref{intro cohft map} (which is based on Theorem~\ref{intro prop:symm_ob}). 
%\gufang{On the one hand, $Z(\boxtimes^n\phi)$  is the universal locus in $X^n$ where the symmetric obstruction theory of Theorem~\ref{intro prop:symm_ob} is isotropic and the square root virtual pullback of an algebraic cycle is well-defined. On the other hand, composing the cycle map $cl$  and the canonical map $can$, algebraic cycles on  $Z(\boxtimes^n\phi)$ give rise to states in the space  $H_{F_0}(X,\varphi_{{\phi}})^{\otimes n}$.  Hence, $\Phi_{g,n,\beta}^{\mathrm{alg}}$ encodes the effect of $\Phi_{g,n,\beta}^{\mathrm{top}}$ on algebraic cycles.}
We expect that the map $\Phi_{g,n,\beta}^{\mathrm{top}}$ can be constructed using an $F_0$-equivariant version of the work of Favero-Kim \cite{FK} which 
is based on the method of fundamental factorizations due to Polishchuk-Vaintrob \cite{PV}. 
We notice that 
$Z(\boxtimes^n\phi)$  is the universal locus in $X^n$ where the symmetric obstruction theory of Theorem~\ref{intro prop:symm_ob} is isotropic and the square root virtual pullback of an algebraic cycle is well-defined.
By composing the cycle map $cl$  and the canonical map $can$, algebraic cycles on  $Z(\boxtimes^n\phi)$ give rise to states in the space  $H_{F_0}(X,\varphi_{{\phi}})^{\otimes n}$.  Hence, $\Phi_{g,n,\beta}^{\mathrm{alg}}$ encodes the effect of $\Phi_{g,n,\beta}^{\mathrm{top}}$ on \textit{algebraic cycles}.
Therefore for the purpose of calculations, the method developed in this paper is expected to be useful in finding values of algebraic classes under the provisional map
$\Phi_{g,n,\beta}^{\mathrm{top}}$.
\begin{remark}
There is a version of Theorem \ref{intro prop:symm_ob} where (equivariant) Chow groups are replaced by (equivariant) Grothendieck groups $K_0(-)$ of coherent sheaves, with similar proof. 

Moreover, we expect a $K$-theoretic analogue of diagram~\eqref{intro comm diag}, where the critical cohomology $H_{F_0}(X,\varphi_{{\phi}})^{\otimes n}$ (resp.~BM homology $H^{BM}(\overline{M}_{g,n})$) is replaced by the Grothendieck group of the matrix factorization category $MF(X^n,\boxplus^n\phi)$ (resp.~the Grothendieck group $K_0(\overline{M}_{g,n})$).
The canonical map $$K_0(Z(\boxplus^n\phi))\to K_0(MF(X^n,\boxplus^n\phi))$$ is always
\textit{surjective}, as noted in \cite[Rmk.~1.6]{CZZ}.
Therefore, the $K$-theoretic version of  $\Phi_{g,n,\beta}^{\mathrm{alg}}$ contains all information about the $K$-theoretic version of the provisional map $\Phi_{g,n,\beta}^{\mathrm{top}}$. We refer to \cite{CTZ} for a recent work in this direction.  
\end{remark}

%\footnote{For example, the Yangian action on critical cohomology for quivers with potentials \cite {RSYZ1, RSYZ2} are constructed via this map. \yl{revise this part}\gufang{tbc}} 

In \S \ref{subsec:crit2}, we discuss two cases where we can define the bottom map 
in diagram \eqref{intro comm diag} using methods of this paper. The first case (see Settings \ref{ass:compact_type}) is when 
$$\phi|_{X^{F_0}}=0, $$
which is motivated by the \textit{compact type} condition in \cite[Def.~4.1.4]{FJR2}. The second case 
is motivated by the \textit{geometric phase} in \cite[Def.~1.4.5]{CFGKS} the details of which are  in Settings \ref{ass:geo_phase}. 
The common features of these two cases are that (i) the (localized) critical cohomology is isomorphic 
to the (localized) Borel-Moore homology of some associated space (Proposition~\ref{prop on iso of crit and bm1}, Eqn.~\eqref{equ on iso of dim re}), and 
(ii) the class $\eta$ of the anti-diagonal $\bar{\Delta}: X\to Z(\boxplus^{2}\phi)$ \eqref{diag on Zboxphi} can be written as 
$$\eta=\eta_i\boxtimes\eta^i\in H_{F_0}(X,\varphi_\phi)_{loc}^{\otimes 2}$$
in the localized critical cohomology (ref.~Lemma~\ref{lem:Casmir}). 
One can then define quasimap classes (Definition~\ref{def of QC2}) and in particular a \textit{quantum product} structure on the critical cohomology:
\begin{equation}\label{eqn:quan_prod_intr}*:H_{F_0}(X,\varphi_\phi)_{loc}\otimes H_{F_0}(X,\varphi_\phi)_{loc} \to H_{F_0}(X,\varphi_\phi)_{loc}[\![z]\!]. \end{equation}
The WDVV type equation implies the associativity of this product.
\begin{theorem}\emph{(Corollary~\ref{cor on ass of prod})}\label{intro thm on wdvv} 
The operation \eqref{eqn:quan_prod_intr} defines an associative multiplication. 
\end{theorem}

\subsection{A variant of quasimaps, vertex functions, and Bethe equations}
There is a variant of quasimaps where one allows \textit{relative points} on the domain curve $C$ and \textit{parametrizes} a component $C_0$ of $C$ by a fixed curve $D$ \cite{CiKM, CiK1}. 
We concentrate on the case when $C$ has genus $0$ (see Definition~\ref{def of para qm 2}).
There is a similar construction of the map \eqref{intro cohft map} in this setting 
%In the setting of the present paper, 
(see~\S \ref{sect on qm with rel pts},~\S \ref{sect on para qm class}), which possesses nice  properties including \textit{gluing and degeneration formulae} (see~\S \ref{sect on dege for}).
Note in this case, we do \textit{not} need $\Ker(R_\chi)=1$ or using twisted stable maps of \cite{AbV} as the data of principal $F$-bundle is fixed, and hence so is the $r$-Spin structures occurring in the theory of balanced twisted maps.

Following works of the Okounkov school \cite{O, PSZ, KZ, KPSZ} on quasimaps to Nakajima quiver varieties, we introduce in the setting of the present paper (i.e.~on critical loci)
several invariants defined using such quasimaps with parametrized component (see \S \ref{subsec:QDE}). Among them, 
we have operators $\bfM^\alpha(z)$, $J(z)$ (Definitions~\ref{defi of qcm}, \ref{defi of vertex fun}\,(4)) which satisfy a \textit{quantum differential equation} (Theorem~\ref{thm:fund_sol}),
and the \textit{vertex function} (Definition~\ref{defi of vertex fun}\,(1))  defined using certain twisted quasimaps from $\mathbb{P}^1$ to 
the critical locus (see~Remark~\ref{rmk on n=0}).
%Computations of the vertex function is much earlier compared to other invariants.  
In \S \ref{sect on qc}, we explain in the two cases mentioned in \S \ref{sect on qch intr}, how $\bfM^\alpha(z)$ (resp.~$J(z)$) can be viewed as a \textit{quantum connection} (resp.~a gauge transformation for the quantum connection).
 
When the critical locus $\Crit(\phi)$ is the Hilbert scheme $\Hilb^n(\mathbb{C}^3)$ of $n$-points on $\mathbb{C}^3$ with Calabi-Yau torus action $F_0\cong (\mathbb{C}^*)^2$ (see Example~\ref{key exam}),
we show (in Proposition~\ref{iso of qm and pt}) that the moduli space of 
 twisted quasimaps with a parametrized component $\mathbb{P}^1$ is isomorphic to the moduli space of  \textit{Pandharipande-Thomas stable pairs} on the total space 
\begin{equation}\label{equ on cy4 intro}\mathrm{Tot}_{\bbP^1}(\calL_1\oplus\calL_2\oplus \calL_3 ) \end{equation} 
of the direct sum of three lines bundles with Calabi-Yau condition $\calL_1\otimes\calL_2\otimes \calL_3\cong \omega_{\bbP^1}$. Moreover quasimap invariants of the formal agrees with stable pair invariants of the Calabi-Yau 4-fold \eqref{equ on cy4 intro} (ref.~Lemma~\ref{lem on vir tange of qm p1} and \cite[\S 5.2]{CMT2}). 

In Propositions~\ref{prop:hypergeom},~\ref{prop:MB_int}, we compute the vertex function of $\Hilb^n(\mathbb{C}^3)$ explicitly, 
in terms of both a power series and a certain contour integral. We also calculate the saddle point equation of the integrant. 
\begin{theorem}\emph{(Proposition~\ref{prop:saddle})}
The saddle point equation of the integrant is given by 
\[\overline{z}=\frac{1}{s_i}\prod_{s=1}^3\prod_{j\neq i}
\frac{s_i-s_j-\hbar_s}{s_i-s_j+\hbar_s}, \quad i=1,\ldots,n. \]
Here $\overline{z} $ is a normalization of the K\"ahler variable, $\{\hbar_s\}_{s=1}^3$ are equivariant variables of the torus $F_0$ with $\sum_s \hbar_s=0$. 
%and $m$ is a formal variable.
\end{theorem} 
Recall that the $F_0$-equivariant critical cohomology of $\coprod_n\Hilb^n(\mathbb{C}^3)$ is a representation of the $(-1)$-shifted affine Yangian $Y_{-1}(\widehat{\fg\fl_1})$ as constructed by \cite{RSYZ2} (see also \cite{RSYZ1}). 
Motivated by \cite{AO, PSZ}, which is based on the work of Nekrasov and Shatashvili \cite{NS}, we expect the above 
saddle point equation coincides with the \textit{Bethe equation} of $Y_{-1}(\widehat{\fg\fl_1})$, which has not been studied by representation theorists 
(see~\S \ref{subsec:Bethe} for more detailed discussions and \cite{GLY} for a related study from physical point of view). 
The saddle point equation should also describe \textit{eigenvalues of quantum multiplications} by quantum line bundles (ref.~\S \ref{sect on ansztz}).

Besides $\Hilb^n(\mathbb{C}^3)$, we also calculate vertex functions and corresponding saddle point equations for some other quivers with potentials, including 
the one that defines \textit{perverse coherent systems} on the resolved conifold $\calO_{\bbP^1}(-1,-1)$ (ref.~\S \ref{sect vertex func on perv}), and 
those that define \textit{higher} $\fs\fl_2$-\textit{spin chains} (ref.~\S \ref{sect on more bethe equ}).

\subsection{Connections to other works}\label{subsec:intro_other_works}
The study of quasimap invariants of Nakajima quiver varieties has obviously lead to exciting progress in many fields in mathematics, including 3d mirror symmetry, symplectic duality, quantum $q$-geometric Langlands, etc. We are not positioned to survey the development here. On the other hand, from the perspective of representation theory of quantum groups, it has become increasingly clear that quivers with potentials provide geometric realizations of equally interesting representations
\cite{BZ,RSYZ2, VV, YZ}. 
As mentioned above, one motivation of the present paper is to set the scene for the investigation of  quasimap invariants of quivers with potentials, in search for analogue results or differences to the counterparts in Nakajima quiver varieties. In a forthcoming investigation, we plan to show a dimensional reduction of quasimap invariants of a quiver with potential to that of a Nakajima quiver variety 
(see \S \ref{sec:dim_red}).

The mathematical setup of the moduli spaces studied in the present paper is similar to that studied in the GLSM setting \cite{FJR2, KL, CFGKS, FK}. In particular, when $F=\bbC^*$ (and $\langle J\rangle$ in~loc.~cit.~is trivial), the moduli stack $QM_{g,n}^{R_{\chi}=\omega_{\mathrm{log}}}(\Crit(\phi),\beta)$ agrees with the one considered in \cite{FJR2, KL}. 
%\yl{one difference is that we do not require $G$ to be subgroup.}\gufang{This is automatic if G has a free orbit.}
Therefore, we expect the up-coming work of Kiem and Park \cite{KP2} to help with the comparison of the invariants.  
It is also interesting to relate the construction here to works on GLSM by Ciocan-Fontanine, Favero, Gu\'er\'e, Kim, Shoemaker \cite{CFGKS}, and Favero and Kim \cite{FK} (in view of diagram \eqref{intro comm diag}), which follow and extend the matrix factorization approach of Polishchuk and Vaintrob \cite{PV}.

Nevertheless, the motivation of the present paper differs from that of GLSM. Physically, GLSM studies  Laudau-Ginzburg potentials, and hence
mainly focuses on (smooth) compact critical loci.  Examples include quintic 3-folds and corresponding Landau-Ginzburg phase, with the aim of studying enumerative geometry and mirror symmetry. The present paper is motivated by the study of $D$-brane effective potentials, the critical loci of which are typically singular and non-compact (though $F_0$-fixed locus is compact,~e.g.~Hilbert schemes of points on $\mathbb{C}^3$).
The latter fits into the framework of cohomological Hall algebras  and representations of quantum groups. 

We also remark that in the present paper for simplicity of the exposition, we imposed several simplifying albeit unnecessary conditions (i.e.~$\Ker(R_\chi)=1$ and $G$ action on $W^{s}$ is free), which can be removed via replacing prestable curves by more general orbicurves in the sense of Abramovich and Vistoli \cite{AGV, AbV},  as did in the works of Cheong, Ciocan-Fontanine, Kim \cite{CCK} in quasimap theory and Fan, Jarvis, Ruan \cite{FJR1, FJR2} in FJRW and GLSM setting. 
%We hope to treat the general case in a future work.  

\subsection*{Acknowledgments}
This work benefits from discussions and communications with many people, including Roman Bezrukavnikov, 
Sasha Garbali, Hiroshi Iritani, Johanna Knapp, Yinbang Lin, Hiraku Nakajima, Paul Norbury,  Andrei Okounkov, Hyeonjun Park, Feng Qu,  Yukinobu Toda, Yaping Yang, Zijun Zhou, Paul Zinn-Justin.
We warmly thank Hyeonjun Park for several very helpful discussions on virtual pullbacks and shifted symplectic structures. 
We are grateful to Taro Kimura, Tasuki Kinjo, Yongbin Ruan, Pavel Safronov,  and Yan Soibelman for insightful feedback on a preliminary draft of this paper. 
We thank the anonymous referee for carefully reading an earlier version of the paper and providing a list of errors and helpful comments.
When preparing the present paper, we sadly learnt the passing of Professor Bumsig Kim, whose works have significantly influenced the present paper. 
We would like to express our admiration to his work and our regret not being able to know the man in person.

The work of Y.~C.~was partially supported by RIKEN Interdisciplinary Theoretical and Mathematical Sciences
Program (iTHEMS), JSPS KAKENHI Grant Number JP19K23397 and Royal Society Newton International Fellowships Alumni 2021 and 2022. 
G.~Z.~is partially supported by the Australian Research Council via DE190101222 and DP210103081.

\subsection*{Statements and Declarations}
We have no conflicts of interest to disclose.

\section{Moduli stacks of quasimaps}\label{sec:CiKM}
The theory of quasimaps began as a new way to compactify 
%is originally used to define a (virtually smooth) compactification for 
the mapping space of smooth curves to GIT quotients, which is different from the  Kontsevich moduli spaces of stable maps (see, e.g., \cite{MOP, MM, Toda, CiK1, CiK2} and references therein). A standard reference for its foundation is \cite{CiKM}. 
%There are huge number of papers 
In this section, we recall some basic notions of quasimaps to a GIT quotient and then apply them to the special case we are interested in.

\subsection{Set-up}

We fix the setting of this paper. 
\begin{setting}\label{setting of glsm}

Let $W$ be a finite dimensional complex vector space, $G$ be a complex reductive algebraic group (referred as \textit{gauge group}) 
and $F$ be a complex torus (referred as \textit{flavor group}). Assume there is a 
group homomorphism 
$H:=G\times F\to  GL(W)$. 

Let $\theta\in\bbX(G)$ be a  character of $G$ which defines a stability condition on $W$ \cite[\S 2]{King} so that the $\theta$-stable locus coincides with the $\theta$-semi-stable locus:
$$W^{ss}=W^s\neq\emptyset, $$ 
on which $G$ acts freely. In particular, the GIT quotient $$W/\!\!/ G:=W^s/G$$ 
is a smooth scheme and coincides with the quotient stack $[W^s/G]$.
%Let $\chi:H\to \bbC^*$ be a character of $H$ coming from a character $\chi:F\to \bbC^*$ of $F$, i.e., $\chi$ does not depend on $G$. 

Let $\chi:F\to \bbC^*$ be a nontrivial character and we extend it naturally to a character 
$$\chi:H\to \bbC^*,$$  
which does not depend on $G$.
Define the Calabi-Yau torus $F_0:=\Ker(\chi)$ as the kernel of $\chi:F\to\bbC^*$. 

Let $$\phi\nequiv 0: W\to\bbC$$ 
be an $H$-equivariant function with the action of $H$ on $\bbC$ given by the character $\chi$ and  
$$\Crit(\phi)\subseteq W$$ be the critical locus of $\phi$, with an induced $H$-action. The stable locus is 
$$\Crit(\phi)^s=\Crit(\phi)\cap W^s$$ 
and 
$$\Crit(\phi)/\!\!/G:=\Crit(\phi)^s/G $$
is the GIT quotient, with an induced $F$-action. 

%Without loss of generality, 
As closed subschemes of $W$, we assume there is a closed embedding $\Crit(\phi)\hookrightarrow Z(\phi)$, where $Z(\phi)$ is the zero locus of $\phi$. 
\end{setting}
\begin{remark}\label{rmk on crit emb to zero}
We remark that the above assumption on closed embedding $\Crit(\phi)\hookrightarrow Z(\phi)$ is a simplifying albeit unnecessary condition. 

In general, $\phi$ is locally constant on the reduced scheme $\Crit(\phi)^{\mathrm{red}}$ of $\Crit(\phi)$. 
Without loss of generality, we may assume all critical values are zero,~i.e.~$\Crit(\phi)^{\mathrm{red}}\subseteq Z(\phi)$ as sets. 
By Hilbert's Nullstellensatz, for some $r>0$, we have an embedding
$$\Crit(\phi) \hookrightarrow  Z(\phi^r)$$
as closed subschemes of $W$. For all examples considered in this paper (see \S \ref{sect vertex func on hilb}, \S \ref{sect on more vertx}), it is easy to check that we can take $r=1$.
When $r>1$, the results of this paper remain hold. We will mention the modification of argument in corresponding sections. 

\end{remark}
Our main working example is the following: 
\begin{example}\label{key exam}
Let $V=\bbC^n$ with $G=\GL_n$ acting in the natural way. 
Let 
$$W=V\oplus \End(V)^{\oplus 3}$$ 
with the induced $G$-action and an additional action of $F=(\bbC^*)^3$ by scaling of the three endomorphisms. 
Let 
$$\phi:W\to \bbC, \quad (i,b_1,b_2,b_3)\mapsto \tr (b_1[b_2,b_3]), $$ 
which is invariant under $G$ action and equivariant under $F$ action ($F$ acts on the target with weight $\chi=(1,1,1)$). 
Let 
$$\theta: G\to \mathbb{C}^*, \quad g\mapsto \det(g)$$ 
be the character of $G$, used as the GIT stability condition. Then the stable locus $W^s$ is the 
open subset of $W$ where the linear span of all vectors obtained by repeated applications of the endomorphisms
to the chosen vector $v\in V$ is the whole $V$.

The function $\phi$ descends to 
$$\phi: W/\!\!/ G=W^s/G\to \mathbb{C}$$ 
such that the critical locus satisfies (e.g.~\cite[pp.~131, Prop.~3.1]{BBS})
$$\Crit(\phi)\cong \Hilb^n(\mathbb{C}^3). $$
The torus $F$ descends to an action on $\Crit(\phi)$. By \cite[Lem.~4.1]{BF2}, $\Hilb^n(\mathbb{C}^3)^F=\Hilb^n(\mathbb{C}^3)^{F_0}$ are finite number of reduced points given by plane partitions/3d Young diagrams of size $n$.
\end{example}

More generally, we consider examples given by quivers with potentials \cite{Gin,DWZ, King}. 
\begin{example}\label{exam of quiv with pot}
A \textit{quiver} $Q = (I, H)$ is a directed graph with $I$ being the set of vertices and $H$ the
set of arrows. For a dimension vector $v=(v_i)_{i\in I}\in \bbN^I=\bbZ^I_{\geqslant 0}$, let 
$$W=\Rep(Q, v):=\bigoplus_{(i\to j)\in H}\Hom(\bbC^{v_i},\bbC^{v_j})$$ be the affine space parameterizing
representations of $Q$. 
%on the $I$-graded complex vector space whose degree $i\in I$-piece is $\bbC^{v_i}$.
A \textit{potential} is a linear combination of cycles in $Q$. Taking the trace of the potential defines a regular function $\phi:W\to \bbC$. 

The groups $G$, $F$ depend on additional data: a subset $I_0$ of $I$ called {\it frozen} vertices. 
Let $$G=\prod_{i\in I\setminus I_0}\GL_{v_i}, $$ which act on $W$ by changing the basis of $\bbC^{v_i}$ for $i\in I\setminus I_0$. 
%\yl{is it true $G$ is a subgroup of $GL(W)$ for this example?}
Let 
%\yl{As cycles could appear in $I\setminus I_0$, $F$ here could have redundant? See example in below.}
$$F=\left(\prod_{i\in I_0}\GL_{v_i}\right)\times (\bbC^*)^{\rk H_1(Q,\bbZ)}, $$ 
where $\GL_{v_i}$ acts on $W$ by changing the basis of $\bbC^{v_i}$ for $i\in I_0$, and $(\bbC^*)^{\rk H_1(Q,\bbZ)}$ depends on a choice of basis of $H_1(Q,\bbZ)$ as cycles in the graph $Q$, with the action on $W$ by scaling the arrows constituting each of the cycles in the aforementioned basis. 
As $F$ in Setting \ref{setting of glsm} is abelian, so here we take the maximal abelian subgroup of the $F$ above. 
This choice of $F$ is inspired by the study of Nakajima quiver varieties \cite[\S 4.2.1]{O}.
%\yl{the above choice is from some reference? }
The action of so-chosen $F$ on $W$ may not be an effective action, and hence in practice we usually choose a subgroup of it as the flavor group. 

Pictorially, we honor the tradition and draw frozen vertices as $\square$ and thawed vertices as $\circ$. For instance, consider the following quiver: 
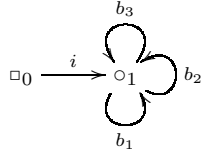
\begin{equation}
\xymatrix{
 \square_{0} \ar[r]^{i} & \circ_1 \ar@(dr,ur)_{b_{2}} \ar@(ru,lu)_{b_{3}} \ar@(ld,rd)_{b_{1}} }
\nonumber \end{equation}
with dimension vector $(1,n)$, potential $\phi=\tr (b_1b_2b_3-b_1b_3b_2)$, frozen vertex $I_0=\{0\}$ and $F=(\bbC^*)^{3}$ scaling $b_i$'s, we get back exactly to Example \ref{key exam}. 
\end{example}

%\yl{write also examples not coming from superpotential? E.g. products of several traces}

\subsection{Stacks of twisted maps}\label{sect on stack of twist map}
%\begin{YC}I slightly rearrange definitions/notations. See whether you agree.\end{YC}
%Let $W$ be an affine algebraic variety with the action of a reductive group $H$. 
%\yl{Do we need $H=G\times F\subseteq GL(W)$?}

%We are interested in maps from a curve to the stack quotient $[W/H]$. 
Recall that a \textit{prestable genus $g$, $n$-pointed curve} over $\bbC$ is $(C,p_1,\ldots,p_n)$ with $C$ being a connected projective curve of arithmetic genus $g$, with at worst nodal singularities, together with $n$ distinct non-singular marked points $p_1,\ldots,p_n$.
%\gufang{A pre-stable genus $g$ curve is also know as semistable curves with $n$ marked points? Are they required to have at least two special points on each $\bbP^1$? CiKM didn't assume this?} 
%Let $\calX$ be an Artin stack locally of finite type over $\bbC$, and $C$ a projective algebraic curve over $\calX$, with $n$ marked smooth points $p_i:\calX\to C$, $i=1,\dots,n$. Assume the geometric fibers of $C$ are prestable and have at worst nodal singularities.

A map from $C$ to the stack quotient $[W/H]$ is equivalent to a pair $(P_H,u)$ 
where $P_H$ is a principal $H$-bundle on $C$ and $u$ is an $H$-equivariant map $P_H\to W$. 
%Such a pair defines an equivariant cycle in $H_2^H(W)$. Its image under 
%the map $H_2^H(W)\to \Hom_{\bbZ}(\Pic^H(W),\bbZ)$ is called the {\it class} of the map:
%$$\beta: \Pic^H(W)\to \bbZ, \quad \beta(L):=\deg_C(u^*(P_H\times_H L)).$$
Let $Map(C,[W/H])$ be the Artin stack of all maps from $C$ to $[W/H]$ \cite{Ols}, which 
has a forgetful map 
$$Map(C,[W/H])\to Map(C,[\pt/H])=\fBun_H(C)$$
to the smooth Artin stack $\fBun_H(C)$ of principal $H$-bundles $P_H$ on $C$.

The \textit{log-canonical bundle} of a marked curve $(C,p_1,\ldots,p_n)$ is the line bundle 
$$\omega_{\mathrm{log}}:=\omega_C \left(\sum^n_{i=1}p_i\right). $$ Without causing confusion, the corresponding $\bbC^*$-bundle is  also denoted by $\omega_{\mathrm{log}}$.

Let $\fBun_H^{\chi=\omega_{\mathrm{log}}}(C)$ be the Artin stack of principal $H$-bundles $P_H$ on $C$ together with an isomorphism 
$\overline{\varkappa}:P_H\times_H\chi\cong \omega_{\mathrm{log}}$.
It is a smooth Artin stack \cite[Lem.~5.2.2]{FJR2} and has a map 
$$\fBun_H^{\chi=\omega_{\mathrm{log}}}(C)\to \fBun_H(C), $$ 
which forgets the isomorphism $\overline{\varkappa}$.

%In what follows, for a principal $H$-bundle $P_H$, we write the induced $F$-bundle as $P_F:=P_H\times_HF$. It is convenient to consider reduction of the structure group from $F$ to a subgroup.
To have a nice moduli stack of twisted quasimaps which will be introduced in the next section, it is convenient to consider reduction of the flavor symmetry $F$ to a one dimensional group.
\begin{definition}\label{defi of R-charge}
An {\it $R$-charge} is a group morphism $R:\bbC^*\to F$. We denote its composition with $\chi$ by $R_\chi:\bbC^*\xrightarrow{R}F\xrightarrow{\chi}\bbC^*$. 
\end{definition}
\begin{definition}
Write $H_R:=G\times\bbC^*$ and define stacks $Map^{\chi=\omega_{\mathrm{log}}}(C,[W/H])$, $Map^{R_\chi=\omega_{\mathrm{log}}}(C,[W/H_R])$ of \textit{twisted maps}
by the following Cartesian diagrams 
\begin{align}\label{diagram defi R twist map} 
\begin{xymatrix}{
Map^{R_{\chi}=\omega_{\mathrm{log}}}(C,[W/H_R]) \ar[d]\ar[r] \ar@{}[dr]|{\Box} & Map^{\chi=\omega_{\mathrm{log}}}(C,[W/H]) \ar@{}[dr]|{\Box} \ar[d]\ar[r]&Map(C,[W/H]) \ar[d]\\
\fBun_{H_R}^{R_{\chi}=\omega_{\mathrm{log}}}(C)\ar[r]&\fBun_H^{\chi=\omega_{\mathrm{log}}}(C)\ar[r]& \fBun_H(C).
}\end{xymatrix}
\end{align}
Here $\fBun_{H_R}^{R_{\chi}=\omega_{\mathrm{log}}}(C)$ is the Artin stack of principal $H_R$-bundles $P$ on $C$ together with an isomorphism 
$\varkappa: P\times_{H_R}R_\chi\cong \omega_{\mathrm{log}}$.
The left bottom map is well-defined because such $P$ and
$\varkappa$ induces an isomorphism $\overline{\varkappa}:P_H\times_H\chi\cong \omega_{\mathrm{log}}$, where $P_{H}:=P\times_{\mathbb{C}^*}R$ is the induced $H$-bundle.  
\end{definition}

%The intersection $F_0:=\ker(\chi)\cap Z_F(\mathrm{Im}(R))$ of the kernel of $\chi:F\to\bbC^*$ with the centralizer of $\mathrm{Im}(R)\subseteq F$ is  referred to as the {\it Calabi-Yau subgroup} or the {\it Calabi-Yau torus} if $F$ is abelian. 
The $H$-equivariant function 
$$\phi:W\to\bbC$$ 
gives rise to the critical locus $\Crit(\phi)\subseteq W$  with an induced $H$-action. 
%The stable locus is $\Crit(\phi)^s=\Crit(\phi)\cap W^s$ and $$\Crit(\phi)/\!\!/G:=\Crit(\phi)^s/G. $$
The equivariant embedding $\Crit(\phi)\inj W$ induces a map of stacks 
$$Map(C,[\Crit(\phi)/H])\to Map(C,[W/H]), $$
which is easily seen as a closed embedding, characterized as classifying maps $(P,u)$ from $(C,p_1,\ldots,p_n)$ such that $u$ lands in the subscheme $\Crit(\phi)$. 
Base change via \eqref{diagram defi R twist map} defines twisted maps to $\Crit(\phi)$.
%$$Map^{R_{\chi}=\omega_{\mathrm{log}}}(C,[\Crit(\phi)/H_R]), \quad Map^{{\chi}=\omega_{\mathrm{log}}}(\Crit(\phi),[W/H_R]). $$

The above construction works in the \textit{relative case}: for any curve $C$ over a base $S$, by working over $S$, the diagram \eqref{diagram defi R twist map} is still well-defined. 
In particular, $Map_S(C,[W/H]\times S)$ is the stack representing morphisms in the category of $S$-stacks.
The map $[W/H]\to [\pt/H]$ induces the map $$Map_S(C,[W/H]\times S)\to Map_S(C,[\pt/H]\times S):=\fBun_{H}(C/S). $$

\subsection{Recollection of quasimaps to GIT quotients}\label{subsec:recol_qas}

We will mainly look at the case when $C/S$ is the universal curve $\mathcal{C}$ over the 
smooth Artin stack $\fM_{g,n}$ of prestable genus $g$, $n$-pointed curves. Note that there is a flat ``stabilization" morphism \cite{B}:
\begin{equation}\label{stab map} st:\fM_{g,n}\to \overline{M}_{g,n} \end{equation}
to the Deligne-Mumford stack $\overline{M}_{g,n}$ of stable genus $g$, $n$-pointed curves. 

In this case, we omit $C$ from the notations and diagram \eqref{diagram defi R twist map} becomes 
\begin{align}\label{diagram defi R twist map 2} 
\begin{xymatrix}{
Map_{g,n}^{R_{\chi}=\omega_{\mathrm{log}}}([W/H_R]) \ar[d]\ar[r] \ar@{}[dr]|{\Box} & Map_{g,n}^{\chi=\omega_{\mathrm{log}}}([W/H]) \ar@{}[dr]|{\Box} \ar[d]\ar[r]&Map_{g,n}([W/H]) \ar[d]\\
\fBun_{H_R,g,n}^{R_{\chi}=\omega_{\mathrm{log}}}\ar[r]&\fBun_{H,g,n}^{\chi=\omega_{\mathrm{log}}}\ar[r]& \fBun_{H,g,n}.
}\end{xymatrix}
\end{align}
Here all bottom stacks are smooth Artin stacks locally of finite type over $\mathbb{C}$ by \cite[Prop.~2.1.1]{CiKM}, \cite[Lem.~5.2.2]{FJR2}.
%For any curve $C$ over a base $S$, let  $\fBun_H(C/S)$ be the moduli stack of principal $H$-bundles on $C/S$. In the special case when $C$ is the universal curve $\mathcal{C}$ over $\fM_{g,n}$, we write $\fBun_H(\mathcal{C}/\fM_{g,n})$ simply as $\fBun_H$, which is smooth Artin stack locally of finite type \cite[Prop.~2.1.1]{CiKM}. 
%We omit $C$ from the notation of $Map$, $Map^{\chi=\omega_{\mathrm{log}}}$, $Map^{R_{\chi}=\omega_{\mathrm{log}}}$, $\fBun_H^{\chi=\omega_{\mathrm{log}}}$, and $\fBun_{H_R}^{R_{\chi}=\omega_{\mathrm{log}}}$ when $C$ is the universal curve on $\fM_{g,n}$.

A $\mathbb{C}$-point in $Map_{g,n}^{R_{\chi}=\omega_{\mathrm{log}}}([W/H_R])$ is a quadruple $\big((C,p_1,\ldots,p_n),P,u,\varkappa\big)$, where 
$(C,p_1,\ldots,p_n)$ is a prestable genus $g$, $n$-pointed curve, $P$ is a principal $(G\times \mathbb{C}^*)$-bundle on $C$ with an isomorphism 
$\varkappa:P/G\times_{ \bbC^*}R_\chi\cong\omega_{\mathrm{log}}$, and $u: P\times_{(G\times  \bbC^*)}(G\times R)\to W$ is a $(G\times F)$-equivariant map. 
\begin{definition}
The \textit{class} of such a quadruple is a map from the character group $\bbX(G)$: 
$$\beta=\beta_{P_G}\in \Hom_{\mathbb{Z}}(\bbX(G), \bbZ), \quad \beta(\xi):=\deg_C(P_G\times_G \bbC_{\xi}), $$
given by the degree of the principal $G$-bundle $P_G:=P/\mathbb{C}^*$ \cite[\S 2.5]{CiKM}.  
\end{definition}
Given a stability condition $\theta: G\to \mathbb{C^*}$ as in Setting \ref{setting of glsm}, we have three conditions on the quadruple $\big((C,p_1,\ldots,p_n),P,u,\varkappa\big)$: 
\begin{enumerate}
\item the image of $u$ lies in the open locus $W^s$ on the entire $C$ but (possibly empty) finitely many points (which are called the {\it base points}); 
\item base points are away from the special points (markings and nodes);
\item $\omega_C (\sum^n_{i=1}p_i)\otimes L_\theta^\epsilon$ is ample for every rational number $\epsilon>0$, where
$L_\theta =P/\bbC^*\times_G \mathbb{C}_\theta$.
\end{enumerate}

\begin{definition}\label{def:QM}
A \textit{genus $g$, $n$-pointed $R$-twisted quasimap} to $W/\!\!/G$ is 
a point in $Map_{g,n}^{R_{\chi}=\omega_{\mathrm{log}}}([W/H_R])$ satisfying condition (1).
It is said to be \textit{prestable} if it satisfies (1), (2).
It is said to be \textit{stable} if it satisfies (1), (2) and (3).
\end{definition}
\begin{remark}\label{rmk on e-stab}
Our stability condition coincides with the $0^+$-stability in \cite[Def.~4.2.13]{FJR2}, \cite[Def.~3.1.2]{CiKM}. 
There is a more general notion of $\epsilon$-stability \cite[Def.~4.2.11]{FJR2}, \cite[Def.~7.1.3]{CiKM}.
\end{remark}
The notions of isomorphisms and families of $R$-twisted quasimaps are the obvious ones as in \cite[\S 3.1]{CiKM} and \cite[\S 4.2]{FJR2}.
Let $$QM_{g,n}^{R_{\chi}=\omega_{\mathrm{log}}}(W/\!\!/ G,\beta)$$ 
denote the \textit{moduli stack of genus} $g$, $n$-\textit{pointed stable} $R$-\textit{twisted quasimaps} of class $\beta\in \Hom_{\bbZ}(\bbX(G), \bbZ)$. 

Since $F$ commutes with $G$ and $F$ is abelian, we have an induced $F$-action on $[W/H_R]$, which induces an action on the stack $Map_{g,n}^{R_{\chi}=\omega_{\mathrm{log}}}(C,[W/H_R])$ by post-composing with the map $u$, i.e. 
for any $f\in F$, we define 
$$f\cdot \big((C,p_1,\ldots,p_n),P,u,\varkappa\big):=\big((C,p_1,\ldots,p_n),P,f\cdot u,\varkappa\big). $$
The $F$-action on $W$ preserves $\theta$-stable locus for any $\theta$ due to the commutativity with $G$, so it preserves $W/\!\!/ G$.
By restricting to the open substack $QM_{g,n}^{R_{\chi}=\omega_{\mathrm{log}}}(W/\!\!/ G,\beta)$, we get an induced $F$-action.

%\begin{YC}I rewrite definitions/notations until here. See whether you agree. If so, let's keep notations consistently overall. \end{YC}

A priori, the evaluation map $ev_i$ on $Map_{g,n}^{\chi=\omega_{\mathrm{log}}}([W/H])$ does not land in $[W/G]$ but rather in $[W/(G\times F)]$. 
By using $R$-charge and quasimap stability, we have the following. 
%\gufang{I think there is indeed a non-properness in general for twisted invariants. See the comment in \cite[(8.2.17)]{O}}
%\yl{\cite{FJR2} proves properness when $X$ is proper?}
\begin{prop}\label{prop:ev_equi}
Let $X:=W/\!\!/ G$. Then for any $i=1,2,\ldots,n$, there exists an $F$-equivariant evaluation map  
$$ev_i: QM_{g,n}^{R_{\chi}=\omega_{\mathrm{log}}}(X,\beta)\to [X/R(\Ker R_{\chi})]. $$ 
\end{prop} 
\begin{proof}
Let $\calP$ be the universal  $(G\times \bbC^*)$-bundle on the universal curve $\cC$ over $QM_{g,n}^{R_{\chi}=\omega_{\mathrm{log}}}(W/\!\!/ G,\beta)$ with 
induced $(G\times F)$-bundle $\calP_{G\times F}:= \calP\times _{(G\times \bbC^*)}(G\times R)$ and universal $(G\times F)$-equivariant map
$$\calP_{G\times F}\to W.$$
Taking quotient by $G$, we obtain a $F$-equivariant map 
$$\mathrm{taut}: \calP_F:=\calP_{G\times F}/G\to [W/G]. $$
Let $p_i: QM_{g,n}^{R_{\chi}=\omega_{\mathrm{log}}}(W/\!\!/ G,\beta) \to \cC$ be the section corresponding to the $i$-th marked point. We can
pullback $\calP_F$ to $QM_{g,n}^{R_{\chi}=\omega_{\mathrm{log}}}(W/\!\!/ G,\beta)$ via $p_i$ and obtain $\calP_F|_{p_i}:=p_i^*\calP_F$.
Restricting to the marked point $p_i$, the map $u$ gives 
$$\mathrm{taut}|_{p_i}: \calP_F|_{p_i}\to W/\!\!/G. $$ 
Write $\calP_{\mathbb{C^*}}:=\calP/G$. 
By pullback the constrain
$$ \calP_{\mathbb{C^*}}\times_{\mathbb{C}^*}R_{\chi}\cong \omega_{\mathrm{log}}$$
to $QM_{g,n}^{R_{\chi}=\omega_{\mathrm{log}}}(X,\beta)$ via $p_i$, we obtain 
\begin{equation}\label{equ on Ppi to triv} \calP_{\mathbb{C}^*}|_{p_i}\times_{\mathbb{C}^*}R_{\chi}\cong \omega_{\mathrm{log}}|_{p_i}\cong QM_{g,n}^{R_{\chi}=\omega_{\mathrm{log}}}(X,\beta)\times \mathbb{C}^*, \end{equation}
where we use the fact that marked points are at smooth points and hence adjunction formula provides a canonical  trivialization of  $\omega_{\mathrm{log}}|_{p_i}$ (ref.~\cite[\S 4.4]{FJR2}). This provides a reduction 
\begin{equation}\label{equ of ga cov} \calP_{\Ker R_\chi,p_i}\hookrightarrow \calP_{\mathbb{C}^*}|_{p_i} \end{equation}
of $\calP_{\mathbb{C}^*}|_{p_i}$ to a principal $\Ker R_\chi$-bundle $\calP_{\Ker R_\chi,p_i}$ (i.e. structure group reduces to $\Ker R_\chi$). 

Combining with the inclusion $R: \Ker R_\chi/\Ker R\hookrightarrow F$, we obtain a reduction 
$$\calP_{\Ker R_\chi,p_i}\times_{\Ker R_\chi}\Ker R_\chi/\Ker R  \hookrightarrow \calP_F|_{p_i}=\calP_{\Ker R_\chi,p_i}\times_{\Ker R_\chi}F $$
of $\calP_F|_{p_i}$ to a principal $\Ker R_\chi/\Ker R\cong R(\Ker R_\chi)$-bundle. Taking quotient by $R(\Ker R_\chi)$, we obtain a section 
of  principal $F/R(\Ker R_\chi)$-bundle $(\calP_F|_{p_i})/R(\Ker R_\chi)$ (ref.~Lemma \ref{lem on pri bdl}), i.e. a trivialization
%Notice that the curve $C$ is assumed to be smooth at $p_i$, and hence adjunction formula provides a canonical  trivialization of  $\omega_{\mathrm{log}}|_{p_i}$. With local coordinate %around $p_i$ given by $z_i$, the trivialization is denoted by $\frac{dz_i}{z_i}$. The trivialization $\frac{dz_i}{z_i}$ gives an $F$-bundle isomorphism
\[\tau: QM_{g,n}^{R_{\chi}=\omega_{\mathrm{log}}}(W/\!\!/ G,\beta)\times (F/R(\Ker R_\chi))\xrightarrow{\cong} (\calP_F|_{p_i})/R(\Ker R_\chi).\]
%\yl{why $\frac{dz_i}{z_i}$ gives trivialization of $F$-bundle? Compare with \cite[\S 4.4]{FJR2}?}
Composing the maps, we obtain (writting $X:=W/\!\!/ G$):
\[ev_i: QM_{g,n}^{R_{\chi}=\omega_{\mathrm{log}}}(X,\beta)\xrightarrow{e} QM_{g,n}^{R_{\chi}=\omega_{\mathrm{log}}}(X,\beta)\times F/R(\Ker R_\chi)
%\xrightarrow{\tau}\calP_F|_{p_i}/R(\Ker R_\chi)
\xrightarrow{\mathrm{taut}|_{p_i}\circ \tau}[X/R(\Ker R_\chi)], \]
$$\big((C,p_1,\ldots,p_n),P,u,\varkappa\big)\mapsto u \circ \tau\left(\big((C,p_1,\ldots,p_n),P,u,\varkappa\big),e\right), $$
where the first map is the identity section of the trivial bundle.
%Starting with $\big((C,p_1,\ldots,p_n),P,u,\varkappa\big)\in QM_{g,n}^{R_{\chi}=\omega_{\mathrm{log}}}(W/\!\!/ G,\beta)$, the above map $ev_i$ gives $u \tau\big((C,p_j,\calP,\varkappa, u),e\big).$

Now we show $ev_i$ is $F$-equivariant.
An element $f\in F$ acts on $\big((C,p_1,\ldots,p_n),P,u,\varkappa\big)$ by
\[f\big((C,p_1,\ldots,p_n),P,u,\varkappa\big)= \big((C,p_1,\ldots,p_n),P,f\cdot u,\varkappa\big).\]
Through $ev_i$, it is mapped to $(f\cdot u) \circ\tau\left(\big((C,p_1,\ldots,p_n),P,f\cdot u, \varkappa\big),e\right)$. 
%The first map send it to $\left(\big((C,p_1,\ldots,p_n),\calP,f\circ u, \varkappa\big),e\right)$.
%The second map send it to $\tau\left(\big((C,p_1,\ldots,p_n),\calP,f\circ u, \varkappa\big),e\right)$.Applying the last map, we get . 
As $\calP_F|_{p_i}$ is the pullback  bundle on $\fBun_{H_R}^{R_{\chi}=\omega_{\mathrm{log}}}$ via the forgetful map 
$QM_{g,n}^{R_{\chi}=\omega_{\mathrm{log}}}(W/\!\!/ G,\beta)\to \fBun_{H_R,g,n}^{R_{\chi}=\omega_{\mathrm{log}}}$, 
so there is a natural identification between the fibers at $\big((C,p_1,\ldots,p_n),P,f\cdot u,\varkappa\big)$ and $\big((C,p_1,\ldots,p_n),P,u,\varkappa\big)$, under which 
$$\tau\left(\big((C,p_1,\ldots,p_n),P,f\cdot u, \varkappa\big),e\right)=\tau\left(\big((C,p_1,\ldots,p_n),P, u, \varkappa\big),e\right). $$ 
Therefore we have  
\[(f\cdot u) \circ\tau\left(\big((C,p_1,\ldots,p_n),P,f\cdot u, \varkappa\big),e\right)
=(f\cdot u) \circ\tau\left(\big((C,p_1,\ldots,p_n),P, u, \varkappa\big),e\right),  \]
which shows the $F$-equivariance.
%Finally by composing the evaluation map $ev_i$ with the projective morphism $$X=W/\!\!/ G\to W/_{\mathrm{aff}}G, $$  we get a proper morphism by \cite[Thm.~4.1.2]{CiKM}. Therefore we know $ev_i$ is also proper. 
\end{proof}
The following standard facts on principal bundles are used in above, whose proof is also sketched for completeness. 
\begin{lemma}\label{lem on pri bdl}
Let $P$ be a principal $F$-bundle on a stack $M$ and $F_0\subseteq F$ be a normal subgroup. Then the followings are equivalent: 
\begin{enumerate}
\item There is a principal $F_0$-bundle $P_0$ and an isomorphism $P_0\times_{F_0}F\cong P$. 
\item There is a principal $F_0$-bundle $P_0$ and an $F_0$-equivariant embedding $P_0\to P$ covering the identity map on the base. 
\item There is a section of the principal $F/F_0$-bundle $P\times_{F}(F/F_0)$.  
\item There is an isomorphism $P\times_{F}(F/F_0)\cong M\times (F/F_0)$ of principal $F/F_0$-bundle. 
\end{enumerate}
\end{lemma}
\begin{proof}
(1)~$\Rightarrow$~(2): The inclusion $F_0\subseteq F$ induces the embedding 
$$P_0\times _{F_0}F_0\to P_0\times_{F_0}F\cong P. $$  
(2)~$\Rightarrow$~(3): Take quotient of the $F_0$-equivariant embedding $P_0\to P$ by $F_0$ gives a section 
$$M\to P/F_0\cong P\times_{F}(F/F_0). $$
(3)~$\Rightarrow$~(4): Write $\bar{F}:=F/F_0$ and $P_{\bar{F}}:=P\times_{F}(F/F_0)$. A section $s: M\to P_{\bar{F}}$ gives a map 
$$M\times \bar{F}\to P_{\bar{F}}, \quad (x,f)\mapsto s(x)\cdot f. $$
It is direct to check this is an isomorphism of principal $\bar{F}$-bundle using the fact that $\bar{F}$ is a group.  \\
%${}$ \\
(4)~$\Rightarrow$~(1): Since $P\times_{F}(F/F_0)\cong P/F_0$, $P$ is a principal $F_0$-bundle over $P\times_{F}(F/F_0)\cong M\times F/F_0$.
Through the identity section $M\to M\times F/F_0$, we can pullback this bundle to $M$, denoted by $P_0$ with a 
$F_0$-equivariant embedding $i: P_0\hookrightarrow P$ covering identity on $M$. We then define a map 
$$\psi:  P_0\times F \to P, \quad (p,f)\mapsto i(p)\cdot f, $$
with $F_0$-action $f_0\cdot (p,f)=(p\cdot f_0, f_0^{-1} f)$ on the domain and right $F_0$-multiplication on the target. 
It is easy to check that $\psi$ factors through $(P_0\times F)/F_0$ and defines a map of principal $F$-bundle
$$P_0 \times_{F_0} F\to P, $$
which must be an isomorphism as $F$ is a group. 
\end{proof}
We recall the following properties of $QM_{g,n}^{R_{\chi}=\omega_{\mathrm{log}}}(W/\!\!/ G,\beta)$ proven in \cite[Lem.~5.3.2, Thm.~5.2.3]{FJR2}.
\begin{theorem}\label{prop:geom_properties} 
Assume $\Ker R_\chi=\{1\}$, then the stack $QM_{g,n}^{R_{\chi}=\omega_{\mathrm{log}}}(W/\!\!/ G,\beta)$ is Deligne-Mumford, separated, and of finite type over $\mathbb{C}$. 
\end{theorem}
\begin{proof}
This is \cite[Theorem~5.2.3]{FJR2}.  A comparison of notations in the present paper and those of~{\it loc.~cit.}~is in order. 
%\begin{equation*}\begin{xymatrix}{ & \mathbb{C}^* &  \\\mathbb{C}^* \ar[r]^{R} \ar[ru]^{R_{\chi}} & F \ar[r]^{\eta\quad \,\,} \ar[u]_{\chi} & \GL(W)}\end{xymatrix}\end{equation*}
The group $\Gamma$ in~\textit{loc}.~\textit{cit}.~in the present setting is the image of $G\times R(\bbC^*)$ in $\GL(W)$. The group $\langle J\rangle$ from~\textit{loc}.~\textit{cit}.~in the present setting is the image of $R(\Ker R_\chi)$  in $\GL(W)$. 
%\gufang{You are right. But the point is that FJR assumes all groups are subgroups of $\GL(W)$ 
%hence there's no difference between $R(\Ker R_\chi)$ and $\Ker R_\chi$. In our case, if $\Ker R$ is non-trivial, even if $R(\Ker R_\chi)=1$ one still needs to parameterize $r$-spin structures on the curve, which would lead to orbi-structure at nodes. That's why this stronger assumption.}
The group $G$ in~\textit{loc}.~\textit{cit}.~in the present setting is chosen to be the image of $G\times R(\Ker R_\chi)$. 
Under the simplifying assumption $\Ker R_\chi=\{1\}$, we claim the composition $\mathbb{C}^* \xrightarrow{R} F\xrightarrow{\eta}\GL(W)$ is injective, 
where $\eta$ denotes the action of $F$ on $W$.  
In fact, in Setting \ref{setting of glsm}, there is a non-trivial function $\phi$ such that for any $t\in \mathbb{C}^*$ and $w\in W$, we have 
$$\phi(t\cdot w):=\phi((\eta\circ R)(t)\cdot w)=R_{\chi}(t)\cdot \phi(w). $$
If there is a $1\neq t\in \mathbb{C}^*$ such that $(\eta\circ R)(t)\cdot w=w$ for any $w\in W$, we get contradiction in the above equality as $R_{\chi}$ is injective. 
Then it is easy to check the orbi-structures on quasimaps in~\textit{loc}.~\textit{cit}.~become trivial, i.e. they are prestable curves used in the present setting. 
%\yl{Is the image of $G\to GL(W)$ (smooth) reductive?}\gufang{This is automatic in char 0.}
%\gufang{This is independent of the potential. In char 0 reductive=linearly reductive. Quotient of a linearly reductive group is linearly reductive. Any group in char 0 is smooth.}
\end{proof}
%\gufang{I added the following remark. In the above, I used a stronger assumption $\Ker R_\chi=1$ to completely remove any necessaty of orbi-curves. (Otherwise, either the base-points or nodes could still have orbi-structure.)}
\begin{remark}\label{rmk on ass on orbi}
In the above theorem, we work under the simplifying assumption that $\Ker R_\chi=\{1\}$.
Without this assumption, in order to get the same separatedness of moduli spaces (similarly the properness in Theorem~\ref{thm on properness} below), one follows \cite{FJR2} and allows prestable marked curves $(C,p_1,\dots,p_n)$ to be a balanced twisted orbicurve in the sense of Abramovich and Vistoli \cite{AbV}. 
In the present paper, for simplicity of exposition, we avoid the full strength of the theory of orbicurves. 
\end{remark}

\subsection{Quasimaps to critical loci}\label{subsec:setup}

\iffalse
Let $\chi:H\to \bbC^*$ be a character. 
Assume now that $V=W$ is smooth, and let $\phi:W\to\bbC$ be an $H$-equivariant function with the action of $H$ on $\bbC$ given by the character $\chi$.
Let $\Crit(\phi)\subseteq W^s$ be the critical locus of $\phi$, that is, the subscheme where $d\phi$ vanishes. Note that there is an induced $H$-action on $\Crit(\phi)$. The stable locus $\Crit(\phi)^s$ is the intersection of $\Crit(\phi)$ with $W^s$.

The equivariant embedding $\Crit(\phi)\inj W$ induces a map of stacks 
$QM_{g,n}(\Crit(\phi)/\!\!/G,\beta)\to QM_{g,n}(W/\!\!/G,\beta)$, which is easily seen as a closed embedding, characterized as classifying maps $(P,u)$ from $(C,p_1,\ldots,p_n)$ such that $u$ lands in the subscheme $\Crit(\phi)$, or equivalently $u$ lands in $\Crit(\phi)^s$ except for finitely many points. 
\fi 

Via the embedding $\Crit(\phi)/\!\!/G\hookrightarrow W/\!\!/G$, we have the closed substack 
\begin{equation}\label{twist qm to crit}QM_{g,n}^{R_{\chi}=\omega_{\mathrm{log}}}(\Crit(\phi)/\!\!/G,\beta)\hookrightarrow QM_{g,n}^{R_{\chi}=\omega_{\mathrm{log}}}(W/\!\!/G,\beta) \end{equation}
of genus $g$, $n$-pointed stable $R$-twisted quasimaps to $\Crit(\phi)/\!\!/G$ of class $\beta$ which is also a separated Deligne-Mumford
stack of finite type by Theorem \ref{prop:geom_properties} (see also \cite[Lemma~5.3.2]{FJR2}). 

As in Proposition \ref{prop:ev_equi}, there are corresponding evaluation maps. 
\begin{prop}\label{prop:ev_equi 2}
Let $C:=\Crit(\phi)/\!\!/G$. Then for any $i=1,2,\ldots,n$, there exists an $F$-equivariant evaluation map  
\begin{equation}\label{equ:ev_equi 2}ev_i: QM_{g,n}^{R_{\chi}=\omega_{\mathrm{log}}}(C,\beta)\to [C/R(\Ker R_{\chi})]. \end{equation} 
\end{prop} 
We have the following properness result. 
%The obstruction theories will be discussed in detail later. In many of the applications below, we assume $F$ further to be abelian.  
\begin{theorem}\label{thm on properness}
Assume $\Ker R_\chi=\{1\}$ and the $F_0$-fixed locus in the affine quotient $(\Crit(\phi)/_{\mathrm{aff}}G)^{F_0}$ is finite, then the $F_0$-fixed locus
 $(QM_{g,n}^{R_{\chi}=\omega_{\mathrm{log}}}(\Crit(\phi)/\!\!/G,\beta))^{F_0}$ is proper. 
\end{theorem}
\begin{proof}
This is a combination of \cite[Thm.~5.4.1]{FJR2} and \cite[\S 4.4]{Kim}. We briefly summarize for the convenience of the readers. Let $\Delta$ be a disc, the generic point of which is $\eta$. We assume on $\eta$ there is an $F_0$-fixed stable quasimap data $(\calC_\eta,\calP_\eta, u_\eta:\calP_\eta\to \Crit(\phi))$. 
The construction of the first 10 paragraphs in the proof of \cite[Thm.~5.4.1]{FJR2} gives sections of $\calC_\eta$ outside of which the quasimap data is a 
balanced twisted pointed stable map landing in $(\Crit(\phi)/\!\!/G)^{F_0}$. The assumption that $(\Crit(\phi)/_{\mathrm{aff}}G)^{F_0}$ is finite implies the properness of $(\Crit(\phi)/\!\!/G)^{F_0}$, hence as in \cite[Thm.~5.4.1]{FJR2} this data extends to a balanced twisted pointed stable map on the entire $\Delta$. 
Forgetting some of the sections and contracts some components of the central fiber of the curve as in  \cite[pp.~281]{FJR2} modifies a pointed stable map into a quasimap data defined outside finitely many points on the central fiber, which agrees with the existing data when restricted to $\eta$. Finally, using \cite[Lem.~4.3.2]{CiKM} and Hartogs' theorem, the quasimap data extends across these finitely many points, hence well-defined on $\Delta$.  The argument in \cite[pp.~282--283]{FJR2} shows the stability. 
 \end{proof}

 %\yl{Be careful about the properness. Does \cite{AbV} imply the properness of moduli of maps from prestable curves (not orbifold curves) to $[X/R(\Ker R_\chi)]$? I do not 
 % think so, see \cite[\S 1.3]{AbV}. }
%\yl{Do we need to assume $F_0$ abelian? \cite{Kim} works with torus case}
%Since $\Pic^{G\times F}(W)\cong \Pic^{G\times F}(\pt)=\bbX(G)\times\bbX(F)$ is the product of character groups, we have an isomorphism of lattices:
%$$\Hom_{\bbZ}(\Pic^{G\times F}(W),\bbZ)\cong \bbX(G)^\vee\times\bbX(F)^\vee. $$ 
%\begin{definition}\label{def:deg_twisting}
%For any $\beta\in\Hom_{\bbZ}(\Pic^{G\times F}(W),\bbZ)$, we write its projections to $\bbX(G)^\vee$ (resp. $\bbX(F)^\vee$) as $\overline\beta$ (resp. $\beta^t$). 
%We call the former the {\it reduced degree} and the latter the {\it degree of twist}. 
%\end{definition}
%\yl{We want $\Hom_{\bbZ}(\Pic^{G\times F}(W),\bbZ)$ or $\Hom_{\bbZ}(\Pic^{G\times \mathbb{C}^*}(W),\bbZ)$? }
%\begin{YC}
%If we do not need the following sentences, we should delete it?
%\end{YC}
%For general $V$ (which is not a vector space), although $\Pic^{G\times F}(V)\cong \Pic^{G\times F}(\pt)$ is no longer true, one still has an action $\Pic^F(\pt)$ on $\Pic^{G\times F}(V)$ and hence any $\sigma\in \Pic^F(\pt)=\bbX(F)$ and $\beta$, the shift $\sigma+\beta$ of $\sigma$ by $\beta$ is defined. 

\section{Shifted symplectic structures}\label{sec:shifted_symp}
%\yl{Ask experts whether this is known to them?}

In this section, following \cite{PTVV, CPTVV}, we construct \textit{shifted symplectic structures} on several derived stacks, which 
will be used to construct \textit{virtual structures} on moduli stacks of quasimaps introduced in the previous section. 

Using standard notations, 
the affine derived  scheme associated to a commutative differential graded algebra (cdga) $A$ is denoted by $\bSpec(A)$.
%The derived mapping stack (relative to a base $S$) is denoted by $\bMap_{\textbf{dSt}/S}(-,-)$. 
The derived fiber product (or homotopy fiber product) of two maps $X\to Z$, $Y\to Z$ between derived stacks 
is denoted by  $X\times^\bfL_Z Y$ (or simply $X\times_Z Y$ if it is clear from the context). 
For a map $f: X\to Y$ between derived stack, the relative tangent (resp.~cotangent) complex is denoted by $\bbT_f$ (resp.~$\bbL_f$) or $\bbT_{X/Y}$ (resp.~$\bbL_{X/Y}$) if we want to emphasis $X$ and $Y$. Expressions such as $f_*, f^*,\otimes$ should be understood in the derived sense unless stated otherwise.
All derived Artin stacks mentioned in this paper are assumed to be locally of finite presentation.

%We refer to \cite{Toen1, Toen2} for introduction to derived algebraic geometry and 
%We refrain from giving a general review of shifted symplectic structures and only collect relevant results in shifted symplectic geometry and refer interested readers to \cite{PTVV, CPTVV} for details.

\subsection{Derived critical locus}\label{subsec:crit}
Let $W$ be a complex vector space with a linear action of a reductive algebraic group of the form $H=G\times F$, together with a regular functon $\phi:W\to \bbC$. We assume that $\phi$ is equivariant with the target endowed with trivial $G$-action and an $F$-action given by a fixed character $\chi:F\to \bbC^*$.  For simplicity, we denote $\bbC_\chi$ to be the associated 1-dimensional representation of $F$, and also of $G\times F$ when no ambiguity arises from the context. 
\begin{definition}
We define the \textit{derived critical locus} $\bCrit^{}(\phi)$ by the homotopy pullback diagram: 
%\begin{equation}\label{equ on sc cl}\bCrit^{}(\phi):=W\times^\bfL_{T^*_W}W, \end{equation}
\begin{equation}\label{equ on sc cl}\begin{xymatrix}{
\bCrit^{}(\phi)  \ar[r]^{ } \ar[d]^{ } \ar@{}[dr]|{\Box} &W \ar[d]^{d\phi}  \\
W \ar[r]^{0  \,\,} & \bfT^*W.
}\end{xymatrix}\end{equation}
%and also a derived stack $\bfZ(d\phi)$ by 
%\begin{equation}\label{diag def zdphi}\begin{xymatrix}{
%\bfT^*_{[W/G]}\otimes\bbC_\chi&[W/G]\ar[l]_{\quad 0}\\
%[W/G]\ar[u]^{d\phi}&\bfZ(d\phi). \ar[l]_{\quad \iota}\ar[u]
%}\end{xymatrix}\end{equation}
%Here the map $\bfZ(d\phi)\to [W/G]$, composing with the structure morphism $[W/G]\to [\pt/G]$ endows $\bfZ(d\phi)$ with a structure morphism $\bfZ(d\phi)\to [\pt/G]$.
\end{definition}
As a Lagrangian intersection, $\bCrit^{}(\phi)$ has a canonical $(-1)$-shifted symplectic structure by \cite[Thm.~0.5]{PTVV}.
Its cotangent complex can be easily calculated as follows. 
\begin{prop}\label{prop on cot cx of cri loc}
We have 
\begin{equation*}\bbL_{\bCrit^{}(\phi)}\cong(0\to W\otimes\calO\to W^*\otimes\calO\to 0),  \end{equation*}
where the right hand side is written as a complex on $\bCrit^{}(\phi)$, the middle map is the differential of $d\phi$ (also known as the Hessian of $\phi$).
%\begin{equation*}\bbL_{\bfZ(d\phi)}\cong(0\to \fg\otimes\calO\to W\otimes\calO\to W^*\otimes\calO\to \fg^*\otimes\calO\to 0),  \end{equation*}
%where the right hand side is written as a $G$-equvariant complex on $W\times^\bfL_{\mu^{-1}(0)}W$. 
%The term $\fg\otimes\calO$ is in homological degree $(-2)$, the first map is the differential of the $G$-action, the second map is the differential of $d\phi$ (also known as the Hessian of $\phi$), the last map is % the differential of $\mu$ (hence dual to the differential of the $G$-action).
\end{prop}
\iffalse
\begin{proof}
These are well-known and we show the second isomorphism using the description \eqref{diag def zdphi}. Note that
\[\bbL_{\bfT^*_{[W/G]}}\cong \left(0\to \fg\otimes\calO\to (W\oplus W^*)\otimes\calO\to \fg^*\otimes\calO\to 0\right),\]
with $(W\oplus W^*)\otimes\calO$ in degree 0, where the RHS is written as a $G$-equivariant complex on $\mu^{-1}(0)$.

Using the exact triangle \[d\phi^*\bbL_{\bfT^*_{[W/G]}}\to \bbL_{[W/G]}\to \bbL_{[W/G]/\bfT^*_{[W/G]}},\] we have \[\bbL_{[W/G]/\bfT^*_{[W/G]}}\cong (0\to \fg\otimes\calO\to W\otimes\calO \to 0),\]
where $W\otimes\calO$ is in degree $(-1)$. 
Derived base-change preserves the relative cotangent complex. Hence, 
\[\bbL_{\bfZ(d\phi)/[W/G]}\cong \left(\bbL_{[W/G]/\bfT^*_{[W/G]}}\right)|_{\bfZ(d\phi)}.\] 
Using the exact triangle, $d\phi^*\bbL_{[W/G]}\to \bbL_{\bfZ(d\phi)}\to \bbL_{\bfZ(d\phi)/[W/G]},$ we obtain the conclusion. 
\end{proof}
\fi
By the construction, $\bCrit^{}(\phi)$ has an action by $G\times F$. It is straightforward to 
calculate the cotangent complex of the derived quotient stack $[\bCrit^{}(\phi)/(G\times F)]$: 
\begin{equation}\label{equ on cot cpx of C/H}\bbL_{[\bCrit^{}(\phi)/(G\times F)]}\cong (0\to  W\otimes\bbC_\chi^*\otimes\calO\to W^*\otimes\calO\to (\mathfrak f\oplus\fg)^*\otimes\calO\to 0), \end{equation}
where the right hand side is written as a $(G\times F)$-equivariant complex on $\bCrit^{}(\phi)$, and 
$\mathfrak g$ and $\mathfrak f$ denotes the Lie algebra of $G$ and $F$ respectively.

%Similarly, we have  
%\begin{equation}\label{equ on cot cpx of Z/F}\bbL_{[\bfZ(d\phi)/F]}\cong (0\to \fg\otimes\bbC_\chi^*\otimes\calO\to W\otimes\bbC_\chi^*\otimes\calO\to W^*\otimes\calO\to (\mathfrak f\oplus\fg)^*\otimes\calO\to 0), \end{equation}where the right hand side is written as a $(G\times F)$-equvariant complex on $W\times^\bfL_{\mu^{-1}(0)}W$.

%\yl{Delete the following words or put them elsewhere.}
%Consider the quotient derived stack $[\bfZ(d\phi)/F]$, we have 
%\[\begin{xymatrix}{&\left[\bfZ(d\phi)/F\right]\ar[ld]\ar[d]&\\[\pt/F]&\left[[\pt/G]/F\right]\ar[l]\ar@{=}[r]&\left[\pt/(F\times G)\right].}\end{xymatrix}\]

\subsection{Derived mapping stacks}\label{sect on der map stack}
%We start with a warm-up case. 

Let $Y$ be a derived Artin stack locally of finite presentation over $\mathbb{C}$ with an action by a complex reductive group $H$, 
%\footnote{Here $H$ is not necessarily $G\times F$ as in the previous section.}, 
$k$ be a Noetherian commutative $\mathbb{C}$-algebra and $C$ be a proper flat family of curves over $k$ with at worst nodal singularities. 

Consider the \textit{derived mapping stack} (relative to $k$): 
\begin{equation}\label{equ on mappi stac}\bMap(C,[Y/H]):=\bMap_{\textbf{dSt}/k}(C,[Y/H]\times \Spec k), \end{equation}
where we omit the inclusion functor from classical stacks to derived stacks for $C$ and $\Spec k$.
By Lurie's representability theorem \cite{Lur} (see also \cite[Cor.~3.3]{Toen2}), we know this is a derived Artin stack locally of finite presentation over $k$.
Let 
$$[u]:C\times \bMap(C,[Y/H])\to [Y/H]$$ 
be the universal morphism and 
$$\pi:C\times \bMap(C,[Y/H])\to  \bMap(C,[Y/H])$$ 
be the projection. 
The tangent complex of $\bMap(C,[Y/H])$ satisfies 
\begin{equation}\label{tang of mapp stac} \bbT_{\bMap(C,[Y/H])}\cong \pi_*[u]^*\bbT_{[Y/H]}. \end{equation}
The map $[Y/H]\to [\pt/H]$ induces a morphism
\begin{equation*}f:\bMap(C,[Y/H])\to \bMap(C,[\pt/H])=:\fBun_H(C), \end{equation*}
where $\fBun_H(C)$ is isomorphic to its classical truncation as $C$ is a curve. 
Base change gives 
\begin{equation}\label{tang of bunf}f^*\bbT_{\fBun_H(C)}\cong \pi_*[u]^*(\mathfrak h\to 0), \end{equation}
where $\mathfrak h$ denotes the Lie algebra of $H$. 
Using the fiber sequence 
\begin{equation}\label{tang cpx tri}\bbT_{\bMap(C,[Y/H])/\fBun_H(C)}\to \bbT_{\bMap(C,[Y/H])} \to f^*\bbT_{\fBun_H(C)}, \end{equation}
we can determine the relative tangent complex. 

%\begin{equation*}\bbT_{\bMap(C,[Y/H])/\fBun_H(C)}=\dR\pi_*[u]^*\bbT_{Z}. \end{equation*}
We spell out things explicitly in the case arising from \S \ref{subsec:crit}, i.e.
\begin{equation}\label{equ for z}Y=\bCrit^{}(\phi):=W\times^\bfL_{T^*W}W, \quad H=G\times F.   \end{equation} 
A $k$-point in $\bMap(C,[\bCrit^{}(\phi)/(G\times F)])$ is a pair $(P,u)$ where $P$ is a principal $(G\times F)$-bundle on $C$, and $u:P\to \bCrit^{}(\phi)$ is a $(G\times F)$-equivariant map. The map $u$ induces  
$$ C\to P\times_{G\times F}\bCrit^{}(\phi), $$ 
whose composition with the projection $P\times_{G\times F}\bCrit^{}(\phi)\to C$ is the identify. 
As the target $\bCrit^{}(\phi)$ is a derived subscheme of $W$, the above map gives rise to a section  of the vector bundle $P\times_{G\times F}W$, 
which without causing confusion is still denoted by $u$.

Let $\calP$ be the universal $(G\times F)$-bundle on $C\times \bMap(C,[\bCrit^{}(\phi)/(G\times F)])$.
The vector bundle $\calP\times_{G\times F}W$ will be referred to very often, hence denoted simply by $\calW$ which satisfies an isomorphism 
$$[u]^*(W\otimes\calO)\cong \calW. $$ 
Note also that $\calP$ is the fiber product of a principal $G$-bundle $\calP_G:=\calP/F$ and a principal $F$-bundle $\calP_F:=\calP/G$ over the base. 
%\yl{Give a reference or explanation?}
Let $\mathfrak g$ and $\mathfrak f$ denote the Lie algebra of $G$ and $F$ respectively. We have 
$$[u]^*(\mathfrak f\otimes\calO)\cong \calP\times_{G\times F}\mathfrak f
\cong\calP_F\times_F\mathfrak f:=ad_{\mathfrak f}\calP,$$ 
which is the adjoint $\mathfrak f$-bundle. Similarly, denote the corresponding adjoint $\mathfrak g $-bundle by $ad_{\mathfrak g}\calP$.  

To sum up, by using \eqref{equ on cot cpx of C/H}, \eqref{tang of mapp stac}, \eqref{tang of bunf}, \eqref{tang cpx tri}, we have the following (relative) tangent complexes.
\begin{prop}
There are canonical isomorphisms 
\begin{align*}
\bbT_{\bMap(C,[\bCrit^{}(\phi)/(G\times F)])} 
&\cong \pi_*[u]^*\bbT_{[\bCrit^{}(\phi)/(G\times F)]} \\
&\cong \pi_*[u]^*\left((\mathfrak f\oplus\fg)\otimes\calO\to W\otimes\calO\to W^*\otimes\bbC_\chi\otimes\calO \right) \\ \nonumber 
&\cong \pi_*\big(ad_{\fg}\calP\oplus ad_{\mathfrak f}\calP \to \calW\to \calW^\vee\otimes (\calP\times_{G\times F}\bbC_\chi) \big), \end{align*}
\begin{align}\label{equ for rel cot cpx2}
\bbT_{\bMap(C,[\bCrit^{}(\phi)/(G\times F)])/\fBun_{G\times F}(C)} 
&\cong \pi_*[u]^*\left( W\otimes\calO\to W^*\otimes\bbC_\chi\otimes\calO \right) \\ \nonumber 
&\cong \pi_*\big(\calW\to \calW^\vee\otimes (\calP\times_{G\times F}\bbC_\chi) \big). \end{align}
%(ii) For the second case in \eqref{equ for z}, with similar notations, we have 
%\begin{align*} \bbT_{\bMap(C,[\bfZ(d\phi)/F])/\fBun_F(C)}&\cong \pi_*[u]^*\left(\fg\otimes\calO\to W\otimes\calO\to W^*\otimes\bbC_\chi\otimes\calO\to \fg^*\otimes\bbC_\chi\otimes\calO\right) \\ \nonumber %&\cong \pi_*\big(ad_{\fg}\calP\to \calW\to \calW^\vee\otimes (\calP\times_{G\times F}\bbC_\chi)\to (ad_{\fg}\calP)^\vee\otimes (\calP\times_{G\times F}\bbC_\chi) \big). \end{align*}
\end{prop}

\subsection{Shifted symplectic structures on $\sigma$-twisted derived mapping stacks I}\label{sect on sss I}
 
%Let $Z$ be a derived Artin stack (locally finitely presented) over a Noetherian commutative ring $k$. 
%Let $\fBun_F^{\chi=\omega}(C)$ be the stack classifying pairs $(\overline{P},\overline{\varkappa})$, where $\overline{P}$ is a principal $F$-bundle on $C$ and 
%\begin{equation}\label{fix twis}\overline{\varkappa}:\overline{P}\times_F\bbC_\chi\cong \omega_{C/k} \end{equation}
%is an isomorphism. It has a map to $\fBun_F(C)$ which forgets $\overline{\varkappa}$. 
%Define its derived enhancement $\fBun_F^{\chi=\omega}(C)$ by the Cartesian diagram:
%\begin{equation}\label{equ on def deri of bun chi}\begin{xymatrix}{\fBun_F^{\chi=\omega}(C) \ar[d]\ar[r]&\fBun_F(C)\ar[d]\\\fBun_F^{\chi=\omega}(C)\ar[r]&\fBun_F(C),}\end{xymatrix} \end{equation}
%where the bottom horizontal map is the forgetful map and the right vertical map is the truncation to the underlying classical stack.
%The base-change in the following Cartesian diagram defines $\bMap^{\chi=\omega}(C,[Z/F])$:
Continue with the setting of the previous section and furthermore fix a character $\chi: H\to \mathbb{C}^*$. 

Consider a derived version of diagram 
\eqref{diagram defi R twist map},~i.e.~we define $\bMap^{\chi=\omega}(C,[Y/H])$ by the homotopy pullback diagram:  
\begin{equation}\label{diag on derived map stk1}\begin{xymatrix}{
\bMap^{\chi=\omega}(C,[Y/H])\ar[d]\ar[r] \ar@{}[dr]|{\Box} &\bMap(C,[Y/H])\ar[d]\\
\fBun_H^{\chi=\omega}(C)\ar[r]&\fBun_H(C), }
\end{xymatrix}\end{equation}
where $\fBun_H^{\chi=\omega}(C)$ is the moduli stack of principal $H$-bundle $P$ on $C$ with $\varkappa:P\times_{H}\mathbb{C}_\chi\cong\omega_{C/k}$, the lower horizontal map is the forgetful map forgetting $\varkappa$ and the right vertical map is induced by $[Y/H]\to [\pt/H]$. 

\begin{definition}
Fix a $k$-point $\sigma$ of $\fBun_H^{\chi=\omega}(C)$, represented by a $H$-bundle $P$ on $C$ with an isomorphism $\varkappa$ as above.
We refer to $\sigma=(P,\varkappa)$ as a {\it twist data}, and define the derived moduli stack $\bMap^\sigma(C,[Y/H])$ of \textit{$\sigma$-twisted maps}\footnote{The idea of doing twist is not new. See \cite{Kim, Dia, O} for examples.} to $Y$ by the following homotopy pullback diagram:
%\begin{equation}\label{equ on sigma twist map}\begin{xymatrix}{
%\bMap^\sigma(C,[Y/H])\ar[d]\ar[r] \ar@{}[dr]|{\Box} &\Spec(k) \ar[d]^{\sigma} \\
%\bMap^{\chi=\omega}(C,[Y/H]) \ar[r]& \fBun_H^{\chi=\omega}(C). }
%\end{xymatrix}\end{equation}
\begin{equation}\label{equ on sigma twist map}\begin{xymatrix}{
\bMap^\sigma(C,[Y/H])\ar[d]\ar[r] \ar@{}[dr]|{\Box} & \bMap^{\chi=\omega}(C,[Y/H]) \ar[d]^{} \\
  \Spec(k)  \ar[r]^{\sigma\quad} & \fBun_H^{\chi=\omega}(C). }
\end{xymatrix}\end{equation}
%\begin{equation}\bMap^\sigma(C,Y):=\bMap^{\chi=\omega}(C,[Y/H])\times_{\fBun_H^{\chi=\omega}(C)}^\bfL \Spec(k) \end{equation}
\end{definition}
%Let $\pi:C\times \bMap^\sigma(C,Y)\to \bMap^\sigma(C,Y)$ be the projection, $\calP$ be the universal $H$-bundle on $C\times \bMap^\sigma(C,Y)$, and $\calW$ the induced $W$-bundle.By \eqref{equ for rel cot cpx}, we have 
%\[\bbT_{\bMap^\sigma(C,Y)}=\pi_* \left(ad_{\fg}\calP\right)\to \pi_* \left(\calW\right)\to \pi_*\left(\calW^\vee\boxtimes\omega_{C/k}\right)\to \pi_*\left((ad_{\fg}\calP)^\vee\boxtimes\omega_{C/k}\right).\] 
The goal of this section is to show that if $Y$ has an $n$-shifted symplectic structure that transforms under $H$ as $\chi^{ }$ (see Definition \ref{def of trans}), 
then $\bMap^\sigma(C,[Y/H])$ has an induced $(n-1)$-shifted symplectic structure. 
This follows from an AKSZ-type construction as \cite[\S 2.1]{PTVV}. The basic idea of \textit{loc.\,cit.}~is as follows: when $H=\{1\}$, we know $\omega_{C/k}\cong \oO$, $Y$ has an $n$-shifted symplectic structure and 
$$\quad \bMap^\sigma(C,[Y/H])=\bMap(C,Y). $$
The $(n-1)$-shifted symplectic structure on this mapping stack is given by 
the pullback of the symplectic structure of $Y$ via the evaluation map 
$$C\times \bMap(C,Y)\to Y$$
and then integrating along $C$ via Serre duality pairing $C(C,\oO)\to k[-1]$. 
For general $H$, we first introduce the notion of shifted symplectic structures that transform under $H$ as $\chi: H\to \mathbb{C}^*$ (ref.~Definition \ref{def of trans}) and then define descent to the stack quotient by $H$ (ref.~Lemma \ref{lem on descent with tw}). Finally we explain  how to do integration on $C$ (ref.~Eqn.~\eqref{eqn:kappa}). 
We prove the existence of shifted symplectic structures in Theorem \ref{thm:symp_no_mark}. 

%\yl{Need to rewrite the remaining part of this section.}

We start with some preparation work which we follow closely the construction and argument as in \cite[\S 1.1 \& \S 1.2]{PTVV}. 
For any derived stack $F$, we have its $\infty$-category of quasi-coherent complexes
\[
\mathbb{L}_\mathrm{Qcoh}(F).
\]
Let $k$ be a Noetherian commutative ring, and $H$ a reductive group scheme on $k$. We write  \[
dg_k^H:=\mathbb{L}_\mathrm{Qcoh}(BH), \,\ dg_k^{gr,H}:=\mathbb{L}_\mathrm{Qcoh}(BH\times B\mathbb{G}_m), \,\ \epsilon\hbox{-}dg_k^{gr,H}:=\mathbb{L}_\mathrm{Qcoh}(BH\times B(\mathbb{G}_m\ltimes \mathbb{G}_a[1])),  
\]
refereed to as the $\infty$-category of $H$-equivariant complexes of $k$-modules, $H$-equivariant graded  complexes of $k$-modules, 
and $H$-equivariant graded mixed complexes of $k$-modules respectively. 

We have an $\infty$-functor
\[(-)^H: \epsilon\hbox{-}dg_k^{gr,H}\to \epsilon\hbox{-}dg_k^{gr}\]
 obtained by pushing forward along the projection \[BH\times B(\mathbb{G}_m\ltimes \mathbb{G}_a[1])\to B(\mathbb{G}_m\ltimes \mathbb{G}_a[1]).\]
Similarly, if $f: H\to G$ is a group scheme homomorphism, we also have the restriction functor, which is an $\infty$-functor 
\begin{equation}
    \label{eqn: restriction1}
f^*: \epsilon\hbox{-}dg_k^{gr,G}\to \epsilon\hbox{-}dg_k^{gr,H}.
\end{equation}
The special case when $H=\Spec k$ is the trivial $k$-group scheme gives a \textit{forgetful functor} 
\[\epsilon\hbox{-}dg_k^{gr,G}\to \epsilon\hbox{-}dg_k^{gr}, \]
to the $\infty$-category of graded mixed complexes of $k$-modules.

In what follows we suppress the forgetful functor from notations when not causing confusions.
%The precise definitions will be given in the proof below. 
%We avoid discussion about the comparison with the $\infty$-category $L_{qcoh}([\pt/F])$ in \cite{PTVV}.

%Let $f:[Y/H]\to [\pt/H]$ be induced by the structure morphism $Y\to \pt$. The following lemma is an analogy of the construction in \cite[Def.~1.13]{PTVV} over $[\pt/H]$, see also \cite[Prop.~1.3.8]{CPTVV}. 

%The $\bfD R$ here is called the {\it internal} de Rham complex in {\it loc.~cit.}. 
\begin{lemma}\label{lem:F-eq_DR}
Let $Y$ be a derived Artin stack (locally of finite presentation) over a Noetherian commutative ring $k$, endowed with an action of a reductive $k$-group scheme $H$. 
Then $\bbL_Y$ is an $H$-equivariant complex. 
Moreover, both de Rham algebra $\bfD R(Y/k)$ and weighted negative cyclic complex $NC^w(\bfD R(Y/k))$ are $H$-equivariant complexes.
\end{lemma}
\begin{proof}
Following \cite{CPTVV}, for any derived stack $F$, we have its $\infty$-category of quasi-coherent algebras
\[\mathrm{cdga}_F := \mathrm{CAlg}(\mathbb{L}_\mathrm{QCoh}F).\] We write 
\[\epsilon\text{-}\mathrm{cdga}^{\mathrm{gr}}_{F} := \mathrm{cdga}_{F\times {B(\mathbb{G}_m\ltimes \mathbb{G}_a[1])}}\hbox{ and }
\mathrm{cdga}^{\mathrm{gr}}_{F} := \mathrm{cdga}_{F\times{B\mathbb{G}_m}}.\] 
%Let $NC^w: \epsilon\text{-}\mathrm{dg}_{\mathrm{gr}}(F)\to \mathrm{dg}_{\mathrm{gr}}(F)$ be the  $\infty$-functor given by the pushforward along the projection \[F\times B(\Gm\ltimes \bbG_a[1])\to F\times B(\Gm).\]
There is an $\infty$-functor
\[
(-){(0)} \colon \epsilon\text{-}\mathrm{cdga}^{\mathrm{gr}}_F
\;\longrightarrow\;
\mathrm{cdga}_F,
\]
sending \( A \in \epsilon\text{-}\mathrm{cdga}^{gr}_F \) to its
 weight-zero part \( A{(0)} \in \mathbb{L}_\mathrm{QCoh}F\), which carries a natural commutative monoid structure, 
hence  an object \( A{(0)} \in \mathrm{cdga}_F \).
This functor admits a left adjoint \cite[Propositions~1.3.8]{CPTVV}:
\[
\mathrm{DR}^{\mathrm{int}} \colon
\mathrm{cdga}_F
\;\longrightarrow\;
\epsilon\text{-}\mathrm{cdga}^{\mathrm{gr}}_F.
\]
Let $\mathrm{dSt}_{/F}$ be the $\infty$-category of derived stacks over $F$.
When $F$ is affine, define \[\mathrm{DR}_F \colon
\mathrm{dSt}_{/F}^{\mathrm{op}}
\;\longrightarrow\;
\epsilon\text{-}\mathrm{cdga}^{\mathrm{gr}}_{F}.\] as the \emph{right Kan extension} of $\mathrm{DR}^{\mathrm{int}} $. For general $F$, set
\[
\mathrm{DR}_F 
:= 
\lim_{\substack{b : T \to F \\ T\ \mathrm{affine}}}
b^{*} \circ \mathrm{DR}_T \circ ( - \times_F T ),
\]
where the limit is taken over all morphisms $b : T \to F$ from derived affine schemes $T$.

Let $Y$ be a derived Artin stack  with an action of  $H$. The $H$-equivariant  cotangent complex $\bbL_Y$ is defined to be $\bbL_{[Y/H]/BH}$ as an object in $\mathbb{L}_\mathrm{QCoh}[Y/H]$.
The $H$-equivariant de Rham complex is defined to be $\bfD R(Y/k):=\mathrm{DR}_{BH}([Y/H]/BH)$, which is an object in $\epsilon$-$dg_k^{gr,H}$. 
The $H$-equivariant weighted negative cyclic complex is $NC^w(\bfD R(Y/k)):=NC^w(\bfD R(Y/k))$ as an object in $dg_k^{gr,H}$.
\end{proof}
\begin{remark}
We will write $\bfD R(Y)=\bfD R(Y/k)$ and $NC^w(\bfD R(Y))=NC^w(\bfD R(Y/k))$ for simplicity when the base ring $k$ 
is clear from the context.
\end{remark}
%Recall $NC^w:\epsilon\hbox{-}dg_k^{gr,H}\to dg_k^{gr,H}$ from \cite[Def.~1.2]{PTVV}. 
For two objects in the $\infty$-category $\epsilon\hbox{-}dg_k^{gr,H}$, the inner homomorphism $\dR\mathcal{H}om_{\epsilon\hbox{-}dg_k^{gr,H}}(-,-)$ gives an object in $\epsilon\hbox{-}dg_k^{gr, H}$. 
%simplicial $k$-module of morphisms which can be enhanced with a grading. 
%We write $\RHom_{dg_k^{gr,H}}(-,-)$ for the $dg_k^{gr,H}$-enhanced morphisms.  
Given a group scheme homomorphism $\chi: H\to \mathbb{G}_m(k)$, we get an object 
$k_\chi\in \epsilon\hbox{-}dg_k^{gr,H}$, concentrated in weight $0$ and homological degree $0$. 
For any object $M\in \epsilon\hbox{-}dg_k^{gr,H}$, we write the ``$\chi$-\textit{eigenspace}" of $M$ as 
\begin{equation}\label{equ on chi tw}
M_\chi:=(\dR\mathcal{H}om_{\epsilon\hbox{-}dg_k^{gr,H}}(k_\chi,M))^H\in \epsilon\hbox{-}dg_k^{gr}. \end{equation}
%\yl{Is \eqref{eqn: restriction2} true if the map $M\to N$ is not iso?}
Then $M_\chi\otimes_k k_\chi$ has a natural $H$-equivariant structure and there is a natural morphism in $\epsilon\hbox{-}dg_k^{gr,H}$:
\[M_\chi\otimes_k k_\chi\to M.\]
\begin{remark}
Let $f:H\to G$ be a surjective group scheme homomorphism, $M\in \epsilon\hbox{-}dg_k^{gr,H}$, $N\in \epsilon\hbox{-}dg_k^{gr,G}$, and $M\to N$ be an equivariant map of graded mixed complexes. 
For a group scheme homomorphism $\chi: G\to \mathbb{G}_m(k)$, let $f^*\chi:=\chi\circ f: H\to \mathbb{G}_m(k)$, 
then we have a map $M_{f^*\chi}\to N_\chi$ of graded mixed complexes making the following diagram commutative
\begin{equation}
    \label{eqn: restriction2}
    \xymatrix{
    M_{f^*\chi}\otimes k_{f^*\chi}\ar[r]\ar[d]& N_\chi\otimes k_\chi\ar[d]\\
    M\ar[r]& N.}
\end{equation}
\end{remark}
\begin{definition}\label{defi of ncchi}
We define $\infty$-functors:
$$NC_\chi(p): \epsilon\hbox{-}dg_k^{gr,H}\to dg_k^{}, \quad M\mapsto NC(p)(M_\chi). $$ 
$$NC_\chi^w:=\bigoplus_p NC_\chi(p): \epsilon\hbox{-}dg_k^{gr,H}\to dg_k^{gr}. $$ 
%$$NC^w_\chi(-):=\RHom_{dg_k^{gr,H}}(\chi,NC^w(-)): \epsilon\hbox{-}dg_k^{gr,H}\to dg_k^{gr}. $$ 
%which we define in the proof. \yl{isn't this a definition?} We spell this out for concreteness. 
\end{definition}
%By Lemma \ref{lem:F-eq_DR}, for any derived Artin stack $Y$ over $k$, which is endowed with an action of a $k$-group scheme $H$, we have 
%$$NC^w(Y):=NC^w(\bfD R(Y))\in dg_k^{gr,H}. $$
For any derived Artin stack $Y$ over $\mathbb{C}$, which is endowed with an action of a complex reductive group $H$, by Lemma \ref{lem:F-eq_DR}, we have 
$$NC^w(Y):=NC^w(\bfD R(Y))\in dg_{\mathbb{C}}^{gr,H}. $$
\textit{Complete reducibility} implies a decomposition 
%which is a direct sum of $NC^w_\chi(Y)\otimes \chi$ based on irreducible representation of $H$ as
$$NC^w(Y)=\bigoplus_{\chi\in \mathrm{Irrep}(H)} NC^w_\chi(Y)\otimes \mathbb{C}_\chi $$
based on irreducible representations of $H$. 
%\begin{remark}
Given a character $\chi: H\to \mathbb{C}^*$, we then have an inclusion 
%via the standard representation of $\mathbb{C}^*$ on $\mathbb{C}$, one gets an irreducible $H$-representation $\chi$ on $\mathbb{C}$. 
$$NC^w_\chi(Y)\otimes \mathbb{C}_\chi\to NC^w(Y). $$
By forgetting the $H$-action, we have a map of graded complexes: 
\begin{equation}\label{ncwmap}NC^w_\chi(Y)\to NC^w(Y) \end{equation}
to the underlying complex of graded $\mathbb{C}$-modules of $NC^w(Y)$ (without causing confusion, here we use the same notation $NC^w(Y)$ for the underlying 
graded complex). 
%\yl{This does not seem to require an identification of $\chi$ with $k$. Does this make sense to you?}
%\end{remark}

Now we are ready to define shifted symplectic structures which transform under $H$ as $\chi$. 
\begin{definition}\label{def of trans}
Let $Y$ be a derived Artin stack over $\mathbb{C}$, endowed with an action of a complex reductive group $H$. 
We say that $Y$ has an $n$-\textit{shifted symplectic structure that transforms under} $H$ \textit{as} $\chi: H\to \mathbb{C}^*$ if 
there is a morphism 
%\yl{in pp. 308 \cite{PTVV}, they wrote $k(2)\to NC_\chi^w(Y)[n-2]$}
$$\Omega: \mathbb{C}[2-n](2)\to NC^w_\chi(Y)$$
of graded complexes of $\mathbb{C}$-modules whose composition with the map \eqref{ncwmap}:
$$NC^w_\chi(Y)\to NC^w(Y)$$
%forgetting the $H$-action
%$$\RHom_{dg_k^{gr,H}}(\chi,NC^w(Y))\mapsto NC^w(Y) $$
defines a $n$-shifted symplectic structure in the sense of \cite[Def.~1.18]{PTVV},~i.e.~the underlying 2-form is non-degenerate. 
%\yl{our $NC^w(Y)(2)$ lies in the $H$-equivariant $dg_k^{gr,H}$. What is the relation with PTVV's setting?}
%$$NC^w(Y)(2)[n-2]\to \RHom(\calO_Y,\wedge^2\bbL_Y[n])$$ is non-degenerate \cite[Def.~1.18]{PTVV} \yl{No $\chi$-twist? Do we take $|NC^w(Y)(2)[n-2]|$? }.
\end{definition}
\begin{remark}
There is a well-defined simplicial set of such shifted symplectic structures. We do not need it here
as we are only concerned with such a structure up to homotopy.
\end{remark}
\begin{remark}
When $k$ is a Noetherian commutative $\mathbb{C}$-algebra, 
by pullback along the structure map $\Spec k \to \Spec \mathbb{C}$, we get a map 
$$\Omega: k[2-n](2)\to NC^w_\chi\left((Y\times \Spec k)/\Spec k\right)$$
of graded complexes of $k$-modules whose composition with 
$$NC^w_\chi\left((Y\times \Spec k)/\Spec k\right) \to NC^w\left((Y\times \Spec k)/\Spec k\right)  $$
defines a $n$-shifted symplectic structure for $Y\times \Spec k$ over $k$. 
\end{remark}
For a $H$-equivariant map $f: A\to B$ between derived Artin stacks over $k$. We can extend Lemma \ref{lem:F-eq_DR} and 
define a \textit{relative de Rham complex} (as \cite[Def.~2.4.2]{CPTVV}):
$$\bfD R(A/B)\in  \epsilon\emph{-}dg_k^{gr,H}. $$
Given a $H$-equivariant commutative square of derived Artin stacks
\[\xymatrix{
A\ar[r]^f\ar[d]_l&B\ar[d]^h\\
M\ar[r]^g&N, 
}\]
by the canonicity in \cite[Prop.~2.4.3]{CPTVV}, we have maps of relative de Rham complexes (in $\epsilon\emph{-}dg_k^{gr,H}$)
\begin{equation}\label{eq canonicity}\bfD R(B/N) \to \bfD R(A/N) \to \bfD R(A/B), \quad \bfD R(M/N) \to \bfD R(A/N) \to \bfD R(A/M), \end{equation}
and in particular 
\begin{equation}\label{eqn:DR_rel}f^*: \bfD R(B/N)\to \bfD R(A/M). \end{equation}
%When $H$ acts on both $A$ and $B$ with $f$ being equivariant, and $H$ acts trivially on $N$ and $M$, 
%Recall Eqn.~\eqref{equ on chi tw}, we get an induced map   \begin{equation}\label{eqn:twist_DR_rel}\bfD R(B/N)_\chi\to \bfD R(A/M)_\chi. \end{equation}
The following lemma relates the invariant part of the relative de Rham algebra with the relative de Rham algebra of the corresponding stack quotients, 
which can be seen as a descent of equivariant forms to the stack quotient. 
\begin{lemma}\label{lem on dR cpx iso}
Let $f: Y\to W$ be a $H$-equivariant map between derived Artin stacks over $\mathbb{C}$, endowed with actions of a complex reductive group $H$. 
Then there is an equivalence 
\[\bfD R([Y/H]/[W/H])\cong \bfD R(Y/W)^H  \]
in $\epsilon\hbox{-}dg^{gr}_{\mathbb{C}}$.  
\end{lemma}
\begin{proof}
We have the following homotopy pullback diagram of derived stacks: 
%\[\xymatrix{Y \ar@/_1.8pc/[dd]_{p}  \ar[r]^{p_Y\quad} \ar[d]_{f}& [Y/H] \ar[d]^{ } \ar@/^1.8pc/[dd]^{q}   \\W \ar[r]^{p_W\quad} \ar[d]_{}& [W/H] \ar[d]^{} \\ \pt \ar[r]^{p_{\pt}\quad }& [\pt/H] }\]
\[\xymatrix{
Y    \ar[r]^{p_Y\quad} \ar[d]_{f} \ar@{}[dr]|{\Box} & [Y/H] \ar[d]^{ }    \\
W \ar[r]^{p_W\quad} & [W/H],   }\]
where $p_Y, p_W$ are quotient maps and right vertical map is the quotient of $f$. 
This implies a $H$-equivariant isomorphism 
\begin{equation}\label{equ on iso of ct on Y/W}p_Y^*\bbL_{[Y/H]/[W/H]}\cong \bbL_{Y/W}, \end{equation}
and a map in $\epsilon\emph{-}dg_{\mathbb{C}}^{gr,H}$ (ref.~Eqn.~\eqref{eqn:DR_rel}):
$$p_Y^*: \bfD R([Y/H]/[W/H])\to \bfD R(Y/W). $$
As the $H$-action is trivial on the LHS, we obtain a map in $\epsilon\emph{-}dg_{\mathbb{C}}^{gr}$:
$$\bfD R([Y/H]/[W/H])\to \bfD R(Y/W)^H. $$
To prove this is an equivalence in $\epsilon\emph{-}dg_{\mathbb{C}}^{gr}$, it is enough to prove the underlying graded complex is an equivalence 
because the forgetful functor 
$$\epsilon\emph{-}dg_{\mathbb{C}}^{gr}\to dg_{\mathbb{C}}^{gr}$$
is conservative (\cite[pp.~292]{PTVV}). As in \cite[Rmk.~2.4.4]{CPTVV}\footnote{Here we use $+1$ shift convention as in \cite[\S 1.2]{PTVV}.}, we have an equivalence in $dg^{gr}_{\mathbb{C}}$:
\begin{align*}
\bfD R([Y/H]/[W/H])&\cong \bigoplus_p\Gamma([Y/H],\Sym^p(\bbL_{[Y/H]/[W/H]}[1])) \\
&\cong \bigoplus_p\RHom_{[Y/H]}(\oO_{[Y/H]},\Sym^p(\bbL_{[Y/H]/[W/H]}[1])) \\
&\cong \bigoplus_p\RHom_{Y}(p_Y^*\oO_{[Y/H]},p_Y^*\Sym^p(\bbL_{[Y/H]/[W/H]}[1]))^H \\
&\cong \bigoplus_p\RHom_{Y}(\oO_{Y},\Sym^p(p_Y^*\bbL_{[Y/H]/[W/H]}[1]))^H \\
&\cong \bigoplus_p\RHom_{Y}(\oO_{Y},\Sym^p(\bbL_{Y/W}[1]))^H \\
&\cong \bfD R(Y/W)^H,  
\end{align*}
where we use \eqref{equ on iso of ct on Y/W} in the fifth equality. 
\end{proof}
Next we introduce a twisted version of the above lemma. We first define $\calL$-twisted relative de Rham complexes. 
\begin{definition}\label{defi of twis dr}
Let $Y$ and $W$ be derived Artin stacks over a Noetherian commutative $\mathbb{C}$-algebra $k$, $\calL$ be a line bundle on $Y$ and $\mathring{\calL}$ 
denote the assciated $\bbC^*$-bundle with a map $\mathring{\calL}\to W$. 

We define the $\calL$-\textit{twisted relative de Rham complex}
$$\bfD R^{\calL}(Y/W):=\bfD R(\mathring{\calL}/W)_{(-1)}\in \epsilon\emph{-}dg_{k}^{gr}$$ 
to be the $\bbC^*$-weight $(-1)$ part of $\bfD R(\mathring{\calL}/W)$. 
%\yl{when twist is trivial, it does not go back to usual de Rham complex}
\end{definition}
Now we state a twisted version of Lemma \ref{lem on dR cpx iso}, which gives a descent of twisted equivariant forms to the quotient stack. 
%\yl{I rewrite the proof of the following lemma, also taking care of $\chi$ v.s. $\chi^{-1}$, see whether you agree.}
\begin{lemma}\label{lem on descent with tw}
Let $f: Y\to W$ be an $H$-equivariant map between derived Artin stacks over $\mathbb{C}$, endowed with actions of a complex reductive group $H$. 
Let $\chi: H\to \mathbb{C}^*$ be a nontrivial character of $H$ and $H_0:=\Ker\chi$.  
Let $\calL_\chi:=[(Y\times\bbC_{\chi^{-1}})/H]$  be the line bundle on $[Y/H]$ with $\mathring{\calL}_\chi$  the associated $\bbC^*$-bundle.
Then there is an isomorphism in $\epsilon\emph{-}dg_{\mathbb{C}}^{gr}$:
\begin{equation}\label{eqn:chi_L0}\bfD R(Y/W)_{\chi^{ }}\cong \bfD R^{\calL_\chi}([Y/H]/[W/H_0]). \end{equation}
By further applying $NC^w(-)$, we obtain  
\begin{equation}\label{eqn:chi_L}
   o: NC^w_{\chi^{ }}(Y/W)\cong NC^w(\bfD R^{\calL_\chi}([Y/H]/[W/H_0])).
\end{equation}
\end{lemma}
\begin{proof}
By the surjectivity of $\chi$, the identity section $Y\to Y\times \bbC^*_{\chi^{-1}}$ induces an isomorphism of quotient stacks 
\begin{equation}\label{equ on iso of Y/H0} [Y/H_0]\cong [(Y\times \bbC_{\chi^{-1}}^*)/H]=:\mathring{\calL}_\chi,  \end{equation}
where we use the convention that the $H$-action on $Y$ is on the left.
Recall Definition \ref{defi of twis dr},
\begin{equation}\label{equ on dR rel1}\bfD R^{\calL_\chi}([Y/H]/[W/H_0])=\bfD R(\mathring{\calL}_\chi/[W/H_0])_{(-1)}. \end{equation}
There is a homotopy pullback diagram (with horizontal maps being quotients by $H_0$): 
\[\xymatrix{
Y\ar[r] \ar[d]_f  \ar@{}[dr]|{\Box} &[Y/H_0]   \ar[d] \\
W\ar[r] &[W/H_0]. 
}\]
Lemma~\ref{lem on dR cpx iso} and Eqn.~\eqref{equ on iso of Y/H0} then imply the following isomorphisms 
\begin{equation}\label{equ on iso of Y/H022} \bfD R(Y/W)^{H_0}\cong \bfD R([Y/H_0]/[W/H_0])\cong \bfD R(\mathring{\calL}_\chi/[W/H_0]). \end{equation}
The action of $H$ on $\bfD R(Y/W)$ induces an action of $\bbC^*_\chi=H/H_0$ on
$\bfD R(Y/W)^{H_0}$ making the above isomorphisms $\bbC^*_\chi$-equivariant. 

Restriction from the group $H$ to $H_0$ gives a map 
\begin{equation}\label{equ on H0 H}\left(\bfD R(Y/W)\otimes \mathbb{C}_\chi\right)^H=\left(\bfD R(Y/W)^{H_0}\otimes \mathbb{C}_\chi\right)^H\to \bfD R(Y/W)^{H_0}. \end{equation}
Here the $H$-action on $\bfD R(Y/W)$ (resp.~$\bbC_\chi$) is from the left (resp.~right). Recall that  
for an representation $V$ of $H$, the weight spaces arising from  left and  right actions are related by 
$$V^{\mathrm{right}}_{\chi^{}}:=\big\{v\in V\,|\, v\cdot h=\chi^{}(h^{-1})\cdot v\,, \,\forall\,\, h\in H\big\}=\big\{v\in V\,|\,h\cdot v=\chi^{-1}(h)\cdot v\,, \,\forall\,\, h\in H\big\}=:V^{\mathrm{left}}_{\chi^{-1}}.  $$
Therefore with left $H$-actions on both $\bfD R(Y/W)$ and $\bbC_\chi$, Eqn.~\eqref{equ on H0 H} becomes a map  
\begin{equation*}\bfD R(Y/W)_{\chi^{ }}:=\left(\bfD R(Y/W)\otimes \mathbb{C}_{\chi^{-1}}\right)^H\to \bfD R(Y/W)^{H_0}. \end{equation*}
%Then Eqn.~\eqref{eqn: restriction2} implies the existence of the following commutative diagram: 
Then we have the following commutative diagram: 
\begin{equation}\label{equ on iso of Y/H023}\xymatrix{
(\bfD R(Y/W)\otimes\mathbb{C}_{\chi^{-1}})^H\ar[r]^{} \ar[d]&(\bfD R([Y/H_0]/[W/H_0]))_{(-1)} \ar[d]\\
\bfD R(Y/W)^{H_0}\ar[r]^{\cong \quad \quad }&\bfD R([Y/H_0]/[W/H_0]), }  \end{equation}
where the $H$-action on $\bfD R(Y/W)^{H_0}$ becomes $\bbC^*$-action on $\bfD R([Y/H_0]/[W/H_0])$ under the map $\chi$. 
By \eqref{eqn: restriction2}, $\chi$-eigenspace maps to weight $(-1)$ eigenspace, i.e. implying the upper horizontal map.
Combining Eqns.~\eqref{equ on dR rel1}, \eqref{equ on iso of Y/H022} and upper horizontal map of diagram \eqref{equ on iso of Y/H023}, we are done. 
\end{proof}

Next we define the integration map. 
Let $X$ and $Y$ be derived Artin stacks over a Noetherian commutative $\mathbb{C}$-algebra $k$, and $\calL$ a line bundle on $X$. 
Let $\tilde X$ be the  $\bbC^*$-bundle obtained by removing the zero-section from the total space of $\calL$. 
The space $\tilde X$ is almost never $\calO$-compact in the sense of \cite[Def.~2.1]{PTVV} since for a perfect complex $E$ on $\tilde X_A:=\tilde X\times\bSpec A$, the dg-module over $A$:
\[C(\tilde X_A,E):=\RHom(\calO,E)\] 
is rarely perfect. Nevertheless the following notion is enough for our purpose. 
\begin{definition}\label{defi of equiv O cpt}
We say $X$ is {\it $\bbC^*$-equivariantly $\calL$-compact} if for any cdga $A$ over $k$, $\oO_{\tilde X_A}$ is a compact object in $D_{qcoh}(\tilde X_A)$ and each graded component of
$C(\tilde X_A,\oO_{\tilde X_A})$, i.e., $C(\tilde X_A,\calL^k)$ for each $k\in\bbZ$, is perfect as a dg-module over $A$. 
\end{definition}
\begin{remark}\label{rmk on o cpt}
A proper Deligne-Mumford stack $X$ (considered as a derived stack) is $\bbC^*$-equivariantly $\calL$-compact for any line bundle $\calL$ on $X$.
%see paragraph before Def. 4.4 of \cite{Toen1}.
% see page 4 of ``Integral Transforms and Drinfeld Centers in Derived Algebraic Geometry" for ref.
\end{remark}
\begin{lemma}
%\yl{Say sth how to conclude the following equality from above definition.}
For $X$ and $\tilde X$ as in Definition \ref{defi of equiv O cpt} and any derived Artin stack $Y$, we have a natural equivalence in $\epsilon\emph{-}dg_{k}^{gr,\bbC^*}$:
%\yl{shouldn't the following map be an equivalence?}
\begin{equation}\label{equ on kapp}\kappa_{Y,\tilde X}:\bfD R((\tilde X\times_kY)/\tilde X)\stackrel{\cong}{\to} \bfD R( Y)\otimes_kC(\tilde X,\calO), \end{equation}
where the $\bbC^*$-equivariant dg-module $C(\tilde X,\calO)$ has weight $0$ with trivial mixed structure. 
\end{lemma}
\begin{proof}
This follows from a similar construction as \cite[pp.~305]{PTVV} which we briefly recall. 
We are indeed constructing a natural equivalence
\begin{equation}\label{equ on nat tr}\bfD R((\tilde X\times_k-)/\tilde X)\to \bfD R( -)\otimes_kC(\tilde X,\calO)\end{equation}
between two functors from the $\infty$-category of derived stacks to $\epsilon\emph{-}dg_k^{gr,\bbC^*}$. The functor $\bfD R$ by construction sends 
$\infty$-colimits to $\infty$-limits. Since $C(\tilde X,\calO)$ is a perfect $\bbC^*$-equivariant dg-module, the tensor functor $-\otimes_kC(\tilde X,\calO)$ preserves $\infty$-limits. Here the limit is taken in the $\infty$-category of $\C^*$-equivariant dg-modules.
Hence, by left Kan extension, it suffices to construct the transform between two functors 
when restricted to derived affine schemes. 

By the natural map  
%$\bfD R(\calO_{\tilde X}/X):=\Sym^*\bbL_{\tilde{X}/X}\to\calO_{\tilde X}$
$\Sym^*\bbL_{\tilde{X}/\tilde X}\stackrel{\cong}{\to}\calO_{\tilde X}$, we know for derived affine schemes $(-)$, there are equivalences
$$\Sym^*(\bbL_{\tilde{X}\times (-)/\tilde X}[1])\cong \Sym^*(\bbL_{\tilde{X}/\tilde X}[1])\otimes \Sym^*(\bbL_{(-)}[1]) \stackrel{\cong}{\to} \calO_{\tilde X}\otimes \bfD R(-). $$
By taking global sections, we obtain  
\[\bfD R((\tilde X\times_k-)/\tilde X)\stackrel{\cong}{\to} C(\tilde X,\bfD R(-)\otimes_k\calO).\]
Using the fact that $\oO_{\tilde X}$ is a compact object in $D_{qcoh}(\tilde X)$, the functor $C(\tilde X,-)$ commutes with colimits and hence we obtain a natural equivalence 
\[\bfD R(-)\otimes_kC(\tilde X,\calO)\stackrel{\cong}{\to} C(\tilde X,\bfD R(-)\otimes_k\calO).\]
The composition of the above two defines the natural transform \eqref{equ on nat tr} on derived affine schemes.
\end{proof}
By Lemma~\ref{lem:F-eq_DR}, both the domain and target of $\kappa_{Y,\tilde X}$ are $\bbC^*$-equivariant. The morphism $\kappa_{Y,\tilde X}$ is also equivariant. 
Notice that the $\bbC^*$-weight $(-1)$  component $C(\tilde X,\calO)_{(-1)}$ is isomorphic to $C(X,\calL)$.  

For $X$ and $\tilde X$ as in Definition \ref{defi of equiv O cpt}, we recall Definition \ref{defi of twis dr}: 
$$\bfD R^\calL((X\times_kY)/\tilde X):=\bfD R((\tilde X\times_kY)/\tilde X)_{(-1)}, $$
which denotes the $\bbC^*$-weight $(-1)$ part of $\bfD R((\tilde X\times_kY/\tilde X))$. 
\begin{definition}
The $\bbC^*$-weight $(-1)$ component of the map \eqref{equ on kapp} is 
\begin{equation}\label{equ on kapp L0}\kappa_{Y,X}^\calL:=(\kappa_{Y,\tilde X})_{(-1)}:\bfD R^\calL((X\times_kY)/\tilde X)\to \bfD R( Y)\otimes_kC(X,\calL).\end{equation}
%Let $p: L\to X$ be a line bundle on $X$. $p_*\oO_{L}=\sum_{i\geqslant 0} L^{-i}$. When $L=\mathbb{C}\to \pt$, $p_*\oO_{L}=\mathbb{C}[u]$. So we see $\mathbb{C}^*$-weight 1 part is $u^{-1}$ and $L$.
Applying functor $NC^w$, we obtain 
\begin{equation}\label{equ on kapp L}\kappa_{Y,X}^\calL:NC^w(\bfD R^\calL((X\times_kY)/\tilde X))\to NC^w(\bfD R( Y)\otimes_kC(X,\calL))\cong NC^w(Y)\otimes_kC(X,\calL).\end{equation}
\end{definition}
Here the isomorphism follows from the $\bbC^*$-equivariantly $\calL$-compactness 
(so that $C(X,\calL)$ is a perfect complex over $k$). 

With the above preparation, we define the integration map. 
\begin{definition}
Assume that $X$ has dimension $d$ and admits a dualizing line bundle $\omega_X$. Let 
$$\mathrm{Serre}: C(X,\omega_X)\to k[-d]$$ 
denote the Serre duality pairing. We define the \textit{integration map} 
\begin{equation}\label{eqn:kappa0} \eta:=\mathrm{Serre}\,\circ\,\kappa^{\omega_X}_{Y,X}: \bfD R^{\omega_X}((X\times_kY)/\tilde X)\to \bfD R(Y)[-d]  \end{equation}
to be the composition of \eqref{equ on kapp L0} with the Serre duality pairing.
By abuse of notation, we also write 
\begin{equation}\label{eqn:kappa} \eta:=\mathrm{Serre}\,\circ\,\kappa^{\omega_X}_{Y,X}: NC^w(\bfD R^{\omega_X}((X\times_kY)/\tilde X))\to NC^w(Y)[-d]  \end{equation}
after applying the functor $NC^w$.  
\end{definition}
%More generally, for any cdga $A$ and perfect complex $E$ on $X_A$, we have
%$$C(X_A,E\otimes\omega_{X_A/A})\cong C(X_A,E^\vee)^\vee[-d]. $$
Note that we have a commutative diagram of graded complexes
\begin{equation*} 
    \begin{xymatrix}{
NC^w(\bfD R^{\omega_X}((X\times_kY)/\tilde X))\ar[r]^{\quad \quad \quad \eta}\ar[d]&NC^w(Y)[-d]\ar[d]\\
\bfD R^{\omega_X}((X\times_kY)/\tilde X)\ar[r]^{ \quad \,\, \eta}&\bfD R(Y)[-d],
}\end{xymatrix}
\end{equation*}
where vertical maps are the projections as \cite[pp.~305]{PTVV}.

%\begin{definition}
%Let $\tilde X$ be the total space of the $\bbC^*$-bundle obtained by removing the zero section of $\omega_X^{-1}$. 
%We say $X$ is $\omega$-compact if $\tilde X$ is $\bbC^*$-equivariantly $\calO$-compact in the sense of Definition \ref{defi of equiv O cpt}. 
%\yl{use other name as there is confusing when $\omega=\oO$.}\end{definition}
%Similarly, let $P$ be a principal $H$-bundle on $X$. For any character $\chi$ of $H$, we then have $$\bfD R(P\times_kY)_\chi\cong \bfD R^{\calL_\chi}(X\times_kY). $$ If furthermore, we are given an isomorphism $\varkappa:\calL_\chi\to \omega_X$, we denote the pair $(P,\varkappa)$ by $\sigma$.
%Let $\bMap^\sigma(X,[Y/H])$ be the fiber of $\bMap(X,[Y/H])\to \bMap(X,[\pt/H])$ over the $k$-point defined by $P$, called the derived stack of $\sigma$-twisted maps.

Now we are ready to prove the main theorem of this section\footnote{After the preparation of the present paper, Pavel Safranov kindly pointed out that a similar result was proven by Ginzburg and Rozenblyum \cite{GR}.}. 
\begin{theorem}\label{thm:symp_no_mark}
Let $X$ be a $d$-dimensional Deligne-Mumford  stack, flat and proper over a Noetherian commutative $\mathbb{C}$-algebra $k$ with a dualizing line bundle $\omega_X$, $Y$ be a derived Artin stack over $\mathbb{C}$, endowed with an action of a complex reductive group $H$. Let $\chi: H\to \mathbb{C}^*$ be a character of $H$ and 
assume $Y$ has an $n$-shifted symplectic structure that transforms under $H$ as $\chi^{ }$. 
Then $\bMap^\sigma(X,[Y/H])$ as defined in \eqref{equ on sigma twist map} has a canonical $(n-d)$-shifted symplectic structure $\Omega_{M}$.
\end{theorem}
%Here $\bMap^\sigma(X,[Z/F])$ is the derived mapping stack \cite[\S 4.3 (4.d)]{Toen1}, which is a derived Artin stack 
%locally of finite type thank to Lurie's representability theorem \cite{Lur} (see also \cite[Cor.~3.3]{Toen2}).
\begin{proof}
We follow closely the argument in \cite[Thm.~2.5]{PTVV}. By base change 
under the structural map $\Spec k\to \Spec \mathbb{C}$, we view $Y$ (resp.~$H$) as a derived stack (resp.~group scheme) over $k$
and often omit writing $\times \Spec k$ for simplicity.  We first construct a closed 2-form on $M:=\bMap^\sigma(X,[Y/H])$. 

By the diagram \eqref{diag on derived map stk1} and \eqref{equ on sigma twist map} (where $C$ is replaced by $X$), we have a commutative diagram:  
%\[\xymatrix{M\times X\ar[r]\ar@/^1.5pc/[rr]^u\ar[d]&\bMap(X,[Y/H])\ar[d]& [Y/H]\ar[d]\\\Spec k\times X\ar[r]\ar@/_1.5pc/[rr]&\bMap(X,BH)&BH}\]
\begin{equation}\label{diag in shif sym str1}\xymatrix{
P  \ar[rr]^{ } \ar[d] & & Y  \ar[d] \\
M\times_k X \ar[r] \ar[d] \ar@/^1.2pc/[rr]^{u} & \bMap(X,[Y/H])\times_k X \ar[r]   \ar[d]& [Y/H] \ar[d] \\
 X  \ar[r]^{\sigma\times_k \id_X \quad \quad \quad \quad } & \bMap(X,BH)\times_k X \ar[r]  &BH,
}\end{equation}
where $P$ is the universal $H$-bundle and $u$ is the universal evaluation map. 
Note that the bundle $P$ is the pullback of an $H$-bundle (denoted by $P_X$ which is determined by the map $X\to BH$ in above) from $X$ by the definition from diagram \eqref{equ on sigma twist map}, i.e. 
\begin{equation}\label{equ on p px}P=P_X\times_k M. \end{equation}
Let $\mathring{\omega_X}$ be the associated $\bbC^*$-bundle of $\omega_X=P_X\times_{H}\mathbb{C}_\chi$ over $X$. As the pullback of $\mathring{\omega_X}$ from $X$ to $\mathring{\omega_X}$ has a canonical section, so the pullback of $P_X$ from $X$ to $\mathring{\omega_X}$ reduces to a $H_0:=\Ker\chi$-bundle (ref.~Lemma \ref{lem on pri bdl}),~i.e.~we have a commutative diagram 
\begin{equation}\label{diag on BH0H}\xymatrix{
\mathring{\omega_X} \ar[r]^{ }\ar[d]&BH_0 \ar[d]\\
X \ar[r]&BH, 
}\end{equation}
where the bottom map defines $P_X$ and is the bottom map in diagram \eqref{diag in shif sym str1}. 
 
Again by diagram \eqref{diag in shif sym str1}, we obtain a map $P\times_H \bbC^*_{\chi^{-1}}\to Y\times_H\bbC^*_{\chi^{-1}}$. Using Eqns.~\eqref{equ on iso of Y/H0},~\eqref{equ on p px}, it becomes 
$$\mathring{\omega_X}\times_k M\to [Y/H_0], $$
which fits into a commutative diagram 
\[\xymatrix{
 \mathring{\omega_X}\times_k M \ar[r]^{ }\ar[d]&[Y/H_0] \ar[d]\\
\mathring{\omega_X}  \ar[r]&BH_0,  }\]
where vertical maps are given by natural projections, the bottom map is the one in diagram \eqref{diag on BH0H}. 

Eqn.~\eqref{eqn:DR_rel} then yields a map  
$$ \bfD R\left([Y/H_0]/BH_0\right) \to \bfD R\left((\mathring{\omega_X}\times_k M)/\mathring{\omega_X}\right). $$
Taking the weight $(-1)$ part of the $\bbC^*$-action, we obtain 
\[u^*:\bfD R^{\calL_\chi}([Y/H]/BH_0)\to \bfD R^{\omega_X}\left((X\times_k M)/\mathring{\omega_X}\right), \]
where $\calL_\chi=Y\times_H\bbC_{\chi^{-1}}$, and we use $[Y/H_0]=[(Y\times \bbC^*_{\chi^{-1}})/H]$ (i.e.~Eqn.~\eqref{equ on iso of Y/H0}). 

Composing with the map \eqref{eqn:chi_L0}, i.e.
$$o: \bfD R(Y)_{\chi^{}}\to \bfD R^{\calL_\chi}([Y/H]/BH_0), $$
and the integration map \eqref{eqn:kappa0}, i.e.
$$\eta: \bfD R^{\omega_X}((X\times_kM)/\mathring{\omega_X})\to \bfD R(M)[-d], $$ 
we obtain 
\begin{equation}\label{compo maps0}\eta\circ u^*\circ o: \bfD R(Y)_{\chi^{}}\to \bfD R(M)[-d].\end{equation}
By abuse of notation, we also write 
\begin{equation}\label{compo maps}\eta\circ u^*\circ o:NC^w_{\chi^{}}(Y)\to NC^w(M)[-d] \end{equation}
after applying functor $NC^w$ to \eqref{compo maps0}.

Combining with the $n$-shifted closed 2-form (after using base change from $\mathbb{C}$ to $k$): 
$$\Omega\in \Hom(k[2-n](2), NC^w_\chi(Y)) $$ 
on $Y$ which transforms as $\chi$ (Definition \ref{def of trans}), we obtain an $(n-d)$-shifted closed 2-form on $M$: 
\begin{equation*}\Omega_{M}:=\eta\circ u^*\circ o\circ \Omega: k[2-n](2)\to NC^w(M)[-d]. \end{equation*}
Then it is enough to show the underlying 2-form  is non-degenerate. For this purpose, we explicitly express the underlying 2-form as follows. 
Let $f: \bSpec A\to M$ be an $A$-point of $M$ corresponding to 
$$f:  X_A:=X\times  \bSpec A\to [Y/H]. $$
Equivalently, we have a principal $H$-bundle $P_A$ on $X\times\bSpec A$ with an $H$-equivariant map 
$$\tilde{f}:P_A\to Y. $$ 
Let $\Omega_0$ be the underlying 2-form of $\Omega$ which defines  
$$\Omega_0: \calO_Y\otimes\chi\to \bbL_Y\wedge\bbL_Y[n]. $$
%Here we use $$\RHom_Y(\oO_Y,\bbL_Y\wedge\bbL_Y\otimes \chi^{-1}[n])\cong \RHom_{\pt}(\mathbb{C},)$$
It is easy to check the descent map $o$ commutes with pullback and we will trace $\Omega_0$ under maps: 
pullback via $\tilde{f}$, descent by $H$-action and the integration. 

Pullback of $\Omega_0$ via $\tilde{f}$ gives 
\[\calO_{P_A}\otimes\chi\xrightarrow{\tilde{f}^*\Omega_0} \tilde{f}^*(\bbL_{Y}\wedge\bbL_{Y})[n].\]
This is an equivariant morphism on $P_A$, which by descent,~i.e.~pushforward and then taking $H$-invariants, defines a morphism of sheaves on $X_A$. 
Recall that the total space of the line bundle $\omega$ is $P_A\times_H\chi$. 
%where $\chi$ is a 1-dimensional affine space with coordinate $x$ and $h\bullet x=\chi(h^{-1})x$ for any $h\in H$. 
Taking the descent of $\calO_{P_A}\otimes\chi$ gives $\omega_{X_A/A}^{-1}$. 
%From $\omega=P_A\times_H\chi$, we have $\oO_{\omega}=(\oO_{P_A}\otimes \Sym(\chi^{-1}))^H$. Pushforward gives $p_*\oO_{\omega}=(p_*\oO_{P_A}\otimes \Sym(\chi^{-1}))^H$. Consider weight 1 part of the $\mathbb{C}^*$-action on fibers, then $\omega^{-1}=(p_*\oO_{P_A}\otimes \chi^{-1})^H$. Note that $H$-action on $p_*\oO_{P_A}$, $\chi^{-1}$ is right v.s. left. To make sense $\chi$-eigenspace, need to make both from the left. So we conclude $\omega^{-1}=(p_*\oO_{P_A}\otimes \chi^{})^H$
Hence, we have 
$$(\tilde{f}^*\Omega_0)^{H\emph{-}\mathrm{desc}}: \omega_{X_A/A}^{-1}\to \left(\tilde{f}^*(\bbL_Y\wedge\bbL_Y)\right)^{H\emph{-}\mathrm{desc}}[n], $$
where $(-)^{H\emph{-}\mathrm{desc}}$ denotes the descent of an equivariant object. 

As in \eqref{tang of mapp stac}, we have 
$$\bbT_fM\cong \dR\Gamma\left(X_A,\left(\tilde{f}^*\bbT_Y\right)^{H\emph{-}\mathrm{desc}}\right), $$
where we do not have term involving Lie algebra of $H$ as we have fixed the twist in \eqref{equ on sigma twist map}. 
%By the relative duality, we have 
%$$\bbL_fM\cong \pi_*\left(\left(\tilde{f}^*\bbL_Y\right)^{H\emph{-}\mathrm{desc}}\boxtimes\omega_{X_A/A}\right)[d]. $$ 

Therefore we get a pairing 
$$(\tilde{f}^*\Omega_0)^{H\emph{-}\mathrm{desc}}:\bbT_fM\otimes \bbT_fM\to A[n-d], $$
$$\dR\Gamma\left(X_A,\left(\tilde{f}^*\bbT_Y\right)^{H\emph{-}\mathrm{desc}}\right)\otimes\, \dR\Gamma\left(X_A,\left(\tilde{f}^*\bbT_Y\right)^{H\emph{-}\mathrm{desc}}\right)\to 
\dR\Gamma\left(X_A,\omega_{X_A/A}\right)\cong A[n-d],$$
where the last map is given by the integration map. The non-degeneracy of the above pairing follows easily from the non-degeneracy of $\Omega_0$. 
%where the last isomorphism comes from relative duality. 
\iffalse 
The non-degeneracy of $\Omega_0$ gives a quasi-isomorphism $\bbL_Y^\vee\otimes\chi\cong \bbL_Y[n]$. 
Therefore 
\[\left(\tilde{f}^*\bbL_Y\right)^{H\emph{-}\mathrm{desc}}\cong \left(\tilde{f}^*(\bbL_Y^\vee\otimes\chi)\right)^{H\emph{-}\mathrm{desc}}[-n]\cong \left(\tilde{f}^*\bbL_Y^\vee\right)^{H\emph{-}\mathrm{desc}}\boxtimes\omega_{X_A/A}^{-1}[-n] \] 
are isomorphic as perfect complexes on $X_A$. Here we use the fact that descent commutes with tensor product as pullback does so. 

Let $\pi: X_A\to \bSpec A$ be the projection. As in \eqref{tang of mapp stac}, we have 
$$\bbT_fM\cong \pi_*\left(\left(\tilde{f}^*\bbT_Y\right)^{H\emph{-}\mathrm{desc}}\right), $$
where we do not have term involving Lie algebra of $H$ as we have fixed the twist in \eqref{equ on sigma twist map}. 
By the relative duality, we have 
$$\bbL_fM\cong \pi_*\left(\left(\tilde{f}^*\bbL_Y\right)^{H\emph{-}\mathrm{desc}}\boxtimes\omega_{X_A/A}\right)[d]. $$ 
Therefore
\[\bbT_fM\cong \pi_*\left(\left(\tilde{f}^*\bbL_Y^\vee\right)^{H\emph{-}\mathrm{desc}}\right)\xrightarrow{\pi_*(\tilde{f}^*\Omega_0)^{H\emph{-}\mathrm{desc}}} \pi_*\left(\left(\tilde{f}^*\bbL_Y\right)^{H\emph{-}\mathrm{desc}}\boxtimes\omega_{X_A/A}\right)[n]
\cong (\bbL_fM)[n-d],  \] 
which finishes the proof the non-degeneracy of $\pi_*(\tilde{f}^*\Omega_0)^{H\emph{-}\mathrm{desc}}$. 
\fi
\end{proof}
Here is an application of the above construction. 
\begin{example}
Let $X=C$ be a smooth projective complex curve and $G$ a complex semi-simple algebraic group. Take $F=\bbC^*$ and $Y$ to be the co-adjoint quotient of the Lie algebra $[\fg^*/G]$, so that $F$ acts on $Y$ by scaling on $\fg^*$. It is known that $[\fg^*/G]\cong\bbT^*[1][\pt/G]$, which has a 1-shifted symplectic structure (e.g.~\cite[\S 1.2.3]{Cal}), which transforms under $F$ by scalar.  
Let $\sigma$ be the pair $(P,\varkappa)$, where $P$ is the principal $\bbC^*$-bundle defined as $\omega_C$ with zero-section removed, and $\varkappa$ is the natural isomorphism $P\times_{\bbC^*}\bbC\cong\omega_C$. The derived stack $\bMap^\sigma(X,[Y/F])$ parameterizes pairs $(P',s)$, where $P'$ is a principal $G$-bundle on $C$ and $s$ is a section of $P'\times_G\fg^*\otimes\omega_C$, and hence is a Hitchin moduli stack of (not necessarily stable) Higgs bundles. 
Theorem~\ref{thm:symp_no_mark} then decorates the Hitchin stack with a symplectic structure in the usual sense. 
It coincides with the symplectic structure constructed in \cite[pp.~310]{PTVV} (ref.~\cite[Lem.~4.3]{GR}).
\end{example}
The main relevant application for this paper is when $X=C$ is a proper curve over $k$ with at worst nodal singularities and $Y=\bCrit^{}(\phi)$, $H=G\times F$ in the setting of \S \ref{subsec:crit}. 
\begin{corollary}\label{cor on exist of -2shifted1}
Notations as above, then $\bMap^\sigma(C,[\bCrit^{}(\phi)/H])$ has a canonical $(-2)$-shifted symplectic structure. 
\end{corollary}
\begin{proof}
By Theorem \ref{thm:symp_no_mark}, it suffices to show that the $(-1)$-shifted symplectic form of $\bCrit^{}(\phi)$ transforms under $H$ as $\chi^{ }$. 
%Recall that we define $NC^w_\chi(Y)$ to be $NC^w((\bfD R(Y)\otimes\chi^{-1})^H)$. 
Note that $W$ is a vector space with $H$-action and $T^*W=W\times W^\vee\otimes \chi$ as $H$-representation so that 
$$d_{dR}\phi: W\to T^*W$$
is an equivariant map. Here we use $d_{dR}$ to denote the de Rham differential.  Let $\{x_i\}$ be a  basis of $W$ and $\{y_i\}$ the dual basis on $W^\vee$, the usual symplectic form 
on $U:=T^*W=W\times W^\vee\otimes \chi$ is of the form $\sum_{i}d_{dR}x_{i}\wedge d_{dR}y_i$, 
which is an element in 
\begin{align*}\Hom_{H}\left(\oO_U,\wedge^2T^*U\otimes \chi^{-1}\right)&\cong \Hom_{H}\left(\mathbb{C},\pi_{U*}(\wedge^2T^*U)\otimes \chi^{-1}\right) \\
&\cong \Hom_{}\left(\mathbb{C},\left(\pi_{U*}(\wedge^2T^*U)\otimes \chi^{-1}\right)^H\right),\end{align*}
and transforms under $H$ as $\chi^{}$.  The $(-1)$-shifted symplectic structure on 
$Y=\bCrit^{}(\phi)$ comes from the Lagrangian intersection of the graph of $d_{dR}\phi$ and the zero section \cite[Thm.~2.9]{PTVV} and 
it is easy to see it transforms under $H$ as $\chi^{}$.

One can also see the statement by explicitly calculating the $(-1)$-shifted symplectic form using the local Darboux theorem \cite[Ex.~5.15]{BBJ}. 
Let $W=\Spec A(0)$, then $\bCrit^{}(\phi)=\bSpec A$, where $A$ is a cdga given by the Koszul complex 
$$A=\left(\cdots\to \wedge^2(\Omega_{A(0)}^{1})^{\vee}\otimes \chi^{-2} \xrightarrow{\cdot d_{dR}\phi} (\Omega_{A(0)}^{1})^{\vee}\otimes \chi^{-1}\xrightarrow{\cdot d_{dR}\phi}  A(0)\right). $$  
Let $\{x_i\}$ be a basis of $W$ and $y_i=\frac{\partial}{\partial x_i}\in (\Omega_{A(0)}^{1})^\vee[1]$ be a basis of the degree $(-1)$ terms of $A$.
Then the $(-1)$-shifted closed 2-form is of form 
$$\Omega_{\bCrit^{}(\phi)}=\sum_{i=1}^n d_{dR}x_i\wedge d_{dR}y_i \in \Hom_{}\left(\mathbb{C},\left(\Omega_{A(0)}^{1}\wedge(\Omega_{A(0)}^{1})^\vee[1]\otimes \chi^{-1}\right)^H\right), $$
which transforms under $H$ as $\chi^{}$. 
%the free $A(0)$-module spanned by the elements $y_i^{-1}$ is isomorphic to $(\Omega_{A(0)}^{1})^\vee\otimes \chi^{-1}[1]$. \yl{Is this twist what we want} The $(-1)$-shifted 2-form $\omega^0\in (\Omega_{A(0)}^{1})^\vee\otimes \Omega_{A(0)}^1[1]$ is therefore $\omega\otimes 1$ with $\omega$ the natural pairing between $(\Omega_{A(0)}^{1})^\vee$ and $\Omega_{A(0)}^1$, which is $H$-invariant. 
\end{proof}

\subsection{Image of shifted symplectic forms to periodic cyclic homology I}\label{sect on vanis 1}

It is often useful to know when the shifted symplectic derived stacks constructed in Theorem \ref{thm:symp_no_mark} 
have local Darboux charts as in \cite{BBJ, BG}, for example to verify the \textit{isotropic condition} of symmetric obstruction theory in the $(-2)$-shifted case (ref.~Definition~\ref{def on iso sym ob},~Theorem~\ref{prop:symm_ob}). One sufficient condition is when the shifted symplectic form maps to zero in the so-called periodic cyclic cohomology (e.g.~\cite{Park2} which is based on \cite{BBJ, BG}). 

Recall that similar to the construction of $NC^w$ in Lemma \ref{lem:F-eq_DR}, there is a \textit{periodic cyclic complex} (ref.~\cite{Lod}, see also \cite[\S 5.2]{BBJ}): 
for each $p\in\bbZ$, we define 
$$PC(p):\epsilon\emph{-}dg^{gr,H}_k\to dg_k^{H}, $$ 
such that 
$$PC^n(E)(p)=\prod_{i\in \mathbb{Z}  }E^{n-2i}(p+i), $$
define also the direct sum 
$$PC^w:=\bigoplus_{p}PC(p): \epsilon\emph{-}dg^{gr,H}_k\to dg_k^{gr, H}. $$
There is a natural transformation of functors: 
\[NC^w \to PC^w, \]
which induces a map on the cohomology 
$$HN^n(-)(p)\to HP^n(-)(p), \,\,\, \forall\,\, n,\, p\in \mathbb{Z}.  $$
%\[\epsilon\emph{-}dg^{gr, H}_{k}\to dg_k^{gr, H}. \]
As in Definition \ref{defi of ncchi}, for any $p\in\bbZ$ and group scheme homomorphism $\chi: G\to \mathbb{G}_m(k)$, we have a functor
$$PC_\chi(p): \epsilon\hbox{-}dg_k^{gr,H}\to dg_k^{}, \quad M\mapsto PC(p)(M_\chi), $$ 
and a natural transformation 
\[NC_\chi(p)\to PC_\chi(p), \]
which induces a map on the cohomology 
$$HN_\chi^n(-)(p)\to HP_\chi^n(-)(p),  \,\,\,  \forall\,\, n,\,p\in \mathbb{Z}.  $$
By the naturality of this map, we immediately have 
\begin{prop}\label{prop on commut of hnp}
In the setting of Theorem \ref{thm:symp_no_mark}, we have a commutative diagram 
\begin{equation*}\begin{xymatrix}{
HN_\chi^{n-2}(Y)(2) \ar[d] \ar[r]   & HN^{n-d-2}(M)(2) \ar[d]\\
 HP_\chi^{n-2}(Y)(2) \ar[r]& HP^{n-d-2}(M)(2), 
}\end{xymatrix} \end{equation*}
where $M:=\bMap^\sigma(X,[Y/H])$ and horizontal maps are obtained by applying $HN^*(-)(2)$, $HP^*(-)(2)$ to the map \eqref{compo maps0}.
\end{prop}
In particular, we have the following vanishing in periodic cyclic cohomology.  
\begin{corollary}
When $n=-1$ and $Y$ is affine, the image of $[\Omega_{M}]$ in $HP^{-3-d}(M)(2)$ is zero.
\end{corollary}
\begin{proof}
By \cite[Prop.~5.6]{BBJ}, which is based on \cite[Prop.~2.6~(ii)]{Emma}, the canonical map $$HN^{-3}(Y)(2) \to HP^{-3}(Y)(2)$$ is zero, 
so is the map $HN_\chi^{-3}(Y)(2)\to HP_\chi^{-3}(Y)(2)$ for $\chi$-eigenspaces. 
From the proof of Theorem \ref{thm:symp_no_mark}, the class $[\Omega_{M}]$ comes from the image of the map $$HN_\chi^{-3}(Y)(2)\to HN^{-3-d}(M)(2). $$ 
Then the claim follows from the commutativity in Proposition \ref{prop on commut of hnp}.
\end{proof}

\subsection{Shifted symplectic structures on $\sigma$-twisted derived mapping stacks II}\label{subsec:marked}

%\subsection{Curves with marked points}

%Let $C$ be a proper flat family of projective curves with at worst nodal singularities over $k$, endowed with a finite number of marked $k$-points $p_1,\dots,p_n$ and  
%Note that $\omega_{\mathrm{log}}|_{p_i}$ has a canonical trivialization given by $\frac{dz}{z}$.
%Consider the stack $\fBun_F^{\chi=\omega_{\mathrm{log}}}(C)$ classifying pairs $(\overline{P},\varkappa)$, where $\overline{P}$ is a principal $F$-bundle on $C$ and $\varkappa:\overline{P}\times_F\bbC_\chi\cong \omega_{\mathrm{log}}$ is an isomorphism.
%as in \eqref{equ on def deri of bun chi}. 
Consider the ``marked point" analogy of diagram \eqref{diag on derived map stk1} with  
$Y=\bCrit^{}(\phi)$, $H=G\times F$ as in the setting of \S \ref{subsec:crit} and $C$ being a proper flat family of curves over $k$ with at worst nodal singularities, 
endowed with smooth $k$-points $p_1,\dots,p_n$ as marked points. Denote 
$$\omega_{\mathrm{log}}:=\omega_{C,\mathrm{log}}=\omega_{C/k}(p_1+\cdots+p_n)$$
to be the log-canonical bundle. 
\begin{definition}
We define $\bMap^{\chi=\omega_{\mathrm{log}}}(C,[\bCrit^{}(\phi)/H])$ by the following homotopy pullback diagram:
\begin{equation}\label{diag on derived map stk2}\begin{xymatrix}{
\bMap^{\chi=\omega_{\mathrm{log}}}(C,[\bCrit^{}(\phi)/H])\ar[d]_\mu\ar[r] \ar@{}[dr]|{\Box} &\bMap(C,[\bCrit^{}(\phi)/H])\ar[d]\\
\fBun_H^{\chi=\omega_{\mathrm{log}}}(C)\ar[r]&\fBun_H(C).
}\end{xymatrix} \end{equation}
\end{definition}
The goal of this section is to extend Corollary \ref{cor on exist of -2shifted1} to the case when domain curve $C$ has marked points and the twist 
is with respect to log-canonical bundle $\omega_{\mathrm{log}}$ rather than $\omega_{C/k}$. 

Consider \textit{evaluation maps} (for simplicity we omit  $(-)\times \Spec k$ in the target)
\[ev^n:=ev_1\times\cdots\times ev_n:\bMap^{\chi=\omega_{\mathrm{log}}}(C,[\bCrit^{}(\phi)/H])\to [\bCrit^{}(\phi)/H]^n, \]
$$ev_{\pt}^n:\fBun_H^{\chi=\omega_{\mathrm{log}}}(C)\to [\pt/H]^n. $$
Composing $ev^n$ with the inclusion $\bCrit^{}(\phi)\hookrightarrow W$, by an abuse of notation, we obtain 
\begin{equation}\label{equ on evn}ev^n:\bMap^{\chi=\omega_{\mathrm{log}}}(C,[\bCrit^{}(\phi)/H])\to [W/H]^n, \end{equation}
whose further composition with projection $[W/H]^n\to [\pt/H]^n$ agrees with the composition $ev_{\pt}^n\circ \mu$. 

%As argued in Proposition~\ref{prop:ev_equi}, the evaluation map $ev^n:=ev_1\times\cdots\times ev_n$ factors through 
%\[ev^n:\bMap^{\chi=\omega_{\mathrm{log}}}(C,[\bCrit^{}(\phi)/H])\to [\bCrit^{}(\phi)/H_0]^n, \quad H_0:=G\times F_0. \]
%Similarly, the evaluation map $\fBun_H^{\chi=\omega_{\mathrm{log}}}(C)= \bMap^{\chi=\omega_{\mathrm{log}}}(C,[\pt/H])\to [\pt/H]^n$ factors through 
%$$ev_{\pt}^n:\fBun_H^{\chi=\omega_{\mathrm{log}}}(C)\to [\pt/H_0]^n. $$
%Composing $ev^n$ with the inclusion $\bCrit^{}(\phi)\hookrightarrow W$, by an abuse of notation, we obtain 
%\begin{equation}\label{equ on evn}ev^n:\bMap^{\chi=\omega_{\mathrm{log}}}(C,[\bCrit^{}(\phi)/H])\to [W/H_0]^n, \end{equation}
%whose further composition with projection $[W/H_0]^n\to [\pt/H_0]^n$ agrees with the composition $ev_{\pt}^n\circ \mu$. 

%\yl{why we need it factors through $[W/H_0]$?}

Let $\pi: \bMap^{\chi=\omega_{\mathrm{log}}}(C,[\bCrit^{}(\phi)/H])\times C\to \bMap^{\chi=\omega_{\mathrm{log}}}(C,[\bCrit^{}(\phi)/H])$ be the projection, 
$\calP$ be the universal $H$-bundle and $\calW:=\calP\times _{H}W$ be the universal $W$-bundle.
We calculate the relative tangent complex of the following map
\begin{equation}\label{eqn:f}
    \textbf{\underline{f}}:=ev^n\times_{[\pt/H]^n}\mu:\bMap^{\chi=\omega_{\mathrm{log}}}(C,[\bCrit^{}(\phi)/H])\to [W/H]^n\times_{[\pt/H]^n}\fBun_H^{\chi=\omega_{\mathrm{log}}}(C).
\end{equation}
\begin{prop}\label{prop on rel tan cpx of f}
We have 
$$\bbT_\textbf{\emph{\underline{f}}}\cong \big(\pi_*\left(\calW\boxtimes (\omega^\vee_{\mathrm{log}}\otimes\omega_{C/k})\right)\to 
\pi_*\left(\calW^\vee\boxtimes\omega_{\mathrm{log}}\right)\big). $$
%And there is a map $$\oO\to \wedge^2\bbL_\textbf{\emph{\underline{f}}}$$
And there is a canonical isomorphism 
$$\bbT_\textbf{\emph{\underline{f}}} \cong \bbL_\textbf{\emph{\underline{f}}}\,[-2]. $$
%which is self-dual under the relative Serre duality of $\pi$.
\end{prop}
\begin{proof}
For simplicity, we use the following shorthands in this proof:
\begin{equation}\label{equ on 3 shorthand}\textbf{\emph{\underline{M}}}:=\bMap^{\chi=\omega_{\mathrm{log}}}(C,[\bCrit^{}(\phi)/H]), \quad
B:=\fBun_H^{\chi=\omega_{\mathrm{log}}}(C), \end{equation} 
which fit into diagram 
\begin{equation}\label{digm on M with marked pts} 
\begin{xymatrix}{
\textbf{\emph{\underline{M}}} \ar@/^1pc/[drr]^{\mu} \ar[dr]^{\textbf{\underline{f}}} \ar@/_1.2pc/[ddr]_{ev^n} &  &\\
& [W/H]^n\times_{[\pt/H]^n}B \ar[r]\ar[d]^{} \ar@{}[dr]|{\Box} &  B \ar[d] \\
& [W/H]^n \ar[r] & [\pt/H]^n. 
}\end{xymatrix}
\end{equation}
By \eqref{equ for rel cot cpx2} and diagram~\eqref{diag on derived map stk2}, base change implies   
\begin{equation}\label{nota T1}\bbT_{\textbf{\emph{\underline{M}}}/B}=\pi_*u^*\left( W \xrightarrow{\mathrm{Hess}(\phi)} W^\vee\otimes \mathbb{C}_{\chi} \right)\cong \pi_*\left( \calW \xrightarrow{\alpha} \calW^\vee\boxtimes\omega_{\mathrm{log}} \right). \end{equation}
By the self-dual property of $\mathrm{Hess}(\phi)$, we know 
$$\left(\calW \to \calW^\vee\boxtimes\omega_{\mathrm{log}}\right)^\vee\boxtimes \omega_{\mathrm{log}}\cong \left(\calW \to \calW^\vee\boxtimes\omega_{\mathrm{log}}\right)[-1], \,\,
\mathrm{with} \,\, \alpha^{\vee}\otimes \omega_{\mathrm{log}}=\alpha, $$ 
%\left(\pi_* \left(\calW\right)\to \pi_*\left(\calW^\vee\boxtimes\omega_{\mathrm{log}}\right) \right).
Let $S:=\{p_1,\ldots, p_n\}\subseteq C$ be the subscheme given by all marked points.
We have 
a fiber sequence 
\begin{equation}
    \label{equ on a dis}
    \bbT_\textbf{\underline{f}}\to \bbT_{\textbf{\emph{\underline{M}}}/B}\to ev^{n*} \bbT_{[W/H]^n/[\pt/H]^n},
\end{equation}
and a quasi-isomorphism 
$$ev^{n*}\bbT_{[W/H]^n/[\pt/H]^n}\cong \pi_*(\calW\boxtimes\calO_S). $$
The map $\bbT_{\textbf{\emph{\underline{M}}}/B}\to ev^{n*}\bbT_{[W/H]^n/[\pt/H]^n}$ is given by 
\[\xymatrix{
\pi_* \left(\calW\right)\ar[r]\ar[d]_{ev^{n}}& \pi_*\left(\calW^\vee\boxtimes\omega_{\mathrm{log}}\right) \\
\pi_*(\calW\boxtimes\calO_S).
}\]
Combining with the short exact sequence 
$$0\to \calO_C(-S)\xrightarrow{s} \calO_C\to \calO_S \to 0, $$ 
we obtain the following representative of the fiber sequence  \eqref{equ on a dis}: 
\begin{equation}\label{eqn:tgt_marked}
\begin{xymatrix}{
 \pi_*\left(\calW\boxtimes \calO_S\right) \\
\pi_*(\calW)\ar[r]\ar[u]^{ev^{n}}&\pi_*\left(\calW^\vee\boxtimes\omega_{\mathrm{log}}\right) \\
\pi_*(\calW(-S))\ar[r]\ar[u]&\pi_*\left(\calW^\vee\boxtimes\omega_{\mathrm{log}}\right), \ar[u]
}\end{xymatrix}
\end{equation}
where the bottom (resp.~middle) row represents $\bbT_{\textbf{\underline{f}}}$ (resp.~$\bbT_{\textbf{\emph{\underline{M}}}/B}$), i.e.
%The bottom row is self-dual under the relative Serre duality of $\pi$. 
\begin{equation}\label{nota T2}\bbT_\textbf{\emph{\underline{f}}}=\pi_*\left( \calW\boxtimes\oO_C(-S) \xrightarrow{\beta=\alpha\circ s} \calW^\vee\boxtimes\omega_{\mathrm{log}} \right), \end{equation}
where $s: \calW\boxtimes\oO_C(-S)\to \calW$ is given by the canonical section $s:\oO_C(-S)\to \oO_C$. 

The following commutative diagram 
\[\xymatrix{
\calW\boxtimes  \omega_{\mathrm{log}}^{-1} \ar[r]^{\,\,\alpha\,\otimes\, \omega_{\mathrm{log}}^{-1}} \ar[d]_{s\,\otimes\, \omega
_C^{-1}}& \calW^\vee  \ar[d]^{s\,\otimes\, \oO_C(S)} \\
\calW\boxtimes\omega_C^{-1} \ar[r]^{\alpha\,\otimes\, \omega_C^{-1} \,\,\,\,\,} & \calW^\vee\boxtimes\oO_C(S)}\]
implies that 
\begin{align*}\beta^{\vee}&=s^{\vee}\circ \alpha^\vee=s^{\vee}\circ (\alpha\otimes \omega_{\mathrm{log}}^{-1})=(s\otimes \oO_C(S))\circ (\alpha\otimes \omega_{\mathrm{log}}^{-1}) \\
&=(\alpha\otimes \omega_{C}^{-1})  \circ  (s\otimes \omega_{C}^{-1})=(\alpha\circ s)\otimes \omega_{C}^{-1}=\beta\otimes \omega_{C}^{-1}. 
\end{align*}
By applying $\pi_*$ and the relative duality, we obtain the desired isomorphism. 
%\begin{equation*}\bbT_\textbf{\emph{\underline{f}}}\cong \bbL_\textbf{\emph{\underline{f}}}\,[-2]. \qedhere \end{equation*}
\end{proof}
Now we are ready to prove the main theorem of this section. 
We use shorthand as \eqref{equ on 3 shorthand}:
$$\textbf{\emph{\underline{M}}}:=\bMap^{\chi=\omega_{\mathrm{log}}}(C,[\bCrit^{}(\phi)/H]), \quad B:=\fBun_H^{\chi=\omega_{\mathrm{log}}}(C). $$

%\yl{By Proposition \ref{prop:symm_ob}, in fact we need to show fiber of  $$\bMap^{R_\chi=\omega_{\mathrm{log}}}(C,[\bCrit^{}(\phi)/H_R])\to 
%[W/H_R]^n\times_{[\pt/H_R]^n}\fBun_{H_R}^{R_\chi=\omega_{\mathrm{log}}}(C)$$ is $(-2)$-shifted.}
\begin{theorem}\label{thm:sympl_marked}
Let $k$ be a  Noetherian commutative ring over $\mathbb{C}$ and  $\sigma: \Spec k\to B$ be a $k$-point. 
Consider base change of diagram \eqref{digm on M with marked pts} under $\sigma$,~i.e.~we define $\textbf{\underline{M}} (k)$ and $K$ by the following homotopy pullback diagrams
\begin{equation}\label{diag on defi of Mk}
\begin{xymatrix}{
\textbf{\underline{M}} (k) \ar[r]^{\bar{\textbf{\emph{f}}}\quad } \ar@{}[dr]|{\Box \quad \,\,}  \ar[d]& \Spec K \ar[r]\ar[d]^{} \ar@{}[dr]|{\Box}  &  \Spec k \ar[d]^\sigma \\
\textbf{\underline{M}} \ar[r]^{\textbf{\emph{\underline{f}}}\quad \quad\quad\,\, } & [W/H]^n\times_{[\emph{pt}/H]^n}B \ar[r] & B. 
}\end{xymatrix}
\end{equation}
As a derived stack over $K$, $\textbf{\underline{M}} (k)$ has a canonical $(-2)$-shifted symplectic structure $\Omega_{\textbf{\textit{\underline{M}}}(k)}$. 
\end{theorem}
\begin{proof}
As in the proof of Theorem~\ref{thm:symp_no_mark}, we have maps in $\epsilon\emph{-}dg_{k}^{gr}$ (here we write $Y$ instead of $Y\times \Spec k$ for short): 
%The diagram \[\xymatrix{\textbf{\emph{\underline{M}}}(k)\ar[d]\times C\ar[r]&[Y/H]\ar[d]\\ C\ar[r]&BH}\]of universal evaluation maps gives
\[k[3](2)\xrightarrow{\Omega} \bfD R(Y)_{\chi}\xrightarrow{u^*\circ o} \bfD R^{\omega_{\mathrm{log}}}((C\times_k \textbf{\emph{\underline{M}}}(k))/\mathring{\omega_{\mathrm{log}}})\xrightarrow{\kappa_{\textbf{\emph{\underline{M}}}(k),C}^{\omega_{C,\mathrm{log}}}}\bfD R(\textbf{\emph{\underline{M}}} (k))\otimes_k C(C,\omega_{C,\mathrm{log}}), \]
where $Y=\bCrit^{}(\phi)$ is the critical locus \eqref{equ on sc cl} and the last map 
$\kappa_{\textbf{\emph{\underline{M}}}(k),C}^{\omega_{C,\mathrm{log}}}$ is defined as map \eqref{equ on kapp L}.

By \eqref{eq canonicity}, there is a map in $\epsilon\emph{-}dg_{k}^{gr}$: 
\[\bfD R(\textbf{\emph{\underline{M}}} (k))\otimes_k C(C,\omega_{C,\mathrm{log}})\xrightarrow{p} \bfD R(\textbf{\emph{\underline{M}}} (k)/\Spec K)\otimes_k C(C,\omega_{C,\mathrm{log}}).\]
In what follows, we show that the composition $p\circ \kappa_{\textbf{\emph{\underline{M}}}(k),C}^{\omega_{C,\mathrm{log}}}\circ u^*\circ o \circ \Omega$
factors through 
\[\bfD R(\textbf{\emph{\underline{M}}} (k) /\Spec K)\otimes_k C(C,\omega_C)\to \bfD R(\textbf{\emph{\underline{M}}} (k) /\Spec K)\otimes_k C(C,\omega_{C,\mathrm{log}}),\]
which is induced by the natural map $\omega_{C}\to \omega_{C,\mathrm{log}}$ and hence we obtain maps in $\epsilon\emph{-}dg_{k}^{gr}$: 
\begin{equation*}k[3](2)\to \bfD R(\textbf{\emph{\underline{M}}} (k) /\Spec K)\otimes_kC(C,\omega_C)\to 
\bfD R(\textbf{\emph{\underline{M}}} (k) /\Spec K)\otimes_kk[-1],\end{equation*}
where the last map is given by Serre duality pairing $C(C,\omega_C)\to k[-1]$.
%Composing with $\Omega: k[3]\to NC_{\chi^{ }}(Y)(2)$, we obtain a map of graded complexes of $k$-modules
%$$\Omega_{\textbf{\textit{\underline{M}}} (k)}: k[4]\to NC(\textbf{\emph{\underline{M}}} (k) /\Spec K)(2)$$
By adjunction, the above map is equivalent to a map in $\epsilon\emph{-}dg_{K}^{gr}$:
\begin{equation}\label{equ on 2 form with marked pt}\Omega_{\textbf{\textit{\underline{M}}} (k)}: K[4](2)\to\bfD R(\textbf{\emph{\underline{M}}} (k) /\Spec K). \end{equation} 
%\yl{what does equivalent mean here?}\gufang{typo corrected}
%\[k[4]\to NC^w(\textbf{\underline{M}} (k) /\Spec K)(2).\]
Now we construct the factorization.
Indeed, by induction we may assume without loss of generality that the number of marked points $n=1$, and let $p_1:\Spec k\to C$ be the marked point.
%\begin{equation*} \begin{xymatrix}{P(k) \ar[r]\ar[d] & Y\times \Spec k \ar[d]^{}  \ar@{^{(}->}[r] & W\times \Spec k \ar[d] & \\
%C\times \textbf{\emph{\underline{M}}}(k) \ar[d]  \ar[r]^{ } & [Y/H]\times \Spec k \ar@{^{(}->}[r]& [W/H]\times \Spec k &  \\
% \textbf{\emph{\underline{M}}}(k) \ar@/_1pc/[u]_{p_1}  \ar[d]  \ar[rr]^{\bar{\textbf{f}}\quad   } & & \Spec K \ar[r]    \ar[d]  & \Spec k  \ar[d]   \\
%\textbf{\textit{\underline{M}}} \ar[rr]^{\textbf{\underline{f}}\quad \quad\quad\quad }  & &   [W/H]\times_{[\pt/H]}B   \ar[r]  & B. }\end{xymatrix}\end{equation*}
Let $P_0$ be the principal $H$-bundle on $\Spec k$ determined by the composition 
$$\Spec k \to B\to [\pt/H]\times \Spec k $$ 
of maps over $k$. 
Thus $P_0$ as a principal $H$-bundle is endowed with a trivialization. 
Then we have 
\begin{equation*} 
\Spec K\cong  P_0\times_H W\cong W\times \Spec k, \end{equation*}
which is a trivial $W$-bundle over $\Spec k$.
The natural map $\Spec K=W\times \Spec k\to [W/H]\times \Spec k$
makes the following diagram commutative
%As we have fixed the principal $H$-bundle on $C$ (given by $\Spec k \to B$), $P(k)$ is in fact the pullback of a bundle on $C$, therefore . 
%Let $P(k)$ be the universal $H$-bundle over $C\times \textbf{\emph{\underline{M}}}(k)$ with universal map to $Y\times \Spec k$  
\begin{equation}\label{diag on SpekK}
\begin{xymatrix}{
& & Y\times \Spec k  \ar[d]  \ar@{^{(}->}[r] & W\times \Spec k \ar[d] & \\
\textbf{\emph{\underline{M}}}(k)  \ar@/_1pc/[rrrd]^{\bar{\textbf{f}}} \ar[r]^{p_1\quad }  & C\times_k \textbf{\emph{\underline{M}}}(k)  \ar[r]^{ } & [Y/H]\times\Spec k \ar[r]& [W/H]\times \Spec k &  \\
& & & \Spec K. \ar[u]\ar@/_4.0pc/@{=}[uu]
}\end{xymatrix}
\end{equation}
In what follows, we write $[Y/H]$ instead of $[Y/H]\times\Spec k$ for short, similarly for $[W/H]$.

With $u, w$ being the universal maps,
we have a commutative diagram 
 \begin{equation*} 
\begin{xymatrix}{ 
\textbf{\textit{\underline{M}}}\times_k C \ar[rr]^{}\ar[d]_{u} & &  B\times_k C \ar[d]^{w} \\
[Y/H] \ar[r]^{} & [W/H] \ar[r]^{}  & [\pt/H]. 
}\end{xymatrix}
\end{equation*}
The lower horizontal maps are quotients of maps $Y\hookrightarrow W\to \pt$. 
The upper map factors through $\alpha$ in below, making the lower-left  square in the following diagram commutative
%\yl{see wether you agree}\gufang{I again re-ordered the wordings. I'm happy with it.}
\begin{equation*} 
\begin{xymatrix}{ 
\textbf{\textit{\underline{M}}}(k)\times_k C \ar[r]^{ } \ar@{}[dr]|{\Box \quad \quad }  \ar[d] & [W/H]\times_{[\pt/H]}C \ar[r]\ar[d]^{} \ar@{}[dr]|{\Box} & C \ar[d]^{\sigma\times \mathrm{id}_C}  \\
\textbf{\textit{\underline{M}}}\times_k C \ar[r]^{\alpha \quad \quad \quad \quad }    \ar[d]_{u} & [W/H]\times_{[\pt/H]}(B\times_k C) \ar[r] \ar[d] \ar@{}[dr]|{\Box}  & B\times_k C \ar[d]^{w} \\
[Y/H]   \ar[r]^{} &  [W/H]  \ar[r]^{} & [\pt/H].
}\end{xymatrix}
\end{equation*}
%Here the other three squares are Cartesian. 
Replacing the $C$'s in above by $\mathring{\omega_{\mathrm{log}}}$, we obtain a commutative diagram 
\begin{equation}\label{diag on MkC2}
\begin{xymatrix}{ 
\textbf{\textit{\underline{M}}}(k)\times_k \mathring{\omega_{\mathrm{log}}} \ar[r]^{ } \ar@{}[dr]|{\Box \quad \quad }  \ar[d] & [W/H]\times_{[\pt/H]}\mathring{\omega_{\mathrm{log}}} \ar[r]\ar[d]^{} \ar@{}[dr]|{\Box} & \mathring{\omega_{\mathrm{log}}} \ar[d]^{\sigma\times \mathrm{id}_{\mathring{\omega_{\mathrm{log}}}}}  \\
\textbf{\textit{\underline{M}}}\times_k \mathring{\omega_{\mathrm{log}}} \ar[r]^{  \quad \quad \quad \quad }    \ar[d]_{\bar{u}} & [W/H]\times_{[\pt/H]}(B\times_k \mathring{\omega_{\mathrm{log}}}) \ar[r] \ar[d] \ar@{}[dr]|{\Box}  & B\times_k \mathring{\omega_{\mathrm{log}}} \ar[d]^{\bar{w}} \\
[Y/H_0]   \ar[r]^{} &  [W/H_0]  \ar[r]^{} & [\pt/H_0]. 
}\end{xymatrix}
\end{equation}
Here the maps $\bar{u}$, $\bar{w}$ exist by a similar argument as that of \eqref{diag on BH0H}.
And we use the fact that 
$$[W/H]\times_{[\pt/H]}(-)\cong [W/H_0]\times_{[\pt/H_0]}(-), \quad \mathrm{where}\,\,  (-)=B\times_k \mathring{\omega_{\mathrm{log}}}\,\, \mathrm{or}\,\, \mathring{\omega_{\mathrm{log}}}, $$
coming from the Cartesian diagram  
\begin{equation*} \begin{xymatrix}{
 [W/H_0] \ar[d] \ar[r] \ar@{}[dr]|{\Box} & [\pt/H_0]  \ar[d]^{}   \\ 
 [W/H]   \ar[r]  & [\pt/H]. }  \end{xymatrix}
\end{equation*}
We claim that the following  diagrams in $\epsilon\emph{-}dg_{k}^{gr}$ are commutative
%\yl{check $\chi$ or $\chi^{-1}$} 
\begin{equation}\label{diag on rel maps}
%\footnotesize
%\xymatrixrowsep{0.4in}
%\xymatrixcolsep{0.16in}
\xymatrix{
\bfD R(Y)_{\chi^{ }}  \ar[r]\ar[d]_{ o} & \bfD R(Y/W)_{\chi^{ }} \ar[d]^{o}\\
\bfD R^{\calL_\chi}([Y/H]/BH_0) \ar[d]_{u^*}\ar[r] &
\bfD R^{\calL_\chi}([Y/H]/[W/H_0]) \ar[d]^{u^*}\\
\bfD R^{\omega_{\mathrm{log}}}((C\times_k \textbf{\emph{\underline{M}}}(k))/\mathring{\omega_{\mathrm{log}}})\ar[d]_{p_1^*}\ar[r]&
\bfD R^{\omega_{\mathrm{log}}}((C\times_k \textbf{\emph{\underline{M}}} (k))/(\mathring{\omega_{\mathrm{log}}}\times_{[\pt/H]}[W/H]))\ar[d]^{p_1^*} \\
\bfD R^{\underline{\mathbb{C}}}((\textbf{\emph{\underline{M}}}(k)\times\mathbb{C}^*)/(\Spec k\times \mathbb{C}^*)) \ar[r]^{ \quad \quad}  \ar@{=}[d]  & 
\bfD R^{\underline{\mathbb{C}}}((\textbf{\emph{\underline{M}}}(k)\times\mathbb{C}^*)/(\Spec K\times \mathbb{C}^*))  \ar@{=}[d]  \\
\bfD R^{\underline{\mathbb{C}}}((\textbf{\emph{\underline{M}}}(k)\times_k(\Spec k\times \mathbb{C}^*))/(\Spec k\times \mathbb{C}^*)) \ar[r]^{ \quad \quad} 
\ar[d]_{\kappa_{\textbf{\emph{\underline{M}}}(k),\Spec k}^{\underline{\mathbb{C}}}} & \bfD R^{\underline{\mathbb{C}}}((\textbf{\emph{\underline{M}}}(k)\times_K(\Spec K\times \mathbb{C}^*))/(\Spec K\times \mathbb{C}^*)) \ar[d]^{\kappa_{\textbf{\emph{\underline{M}}}(k),\Spec K}^{\underline{\mathbb{C}}}} \\
\bfD R^{ }(\textbf{\emph{\underline{M}}}(k)) \ar[r]^{ p \quad \quad} & \bfD R^{}(\textbf{\emph{\underline{M}}}(k)/\Spec K). } \end{equation} 
The commutativity of the first square follows easily from the definition of the map $o$ \eqref{eqn:chi_L}. 
The second square commutes by using the commutativity of diagram \eqref{diag on MkC2} and the canonicity of relative de Rham complexes \eqref{eq canonicity}.  
In the third square, the commutativity follows from the commutativity of diagrams \eqref{diag on SpekK},  \eqref{diag on MkC2}.
And we also use the fact that $p_1^*\omega_{\mathrm{log}}$ is trivial on $\textbf{\emph{\underline{M}}}(k)$.
%The last square commutes by the definition of map $\delta=\kappa_{\textbf{\emph{\underline{M}}}(k),C}^{\omega_{C,\mathrm{log}}}$ \eqref{equ on kapp L}. 
In the last square, the commutativity follows from the definition of the map in \eqref{equ on kapp L}.

As the composition $k[3](2)\xrightarrow{\Omega} \bfD R(Y)_{\chi^{ }}\to \bfD R(Y/W)_{\chi^{ }}$ has a null-homotopy given by the Lagrangian fibration structure $Y\to W$ \cite[Rmk.~3.12]{Gra}\footnote{We thank Hyeonjun Park for pointing out this to us.}, this induces a  null-homotopy of the map $p\circ \kappa_{\textbf{\emph{\underline{M}}}(k),\Spec k}^{\underline{\mathbb{C}}} \circ p_1^* \circ u^*\circ o \circ \Omega$. 
Using the following commutative diagram in $\epsilon\emph{-}dg_{k}^{gr}$ (below $r$ is given by the restriction $\omega_{C,\mathrm{log}}\to \omega_{C,\mathrm{log}}|_{p_1}=\oO_{p_1}$):
%\yl{commutativity follows from $p_1^*=p_1^*\circ \delta$}
\begin{equation}\label{comm diagra on restric}
\xymatrixrowsep{0.4in}
\xymatrixcolsep{0.5in}
\xymatrix{  
 & \bfD R (\textbf{\emph{\underline{M}}} (k)) \otimes_k C(C,\omega_{C,\mathrm{log}})  \ar[d]^{p}  \\
\bfD R^{\omega_{\mathrm{log}}}((C\times_k \textbf{\emph{\underline{M}}}(k))/\mathring{\omega_{\mathrm{log}}})  \ar[ur]^{\kappa_{\textbf{\emph{\underline{M}}}(k),C}^{\omega_{C,\mathrm{log}}}} \ar[r]^{p\circ \kappa_{\textbf{\emph{\underline{M}}}(k),C}^{\omega_{C,\mathrm{log}}}  \quad\quad\,\, } \ar[d]_{ p_1^*} & 
\bfD R (\textbf{\emph{\underline{M}}} (k) /\Spec K) \otimes_k C(C,\omega_{C,\mathrm{log}}) \ar[d]^{r} \\
\bfD R^{\underline{\mathbb{C}}}((\mathbb{C}^*\times \textbf{\emph{\underline{M}}}(k))/\mathbb{C}^*) \ar[d]_{\kappa_{\textbf{\emph{\underline{M}}}(k),\Spec k}^{\underline{\mathbb{C}}}} & 
\bfD R (\textbf{\emph{\underline{M}}}(k)/\Spec K)\otimes_k C(C,\oO_{\{p_1\}}) \ar@{=}[d]   \\
\bfD R (  \textbf{\emph{\underline{M}}} (k)) \ar[r]^{p \quad\quad }  & \bfD R (\textbf{\emph{\underline{M}}} (k) /\Spec K),
} \end{equation}
we know $r\circ p\circ \kappa_{\textbf{\emph{\underline{M}}}(k),C}^{\omega_{C,\mathrm{log}}}\circ u^*\circ o \circ \Omega$ is also null-homotopy.  
This null-homotopy induces a factorization of $p\circ \kappa_{\textbf{\emph{\underline{M}}}(k),C}^{\omega_{C,\mathrm{log}}}\circ u^*\circ o \circ \Omega$ through $\bfD R(\textbf{\emph{\underline{M}}} (k) /\Spec K)\otimes_k C(C,\omega_{C,\mathrm{log}}(-p_1))$ as claimed, and we obtain a canonical $(-2)$-shifted closed 2-form as \eqref{equ on 2 form with marked pt}. Its underlying 2-form becomes the pairing in Proposition~\ref{prop on rel tan cpx of f} which is non-degenerate. 
%Using the description of the cotangent complex in \eqref{eqn:f}, the factorization of $p\circ \kappa_{\textbf{\emph{\underline{M}}}(k),C}^{\omega_{C,\mathrm{log}}}\circ u^*\circ o \circ \Omega$ through $\bfD R(\textbf{\emph{\underline{M}}} (k) /\Spec K)\otimes_k C(C,\omega_{C,\mathrm{log}}(-p_1))$ in the above induced by the null-homotopy becomes the pairing in Proposition~\ref{prop on rel tan cpx of f}. 
%Indeed, by \cite[Def.~2.21]{Gra}, the null-homotopy induces quasi-isomorphisms
%\begin{equation}\label{eqn:Lag}\xymatrix{\bbT_{Y/W} \ar[r]^{\cong \quad \quad } \ar[d]_{\cong}& \bbL_{W}|_Y[-1] \ar[d]^{\cong}   \\W^\vee\otimes \oO \ar@{=}[r]& W^\vee\otimes \oO.}\end{equation}

Indeed, following notations in \eqref{nota T1}, \eqref{nota T2} with $S=\{p_1\}$, consider the following perfect complexes on $\textbf{\emph{\underline{M}}} (k)\times_k C$: 
$$\mathcal{T}:=\left(\calW \xrightarrow{\alpha} \calW^\vee\boxtimes\omega_{\mathrm{log}}\right), \quad \mathcal{T}':=\left(\calW(-S) \xrightarrow{\beta=\alpha\circ s} \calW^\vee\boxtimes\omega_{\mathrm{log}}\right). $$
Using $u^*\circ o \circ \Omega$, 
we obtain a pairing
%\yl{I think the pairing is on $\textbf{\emph{\underline{M}}} (k)$ not the product, and we need to pushforward}
\begin{equation}\label{equ on pair1}\mathcal{T}^{\otimes 2}\to  \omega_{\mathrm{log}}[-1].  \end{equation} 
which can be rewritten as a  quasi-isomorphism $\mathcal{T}\cong \mathcal{T}^\vee\boxtimes \omega_{\mathrm{log}}[-1]$. 

There is a map $s:\mathcal{T}'\to\mathcal{T}$ given by the canonical section of $\oO(S)$, which fits into a commutative diagram 
\[\xymatrix{
\mathcal{T}' \ar@{.>}[dd]^{ }  \ar[r]^{s} &  \mathcal{T} \ar[d]^{\cong} \ar[r]^{r \quad \quad \quad \quad} &\mathcal{T}|_{\textbf{\emph{\underline{M}}}(k)\times p_1}=ev_{p_1}^*\bbT_Y \ar[d]_{\Omega}^{} \\
   & \mathcal{T}^\vee\boxtimes \omega_{\mathrm{log}}[-1]  \ar[d]^{s^\vee} \ar[r]& \mathcal{T}^\vee|_{\textbf{\emph{\underline{M}}}(k)\times p_1}[-1]=ev_{p_1}^*\bbL_Y[-1] 
   \ar[d]_{s^\vee|_{p_1}}^{} \\
(\mathcal{T}')^\vee\boxtimes\omega_{C}[-1] \ar[r]  & (\mathcal{T}')^\vee\boxtimes \omega_{\mathrm{log}}[-1]   \ar[r] & (\mathcal{T}')^\vee|_{\textbf{\emph{\underline{M}}}(k)\times p_1}[-1].
}\]
Note that $\bbT_Y=(\bbT_W|_Y\to \bbL_W|_Y)$ with nondegenerate pairing $\Omega$ and $(0\to \bbL_W|_Y)$ is an isotropic subcomplex by the Lagrangian fibration structure. 
And we have 
$$r\circ s(\mathcal{T}')=s|_{p_1}((\mathcal{T}')|_{\textbf{\emph{\underline{M}}}(k)\times p_1})=ev_{p_1}^*(0\to \bbL_W|_Y).$$
%Then for any $v\in (\mathcal{T}')|_{\textbf{\emph{\underline{M}}}(k)\times p_1}[-1]$, we have 
%\begin{align*}(s|_{p_1}\circ\Omega\circ r\circ s)(v)&=(s^\vee|_{p_1}(\Omega\,(0\to \bbL_W|_Y))(v) \\&=\langle s|_{p_1}(v), \rangle\end{align*}
Therefore the map $s^\vee|_{p_1}\circ \Omega\circ r\circ s$ has a null-homotopy, which induces a map $\mathcal{T}'\to (\mathcal{T}')^\vee\boxtimes\omega_{C}[-1]$, $\pi_*$ of
which is the one in Proposition~\ref{prop on rel tan cpx of f}. 
%Applying $\pi_*$ to which gives the map $\bbT_{\textbf{\emph{\underline{M}}} (k)}\to \bbT_{\textbf{\emph{\underline{M}}} (k)/K}$.
%Therefore, the verification of the non-degeneracy of the underlying 2-form of \eqref{equ on 2 form with marked pt} is the same as in the proof of Theorem~\ref{thm:symp_no_mark}.
%\yl{need to argue using Lag fib structure of $Y\to W$.}\gufang{The above paragraph added.}
\end{proof}

\subsection{Image of shifted symplectic forms to periodic cyclic homology II}
As in \S \ref{sect on vanis 1}, we show a vanishing of shifted symplectic forms in periodic cyclic homology, which will be used to verify the isotropic condition in the proof of Theorem \ref{prop:symm_ob}. 

Let $Z\subseteq W^n$ be a $H$-invariant closed subscheme such that $Z\subseteq Z(\boxplus^n\phi)$, where $Z(\boxplus^n\phi)$ denotes the zero locus of the function 
$$\boxplus^n\phi: W^n\to \mathbb{C}, \quad (x_1,\ldots,x_n)\mapsto \sum_{i=1}^n\phi(x_i). $$ 
Consider the Cartesian diagram of stacks
\begin{equation}\label{diag def iota}
\begin{xymatrix}{
\Spec K' \ar[r]^{\iota} \ar[d]^{} \ar@{}[dr]|{\Box}  &  \Spec K \ar[d] \\
[Z/H^n]\times_{[\pt/H]^n}B \ar[r]^{ } & [W/H]^n\times_{[\pt/H]^n}B, 
}\end{xymatrix}
\end{equation}
where the right vertical map is given as diagram \eqref{diag on defi of Mk}. 
\begin{prop}\label{lem:HN_HP}
%(1) Any $(-1)$-shifted symplectic structure on $Y$ maps to zero under the canonical map $HN^{-3}(Y)(2) \to HP^{-3}(Y)(2)$.
%(2) We have a commutative diagram
%\[\xymatrix{HN^{-3}_\chi(Y)(2)\ar[r]\ar[d]& HN^{-4}(\textbf{\textbf{\underline{M}}} (k)/\Spec K)(2)\ar[d]\\
%HP^{-3}_\chi(Y)(2)\ar[r]& HP^{-4}(\textbf{\textbf{\underline{M}}} (k) /\Spec K)(2),}\]
%where the horizontal maps are defined by applying $HN^n(-)(p)$, $HP^n(-)(p)$ to the construction in Theorem \ref{thm:sympl_marked} and 
%the vertical maps are given by the natural transformations. 
%Let $\iota:\Spec K'\to \Spec K$ be a map so that the function $\boxtimes^n\phi$ vanishes when restricted to $\Spec K'$ \yl{write more clearly what this means}.
Let $\textbf{\underline{M}}'(k)/\Spec K'$ be the base-change of $\textbf{\underline{M}} (k)/\Spec K$ by the map $\iota$ in diagram \eqref{diag def iota} and $\Omega_{\textbf{\underline{M}} (k)}$ be the shifted symplectic form constructed in Theorem \ref{thm:sympl_marked}. %\yl{ do base change on classical moduli?}
Then the pullback class $\iota^*[\Omega_{\textbf{\underline{M}} (k)}]$ goes to zero under the map 
$$HN^{-4}(\textbf{\underline{M}} '(k)/\Spec K')(2)\to HP^{-4}(\textbf{\underline{M}}' (k)/\Spec K')(2). $$
\end{prop}
\begin{proof}
%(1) is showed in \cite[Prop.~5.6]{BBJ} (which is essentially proven in \cite[Prop.~2.6~(ii)]{Emma}).
%(2) follows from applying $PC^w$ to the construction in Theorem \ref{thm:sympl_marked} and the first statement.
As in above, without loss of generality, we consider the case when there is only one marked point $p_1\in C$. Let 
$Y=\bCrit^{}(\phi)$ be as in \eqref{equ on sc cl}.
%Combining diagrams~\eqref{diag on rel maps},~\eqref{comm diagra on restric}
There are commutative diagrams in $dg_k^{gr}$: 
\begin{equation}\label{diag on iso1}
{\footnotesize \begin{xymatrix}{
 NC^w(\bfD R(Y)_{\chi^{ }})  \ar[d]_{T_1} \ar[r]^p   & NC^w(\bfD R(Y/W)_{\chi^{ }})  \ar[d]^{T_2}   \\ 
PC^w(\bfD R(Y)_{\chi^{ }})  \ar[r]^{p'}  \ar[d]_{R^{PC}_{1}}  & PC^w(\bfD R(Y/W)_{\chi^{ }}) \ar[d]^{R^{PC}_{2}}  \\
 PC^w(\textbf{\textit{\underline{M}}}(k)/ K)\otimes_k C(C,\omega_{C,\mathrm{log}}) \ar[r]^r  & 
PC^w(\textbf{\textit{\underline{M}}}(k)/ K)\otimes_k C(C,\oO_{p_1})
 }  \end{xymatrix}}
\end{equation}
and 
\begin{equation}\label{diag on iso2}
{\footnotesize \xymatrix{
%\xymatrixrowsep{0in} \xymatrixcolsep{0in}
k[3](2) \ar[r]^{\Omega_Y} \ar[d]_{\overline{\Omega}} & NC^w(\bfD R(Y)_{\chi^{ }}) \ar[r]^p \ar[d]_{R^{NC}_{1}} & NC^w(\bfD R(Y/W)_{\chi^{ }})  \ar[d]^{R^{NC}_{2}} \\
NC^w(\textbf{\textit{\underline{M}}}(k)/ K)\otimes_k C(C,\omega_{C}) \ar[r] \ar[d]_{S_0} & NC^w(\textbf{\textit{\underline{M}}}(k)/ K)\otimes_k C(C,\omega_{C,\mathrm{log}}) \ar[r] \ar[d]_{S_1} & 
NC^w(\textbf{\textit{\underline{M}}}(k)/ K)\otimes_k C(C,\oO_{p_1}) \ar[d]^{S_2}  \\
PC^w(\textbf{\textit{\underline{M}}}(k)/ K)\otimes_k C(C,\omega_{C}) \ar[r]^{s \quad }  & PC^w(\textbf{\textit{\underline{M}}}(k)/ K)\otimes_k C(C,\omega_{C,\mathrm{log}}) \ar[r]^{r}  & 
PC^w(\textbf{\textit{\underline{M}}}(k)/ K)\otimes_k C(C,\oO_{p_1}). } }
\end{equation}
Here in diagram~\eqref{diag on iso1}, the lower square commutes by applying $PC^w$ to 
diagrams~\eqref{diag on rel maps},~\eqref{comm diagra on restric}, and in diagram \eqref{diag on iso2}, 
the middle and lower horizontal sequences are fiber sequences and commutativity of the right upper square follows from applying $NC^w$ to 
diagrams~\eqref{diag on rel maps},~\eqref{comm diagra on restric}.  
Obviously, we have equivalences 
$$R^{PC}_{i}\circ T_i=S_i \circ R^{NC}_{i}, \quad i=1,2. $$
As noted in the proof of Theorem \ref{thm:sympl_marked}, 
the map $\overline{\Omega}$ is induced by 
the null-homotopy $$ p\circ \Omega_Y\stackrel{\gamma}{\leadsto}  0$$ from 
the Lagrangian fibration structure on $Y\to W$.
%\begin{equation}\label{equ on map before serre}\overline{\Omega}: k[3](2)\to NC^w(\textbf{\textit{\underline{M}}}(k)/ K)\otimes_k C(C,\omega_{C}), \end{equation} 
The composition of $\overline{\Omega}$ with Serre duality defines the $(-2)$-shifted symplectic form 
$\Omega_{\textbf{\textit{\underline{M}}} (k)}$ in \eqref{equ on 2 form with marked pt}. 
Therefore, to prove the proposition, it is enough to show the composition map 
\begin{equation}\label{equ on map before serre2}S_0\circ \overline{\Omega}: K[3](2)\to PC^w(\textbf{\textit{\underline{M}}}(k)/ K)\otimes_k C(C,\omega_{C}) \end{equation} 
is null-homotopic after the specified base-change \eqref{diag def iota}. 
Note that the map \eqref{equ on map before serre2} is determined by $s\,\circ\,S_0\,\circ\,\overline{\Omega}$ and the null-homotopy
\begin{align}\label{eqn on path1}r\circ\, s\,\circ\, S_0\,\circ\, \overline{\Omega}&= S_2\circ R^{NC}_{2}\circ p\circ \Omega_Y\\ \nonumber 
&=R^{PC}_{2}\circ T_2\circ p\circ \Omega_Y  \stackrel{R^{PC}_{2}\circ T_2(\gamma)}{\leadsto} 0. \end{align} 
%&=R^{PC}_{2}\circ p'\circ T_1\circ \Omega_Y 
Thanks to \cite[Prop.~5.6]{BBJ} which in turn is based on 
\cite[Prop.~2.6~(ii)]{Emma},  the canonical map $HN^{-3}(Y)(2) \to HP^{-3}(Y)(2)$ is zero.
Hence we have a null-homotopy
$$T_1 \circ \Omega_Y\stackrel{\eta}{\leadsto} 0, $$
which gives 
\begin{align}\label{eqn on path2}r\circ\,s\,\circ\,S_0\,\circ\,\overline{\Omega}&=r\circ S_1\circ R^{NC}_{1}\circ \Omega_Y \\ \nonumber
&=r\circ\,R^{PC}_{1}\circ T_1 \circ \Omega_Y \\ \nonumber
&=R^{PC}_{2}\circ p'\circ T_1 \circ \Omega_Y  
%&=R^{PC}_{2}\circ T_2 \circ p \circ \Omega_Y \\
\stackrel{R^{PC}_{2}\circ p'(\eta)}{\leadsto} 0. \end{align} 
Composing the paths \eqref{eqn on path1}~and~\eqref{eqn on path2} determines a loop (denoted by $R_2\left(T_2(\gamma)\circ p'(\eta)\right)$) in 
$$|PC^{-3}(\textbf{\textit{\underline{M}}}(k)/ K)\otimes_k C(C,\oO_{p_1})(2)|$$
which comes from a loop (denoted by $T_2(\gamma)\circ p'(\eta)$) in $|PC^{-3}(\bfD R(Y/W)_{\chi^{ }})(2)|$. We are left to show it is trivial 
after the specified base change. 

We first describe the null-homotopy $\eta$. 
In the coordinates used in the proof of Corollary~\ref{cor on exist of -2shifted1},
$$\Omega_Y=\sum_{i=1}^nd_{dR}x_i\wedge d_{dR}y_i, \quad T_1\circ\Omega_Y=(d+d_{dR}) \alpha\in PC^{-3}(\bfD R(Y)_{\chi^{ }})(2), $$ 
where 
$$\alpha=\sum_{i=1}^ny_id_{dR}x_i +\phi \in \bfD R(Y)_{\chi^{ }}^{-2}(1)\oplus \bfD R(Y)_{\chi^{ }}^{0}(0)\subset PC^{-4}(\bfD R(Y)_{\chi^{ }})(2). $$
Indeed, taking the realization of $PC^{-3}(\bfD R(Y)_{\chi^{ }})(2)$ as in \cite[Def.~5.5]{BBJ}, we have 
$$d_{dR}\left(\sum_{i=1}^ny_id_{dR}x_i\right)=\Omega_Y, \quad d\left(\sum_{i=1}^ny_id_{dR}x_i\right)=-d_{dR}\phi, \quad d\phi=0, $$ 
where the last vanishing is because $\phi$ is a polynomial on variables $x_i$ and $dx_i=0$ (\cite[Ex.~5.15]{BBJ}).

Next we describe the null-homotopy $T_2(\gamma)$ using the above presentation. 
As $d_{dR}x_i$ are sections of $\bbL_{Y}$ coming from $\bbL_{W}|_Y$ which maps to $0$ via 
$\bbL_{Y}\to \bbL_{Y/W}$, therefore we get 
$$T_2\circ p\circ \Omega_Y=p'\circ T_1\circ \Omega_Y=0. $$ 
Similarly, we also have 
$$p'(\alpha)=p'(\phi)\in \bfD R(Y/W)_{\chi^{ }}^{0}(0), $$
which is shown to vanish after the base change in below. 

As the number of marked points is assumed to be one, consider the homotopy pullback diagram 
\begin{equation}\label{equ on Y'Z}
\begin{xymatrix}{
Y'\ar[r]^{} \ar[d]^{} \ar@{}[dr]|{\Box}  &  Z \ar[d]^{\iota} \\
Y \ar[r]^{ } &W, 
}\end{xymatrix}
\end{equation}
and the base change map
$\iota^*:\bfD R(Y/W)_{\chi^{ }}^{0}(0)\to \bfD R(Y'/Z)_{\chi^{ }}^{0}(0)$. By our assumption 
$$\iota^*\circ p'(\phi)=p' \circ\iota^* (\phi)=0. $$ 
Therefore the loop $T_2(\gamma)\circ p'(\eta)$ becomes trivial after going to $|PC^{-3}(\bfD R(Y'/Z)_{\chi^{ }})(2)|$. 
There are similar diagrams as \eqref{diag on rel maps}, \eqref{comm diagra on restric} after base change via \eqref{equ on Y'Z}, therefore we have similar diagrams as \eqref{diag on iso1}, \eqref{diag on iso2} after the base change. The commutativity of the diagram 
\begin{equation*}
{ \begin{xymatrix}{
PC^w(\bfD R(Y/W)_{\chi^{ }})  \ar[r]^{}  \ar[d]_{R^{PC}_{2}}  & PC^w(\bfD R(Y'/Z)_{\chi^{ }}) \ar[d]^{R^{PC}_{2}}  \\
 PC^w(\textbf{\textit{\underline{M}}}(k)/ K)\otimes_k C(C,\oO_{p_1}) \ar[r]^{\underline{\iota}}  & 
PC^w(\textbf{\textit{\underline{M}}}'(k)/ K')\otimes_k C(C,\oO_{p_1}) }  \end{xymatrix}}
\end{equation*}
implies that the loop $R_2\left(T_2(\gamma)\circ p'(\eta)\right)$ in $|PC^{-3}(\textbf{\textit{\underline{M}}}(k)/ K)\otimes_k C(C,\oO_{p_1})(2)|$ becomes 
the trivial loop under the map $\underline{\iota}$, therefore our claim holds. 
%The last row of \eqref{diag on iso2} is a fiber sequence, hence the map \eqref{equ on map before serre2} becomes null-homotopic after base change 
%to $\textbf{\textit{\underline{M}}}'(k)/\Spec K'$.
\end{proof}
\begin{remark}\label{rmk on iso still holds}
The above result remains hold if we replace $Z$ in diagram \eqref{diag def iota} by a closed subscheme in $Z\left((\boxplus^n\phi)^r\right)$ with $r\geqslant 1$.  
%\yl{Double check this}
\end{remark}

\section{Virtual pullbacks}
We retain notations from \S \ref{sec:CiKM}. 
We recollect general theory of virtual pullbacks arising from 
$(-2)$-shifted symplectic structures and then apply to our setting. The theory is a rather recent development coming out of defining Donaldson-Thomas type invariants for Calabi-Yau 4-folds \cite{BJ, OT, CGJ1, CGJ2} (see also \cite{CL1,CL2,CL3}). Our main reference is the virtual pullback construction of Park \cite{Park} which makes the virtual class construction 
of Oh-Thomas \cite{OT} functorial.  

%In this section and the next, we write $QM_{g,n}^{R_{\chi}=\omega_{\mathrm{log}}}(\Crit(\phi)/\!\!/G,\beta)$ simply as $QM_{g,n}(\beta)$ or $QM$. We drop $\beta$ from the notation if we take union over all possible $\beta$'s.
\label{sec:fund}
\subsection{Virtual pullbacks via symmetric obstruction theory}\label{sect on vir pullback}
First recall relevant notions and results from \cite{Park}. 

\begin{definition}(\cite[Prop.~1.7,~\S A.2]{Park})\label{defi on sym cx}
A \textit{symmetric complex}  $\bbE$ on an algebraic stack
$\calX$ consists of the following data:
\begin{enumerate}
    \item A perfect complex $\bbE$ of tor-amplitude $[-2,0]$ on $\calX$.
    \item A non-degenerate symmetric form $\theta$ on $\bbE$, i.e. a morphism
$$\theta : \calO_{\calX} \to (\bbE \otimes\bbE)[-2]$$
in the derived category of $\calX$, invariant under the transposition
$\sigma : \bbE \otimes \bbE \to  \bbE \otimes \bbE$, and the induced map 
$\iota_{\theta} : \bbE^\vee\to\bbE[-2]$ is an isomorphism.
    \item An orientation $o$ of $\bbE$,~i.e.~an isomorphism~$o:\calO_\calX \to \det(\bbE)$ of line
bundles such that $\det(\iota_\theta)=o\circ o^\vee$.
\end{enumerate}
\end{definition}
\begin{remark}\label{rmk on ori}
If a symmetric complex is of form $$\bbE=(\mathbb{V}\xrightarrow{\varphi} \mathbb{V}^\vee), $$ 
where $\mathbb{V}$ is a perfect complex 
of tor-amplitude $[0,1]$ and $\varphi$ is self-dual under the isomorphism $\iota_{\theta}$ above. Then we have a canonical isomorphism 
$\det(\bbE)\cong \calO_\calX$ and \textit{orientations} of $\bbE$ are given by 
\begin{equation}\label{cho of sign}\calO_\calX\xrightarrow{\pm (-\sqrt{-1})^{\rk(\mathbb{V})}} \calO_\calX \end{equation}
on each connected component of $\calX$ (e.g.~\cite[Eqns.~(59),~(63)]{OT}). We choose the plus sign in above as a canonical choice of orientation. 
\end{remark}
%We write its virtual normal cone as $\fC_{\bbE}$.
For a symmetric complex $\bbE$, there is a quadratic function (\cite[Prop.~1.7, \S A.2]{Park}):
\begin{equation}\label{def of qua eq1}\fq_\bbE:\fC_{\bbE}\to \bbA^1_{\calX}, \end{equation}
from the virtual normal cone $\fC_{\bbE}$ of $\bbE$, 
characterized by some naturality conditions. 
For example, for a Deligne-Mumford morphism $f:\calY\to \calX$ between algebraic stacks, we have 
\begin{equation}\label{def of qua eq2}f^*\fq_{\bbE}=\fq_{f^*\bbE}:f^*\fC_{\bbE}=\fC_{f^*\bbE}\to \bbA_{\calY}. \end{equation}
When $\bbE=E[1]$ for a special orthogonal bundle $E$, $\fq_\bbE$ is given by the quadratic form on $E$. 

\begin{definition}(\cite[Def.~1.9,~\S A.2]{Park})\label{def on sym ob}
A \textit{symmetric obstruction
theory} for a Deligne-Mumford morphism $f : \calX\to \calY$ between algebraic stacks   is a morphism $\phi : \bbE \to \bfL_f$ in
the derived category of $\calX$ such that
\begin{enumerate}
    \item $\bbE$ is a symmetric complex.
    \item $\phi$ is an obstruction theory in the sense of Behrend-Fantechi \cite{BF}, i.e., $h^0(\phi)$ is an isomorphism and $h^{-1}(\phi)$ is surjective,
where $\bfL_f := \tau^{\geqslant  -1}\bbL_f$ is the truncated cotangent complex.
\end{enumerate}
\end{definition}
\begin{remark}
Do not confuse this with the ``symmetric obstruction theory" in the sense of Behrend-Fantechi \cite{BF2} where ``obstruction is dual to deformation". 
\end{remark}
%The cone stack $\fC_{\bbE}$ is called the virtual normal cone.
The obstruction theory $\phi$ induces a closed
embedding of the intrinsic normal cone 
$$\fC_f\inj \fC_{\bbE}. $$
\begin{definition}\label{def on iso sym ob}
A symmetric obstruction theory $\phi : \bbE \to \bfL_f$  is {\it isotropic} if the intrinsic normal
cone $\fC_f$ is isotropic in the virtual normal cone $\fC_{\bbE}$, i.e. the restriction
$\fq_{\bbE}|_{\fC_f}
: \fC_f\inj \fC_{\bbE}\to \bbA^1$
 vanishes.
\end{definition}
%\begin{remark} By \cite[Ex.~1.17]{Park} which uses \cite{BBBJ}, having $(-2)$-shifted symplectic structure automatically implies isotropic condition. \end{remark}
Isotropic symmetric obstruction theory implies the existence of square root virtual pullback which we now briefly recall. 
For a symmetric complex $\bbE$ on an algebraic stack $\calX$, let $\fQ(\bbE)$ be the zero locus of the quadratic function $\fq_{\bbE}:\fC_\bbE\to \bbA^1_\calX$,
there is a \textit{square root Gysin pullback} \cite[Def.~A.2]{Park} 
$$\sqrt{0^!_{\fQ(\bbE)}}: A_*(\fQ(\bbE)) \to A_*(\calX), $$
if $\calX$ is a quotient of a separated Deligne-Mumford stack by an algebraic group.
\begin{definition}
 Assume that $f:\calX\to\calY$ is a Deligne-Mumford morphism between algebraic stacks with an 
isotropic symmetric obstruction theory $\phi:\bbE\to\bfL_f$. It induces  a closed embedding $a:\fC_f\to \fQ(\bbE)$.
The \textit{square root virtual pullback} is the composition 
\begin{equation}\label{eqn:sqrt_pull}
    \sqrt{f^!}:A_*(\calY)\xrightarrow{sp_f}A_*(\fC_f )\xrightarrow{a_*} A^{*} (\fQ(\bbE))\xrightarrow{\sqrt{0^!_{\fQ(\bbE)}}} A_{*}(\calX ),
\end{equation}
 where $sp_f:A_*(\calY)\to A_*(\fC_f)$
is the specialization map (\cite[Const.~3.6]{Man}).
\end{definition}
The map $\sqrt{f^!}$ commutes with \textit{projective pushforwards, smooth pullbacks, and Gysin pullbacks for regular immersions}. 
Moreover, it has a \textit{functoriality} with respect to morphisms compatible with symmetric obstruction theories \cite[Thm.~A.4]{Park} as explained below.

Let $f:\calX\to\calY$ be a Deligne-Mumford (DM) morphism of algebraic stacks having reductive stabilizer groups and affine diagonals\footnote{The original assumptions of \cite[Thm.~A.4]{Park} are (1) $\calY$ is the quotient of a DM stack by a linear algebraic group, (2) $\calX$ has the resolution property and 
(3) $f$ is quasi-projective. We learned from Hyeonjun Park that (1)--(3) can be replaced by the assumption stated above where details will appear in a forthcoming work \cite{BP}.}, 
which are satisfied if $\calX$ and $\calY$ are quotient stacks of 
separated DM stacks by algebraic tori.
Let $g:\calY\to\calZ$ be a DM morphism of algebraic stacks.
Assume $\phi_g:\bbE_g\to \bfL_g$, $\phi_{g\circ f}:\bbE_{g\circ f}\to \bfL_{g\circ f}$ are isotropic symmetric obstruction theories, 
$\phi_f:\bbE_f\to \bfL_f$ is a perfect obstruction theory \cite{BF} and they are {\it compatible}, i.e. there exists a perfect complex $\bbD$ and morphisms 
$\alpha:\bbE_{g\circ f}\to \bbD$ and $\beta:f^*\bbE_g\to \bbD$ fitting into diagram \eqref{eqn:triang_obst} of exact triangles and preserves orientation 
(the orientation of $\bbE_{g\circ f}$ is given by the orientation of $\bbE_{g}$).
\begin{equation}
   \label{eqn:triang_obst}
\xymatrix{\bbD^\vee[2]\ar[r]^{\alpha^\vee}\ar[d]_{\beta^\vee}&\bbE_{g\circ f}\ar[d]_\alpha\ar[r]^\delta&\bbE_f\ar@{=}[d]\\
f^*\bbE_g\ar[r]^\beta\ar[d]_{f^*\phi_g}&\bbD\ar[r]^\gamma\ar[d]_{\phi'_{g\circ f}}&\bbE_f\ar[d]_{\phi'_f}\\
\tau^{\geqslant -1}f^*\bfL_g\ar[r]&\bfL_{g\circ f}\ar[r]&\bfL_f'.}
\end{equation}
Here $\phi_{g\circ f}=\phi_{g\circ f}'\circ\alpha$, $\phi_f=r\circ \phi'_f$ with $\bfL'_f$ is the cone of 
$\tau^{\geqslant  -1}f^*\bfL_g\to\bfL_{g\circ f}$ and $r:\bfL'_f\to \bfL_f$ the truncation. 
Then we have \begin{equation}\label{eqn:Park}
\sqrt{(g\circ f)^!}=f^!\circ\sqrt{g^!},\end{equation}
where $f^!$ is the virtual pullback of Manolache \cite{Man}. 
Finally, we remark that the above extends to the equivariant setting when there is a torus action.

\subsection{Virtual pullbacks for moduli stacks of quasimaps}\label{subsec:obst}
In this section, let 
$$Y=\bCrit^{}(\phi):=W\times^\bfL_{T^*W}W, \quad H=G\times F$$ 
be as in \eqref{equ for z} and 
$\mathcal{C}\to\mathfrak{M}_{g,n}$ be the universal family. 
%without causing confusion, we use the same notation for the corresponding derived enhancement. 
Denote 
\begin{equation}\label{def of mod stack of mapping}
\bMap_{g,n}([Y/H]):=\bMap_{\textbf{dSt}/\mathfrak{M}_{g,n}}(\mathcal{C},[Y/H]\times \mathfrak{M}_{g,n}) 
\end{equation}
to be the derived mapping stack of $\mathcal{C}$ to $[Y/H]$  (relative to $\mathfrak{M}_{g,n}$) as in \cite[\S 4.3 (4.d)]{Toen1}, where 
we omit the inclusion functor from classical stacks to derived stacks for $\mathcal{C}$ and $\mathfrak{M}_{g,n}$. 
This is the ``global" version of derived stack \eqref{equ on mappi stac} when $\Spec(k)$ is replaced by $\mathfrak{M}_{g,n}$.
By Lurie's representability theorem \cite{Lur} (see also \cite[Cor.~3.3]{Toen2}), this is a derived Artin stack 
locally of finite presentation\footnote{In fact by \cite[Prop.~1.3.3.4]{TV}, it is enough to check it for an atlas $\{U_i\}$ on $\mathfrak{M}_{g,n}$. By taking some 
etale cover of $U_i$, we may assume $U_i\times_{\mathfrak{M}_{g,n}}\mathcal{C}$ is a scheme (e.g.~\cite[Tag 0E6F]{stack}).
Then we are reduced to the case \eqref{equ on mappi stac}.}. 

Let $H_R:=G\times \mathbb{C}^*$ and recall $R$-charge $R: \mathbb{C}^*\to F$ and $R_\chi: \mathbb{C}^*\to \mathbb{C}^*$ as in Definition \ref{defi of R-charge}. 
Consider the derived version of diagram \eqref{diagram defi R twist map 2}: 
\begin{definition}
We define derived stacks 
$\bMap_{g,n}^{R_{\chi}=\omega_{\mathrm{log}}}([Y/H_R])$ and $\bMap_{g,n}^{\chi=\omega_{\mathrm{log}}}([Y/H]) $ by the following homotopy pullback diagrams: 
\begin{equation}\label{fiber diag on mgn}
\begin{xymatrix}{
\bMap_{g,n}^{R_{\chi}=\omega_{\mathrm{log}}}([Y/H_R]) \ar[d]_{\mu} \ar[r]^{\textbf{h}} \ar@{}[dr]|{\Box} &\bMap_{g,n}^{\chi=\omega_{\mathrm{log}}}([Y/H]) \ar[d]  \ar[r]  \ar@{}[dr]|{\Box} & \bMap_{g,n}([Y/H]) \ar[d]\\
\fBun_{H_R,g,n}^{R_{\chi}=\omega_{\mathrm{log}}} \ar[r]^{\eta} & \fBun_{H,g,n}^{\chi=\omega_{\mathrm{log}}}\ar[r] & \fBun_{H,g,n}.
}\end{xymatrix} \end{equation} 
\end{definition}
Here the right square is the ``global" version of \eqref{diag on derived map stk2} when $\Spec(k)$ is replaced by $\mathfrak{M}_{g,n}$. 
And $\bMap_{g,n}^{R_{\chi}=\omega_{\mathrm{log}}}([Y/H_R])$ is the derived stack of quadruple $\big((C,p_1,\ldots,p_n),P,u,\varkappa\big)$, where 
$(C,p_1,\ldots,p_n)$ is a prestable genus $g$, $n$-pointed curve, $P$ is a principal $H_R$-bundle on $C$ with an isomorphism 
$\varkappa:P/G\times_{ \bbC^*}R_\chi\cong\omega_{\mathrm{log}}$, and $u: P\times_{H_R}(G\times R)\to Y$ is a $(G\times F)$-equivariant map. 

%where $\fBun_{\bbC^*,g,n}^{R_{\chi}=\omega_{\mathrm{log}}}$ is the moduli stack of pairs $(P_{\mathbb{C}^*},\varkappa)$, where $P_{\mathbb{C}^*}$ is a principal $\bbC^*$-bundle on a fiber $C$ of $\mathcal{C}\to \mathfrak{M}_{g,n}$ and $\varkappa:P_{\mathbb{C}^*}\times_{\bbC^*}R_\chi\cong \omega_{C,\mathrm{log}}$ is an isomorphism.
%as in \eqref{equ on def deri of bun chi}. 
%One similarly defines $\fBun_{H_R,g,n}^{R_{\chi}=\omega_{\mathrm{log}}}$ to be the moduli stack of pairs $(P,\varkappa)$, where $P$ is a principal $H_R=G\times \bbC^*$-bundle on a fiber $C$ of $\mathcal{C}\to \mathfrak{M}_{g,n}$ and $\varkappa:P\times_{H_R}\bbC_\chi\cong \omega_{C,\mathrm{log}}$ is an isomorphism.
%Moreover the left vertical map in diagram \eqref{fiber diag on mgn} factors through it, i.e., we have 
%\begin{equation}\label{equ on sev forget map0}\bMap^{R_{\chi}=\omega_{\mathrm{log}}}_{g,n}([Z/\bbC^*])\stackrel{\mu}{\to} \fBun_{H_R,g,n}^{R_{\chi}=\omega_{\mathrm{log}}} \stackrel{\nu}{\to} \fBun_{\bbC^*,g,n}^{R_{\chi}=\omega_{\mathrm{log}}}. \end{equation}

%\yl{Whether it is enough to map to $[W/H_R]$?}

Consider the left two terms of diagram \eqref{fiber diag on mgn}. 
We have the product of evaluation maps
\begin{equation}\label{prd of eva map}ev^n:=ev_1\times\cdots\times ev_n: \bMap_{g,n}^{R_{\chi}=\omega_{\mathrm{log}}}([Y/H_R])  \to [Y/H_R]^n\hookrightarrow [W/H_R]^n, \end{equation}
and the structure map $[W/H_R]\to [\pt/H_R]$. 
They fit into the following diagram 
\begin{equation}
\label{diag on map f}
\begin{xymatrix}{
 \bMap_{g,n}^{R_{\chi}=\omega_{\mathrm{log}}}([Y/H_R]) \ar@/^1pc/[drr]^{\mu} \ar[dr]^{\bf f} \ar@/_1.5pc/[ddr]_{ev^n} &  & \\
& [W/H_R]^n\times_{[\pt/H_R]^n}\fBun_{H_R,g,n}^{R_{\chi}=\omega_{\mathrm{log}}} \ar[r]\ar[d]^{} \ar@{}[dr]|{\Box} &  \fBun_{H_R,g,n}^{R_{\chi}=\omega_{\mathrm{log}}} \ar[d] \\
& [W/H_R]^n \ar[r] & [\pt/H_R]^n,
}\end{xymatrix}
\end{equation}
where $\textbf{f}:= \mu\times_{[\pt/H_R]^n} ev^n$ is the induced map to the fiber product. 

We have a similar diagram when considering the middle two terms of diagram \eqref{fiber diag on mgn}: 
\begin{equation} 
\label{diag on map f another}
\begin{xymatrix}{
 \bMap_{g,n}^{\chi=\omega_{\mathrm{log}}}([Y/H]) \ar@/^1pc/[drr]^{\underline{\mu}} \ar[dr]^{\bf \underline{f}} \ar@/_1.5pc/[ddr]_{ev^n} &  &\\
& [W/H]^n\times_{[\pt/H]^n}\fBun_{H,g,n}^{\chi=\omega_{\mathrm{log}}} \ar[r]\ar[d]^{} \ar@{}[dr]|{\Box} &  \fBun_{H,g,n}^{\chi=\omega_{\mathrm{log}}} \ar[d] \\
& [W/H]^n \ar[r] & [\pt/H]^n.
}\end{xymatrix}
\end{equation}
These two diagrams are related by the following base change. 
\begin{lemma}\label{lem on relate Rchi twist chi twist}
We have the following homotopy pullback diagrams of derived stacks
\begin{equation}\label{diag compare f and underly f}
\begin{xymatrix}{
\bMap_{g,n}^{R_{\chi}=\omega_{\mathrm{log}}}([Y/H_R]) \ar[r]^{\textbf{\emph{h}}} \ar[d]_{\bf f} \ar@{}[dr]|{\Box}  &  \bMap_{g,n}^{\chi=\omega_{\mathrm{log}}}([Y/H]) \ar[d]^{\bf \underline{f}}   \\ 
 [W/H_R]^n\times_{[\emph{pt}/H_R]^n}\fBun_{H_R,g,n}^{R_{\chi}=\omega_{\mathrm{log}}}   \ar[r]\ar[d]^{} \ar@{}[dr]|{\Box} &  [W/H]^n\times_{[\emph{pt}/H]^n}\fBun_{H,g,n}^{\chi=\omega_{\mathrm{log}}} \ar[d] \\
\fBun_{H_R,g,n}^{R_{\chi}=\omega_{\mathrm{log}}} \ar[r]^{\eta} & \fBun_{H,g,n}^{\chi=\omega_{\mathrm{log}}}.
}\end{xymatrix}
\end{equation}
\end{lemma}
\begin{proof}
The map $W\to \pt$ and $H_R\to H$ induce a Cartesian diagram of smooth stacks  
\begin{equation}\label{base cha of smooth quot stack}
\begin{xymatrix}{
 [W/H_R] \ar[d] \ar[r] \ar@{}[dr]|{\Box} & [\pt/H_R]  \ar[d]^{}   \\ 
 [W/H]   \ar[r]  & [\pt/H]. }  \end{xymatrix}
\end{equation}
As the horizontal maps are smooth, so it is also a homotopy pullback diagram of derived stacks. 
Combining this with diagrams \eqref{diag on map f}, \eqref{diag on map f another} and a diagram chasing, we obtain a commutative diagram of derived stacks:
\begin{equation*} 
\begin{xymatrix}{
 \bMap_{g,n}^{R_{\chi}=\omega_{\mathrm{log}}}([Y/H_R]) \ar[r] \ar[d] &  [W/H_R]^n\times_{[\pt/H_R]^n}\fBun_{H_R,g,n}^{R_\chi=\omega_{\mathrm{log}}} \ar[d] \ar[r]  &  \fBun_{H_R,g,n}^{R_\chi=\omega_{\mathrm{log}}}   \ar[d]^{\eta} \\
\bMap_{g,n}^{\chi=\omega_{\mathrm{log}}}([Y/H]) \ar[r] & [W/H]^n\times_{[\pt/H]^n}\fBun_{H,g,n}^{\chi=\omega_{\mathrm{log}}} 
\ar[r] &  \fBun_{H,g,n}^{\chi=\omega_{\mathrm{log}}}  }\end{xymatrix}
\end{equation*}
where the right and outer squares are homotopy pullback diagrams, so is the left square.
\end{proof}
%For any $H$-invariant closed subscheme $Z\subseteq W^n$, we can do base change of diagram \eqref{diag compare f and underly f} via the embedding $Z\subseteq W^n$.  
\begin{lemma}\label{lem on relate Rchi twist chi twist2}
Let $Z\subseteq W^n$ be a $H$-invariant closed subscheme.
Then we have the following homotopy pullback diagram of derived stacks
\begin{equation}\label{diag compare f and underly f2}
\begin{xymatrix}{
 [Z/H_R^n]\times_{[\emph{pt}/H_R]^n}\fBun_{H_R,g,n}^{R_{\chi}=\omega_{\mathrm{log}}}   \ar[r]\ar[d]^{} \ar@{}[dr]|{\Box} &  [Z/H^n]\times_{[\emph{pt}/H]^n}\fBun_{H,g,n}^{\chi=\omega_{\mathrm{log}}} \ar[d] \\
\fBun_{H_R,g,n}^{R_{\chi}=\omega_{\mathrm{log}}} \ar[r]^{\eta} & \fBun_{H,g,n}^{\chi=\omega_{\mathrm{log}}}.
}\end{xymatrix}
\end{equation}
Here we treat classical stacks as derived stacks via the natural inclusion. 
\end{lemma}
\begin{proof}
Extending diagrams \eqref{diag on map f}, \eqref{base cha of smooth quot stack}, we have Cartesian diagrams of classical stacks: 
%\begin{equation}\label{base cha of smooth quot stack2}\begin{xymatrix}{
% [Z/H_R^n] \ar[d] \ar[r] \ar@{}[dr]|{\Box} &[W^n/H_R^n] \ar[d] \ar[r] \ar@{}[dr]|{\Box} & [\pt/H_R^n]  \ar[d]^{}   \\ 
%[Z/H^n]  \ar[r] &  [W^n/H^n]   \ar[r]  & [\pt/H^n]. }  \end{xymatrix}\end{equation}
\begin{equation} 
\label{diag on map f another2}
\begin{xymatrix}{
[Z/H_R^n]\times_{[\pt/H_R]^n}\fBun_{H_R,g,n}^{R_\chi=\omega_{\mathrm{log}}}  \ar[r]   \ar[d]  \ar@{}[dr]|{\Box}  & [W/H_R]^n\times_{[\pt/H_R]^n}\fBun_{H_R,g,n}^{R_\chi=\omega_{\mathrm{log}}} \ar[r]\ar[d]^{} \ar@{}[dr]|{\Box} &  \fBun_{H_R,g,n}^{R_\chi=\omega_{\mathrm{log}}} \ar[d] \\
[Z/H_R^n]  \ar[r] \ar[d] \ar@{}[dr]|{\Box} &  [W^n/H_R^n] \ar[r] \ar[d] \ar@{}[dr]|{\Box} & [\pt/H_R^n] \ar[d] \\
[Z/H^n]  \ar[r] &  [W^n/H^n] \ar[r] & [\pt/H^n].
}\end{xymatrix}
\end{equation}
We claim the right two vertical maps are smooth, so the diagrams are also homotopy pullback diagrams. 
To prove the right upper vertical map is smooth, recall the following Cartesian diagram 
\begin{equation*} 
\begin{xymatrix}{
 \prod_{i=1}^n p_i^*\mathcal{P}^n \ar[d] \ar[r]  \ar@{}[dr]|{\Box}  & \mathcal{P}^n \ar[d]^{} \ar[r]^{} \ar@{}[dr]|{\Box}  & \pt \ar[d]^{} \  \\
 \fBun_{H_R,g,n}^{R_\chi=\omega_{\mathrm{log}}} \ar[r]^{\quad \quad \prod_{i=1}^n p_i}
 & \mathcal{C}^n \ar[r] & [\pt/H_R]^n, }\end{xymatrix}
\end{equation*}
where $\mathcal{C}$ is the universal curve, $\mathcal{P}$ is the universal $H_R$-bundle over $\mathcal{C}$ and 
$p_i$ is given by the $i$-th marked point. Since $H_R$ and $\fBun_{H_R,g,n}^{R_\chi=\omega_{\mathrm{log}}}$ are smooth, so is $\prod_{i=1}^n p_i^*\mathcal{P}^n$, 
therefore the claim holds. 
The right lower vertical map is smooth as $[H/H_R]$ is smooth.  

Similarly we also have the homotopy pullback diagram 
\begin{equation} 
\label{diag on map f another3}
\begin{xymatrix}{
[Z/H^n]\times_{[\pt/H]^n}\fBun_{H,g,n}^{\chi=\omega_{\mathrm{log}}}  \ar[r]   \ar[d]  \ar@{}[dr]|{\Box}  & [W/H]^n\times_{[\pt/H]^n}\fBun_{H,g,n}^{\chi=\omega_{\mathrm{log}}} \ar[r]\ar[d]^{} \ar@{}[dr]|{\Box} &  \fBun_{H,g,n}^{\chi=\omega_{\mathrm{log}}} \ar[d] \\
[Z/H^n]  \ar[r] &  [W^n/H^n] \ar[r] & [\pt/H^n].
}\end{xymatrix}
\end{equation}
By a diagram chasing on \eqref{diag on map f another2}, \eqref{diag on map f another3}, we obtain \eqref{diag compare f and underly f2}. 
 \end{proof}
\begin{remark}\label{rmk on facti}
As argued in Proposition \ref{prop:ev_equi}, we have a factorization of the evaluation map 
$$\fBun_{H_R,g,n}^{R_{\chi}=\omega_{\mathrm{log}}}\to [\pt/(G\times R(\Ker R_\chi))]\to  [\pt/H_R]. $$
Combining with the Cartesian diagram (as in \eqref{base cha of smooth quot stack}):
\begin{equation*} 
\begin{xymatrix}{
 [W/(G\times R(\Ker R_\chi))] \ar[d] \ar[r] \ar@{}[dr]|{\Box}  & [\pt/(G\times R(\Ker R_\chi))]  \ar[d]^{}   \\ 
 [W/H_R]   \ar[r]  & [\pt/H_R], }  \end{xymatrix}
\end{equation*}
we obtain an isomorphism of stacks: 
\begin{equation}\label{iso of bas chag}[W/(G\times R(\Ker R_\chi))]^n\times_{[\pt/(G\times R(\Ker R_\chi)]^n}\fBun_{H_R,g,n}^{R_{\chi}=\omega_{\mathrm{log}}} \cong [W/H_R]^n\times_{[\pt/H_R]^n}\fBun_{H_R,g,n}^{R_{\chi}=\omega_{\mathrm{log}}}. \end{equation}
If $Z\subseteq W^n$ is a $H$-invariant closed subscheme, we similarly have an isomorphism of stacks: 
\begin{equation}\label{iso of bas chag2}[Z/(G\times R(\Ker R_\chi))^n]\times_{[\pt/(G\times R(\Ker R_\chi)]^n}\fBun_{H_R,g,n}^{R_{\chi}=\omega_{\mathrm{log}}} \cong [Z/H_R^n]\times_{[\pt/H_R]^n}\fBun_{H_R,g,n}^{R_{\chi}=\omega_{\mathrm{log}}}. \end{equation}
\end{remark}
Now consider the \textit{classical truncation} of $\bff$ and $\bf\underline{f}$: 
\begin{align}\label{equ on cla f} 
\begin{split} 
& f=t_0(\bff): M:=t_0\left(\bMap^{R_{\chi}=\omega_{\mathrm{log}}}_{g,n}([Y/H_R]) \right) \to 
[W/H_R]^n\times_{[\pt/H_R]^n}\fBun_{H_R,g,n}^{R_{\chi}=\omega_{\mathrm{log}}}, \\
& \underline{f}=t_0(\textbf{\underline{f}}): \underline{M}:=t_0\left(\bMap^{\chi=\omega_{\mathrm{log}}}_{g,n}([Y/H]) \right) \to [W/H]^n\times_{[\pt/H]^n}\fBun_{H,g,n}^{\chi=\omega_{\mathrm{log}}}.
\end{split} 
\end{align}
The restrictions $\bbE_f:=\bbL_\textbf{f}|_{M}$, $\bbE_{\underline{f}}:=\bbL_\textbf{\underline{f}}|_{\underline{M}}$ of the derived cotangent complexes 
to the classical truncations induce morphisms in derived categories (ref.~\cite[Prop.~1.2]{STV}):
$$\bbE_f\to \bbL_f, \quad \bbE_{\underline{f}}\to \bbL_{\underline{f}},  $$ 
whose compositions with the truncation $\bbL_{\bullet} \to \tau^{\geqslant  -1}\bbL_{\bullet}=:\bfL_{\bullet}$ give morphisms
$$\phi_f: \bbE_f\to \bfL_{f}, \quad \phi_{\underline{f}}: \bbE_{\underline{f}}\to \bfL_{\underline{f}}.$$
\begin{theorem}\label{prop:symm_ob} 
Let $Z\subseteq W^n$ be a $H$-invariant closed subscheme such that $\Crit(\phi)^n\subseteq Z\subseteq Z(\boxplus^n\phi)$.
Then after base change to $[Z/H_R^n]\times_{[\emph{pt}/H_R]^n}\fBun_{H_R,g,n}^{R_\chi=\omega_{\mathrm{log}}}$ $($resp.~$[Z/H^n]\times_{[\emph{pt}/H]^n}\fBun_{H,g,n}^{\chi=\omega_{\mathrm{log}}}$$)$,
$\phi_f$ $($resp.~$\phi_{\underline{f}}$$)$ are isotropic symmetric obstruction theories in the sense Definitions \ref{def on sym ob}, \ref{def on iso sym ob}. 
\end{theorem}
\begin{proof}
%Let $\underline{M}:=Map_{g,n}^{\chi=\omega_{\mathrm{log}}}([Y/H])$ be the classical truncation of $\bMap_{g,n}^{\chi=\omega_{\mathrm{log}}}([Y/H])$. 
By Proposition \ref{prop on rel tan cpx of f}, we know $\bbE_{\underline{f}}:=\bbL_{\textbf{\underline{f}}}|_{\underline{M}}$ is a symmetric complex.  
Lemma \ref{lem on relate Rchi twist chi twist} implies 
$$\bbL_\textbf{f}=\textbf{h}^*\bbL_\textbf{\underline{f}}. $$
Therefore $\bbE_f:=\bbL_\textbf{f}|_{M}$ is also a symmetric complex, which 
we spell out explicitly as follows.
Let $\pi:\mathcal{C}\to M$ denote the universal curve, $\mathcal{P}\to \mathcal{C}$ be the universal $H_R$-bundle and $\calW:=\mathcal{P}\times_{H_R}W$. 
As in Proposition \ref{prop on rel tan cpx of f}, we have 
\begin{equation}\label{equ on repe1}
\bbE_f\cong \left( \dR\pi_*\left(\calW\boxtimes\left(\omega_{\pi,\mathrm{log}}^{\vee}\otimes \omega_{\pi}\right)\right)\to \dR\pi_*\left(\calW^\vee\boxtimes\omega_{\pi,\mathrm{log}}\right)\right). \end{equation} 
% the map is self dual as it is induced from Hessian $Hess(\phi): W\to W^{\vee}\otimes \mathbb{C}_{\chi}$ which is self-dual (up to twist back \mathbb{C}_{\chi})
The relative Serre duality 
$$\dR\pi_*(\calW^\vee\boxtimes\omega_{\pi,\mathrm{log}})\cong \left(\dR\pi_*\left(\calW\boxtimes\left(\omega_{\pi,\mathrm{log}}^{\vee}\otimes \omega_{\pi}\right)\right)\right)^\vee[-1] $$
defines a non-degenerate symmetric form on $\bbE_f$: 
$$\calO\to (\bbE_f\otimes\bbE_f)[-2], $$
and a canonical choice of orientation of it (ref.~Remark \ref{rmk on ori}) by the canonical trivialization 
$$\calO\cong \det(\dR\pi_*(\calW^\vee\boxtimes\omega_{\pi,\mathrm{log}}))\otimes \det(\dR\pi_*(\calW\boxtimes(\omega^\vee_{\pi,\mathrm{log}}\otimes \omega_{\pi}))[1]). $$
%of the determinant of $\bbE_f$.
The tor-amplitudes of $\bbE_f$, $\bbE_{\underline{f}}$ are obviously in $[-2,0]$.
%As $QM$ is Deligne-Mumford, the tor-amplitude of $\bbE_f$ is in non-positive degree, duality implies it is in fact of tor-amplitude in $[-2,0]$.
By \cite[Prop.~1.2]{STV}, we know both $\phi_f$ and $\phi_{\underline{f}}$
satisfy that $h^0$ is isomorphic and $h^{-1}$ is surjective, i.e. they are symmetric obstruction theories. 

Next we show the isotropic condition after the specified base change.
Consider the base change of \eqref{equ on cla f} via the embedding $Z\subseteq W^n$: 
\begin{equation*} 
\begin{xymatrix}{
M \ar[d]_{f} \ar[r]^{=} \ar@{}[dr]|{\Box}  & M  \ar[d]^{f}   \\ 
 [Z/H_R^n]\times_{[\pt/H_R]^n}\fBun_{H_R,g,n}^{R_{\chi}=\omega_{\mathrm{log}}}  \ar[r]  &[W/H_R]^n\times_{[\pt/H_R]^n}\fBun_{H_R,g,n}^{R_{\chi}=\omega_{\mathrm{log}}}, }  \end{xymatrix}
\end{equation*}
\begin{equation*} 
\begin{xymatrix}{
\underline{M} \ar[d]_{\underline{f}} \ar[r]^{=} \ar@{}[dr]|{\Box}  & \underline{M}  \ar[d]^{\underline{f}}   \\ 
 [Z/H^n]\times_{[\pt/H]^n}\fBun_{H,g,n}^{\chi=\omega_{\mathrm{log}}}  \ar[r]  &[W/H]^n\times_{[\pt/H]^n}\fBun_{H,g,n}^{\chi=\omega_{\mathrm{log}}}, }  \end{xymatrix}
\end{equation*}
where the base change of $M$ and $\underline{M}$ keeps the same as evaluation maps factor through $\Crit(\phi)^n\subseteq Z$.
Here we denote the maps after base change using same notations for simplicity.  

Combining with Lemmata \ref{lem on relate Rchi twist chi twist} and \ref{lem on relate Rchi twist chi twist2}, we obtain a Cartesian diagram of classical stacks
\begin{equation*} 
\begin{xymatrix}{
M  \ar[r]^{h=t_0(\textbf{h})}\ar[d]_{f} \ar@{}[dr]|{\Box} &  \underline{M} \ar[d]^{\underline{f}} \\
 [Z/H_R^n]\times_{[\pt/H_R]^n}\fBun_{H_R,g,n}^{R_{\chi}=\omega_{\mathrm{log}}}   \ar[r]   &  [Z/H^n]\times_{[\pt/H]^n}\fBun_{H,g,n}^{\chi=\omega_{\mathrm{log}}}.  
}\end{xymatrix}
\end{equation*}
And the pullback of $\phi_{f}$ and $\phi_{\underline{f}}$ defines symmetric obstruction theories on the vertical maps. 

Therefore, we have embeddings of cone stacks   
%\begin{equation*}\begin{xymatrix}{
%\fC_{f} \ar@{^{(}->}[r]^{ }  \ar@{^{(}->}[d]^{ }  & h^*\fC_{\underline{f}} \ar@{^{(}->}[d]^{ } &  \\
%\fC_{\bbL_{\textbf{f}}|_M}  \ar[r]^{\cong\,\,\,} & h^*\fC_{\bbL_{\textbf{\underline{f}}}|_{\underline{M}}} \ar[r]^{\fq_{\bbL_{\textbf{f}}|_M}}  & \bbA^1_{M},
%}\end{xymatrix} \end{equation*}
\begin{equation*}\begin{xymatrix}{
\fC_{f} \ar@{^{(}->}[r]^{ }  \ar@{^{(}->}[d]^{ }  & h^*\fC_{\underline{f}} \ar@{^{(}->}[d]^{ } &  \\
\fC_{\bbE_{f}}  \ar[r]^{\cong\,\,\,} & h^*\fC_{\bbE_{\underline{f}}} \ar[r]^{\fq_{\bbE_{f}}}  & \bbA^1_{M},
}\end{xymatrix} \end{equation*}
where the horizontal embedding follows from \cite[Prop.~2.26]{Man}.
%$$\fC_{f} \hookrightarrow \fN_f:=\fC_{\bbL_{f}} \hookrightarrow \fC_{\bbE_f}, $$
%where the first (second) term is called the intrinsic normal cone (sheaf) of $f$ \cite[\S 7]{BF}. 
By Equ.~\eqref{def of qua eq2} and the above diagram, to show $\fq_{\bbE_f}|_{\fC_{f}}=0$, it is enough to show 
$\fq_{\bbE_{\underline{f}}}|_{\fC_{\underline{f}}}=0$, which 
we prove by taking a cover. 

For any flat morphism $\underline{\sigma}: \Spec k\to \fBun_{H,g,n}^{\chi=\omega_{\mathrm{log}}}$, denote its base change to 
$[W/H]^n\times_{[\pt/H]^n}\fBun_{H,g,n}^{\chi=\omega_{\mathrm{log}}}$ (resp.~$[Z/H^n]\times_{[\pt/H]^n}\fBun_{H,g,n}^{\chi=\omega_{\mathrm{log}}}$) 
by $\Spec K$ (resp.~$\Spec K'$), i.e.
\begin{equation*}\begin{xymatrix}{
\Spec K' \ar[d]_{ }\ar[r]^{\sigma\quad \quad \quad\quad \quad  }\ar@{}[dr]|{\Box} &[Z/H^n]\times_{[\pt/H]^n}\fBun_{H,g,n}^{\chi=\omega_{\mathrm{log}}} \ar[d]^{ } \\
\Spec K \ar[d]_{ }\ar[r]^{}\ar@{}[dr]|{\Box} &[W/H]^n\times_{[\pt/H]^n}\fBun_{H,g,n}^{\chi=\omega_{\mathrm{log}}} \ar[d]^{ } \\
\Spec k \ar[r]^{\underline{\sigma}\quad} &\fBun_{H,g,n}^{\chi=\omega_{\mathrm{log}}}. 
}\end{xymatrix} \end{equation*}
Here the fiber products are affine as the right vertical maps are affine. 

We have the following Cartesian diagram (below $\underline{M}'(k)$ is defined by the diagram): 
\begin{equation*}\begin{xymatrix}{
\fC_{\underline{g}}\cong \bar{\sigma}^*\fC_{ \underline{f}} \ar@{^{(}->}[d]_{i}\ar[r]^{}& \fC_{ \underline{f}}  \ar@{^{(}->}[d]^{j}  \\
\fC_{\bar{\sigma}^*(\bbE_{\underline{f}})}\cong \bar{\sigma}^*\fC_{\bbE_{\underline{f}}}   \ar[d]^{ }\ar[r]^{\quad\quad \quad \hat{\sigma}}&\fC_{\bbE_{\underline{f}}} \ar[d]^{ }\ar[r]^{\fq_{\bbE_{\underline{f}}}}  & \bbA^1_{\underline{M}} \\
\underline{M}'(k) \ar[d]_{\underline{g}}\ar[r]^{\bar{\sigma}}&\underline{M} \ar[d]^{\underline{f}}  \\
\Spec K' \ar[r]^{\sigma\quad \quad\quad \quad\quad } & [Z/H^n]\times_{[\pt/H]^n}\fBun_{H,g,n}^{\chi=\omega_{\mathrm{log}}}. 
}\end{xymatrix} \end{equation*}
%\begin{equation*}\begin{xymatrix}{
%\fC_{\underline{g}}\cong \bar{\sigma}^*\fC_{ \underline{f}} \ar@{^{(}->}[d]_{i}\ar[r]^{}& \fC_{ \underline{f}}  \ar@{^{(}->}[d]^{j}  \\
%\fC_{\bar{\sigma}^*(\bbL_{\textbf{\underline{f}}}|_{\underline{M}})}\cong \bar{\sigma}^*\fC_{\bbL_{\textbf{\underline{f}}}|_{\underline{M}}}   \ar[d]^{ }\ar[r]^{\quad\quad \quad \hat{\sigma}}&\fC_{\bbL_{\textbf{\underline{f}}}|_{\underline{M}}} \ar[d]^{ }\ar[r]^{\fq_{\bbL_{\textbf{\underline{f}}}|_{\underline{M}}}}  & \bbA^1_{\underline{M}} \\
%\underline{M}(k) \ar[d]_{\underline{g}}\ar[r]^{\bar{\sigma}}&\underline{M} \ar[d]^{\underline{f}}  \\
%\Spec K \ar[r]^{\sigma\quad \quad\quad \quad\quad } & [W/H]^n\times_{[\pt/H]^n}\fBun_{H,g,n}^{\chi=\omega_{\mathrm{log}}}. 
%}\end{xymatrix} \end{equation*}
Here the isomorphism in the left up corner follows from \cite[Prop.~2.26]{Man}.
By the base change property \eqref{def of qua eq2}, 
to show $\fq_{\bbE_{\underline{f}}}|_{\fC_{\underline{f}}}=0$, it is enough to show $\fq_{\bbE_{\underline{f}}}\circ \hat{\sigma}\circ i=0$,~i.e.~$\fq_{\bar{\sigma}^*(\bbE_{\underline{f}})}|_{\fC_{\underline{g}}}=0$ for any diagram as above.
%Obviously we can further reduce to show the vanishing on any atlas $U\to \underline{M}(k)$ when $k$ is a local ring. 

Note that $\underline{M}'(k)$ (over $K'$) has a derived enhancement to a $(-2)$-shifted symplectic derived stack $(\textbf{\textit{\underline{M}}}'(k),\iota^*\Omega_{\textbf{\textit{\underline{M}}}(k)})$ (over $K'$)
%$(\textbf{\textit{\underline{M}}} (k),\Omega_{\textbf{\textit{\underline{M}}} (k)})$ 
as constructed in Proposition~\ref{lem:HN_HP} via Theorem~\ref{thm:sympl_marked} , where the complex 
$(\bar{\sigma}^*\bbE_{\underline{f}})$ is the restriction of the derived cotangent complex to the underlying classical part. 
By Proposition~\ref{lem:HN_HP}, the
image of $[\iota^*\Omega_{\textbf{\textit{\underline{M}}}(k)}]$ under the map 
$$HN^{-4}(\textbf{\textit{\underline{M}}}'(k)/\Spec K')(2)\to HP^{-4}(\textbf{\textit{\underline{M}}}'(k)/\Spec K')(2) $$
is zero.
By \cite{Park2} which is based on \cite{BG, BBJ}, we know $\fq_{\bar{\sigma}^*(\bbE_{\underline{f}})}|_{\fC_{\underline{g}}}=0$, hence we are done. 
\end{proof}
\begin{remark}
Similar to \cite{OT}, the symmetric obstruction theory constructed above depends only on the underlying 
$(-2)$-shifted 2-form of the shifted symplectic 
structure constructed in \S \ref{sect on sss I}, \S \ref{subsec:marked}. Nevertheless, the verification of the isotropic condition in symmetric obstruction theory here relies on the $(-2)$-shifted symplectic 
structure. 
\end{remark}
Next we define virtual pullbacks for moduli stacks of quasimaps. 
Let $$QM_{g,n}^{R_{\chi}=\omega_{\mathrm{log}}}(\Crit(\phi)/\!\!/G,\beta)\subset t_0\left(\bMap^{R_{\chi}=\omega_{\mathrm{log}}}_{g,n}([Y/H_R]) \right)$$ 
be the open substack where quasimap stability (in Definition \ref{def:QM}) is imposed. 
In the rest of this section, we work under the following simplifying condition. 
\begin{ass}\label{ass on Rchi}
We assume $\Ker(R_{\chi})=1$. 
\end{ass}
By Proposition \ref{prop:ev_equi 2}, we know the evaluation maps factor through the stable locus:
$$ev_i: QM_{g,n}^{R_{\chi}=\omega_{\mathrm{log}}}(\Crit(\phi)/\!\!/G,\beta)\to \Crit(\phi)^s/G \subset W^s/G, \quad \forall\,\,i=1,2,\ldots,n. $$
Therefore the map \eqref{equ on cla f} restricts to  
\begin{equation}\label{equ on f qm} f: QM_{g,n}^{R_{\chi}=\omega_{\mathrm{log}}}(\Crit(\phi)/\!\!/G,\beta)\to [W^s/G]^n\times_{[\pt/G]^n}\fBun_{H_R,g,n}^{R_{\chi}=\omega_{\mathrm{log}}}, \end{equation}
where we have used the isomorphism \eqref{iso of bas chag} for the target. 

Let $Z\subseteq W^n$ be an $H$-invariant closed subscheme  
such that there are closed embedding
\begin{equation}\label{equ on Z}\Crit(\phi)^n \hookrightarrow Z \hookrightarrow Z(\boxplus^n\phi).  \end{equation}
Denote the stable locus by 
\begin{equation}\label{equ on st locus}Z^s:=Z\cap (W^s)^n. \end{equation}
By base change of \eqref{equ on f qm}, we obtain 
\begin{equation}\label{equ on f qm2} f: QM_{g,n}^{R_{\chi}=\omega_{\mathrm{log}}}(\Crit(\phi)/\!\!/G,\beta)\to [Z^s/G^n]\times_{[\pt/G]^n}\fBun_{H_R,g,n}^{R_{\chi}=\omega_{\mathrm{log}}}, \end{equation}
where the domain keeps the same as evaluation maps factor through $[(\Crit(\phi)^n)^s/G^n]\subset [Z^s/G^n]$. 

By Theorem \ref{prop:symm_ob}, the base change of $\phi_f$ to \eqref{equ on f qm2} 
gives an isotropic symmetric obstruction theory which enables us to define a square root virtual pullback as \eqref{eqn:sqrt_pull}.
By Proposition \ref{prop:ev_equi}, the map $f$ is $F$-equivariant. As the Hessian of $\phi$
in Proposition \ref{prop on cot cx of cri loc} is equivariant under the action of Calabi-Yau torus $F_0\subseteq F$, so 
the symmetric obstruction theory $\phi_f$ is $F_0$-equivariant. 

To sum up, we have the following $F_0$-\textit{equivariant square root virtual pullback}. 
%We define virtual classes for QM moduli spaces.  
%\yl{Mention about $F$-equivariance}
\begin{definition}\label{defi of qm vir class}
Let  $Z\subseteq W^n$ be a $H$-invariant closed subscheme such that \eqref{equ on Z} holds. Then we have a group homomorphism 
\begin{equation}\label{eqn on virt pb}\sqrt{f^!}: A^{F_0}_*\left(\fBun_{H_R,g,n}^{R_{\chi}=\omega_{\mathrm{log}}}\times_{[\pt/G]^n}[Z^s/G^n]\right)\to 
A^{F_0}_*\left(QM_{g,n}^{R_{\chi}=\omega_{\mathrm{log}}}(\Crit(\phi)/\!\!/G,\beta)\right). \end{equation}
Here $A^{F_0}_{*}(-)$ denotes the $F_0$-equivariant Chow group.
\end{definition}
\begin{remark}
One can also define square root virtual pullbacks in $K$-theory by \cite[App.~B]{Park}.
By \cite[Prop.~1.15, Def.~A.3]{Park}, we know the above pullback map is determined by the pullback map when 
$Z=Z(\boxplus^n\phi)$.   
\end{remark}
\begin{remark}
The \textit{degree shift} in the above group homomorphism is calculated by 
$$\rk_{\mathbb{C}} \dR\pi_*\left(\calW\boxtimes\left(\omega_{\pi,\mathrm{log}}^{\vee}\otimes \omega_{\pi}\right)\right)=
\int_{\beta}c_1(P\times_{(G\times \mathbb{C}^*)} W)+(1-g-n)\dim_{\mathbb{C}} W, $$
where $P$ is any principal $(G\times \mathbb{C}^*)$-bundle on a genus $g$ curve $C$. 
\end{remark}
%\begin{YC}Include the explicit virtual dim here. \end{YC}
\begin{remark}\label{rmk on pullback without embed}
Here we work in Setting \ref{setting of glsm} and have embedding $\Crit(\phi)\hookrightarrow Z(\phi)$.
Note that $\Crit(\phi)^n=\Crit(\boxplus^n\phi)$ as closed subscheme in $W^n$.
In general, by Remark \ref{rmk on crit emb to zero}, for some $r\geqslant 1$, we have an embedding $\Crit(\boxplus^n\phi)\hookrightarrow Z((\boxplus^n\phi)^r)$ as closed subschemes in $W^n$. Using Remark \ref{rmk on iso still holds}, we may simply take $\Crit(\phi)^n\hookrightarrow Z \hookrightarrow Z((\boxplus^n\phi)^r)$ in Definition \ref{defi of qm vir class} and hereafter in general.
\end{remark}

\subsection{Properties of virtual pullbacks}\label{sect on glu}
In this section, we show several properties of virtual pullback \eqref{eqn on virt pb} which will be used to prove a gluing formula in \S \ref{subsec:CohFT_glue}. 
The formulation is similar to case of (twisted) Gromov-Witten theory, and quasimap theory to smooth GIT quotients  \cite[\S 5.3]{AGV}, \cite[pp.~608]{B}, \cite[\S 6.3]{CiK1}, \cite[\S 2.3.3]{CiK3}. 

We write $QM_{g,n}^{R_{\chi}=\omega_{\mathrm{log}}}(\Crit(\phi)/\!\!/G,\beta)$ simply as $QM_{g,n}(\beta)$ or $QM_{g,n}$ if $\beta$ is not relevant in the discussion. 
Let $Z\subseteq W^n$ be an $H$-invariant closed subscheme  such that \eqref{equ on Z} holds.
We work under Assumption \ref{ass on Rchi}. 

\subsubsection{Normalization of nodal curves}

Let $n_1,n_2,g_1,g_2$ be non-negative integers and 
$$n=n_1+n_2, \quad g=g_1+g_2. $$
%Let $\mathfrak{M}_{g,n}$ be the moduli stack of prestable curves of genus $g$ with $n$-marked points which is a smooth Artin stack locally of finite type. 
We have the following Cartesian diagram: 
%which defines $QM_{g,n,\mathrm{node}}$ and $f_{\mathrm{node}}=\mu_{\mathrm{node}}\times ev^n$  in particular consists of nodal curves hence the name.
\begin{align} 
\label{eqn:degeneration0}
\xymatrix{
\coprod_{\beta_1+\beta_2=\beta}QM_{g_1,n_1+1}(\beta_1)  \times_{X} QM_{g_2,n_2+1}(\beta_2) \ar[r]^{ } \ar[d]_{ }  \ar@{}[dr]|{\Box} &  QM_{g,n}(\beta) \ar[d]^{ }    \\
\fBun_{H_R,g_1,n_1+1}^{R_{\chi}=\omega_{\mathrm{log}}}\times_{[\pt/G]} \fBun_{ H_R,g_2,n_2+1}^{R_{\chi}=\omega_{\mathrm{log}}} 
\ar[r]^{  } \ar[d] \ar@{}[dr]|{\Box} & 
\fBun_{H_R,g,n}^{R_{\chi}=\omega_{\mathrm{log}}}  \ar[d]   \\
\mathfrak{M}_{g_1,n_1+1}  \times \mathfrak{M}_{g_2,n_2+1}   \ar[r]^{\quad\quad\quad gl} & \mathfrak{M}_{g,n},     }
\end{align}
where $gl$ in the bottom is the gluing morphism that identifies $(n_1+1)$-th and $(n_2+1)$-th marked point, which is finite and unramified \cite[Prop.~5.2.2, Lem.~6.2.4]{AGV}.
In fact, it is the composition of a finite flat morphism and a base change of a regular closed immersion (see diagram \eqref{big diag gl for}).
% the map $i_{\mathfrak{Q}}$ is a finite map between smooth stack, so it is flat. similar statement in \cite[Lem. 6.2.3]{AGV}

%By diagram \eqref{diag on map f}, 
The upper-right  vertical map in above factors through the following map $f$, and we obtain the following 
 Cartesian diagram:
\begin{align} 
\label{eqn:degeneration}
\xymatrix{
\coprod_{\beta_1+\beta_2=\beta}QM_{g_1,n_1+1}(\beta_1)  \times_{X} QM_{g_2,n_2+1}(\beta_2) \ar[r]^{\quad\quad\quad\quad\quad\quad\quad\,\, gl} \ar[d]_{f_{\mathrm{node}}} \ar@{}[dr]|{\Box}  &  QM_{g,n}(\beta) \ar[d]^{f}    \\
\fBun_{H_R,g_1,n_1+1}^{R_{\chi}=\omega_{\mathrm{log}}}\times_{[\pt/G]} \fBun_{ H_R,g_2,n_2+1}^{R_{\chi}=\omega_{\mathrm{log}}} \times_{[\pt/G]^n} [Z^s/G^n]
\ar[r]^{ \quad\quad\quad\quad \quad \,\,\, gl} \ar[d] \ar@{}[dr]|{\Box} & 
\fBun_{H_R,g,n}^{R_{\chi}=\omega_{\mathrm{log}}} \times_{[\pt/G]^n}  [Z^s/G^n] \ar[d]   \\
\mathfrak{M}_{g_1,n_1+1}  \times \mathfrak{M}_{g_2,n_2+1}   \ar[r]^{\quad\quad\quad gl} & \mathfrak{M}_{g,n},     }
\end{align}
where $Z\subseteq W^n$ is any $H$-invariant closed subscheme which satisfies \eqref{equ on Z}. 
%(Our QM spaces differ from \cite{CiK1, CiKM} by imposing $d\phi=0$ and twisting condition. It is easy to see they are preserved by the gluing morphism, so the above diagram commutes.)
By Theorem \ref{prop:symm_ob} and Definition \ref{defi of qm vir class}, the map $f$ has an isotropic symmetric obstruction theory whose 
pullback along $gl$ gives an isotropic symmetric obstruction theory of $f_{\mathrm{node}}$ (e.g.~\cite[Eqn.~(1.14)]{Park}). 
The following is straightforward from \cite[Prop.~1.15, Def.~A.3]{Park}.
\begin{prop}\label{prop on gluing mor}
Notations as above, we have
\begin{equation*}
(1) \, \,\, gl^!\circ \sqrt{f^!}=\sqrt{f_{\mathrm{node}}^!}\circ gl^!, \quad (2) \,\,\, \sqrt{f^!}\circ gl_*=gl_*\circ\sqrt{f_{\mathrm{node}}^!}.
\end{equation*}
\end{prop}

%This proposition implies that virtual fundamental class of $QM_{g,n}$ specializes to that of $ QM_{g_1,n_1+1}  \times_{W/\!\!/G} QM_{g_2,n_2+1}$.

\subsubsection{Gluing nodal curves}

%$$\fBun_{H_R,g_1,n_1+1}^{R_{\chi}=\omega_{\mathrm{log}}}\times_{[\pt/H_R]} \fBun_{ H_R,g_2,n_2+1}^{R_{\chi}=\omega_{\mathrm{log}}} \times_{[\pt/H_R]^n} [W/H_R]^n$$
Recall Setting \ref{setting of glsm} and let $X:=W^s/G$. Denote 
\begin{equation}\label{ord 2 autom}\sigma:W\to W \end{equation}
to be an \textit{automorphism} 
commuting with the action of $G\times F_0$, so that $\sigma^{*}\phi=-\phi$.  

The above automorphism obviously preserves $\Crit(\phi)$:
$$\sigma:\Crit(\phi)\to \Crit(\phi), $$
and also induces an automorphism on $X$:
$$ \sigma:X\to X. $$
\begin{remark}\label{rmk on sigma}
%Notice here that $\sigma^*_{\check\bullet}$ on LHS of Lemma \ref{prop:Can_pairing} on the Borel-Moore homology $H_{F_0}^{BM}(C^{I\sqcup \{\bullet,\check\bullet\}})$. 
(1) The motivation to add the automorphism $\sigma$ is to make the virtual pullback work for gluing curves, see Remark \ref{rmk on sigma2}.

(2) When the involution $\sigma$ is homotopic to the identify, i.e., fitting into a continuous $[0,1]$-family of operators $X\to X$ that commutes with $F_0$-action, then 
%\yl{chow gp version?} 
$$\sigma^*_{\check\bullet}=\id: H_{F_0}^{BM}(X)\to H_{F_0}^{BM}(X). $$
(3) Recall the $R$-charge as in Definition \ref{defi of R-charge} and Setting \ref{setting of glsm}. Assume the composition
$$R_\chi: \bbC^*\stackrel{R}{\to} F \stackrel{\chi}{\to} \bbC^* $$
is a nontrivial map (so it is also surjective), then we can take $\sigma\in R_\chi^{-1}(-1)$ to be a preimage of $-1$.
Then $\sigma$ is homotopic to identify, commutes with the action of $G\times F_0$ and satisfies $\sigma^{*}\phi=-\phi$. \\
(4) Under Assumption \ref{ass on Rchi}, we simply take $\sigma=-1$ and the automorphism \eqref{ord 2 autom} has order two. 
%One example is when there is a $\bbC^*$-action on $X$, which scales $\phi$ and commutes with $F_0$-action (in particular applicable to our Setting \ref{setting of glsm}), 
%In Setting \ref{setting of glsm}, as the character $\chi: F\to \mathbb{C}^*$ is surjective, the condition holds. 
\end{remark}
%\begin{definition}Let $\sigma$ be an involution as \eqref{ord 2 autom}.Denote the closed embedding $$\bar{\Delta}: X\to X\times X, \quad x\mapsto (\sigma(x),x). $$ 
%\end{definition}
Recall the notation $(-)^s$ for stable locus \eqref{equ on st locus}. Then there are Cartesian diagrams
\begin{equation}\label{diag on Zboxphi}
\xymatrix{
\frac{Z(\boxplus^{n_1}\phi)^s}{G^{n_1}}\times \frac{Z(\boxplus^{n_2}\phi)^s}{G^{n_2}}
\times X  \ar[d]_{} \ar[r]^{ }  \ar@{}[dr]|{\Box} &  X^{n_1}\times X^{n_2} \times X  \ar[d]^{\bar{\Delta}}  \ar@{=}[r] \ar@{}[dr]|{\Box} & X^{n_1} \times X^{n_2} \times X \ar[d]^{\Delta} \\
%Z(d\phi)/\!\!/G  \ar[r]^{\Delta \quad \quad } \ar[d]_{ }  & Z(d\phi)/\!\!/G  \times Z(d\phi)/\!\!/G \ar[d]_{ }    \\
\frac{Z(\boxplus^{n_1}\phi)^s}{G^{n_1}}\times \frac{Z(\boxplus^{n_2}\phi)^s}{G^{n_2}}
\times \frac{Z(\boxplus^{2}\phi)^s}{G^{2}} \ar[r]^{} & 
X^{n_1} \times X^{n_2} \times X \times X \ar[r]^{\sigma^{-1}}  & X^{n_1} \times X^{n_2} \times X \times X ,     }
\end{equation}
where $\Delta$ is given by the \textit{diagonal embedding} $X\to X\times X$, i.e. 
$$\Delta(x_1,\ldots,x_{n_1},y_1,\ldots,y_{n_2},x)=(x_1,\ldots,x_{n_1},y_1,\ldots,y_{n_2},x,x), $$
$\sigma$ is applied to $X^{n_1}\times X$, i.e.
$$\sigma(x_1,\ldots,x_{n_1},y_1,\ldots,y_{n_2},x,y)=(\sigma x_1,\ldots,\sigma x_{n_1},y_1,\ldots,y_{n_2},\sigma x,y), $$
and $\bar{\Delta}$ satisfies 
$$\bar{\Delta}(x_1,\ldots,x_{n_1},y_1,\ldots,y_{n_2},x)=(\sigma x_1,\ldots,\sigma x_{n_1},y_1,\ldots,y_{n_2},\sigma x,x). $$

%Recall the notation $(-)^s$ for stable locus \eqref{equ on st locus}. 
Consider the following Cartesian diagram (which defines $f_{\bar{\Delta}}$):
\begin{align}\label{diag on diagonal} 
{\tiny
\xymatrix{
QM_{g_1,n_1+1}  \times_{\bar{\Delta}(X)} QM_{g_2,n_2+1}  \ar[d]_{i_{\bar{\Delta}}} \ar[r]^{f_{\bar{\Delta}} \quad \quad\quad \quad\quad\quad\quad\quad\quad\quad\quad\quad\quad\quad}  \ar@{}[dr]|{\Box} &  
\fBun_{H_R,g_1,n_1+1}^{R_{\chi}=\omega_{\mathrm{log}}}\times_{[\pt/G]^{n_1+1}}\times \frac{Z(\boxplus^{n_1}\phi)^s}{G^{n_1}}\times \frac{Z(\boxplus^{n_2}\phi)^s}{G^{n_2}}
\times X \times_{[\pt/G]^{n_2+1}}\fBun_{ H_R,g_2,n_2+1}^{R_{\chi}=\omega_{\mathrm{log}}}    \ar[d]^{\bar{\Delta} }  \\
%Z(d\phi)/\!\!/G  \ar[r]^{\Delta \quad \quad } \ar[d]_{ }  & Z(d\phi)/\!\!/G  \times Z(d\phi)/\!\!/G \ar[d]_{ }    \\
QM_{g_1,n_1+1}  \times QM_{g_2,n_2+1} \ar[r]^{f_1\times f_2 \quad \quad\quad \quad\quad\quad\quad\quad\quad\quad\quad\quad\quad\quad\quad\quad } & 
\fBun_{H_R,g_1,n_1+1}^{R_{\chi}=\omega_{\mathrm{log}}}\times_{[\pt/G]^{n_1+1}} 
\frac{Z(\boxplus^{n_1}\phi)^s}{G^{n_1}}\times \frac{Z(\boxplus^{n_2}\phi)^s}{G^{n_2}}
\times \frac{Z(\boxplus^{2}\phi)^s}{G^{2}} \times_{[\pt/G]^{n_2+1}}\fBun_{ H_R,g_2,n_2+1}^{R_{\chi}=\omega_{\mathrm{log}}} ,   }}
\end{align}
where $\bar{\Delta}$ is given by the embedding $\bar{\Delta}: X\to \frac{Z(\boxplus^{2}\phi)^s}{G^{2}}$ as in \eqref{diag on Zboxphi}, 
$f_1, f_2$ are defined as the map $f$ in \eqref{equ on f qm2}. Since $\frac{Z(\boxplus^{n_1}\phi)^s}{G^{n_1}}\times \frac{Z(\boxplus^{n_2}\phi)^s}{G^{n_2}}
\times \frac{Z(\boxplus^{2}\phi)^s}{G^{2}}$ satisfies \eqref{equ on Z} with respect to  the function $\boxplus^{n_1+n_2+2} \phi$, the map
$f_1\times f_2$ has a square root virtual pullback.\footnote{In general, we have a closed embedding 
$Z((\boxplus^{n_1}\phi)^{r_1})\times Z((\boxplus^{n_2}\phi)^{r_2})\times Z((\boxplus^{2}\phi)^{r_3})\hookrightarrow Z((\boxplus^{n_1+n_2+2}\phi)^{r})$ if $r\geqslant r_1+r_2+r_3$. By 
Remark \ref{rmk on iso still holds}, we have a square root virtual pullback without the condition $\Crit(\phi)\hookrightarrow Z(\phi)$ in Setting \ref{setting of glsm}.} 
Again by \cite[Prop.~1.15, Def.~A.3]{Park}, we have  
\begin{prop}
\begin{equation}\label{eqn:Delta}
    \sqrt{(f_{1}\times f_{2})^!}\circ \bar{\Delta}_*=i_{\bar{\Delta}*}\circ\sqrt{f_{\bar{\Delta}}^!}.
\end{equation}
\end{prop}
%In what follows, and in \S~\ref{subsec:CohFT_glue} we simply write $ev^{n+2}:=ev^{n_1+1}\times ev^{n_2+2}$, $ev^{n+1}:=ev^n\times ev_\Delta$, $\mu_1:= \mu_{g_1,n_1+1}$, $\mu_2:=\mu_{g_2,n_2+1}$
%We also have the morphism of derived stacks 
%\begin{equation}QM_{g,n,\mathrm{node}}\to QM_{g_1,n_1+1}  \times_{X} QM_{g_2,n_2+1} ,\end{equation} which induces an isomorphism on the underlying classical stacks. 
%Under this map, $\mu_1\times_X\mu_2=\mu_{\mathrm{node}}$ and hence $f_{\mathrm{node}}=\mu_1\times_X\mu_2\times ev^n$.
Consider also the Cartesian diagram (which defines $f_{\Delta}$):
\begin{align}\label{diag on diagonal2} 
{\tiny
\xymatrix{
QM_{g_1,n_1+1}  \times_{X} QM_{g_2,n_2+1}  \ar[d]_{i} \ar[r]^{f_{\Delta} \,\, \quad \quad\quad \quad\quad\quad\quad\quad\quad\quad\quad\quad\quad\quad}  \ar@{}[dr]|{\Box} &  
\fBun_{H_R,g_1,n_1+1}^{R_{\chi}=\omega_{\mathrm{log}}}\times_{[\pt/G]^{n_1+1}}\times \frac{Z(\boxplus^{n_1}\phi)^s}{G^{n_1}}\times \frac{Z(\boxplus^{n_2}\phi)^s}{G^{n_2}}
\times X \times_{[\pt/G]^{n_2+1}}\fBun_{ H_R,g_2,n_2+1}^{R_{\chi}=\omega_{\mathrm{log}}}    \ar[d]^{\Delta }  \\
%Z(d\phi)/\!\!/G  \ar[r]^{\Delta \quad \quad } \ar[d]_{ }  & Z(d\phi)/\!\!/G  \times Z(d\phi)/\!\!/G \ar[d]_{ }    \\
QM_{g_1,n_1+1}  \times QM_{g_2,n_2+1} \ar[r]^{f_1\times f_2  \,\,\, \quad\quad \quad\quad\quad\quad\quad\quad\quad\quad\quad\quad\quad\quad\quad } & 
\fBun_{H_R,g_1,n_1+1}^{R_{\chi}=\omega_{\mathrm{log}}}\times_{[\pt/G]^{n_1+1}} 
\frac{Z(\boxplus^{n_1}\phi)^s}{G^{n_1}}\times \frac{Z(\boxplus^{n_2}\phi)^s}{G^{n_2}}
\times X\times X \times_{[\pt/G]^{n_2+1}}\fBun_{ H_R,g_2,n_2+1}^{R_{\chi}=\omega_{\mathrm{log}}} ,   }}
\end{align}
where $\Delta$ is given by the diagonal embedding $X\to X\times X$.

The automorphism \eqref{ord 2 autom} naturally induces an automorphism on the moduli stacks of quasimaps. 
\begin{definition}
Let $\sigma$ be an automorphism as \eqref{ord 2 autom}. 
We define the induced \textit{automorphism}   
$$\sigma: QM_{g,n}^{R_{\chi}=\omega_{\mathrm{log}}}(\Crit(\phi)/\!\!/G,\beta)\to QM_{g,n}^{R_{\chi}=\omega_{\mathrm{log}}}(\Crit(\phi)/\!\!/G,\beta), $$ 
$$\sigma\big((C,p_1,\ldots,p_n),P,u,\varkappa\big):=\big((C,p_1,\ldots,p_n),P,\sigma\cdot u,\varkappa\big), $$
where $\sigma\cdot u: P \stackrel{u}{\to} \Crit(\phi)\stackrel{\sigma}{\to} \Crit(\phi)$ is the composition of $u$ and $\sigma$. 
\end{definition}
In the discussions below, we use the following shorthands:
$$QM_i:=QM_{g_i,n_i+1}, \quad \fB_i:=\fBun_{H_R,g_i,n_i+1}^{R_{\chi}=\omega_{\mathrm{log}}}, \quad i=1,2, $$
$$ \mathcal{B}:=\fBun_{H_R,g_1,n_1+1}^{R_{\chi}=\omega_{\mathrm{log}}}\times_{[\pt/G]} \fBun_{ H_R,g_2,n_2+1}^{R_{\chi}=\omega_{\mathrm{log}}} \times_{[\pt/G]^n} \frac{Z(\boxplus^{n_1}\phi)^s}{G^{n_1}}\times \frac{Z(\boxplus^{n_2}\phi)^s}{G^{n_2}}, $$
$$\mathcal{Z}:=\mathcal{B}\times_{[\pt/G]} X. $$
\begin{lemma}\label{lemma on sigma}
We have a commutative diagram
\begin{align} 
\label{diag cpr Deltabar}
\xymatrix{
QM_{g_1,n_1+1}  \times_{\bar{\Delta}(X)} QM_{g_2,n_2+1}  \ar[r]^{\quad\quad\quad\quad\quad\,\, f_{\bar{\Delta}}} \ar[d]^{\sigma}_{\cong}  & \mathcal{Z} \ar@{=}[d]    \\
QM_{g_1,n_1+1}  \times_{X} QM_{g_2,n_2+1}  \ar[r]^{\quad\quad\quad\quad\quad\,\, f_{\Delta}}   &   \mathcal{Z},    }
\end{align}
where $\sigma$ is a canonical isomorphism.
\end{lemma}
\begin{proof}
We have the following commutative diagram 
\begin{equation*} 
\footnotesize{ 
\xymatrix{
QM_1\times_{\bar{\Delta}(X)} QM_2  \ar@{.>}[d]^{\sigma}    \ar@/_3.6pc/[ddd]^{i_{\bar{\Delta}}} \ar@/^1.5pc/[dr]^{f_{\bar{\Delta}}} & & \\
QM_1\times_X QM_2 \ar[d]^{i} \ar[r]^{f_{\Delta} \quad \quad\quad\quad\quad\quad\quad\quad\quad}     &  \fB_1\times_{BG^{n_1+1}}\frac{Z(\boxplus^{n_1}\phi)^s}{G^{n_1}}\times \frac{Z(\boxplus^{n_2}\phi)^s}{G^{n_2}}
\times X \ar[r]^{ } \times_{BG^{n_2+1}} \fB_2   \ar@/_1.5pc/[dd]_{\bar{\Delta} }
\ar@{}[dr]|{\Box} & X^{n_1}\times X^{n_2} \times X \ar[d]^{\Delta} \\
QM_1\times QM_2 \ar[d]^{\sigma} \ar@/^0.6pc/[rr]^{} &     & X^{n_1}\times X^{n_2} \times X^2 \ar[d]^{\sigma}  \\
QM_1\times QM_2  \ar[r]^{f_1\times f_2 \quad\quad\quad\quad\quad\quad\quad\quad\quad\quad\quad } &  \fB_1\times_{BG^{n_1+1}}\frac{Z(\boxplus^{n_1}\phi)^s}{G^{n_1}}\times \frac{Z(\boxplus^{n_2}\phi)^s}{G^{n_2}}
\times \frac{Z(\boxplus^{2}\phi)^s}{G^{2}} \ar[r]^{ } \times_{BG^{n_2+1}} \fB_2 & X^{n_1}\times X^{n_2}   \times X^2,   } }
\end{equation*}
parts of which are the Cartesian diagrams \eqref{diag on diagonal}, \eqref{diag on diagonal2}, and $\sigma$ acts on $QM_1$ and $X^{n_1}\times X$. 

By a diagram chasing, there exists a canonical map $\sigma: QM_1\times_{\bar{\Delta}(X)}QM_2 \to QM_1\times_X QM_2$ making the above diagram commutative. The map $\sigma$ is furthermore an isomorphism, as are the maps 
\begin{equation*} \sigma: QM_1\times QM_2 \to QM_1\times QM_2, \quad \sigma: X^{n_1}\times X^{n_2}   \times X^2\to X^{n_1}\times X^{n_2} \times X^2. \qedhere \end{equation*}
\end{proof}

Let $\mathcal{C}$ be the universal curve over $QM_{g_1,n_1+1}  \times_{X} QM_{g_2,n_2+1}$ and $\mathcal{C}'$ be the pullback of the universal curve  from 
$QM_{g_1,n_1+1}  \times QM_{g_2,n_2+1}$ along $i$. As in \cite[pp.~607--608]{B}, there is a commutative diagram 
\begin{align}\label{diag on CandC'} \xymatrix{ \mathcal{C}'  \ar^{p}[rr]\ar[rd]^{\pi'} && \mathcal{C}  \ar[ld]_{\pi} \\
& QM_{g_1,n_1+1}  \times_{X} QM_{g_2,n_2+1} \ar@/_1.6pc/[ur]^{\quad \quad \quad \quad \quad \quad \quad \quad x } \ar@/^1.6pc/[ul]_{x_1,x_2 \quad \quad \quad  \quad \quad \quad  \quad \quad \quad  \quad \quad \quad }}
\end{align}
where $x_1,x_2$ are marked points where $ev_{n_1+1}, ev_{n_2+1}$ are evaluated and 
$p$ is the (universal) partial normalization which glues $x_1,x_2$ to $x:=p\circ x_1=p \circ x_2$ (which becomes a node).   

Let $\mathcal{P}$ be the universal principal $H_R=(G\times \mathbb{C}^*)$-bundle on $\mathcal{C}$. There is a Cartesian diagram
\begin{align}\label{diag on univeral bdl}  \xymatrix{
p^*(\mathcal{P}\times_{H_R} W) \ar[r]^{\,\,\widetilde{p}} \ar[d]^{\rho'} \ar@{}[dr]|{\Box}  &  \mathcal{P}\times_{H_R} W  \ar[d]_{\rho}   \\
\mathcal{C}'  \ar[r]^{p} \ar@/^1pc/[u]^{u'} & \mathcal{C} \ar@/_1pc/[u]_{u},   }
\end{align}
where $u$ is the universal section and $u'$ is its pullback.  

For any locally free sheaf $E$ on $\mathcal{C}$, we have evaluation maps
$$u_i: p^*E\to x_{i*}x_i^*p^*E=x_{i*}x^*E, \quad i=1,2. $$
By pushforward to $\mathcal{C}$ via $p$, we obtain a short exact sequence of sheaves
\begin{equation}\label{equ on norm seq ses}0\to E \to p_*p^*E \stackrel{u}{\to}x_*x^*E\to 0,   \end{equation}
where $u=p_*{u_2}-p_*{u_1}$. Equivalently, this is given by applying $-\otimes E$ to 
\begin{equation}\label{equ on norm seq ses2}0\to \oO_{\mathcal{C}} \to p_*\oO_{\mathcal{C}'} \to x_*\oO\to 0. \end{equation}
Applying $\dR\pi_*$ to \eqref{equ on norm seq ses}, we obtain an exact triangle
\begin{equation}\label{equ on dist tri on norm seq}\dR\pi_*E\to \dR\pi'_*p^*E \to x^*E.  \end{equation}
Note also the following short exact sequences (e.g.~\cite[pp.~91]{ACGH}): 
\begin{eqnarray}
\label{eqn:can_norm}
0\to p_*\omega_{\pi'}\to\omega_\pi\to x_*\calO\to 0, \\  
\label{eqn:can_norm2}
%\\\notag\hbox{and}\\\notag
0\to \omega_{\pi}\to p_*\omega_{\pi'}(x_1+x_2)\to x_*\calO\to 0, 
\end{eqnarray}
%\begin{YC}I use $p_*\omega_{\pi'}(x_1+x_2)$ instead of $p_*\omega_{\pi'\mathrm{log}}$. Does this make sense to you? \end{YC}
where the first sequence follows from the dual of \eqref{equ on norm seq ses2}. 

Recall Lemma \ref{lemma on sigma}, we have a commutative diagram: 
%(where $Z=\frac{Z(\boxplus^{n_1}\phi)^s}{G^{n_1}}\times \frac{Z(\boxplus^{n_2}\phi)^s}{G^{n_2}}$): 
\begin{align}\label{diag on moduli rel to Bun} 
{\footnotesize
\xymatrix{
\overline{\mathcal{Y}}:=QM_{g_1,n_1+1}  \times_{\bar{\Delta}(X)} QM_{g_2,n_2+1} \ar^{f_{\bar{\Delta}}}[r]  \ar^{\sigma}_{\cong}[d]  & \mathcal{Z} \ar@{=}[d]   \\
\mathcal{Y}:=QM_{g_1,n_1+1}  \times_{X} QM_{g_2,n_2+1} \ar^{f_{\Delta}}[r] \ar[rd]_{f_{\mathrm{node}}\quad \quad \quad\quad\quad } & 
\mathcal{Z}=\mathcal{B}\times_{[\pt/G]} X   \ar[d]^{p_{\mathcal{B}}} \\
&  \mathcal{B}=\fBun_{H_R,g_1,n_1+1}^{R_{\chi}=\omega_{\mathrm{log}}}\times_{[\pt/G]} \fBun_{ H_R,g_2,n_2+1}^{R_{\chi}=\omega_{\mathrm{log}}} \times_{[\pt/G]^n} \frac{Z(\boxplus^{n_1}\phi)^s}{G^{n_1}}\times \frac{Z(\boxplus^{n_2}\phi)^s}{G^{n_2}},}}
\end{align}
where the fiber product in $\mathcal{Z}$ is given by evaluation maps $\mathcal{B}\to [\pt/G]$, $X\to [\pt/G]$ at the node (obtained by identifying the two marked points), 
$p_\Delta$ is the projection, and $f_{\bar{\Delta}}$ is as in diagram \eqref{diag on diagonal}. 

This gives rise to a commutative diagram: 
%Notice that $\calZ$ is $\calY\times [X]$ with $p_\Delta$ being the projection, hence base-change $f_{\mathrm{node}}$ along $p_\Delta$ gives the following diagram
\begin{equation}\label{diag def fnote bar}
\xymatrix{
\mathcal{Y}\ar[d]_{\id_{\calY}\times ev_\Delta}\ar[drr]^{f_{\Delta} \quad }  &  & \overline{\mathcal{Y}} \ar[ll]_{\sigma} \ar[d]^{f_{\bar{\Delta}}} \\
\mathcal{Y}\times_{[\pt/G]} X\ar[d]_{p} \ar[rr]^{f_{\mathrm{node}}\times\id_X\quad} \ar@{}[drr]|{\Box} & & \mathcal{Z}=\mathcal{B}\times_{[\pt/G]} X \ar[d]^{\bar{p}} \\
\mathcal{Y} \ar[rr]^{f_{\mathrm{node}} }& & \mathcal{B},
}\end{equation}
where $ev_\Delta$ is the evaluation map at the node (obtained by identifying the two marked points in $QM_{g_1,n_1+1}$ and $QM_{g_2,n_2+1}$). 

As $X=[W^s/G]$ is smooth and affine over $[\pt/G]$, the map $\id_{\calY}\times ev_\Delta$ is a regular embedding by \cite[Def.~1.20]{Vis}, \cite[App.~B.7.3]{Fu}, therefore there is a Gysin pullback $(\id_{\calY}\times ev_\Delta)^!$. By diagram \eqref{diag on diagonal}, $f_{\bar{\Delta}}$ has a square root virtual pullback $\sqrt{f_{\bar{\Delta}}^!}$
such that \eqref{eqn:Delta} holds. 
Similarly, $(f_{\mathrm{node}}\times\id_X)$ has a square root virtual pullback $\sqrt{(f_{\mathrm{node}}\times\id_X)^!}$
as the base change by $f_{\mathrm{node}}$ which comes as the base change of $f$ in diagram \eqref{eqn:degeneration}.
\begin{remark}\label{rmk on sigma2}
The map $f_{\Delta}$ does not clearly have a square root virtual pullback as $\frac{Z(\boxplus^{n_1}\phi)^s}{G^{n_1}}\times \frac{Z(\boxplus^{n_2}\phi)^s}{G^{n_2}}\times X$ does not satisfy condition \eqref{equ on Z}, and this is the point we need to introduce $\sigma$ and $f_{\bar{\Delta}}$.
\end{remark}

The rest of this section is to prove the following compatibility.  
\begin{prop}\label{prop on gluing form}
Notations as above, we have
\begin{align*}\sigma^*\circ (\id_{\calY}\times ev_\Delta)^!\circ \sqrt{(f_{\mathrm{node}}\times\id_X)^!}=\sqrt{f_{\bar{\Delta}}^!}. \end{align*}
\end{prop}
\begin{proof}
%\yl{Check whether this is really without loss of generality.}
Without loss of generality, we assume for simplicity that $n_1=n_2=0$, so that 
\begin{equation}\label{equ in compat-1}\omega_{\pi,\log}=\omega_{\pi}, \end{equation}
and $\calC'$ has only two marked points $x_1$
and $x_2$ glued to the node $x$ in $\calC$. We will use the functoriality of square root virtual pullback \eqref{eqn:Park} to prove the claim. 
For this purpose, it is enough to construct diagram \eqref{eqn:triang_obst} in this setting for maps 
$$(\id_{\calY}\times ev_\Delta)\circ \sigma, \,\,\, (f_{\mathrm{node}}\times\id_X),\,\,\, f_{\bar{\Delta}}. $$
By base change along 
$$\frac{Z(\boxplus^{n_1}\phi)^s}{G^{n_1}}\times \frac{Z(\boxplus^{n_2}\phi)^s}{G^{n_2}}\hookrightarrow X^{n_1+n_2},$$
we further reduce the construction of diagram \eqref{eqn:triang_obst} to the case where $\frac{Z(\boxplus^{n_1}\phi)^s}{G^{n_1}}\times \frac{Z(\boxplus^{n_2}\phi)^s}{G^{n_2}}$ in 
diagrams \eqref{diag on moduli rel to Bun} and \eqref{diag def fnote bar} is replaced by $X^{n_1+n_2}$.
As $\sigma$ is an isomorphism (Lemma~\ref{lemma on sigma}), we first construct diagram \eqref{eqn:triang_obst} for maps:
\[e:=\id\times ev_\Delta,\quad h:=f_{\mathrm{node}}\times\id_X,\quad g:=f_{\Delta}. \]
We introduce some shorthand notations used only in this proof: 
\begin{equation}\label{equ in compat0}
\calW:=u^*T_{\rho}\cong\mathcal{P}\times_{H_R} W,\,\,\,  A:=\dR\pi'_*(u'^*T_{\rho'}(-x_1-x_2))
, \,\,\, B:=\dR\pi_*(\calW), \,\,\, C:=x^*\calW. \end{equation}
By a base change in diagram \eqref{diag on univeral bdl}, we get 
\begin{equation}\label{equ in compat1}\widetilde{p}^*T_{\rho}\cong T_{\rho'}. \end{equation}
Therefore using diagrams \eqref{diag on CandC'} and \eqref{diag on univeral bdl}, we have
\begin{equation}\label{equ in compat1.5}A\cong \dR\pi'_*(u'^*\widetilde{p}^*T_{\rho}(-x_1-x_2))\cong \dR \pi_* p_*(p^*u^*T_{\rho}(-x_1-x_2))\cong 
\dR \pi_* (\calW\otimes p_*\oO_{\mathcal{C}'}(-x_1-x_2)).  \end{equation}
By applying  $\dR\pi_*(\calW\otimes-)$ to \eqref{equ on norm seq ses2}, or equivalently applying \eqref{equ on dist tri on norm seq} to $E=\calW$, we obtain an exact triangle  
\begin{equation}\label{equ in compat1.6} B\to \dR\pi'_*(p^*\calW)\to x^*\calW. \end{equation}
By applying $\dR\pi_*(\calW\otimes p_*(-))$ to the short exact sequence 
$$0\to \oO_{\mathcal{C}'}(-x_1-x_2)\to \oO_{\mathcal{C}'} \to x_{1*}\oO\oplus x_{2*}\oO\to 0, $$
we obtain an exact triangle
\begin{equation}\label{equ in compat1.7} A\to \dR\pi'_*(p^*\calW)\to x^*\calW\oplus x^*\calW. \end{equation}
%and the obvious map $x_1^*p^*\calW\oplus x_2^*p^*\calW\to x^*\calW $ the kernel of which is the diagonal copy of $x^*\calW$,we obtain 
Combining \eqref{equ in compat1.6}, \eqref{equ in compat1.7} and the quotient map (whose kernel is the diagonal $x^*\calW$):
$$x^*\calW\oplus x^*\calW\to x^*\calW, $$
we obtain an exact triangle
\begin{equation}\label{equ in compat2}A\xrightarrow{\alpha_0} B \to C. \end{equation}
By applying $\dR\pi_*(\calW^\vee\otimes-)$ to \eqref{eqn:can_norm2}, we obtain an exact triangle
\begin{equation}\label{equ in compat3}\dR\pi_*(\calW^\vee\otimes \omega_{\pi})\xrightarrow{\alpha^0} \dR\pi_*(\calW^\vee\otimes p_*\omega_{\pi'}(x_1+x_2))\to 
\dR\pi_*(\calW^\vee\otimes x_*\oO). \end{equation}
Applying relative duality for $p$, we have
%\begin{YC}Be careful, base is not DM stack. Note $\pi^!\oO\cong \omega_{\pi}$. \end{YC}
$$\dR \mathcal{H}om(p_*\oO_{\cC'}(-x_1-x_2),\omega_{\pi})\cong p_*(\oO_{\cC'}(x_1+x_2)\otimes p^!\omega_{\pi})\cong p_*(\omega_{\pi'}(x_1+x_2)). $$
Then it is easy to see $\alpha^0$ is dual to $\alpha_0$ under the isomorphism \eqref{equ in compat1.5} and relative duality for $\pi$.

Noticing that 
\begin{align*}
\dR\pi_*(\calW^\vee\otimes x_*\oO)\cong \dR\pi_*x_*(x^*\calW^\vee\otimes \oO)\cong x^*\calW^\vee. 
\end{align*}
Then \eqref{equ in compat3} becomes
\begin{comment}
Similar as \eqref{equ in compat1.5}, we have 
\begin{align*}
\dR\pi_*(u^*T_{\rho}^\vee\otimes p_*\omega_{\pi'})&\cong \dR\pi_*p_*(p^*u^*T_{\rho}^\vee\otimes \omega_{\pi'}) \\
&\cong \dR\pi'_*(u'^{*}\widetilde{p}^*T_{\rho}^\vee\otimes \omega_{\pi'}) \\
&\cong \dR\pi'_*(u'^{*}T_{\rho'}^\vee\otimes \omega_{\pi'}),
\end{align*}
\begin{align*}
\dR\pi_*(u^*T_{\rho}^\vee\otimes x_*\oO)\cong \dR\pi_*x_*(x^*u^*T_{\rho}^\vee\otimes \oO)\cong x^*u^*T_{\rho}^\vee. 
\end{align*}
Then \eqref{equ in compat3} becomes 
\begin{equation}\label{equ in compat4}\dR\pi'_*(u'^{*}T_{\rho'}^\vee\otimes \omega_{\pi'})\to \dR\pi_*(u^*T_{\rho}^\vee\otimes \omega_{\pi})\to 
x^*u^*T_{\rho}^\vee.   \end{equation}
Using relative duality:
\begin{equation}\dR\pi'_*(u'^{*}T_{\rho'}^\vee\otimes \omega_{\pi'})[1]\cong 
\left(\dR\pi'_*(u'^{*}T_{\rho'})\right)^{\vee},\,\, \dR\pi_*(u^*T_{\rho}^\vee\otimes \omega_{\pi})[1]\cong
\left(\dR\pi_*(u^*T_{\rho})\right)^{\vee},  \nonumber \end{equation}
the exact triangle \eqref{equ in compat4} becomes (using notation \eqref{equ in compat0}): 
\end{comment}
\begin{equation}\label{equ in compat5} 
B^\vee[-1]\xrightarrow{\alpha^0}  A^\vee[-1]\to  C^\vee.
 \end{equation}
Consider the following symmetric complexes  
$$\mathbb{E}_{g}=(A\xrightarrow{d_g} A^{\vee}[-1]), \quad  e^*\mathbb{E}_{h}=(B\xrightarrow{d_h}  B^{\vee}[-1]), $$
%$$\mathbb{E}_{g}=(B\xrightarrow{d_g}  B^{\vee}[-1]), \quad e^*\mathbb{E}_{h}=(A\xrightarrow{d_h} A^{\vee}[-1]),  $$
where $\mathbb{E}_{g}$ is the pullback of the direct sum of symmetric complexes $\bbE_{f_1}, \bbE_{f_2}$ (each one as defined in Theorem \ref{prop:symm_ob}) via the diagonal base change in \eqref{diag on diagonal2} and 
$\mathbb{E}_{h}$ is the pullback of $\bbE_{f}$ via base change from $f$ to $f_{\mathrm{node}}$ defined in \eqref{eqn:degeneration} and base change 
from $f_{\mathrm{node}}$ to $h:=f_{\mathrm{node}}\times\id_X$ as in \eqref{diag def fnote bar}.
%Here the degrees of $A$ and $B$ are in $-2,-1$ and degrees of $A^{\vee}[-1]$ and $B^{\vee}[-1]$ are  in $-1,0$.
Define $$\mathbb{D}:=(B\xrightarrow{\alpha^0\circ d_h}  A^{\vee}[-1]), \quad \mathbb{E}_e:=C^\vee[1]. $$ 
%\yl{Do we need $\mathbb{D}$ has tor-amplitude in $[-2,0]$?}
Here the map in $\bbD$ is the composition of the differential in $e^*\bbE_h$ with $\alpha^0$ in \eqref{equ in compat5} and $\mathbb{E}_e$ is a vector bundle concentrated in degree $-1$ and coincides with the cotangent complex of map $e$.

With notations as above, we define morphisms $\alpha: \bbE_{g}\to \bbD$, $\beta: \bbD^{\vee}[2]\to e^*\mathbb{E}_{h}$ by  
%as the map of complexes with identity on $A^\vee$ and $\alpha_0:A\to B$; $\beta: \bbD^\vee\to i^*\bbE_h$ as identify on $B^\vee$ and $\alpha_0:A\to B$. 
\begin{equation}\label{diag of alpha and D}\xymatrix{
\bbE_{g}\ar@{=}[r]\ar[d]_{\alpha}& \big(A\ar[r]^{d_g\quad }\ar[d]^{\alpha_0}\,\, &A^{\vee}[-1]\big) \ar@{=}[d] & \bbD^{\vee}[2]\ar@{=}[r]\ar[d]_{\beta}& \big(A\ar[r]^{d_h\circ \alpha_0\quad}\ar[d]^{\alpha_0}\,\, &B^{\vee}[-1]\big) \ar@{=}[d] \\
\bbD\ar@{=}[r]& \big(B\ar[r]^{\alpha^0\circ d_h \quad }\,\, &A^{\vee}[-1]\big), &e^*\mathbb{E}_{h}\ar@{=}[r]& \big(B\ar[r]^{d_h \quad}\,\, &B^{\vee}[-1]\big).}
\end{equation}
%With notations as above, we define morphisms $\alpha: \bbE_{g}\to \bbD$, $\beta: \bbD^{\vee}\to e^*\mathbb{E}_{h}^{\vee}$ by  
%\begin{equation*}\xymatrix{
%\bbE_{g}\ar@{=}[r]\ar[d]_{\alpha}& \big(B\ar[r]^{d_g\quad }\ar@{=}[d]\,\, &B^{\vee}[-1]\big) \ar[d]^{\alpha^0} & \bbD^{\vee}\ar@{=}[r]\ar[d]_{\beta}& \big(A\ar[r]^{d_g\circ \alpha_0\quad}\ar@{=}[d]\,\, &B^{\vee}[-1]\big) \ar[d]^{\alpha^0} \\
%\bbD\ar@{=}[r]& \big(B\ar[r]^{\alpha^0\circ d_g \quad }\,\, &A^{\vee}[-1]\big), &e^*\mathbb{E}_{h}^{\vee}\ar@{=}[r]& \big(A\ar[r]^{d_h \quad}\,\, &A^{\vee}[-1]\big).}
%\end{equation*}
We claim that $d_g=\alpha^0\circ d_h\circ\alpha_0$, so $\alpha$ is well-defined.   
%As in diagrams \eqref{diag on CandC'}, \eqref{diag on univeral bdl}, let $p:\mathcal{C}'\to \mathcal{C}$ be the normalization of the node $x$ with preimage $x_1,x_2$, and 
%Write $\calW:=\calP\times_HW$ as a shorthand. 
Note that $d_h$ is given by applying $\dR\pi_*$ to 
$$\calW\xrightarrow{\mathrm{Hess}_\phi}\calW^\vee\otimes\omega_{\pi}, $$ 
and $d_g$ is given by applying $\dR\pi_*p_*$ to the composition 
$$p^*\calW\otimes \oO_{C'}(-x_1-x_2)\to p^*\calW \xrightarrow{\mathrm{Hess}_\phi} 
p^*\calW^{\vee}\otimes \omega_{\pi'}(x_1+x_2), $$ 
where the first map is induced by the natural inclusion $\oO_{\mathcal{C}'}(-x_1-x_2)\to \oO_{\mathcal{C}'}$.

By adjunction, we have a commutative diagram 
\[\xymatrix{p_*p^*\calW\ar[r]^{\mathrm{Hess}_\phi \quad\quad }&p_*p^*\left(\calW^\vee\otimes \omega_{\pi}\right)  \\
\calW\ar[u]\ar[r]^{\mathrm{Hess}_\phi \quad \quad}&\calW^\vee\otimes\omega_\pi. \ar[u]_{ }& } \]
Applying $\dR\pi_*$ to it, we get a commutative diagram:
\[\xymatrix{\dR\pi_*p_*p^*\calW\ar[r]&\dR\pi_*p_*p^*\left(\calW^\vee\otimes \omega_{\pi}\right)\cong \dR\pi_*\left(\calW^\vee\otimes p_*\omega_{\pi'}(x_1+x_2)\right)  \\
\dR\pi_*\calW\ar[u]\ar[r]^{d_h \quad \quad}&\dR\pi_*\left(\calW^\vee\otimes\omega_\pi\right), \ar[u]_{\alpha^0}& } \]
where the isomorphism uses Eqn.~\eqref{equ in compat-1} and $p^*\omega_{\pi,\log}=\omega_{\pi',\log}$. 

By definition, $\alpha_0$ fits into 
\[\xymatrix{
\dR \pi_* (\calW\otimes p_*\oO_{\mathcal{C}'}(-x_1-x_2)) \ar[dr]_{\alpha_0}\ar[r]&\dR\pi_*p_*p^*\calW\cong \dR\pi_*(\calW\otimes p_*\oO_{\mathcal{C}'}) \\
& \dR\pi_*\calW.\ar[u]
}
\]
By a diagram chasing, we get $d_g=\alpha^0\circ d_h\circ\alpha_0$.

Cones of both $\alpha^\vee[2]$ and $\beta^\vee[2]$ are $\mathbb{E}_e$ 
%\yl{The cones are not $\mathbb{E}_e$...}
and they fit into a commutative diagram:
\begin{align}\label{vir pull dia1}
\xymatrix{
\bbD^\vee[2]\ar[r]^{\alpha^\vee[2]}\ar[d]_{\beta}&\bbE_{g}\ar[d]_\alpha\ar[r]&\bbE_e\ar@{=}[d]\\
e^*\bbE_h\ar[r]^{\beta^\vee[2]}&\bbD\ar[r]&\bbE_e, }
\end{align}
where we use $\bbE^{\vee}[2]\cong \bbE$ for $\bbE=\bbE_{g}$ and $e^*\mathbb{E}_{h}$.

Next we construct the bottom part of diagram \eqref{eqn:triang_obst}. This is done by considering derived stacks and the restriction of their cotangent complexes
to their classical truncations.  

Consider the derived enhancement of $f$ in diagram \eqref{eqn:degeneration} where the $Z$ is replaced by $W^n$ (exactly as in diagram \eqref{diag on map f}). The homotopy pullback via the following 
diagram defines a derived enhancement of $f_{\mathrm{node}}$: 
\begin{align} 
\label{eqn:degeneration der}
\xymatrix{
\textbf{Q}M_{\mathrm{node}} \ar[r]^{\quad\quad\quad\quad } \ar[d]_{\bf f_{\textbf{node}}} \ar@{}[dr]|{\Box}  &  \textbf{Q}M_{g,n} \ar[d]^{\bf f}    \\
\fBun_{H_R,g_1,n_1+1}^{R_{\chi}=\omega_{\mathrm{log}}}\times_{[\pt/G]} \fBun_{ H_R,g_2,n_2+1}^{R_{\chi}=\omega_{\mathrm{log}}} \times_{[\pt/G]^n} X^n
\ar[r]^{ \quad\quad\quad\quad \quad \,\, gl} \ar[d] \ar@{}[dr]|{\Box} & 
\fBun_{H_R,g,n}^{R_{\chi}=\omega_{\mathrm{log}}} \times_{[\pt/G]^n}  X^n \ar[d]   \\
\mathfrak{M}_{g_1,n_1+1}  \times \mathfrak{M}_{g_2,n_2+1}   \ar[r]^{\quad\quad\quad gl} & \mathfrak{M}_{g,n},     }
\end{align}
where the underlying classical stack satisfies 
$$t_0(\textbf{Q}M_{\mathrm{node}} )\cong QM_{g_1,n_1+1}  \times_{X} QM_{g_2,n_2+1}=:\mathcal{Y}. $$
Further homotopy pullback via diagram \eqref{diag def fnote bar} defines 
a derived enhancement of $h=f_{\mathrm{node}}\times \id_X$. Let $\bbE_h$ be the restriction of the (derived) cotangent complex to its classical truncation, then we obtain a symmetric obstruction theory (~Theorem~\ref{prop:symm_ob}):
$$\phi_h:\bbE_h\to \bfL_h:=\tau^{\geqslant  -1}\bbL_{h}. $$ 
Consider two derived enhancements of $g=f_\Delta$ as follows.  
One of them is constructed via diagram \eqref{diag on diagonal2}: by considering derived enhancement $\bf f_1, \bf f_2$ of $f_1, f_2$ where the $\frac{Z(\boxplus^{n_1}\phi)^s}{G^{n_1}}\times \frac{Z(\boxplus^{n_2}\phi)^s}{G^{n_2}}$ in \eqref{diag on diagonal2} is replaced by $X^{n_1+n_2}$ (as $\bf f$ in diagram \eqref{diag on map f}) and then define $\textbf{f}_\Delta $ to be such that  diagram: 
\begin{align}\label{diag on diagonal der} 
\xymatrix{\bfQ M_{g_1,n_1+1}  \times_{X} \bfQ M_{g_2,n_2+1}  \ar[d]_{\textbf{i}} \ar[r]^{\textbf{f}_\Delta \quad \quad\quad \quad\quad  } \ar@{}[dr]|{\Box} &  
\fB_1\times_{[\pt/G]^{n_1+1}} X^{n_1+1} \times_X X^{n_2+1}  \times_{[\pt/G]^{n_2+1}}\fB_2   \ar[d]^{\Delta }  \\
%Z(d\phi)/\!\!/G  \ar[r]^{\Delta \quad \quad } \ar[d]_{ }  & Z(d\phi)/\!\!/G  \times Z(d\phi)/\!\!/G \ar[d]_{ }    \\
\bfQ M_{g_1,n_1+1}  \times \bfQ M_{g_2,n_2+1} \ar[r]^{\bf f_1\times \bf f_2 \quad \quad\quad \quad\quad } & 
\fB_1\times_{[\pt/G]^{n_1+1}} X^{n_1+1} \times X^{n_2+1}  \times_{[\pt/G]^{n_2+1}}\fB_2 ,     }
\end{align}
is homotopy pullback,
where $\fB_i:=\fBun_{H_R,g_i,n_i+1}^{R_{\chi}=\omega_{\mathrm{log}}}$ ($i=1,2$).
Then $\bbE_g$ defined above satisfies $\bbE_g=\bbL_{\textbf{f}_\Delta}|_{\mathcal{Y}}$ which gives rise to the  symmetric obstruction theory 
$$\phi_g:\bbE_g\to \bfL_g. $$
The other derived enhancement is defined similarly, by replacing those $X$ in the right hand side of diagram \eqref{diag on diagonal der} by 
$C=[(W\times^\bfL_{T^*W}W)^s/G]\hookrightarrow X=[W^s/G]$, i.e. as the homotopy pullback of derived stacks:
\begin{equation}\label{diag on diagonal der 2} 
\xymatrix{ 
\bfQ M_{g_1,n_1+1}  \times_C \bfQ M_{g_2,n_2+1} \ar[d]_{\tilde{\textbf{i}}} \ar[r]^{\tilde{\textbf{f}}_\Delta \quad \quad\quad \quad\quad   } \ar@{}[dr]|{\Box} &  
\fB_1\times_{[\pt/G]^{n_1+1}} C^{n_1+1}\times_C \times C^{n_2+1}  \times_{[\pt/G]^{n_2+1}}\fB_2 \ar[d]^{\Delta }  \\
\bfQ M_{g_1,n_1+1}  \times \bfQ M_{g_2,n_2+1} \ar[r]^{\tilde{\textbf{f}}_1\times \tilde{\textbf{f}}_2 \quad \quad\quad \quad\quad   } & 
\fB_1\times_{[\pt/G]^{n_1+1}} C^{n_1+1} \times  C^{n_2+1}  \times_{[\pt/G]^{n_2+1}}\fB_2.    }
\end{equation}
%where $\fBun_i:=\fBun_{H_R,g_i,n_i+1}^{R_{\chi}=\omega_{\mathrm{log}}}$ ($i=1,2$) and $Z=[W\times^\bfL_{T^*W}W/H_R]\hookrightarrow X=[W/H_R]$. 
Note by \eqref{prd of eva map}, the maps $\bf f_1, \bf f_2$ 
in the first derived enhancement factors through $\tilde{\textbf{f}}_1,\tilde{\textbf{f}}_2$ respectively, so we have 
the following commutative diagram of derived stacks
%The target of $\tilde{\textbf{f}}_i$ the evaluate
\[\xymatrix{
 \bfQ M_{g_1,n_1+1}  \times_C \bfQ M_{g_2,n_2+1} \ar[r]^{\tilde{\textbf{f}}_\Delta \quad \quad\quad \quad\quad   } \ar[d]_{\textbf{j}}  & \fB_1\times_{[\pt/G]^{n_1+1}} C^{n_1+1}\times_C \times C^{n_2+1}  \times_{[\pt/G]^{n_2+1}}\fB_2  \ar[d]^{\iota }   \\
 \bfQ M_{g_1,n_1+1}  \times_X \bfQ M_{g_2,n_2+1}  \ar[r]^{\textbf{f}_\Delta \quad \quad\quad \quad\quad   } & \fB_1\times_{[\pt/G]^{n_1+1}} X^{n_1+1}\times_X \times X^{n_2+1}  \times_{[\pt/G]^{n_2+1}}\fB_2, }
\] 
where $\iota$ is induced by the natural inclusion $C\hookrightarrow  X$. 
The classical truncation of $\textbf{j}$ induces an isomorphism of classical stacks 
$$t_0\left(\bfQ M_{g_1,n_1+1}  \times_C \bfQ M_{g_2,n_2+1} \right)\cong t_0\left(\bfQ M_{g_1,n_1+1}  \times_X \bfQ M_{g_2,n_2+1} \right)=\mathcal{Y}, $$
because maps in $Q M_{g_i,n_i+1}$ already evaluate at $C$. Consider the restriction of the cotangent complexes to the classical truncation, we have 
a commutative diagram 
\begin{align}\label{vir pull dia2}
\xymatrix{
  \bbE_g\cong(\textbf{j}^*\bbL_{\textbf{f}_{\Delta}})|_{\mathcal{Y}} \ar[r]^{} \ar[d]_{ }  & \bbL_{\textbf{f}_{\Delta}\circ \textbf{j}}|_{\mathcal{Y}}\cong \bbD  \ar[d]^{ }   \\
 \bbL_{g} \ar[d]_{ }    \ar@{=}[r] &\bbL_{g}  \ar[d]_{ }   \\
 \bfL_g=\tau^{\geqslant  -1}\bbL_{g}  \ar@{=}[r] &\bfL_g=\tau^{\geqslant  -1}\bbL_{g}.
 } \end{align}
It is straightforward to check the upper horizontal map coincides with $\alpha$ in diagram \eqref{diag of alpha and D}. 

Using Lemma \ref{lem on ynode der cmp}, we then have a commutative diagram of derived stacks  
\begin{equation}\label{diag def fnote bar der}
\xymatrix{
\textbf{Q}M_{\mathrm{node}}\ar[d]_{\textbf{e}:=\id_{}\times ev_\Delta}   & & \bfQ M_{g_1,n_1+1}  \times_C \bfQ M_{g_2,n_2+1} \ar[ll]_{\textbf{r} \quad \quad\quad } \ar[d]^{\textbf{g}:=\iota\circ \tilde{\textbf{f}}_\Delta=\textbf{f}_{\Delta}\circ \textbf{j}} \\
\textbf{Q}M_{\mathrm{node}}\times_{[\pt/G]} X\ar[d] \ar[rr]_{\bf h:=\bf f_{\textbf{node}}\times\id_X} \ar@{}[drr]|{\Box}& & \mathcal{Z}=\mathcal{B}\times_{[\pt/G]} X \ar[d] \\
\textbf{Q}M_{\mathrm{node}} \ar[rr]^{\bf f_{\textbf{node}} }& & \mathcal{B},
}\end{equation}
where $\mathcal{B}=\fBun_{H_R,g_1,n_1+1}^{R_{\chi}=\omega_{\mathrm{log}}}\times_{[\pt/G]} \fBun_{ H_R,g_2,n_2+1}^{R_{\chi}=\omega_{\mathrm{log}}} \times_{[\pt/G]^n} X^n$. 
Restricting cotangent complexes to the classical truncations and using the fact that $t_0(\textbf{r})$ is an isomorphism and $\bbL_{\bf e\circ \bf r}|_{\mathcal{Y}}\cong \bbL_{\bf e}|_{\mathcal{Y}}$ (which follows from Lemma \ref{lem on ynode der cmp}), we obtain a commutative diagram 
\begin{align}\label{vir pull dia3}
\xymatrix{
(\textbf{e}^*\bbL_{\bf h})|_{\mathcal{Y}}\cong e^*\bbE_h \ar[d]_{ }\ar[r]^{\quad \,\,\, \beta^\vee[2]}&\bbL_{\bf g}|_{\mathcal{Y}}\cong \bbD\ar[r]\ar[d]_{ }&\bbL_{\bf e}|_{\mathcal{Y}}\cong\bbE_e\ar[d]_{ } \\
e^*\bbL_h\ar[r] \ar[d]_{ }  &\bbL_g\ar[r] \ar[d]_{ } &\bbL_e \ar[d]_{ }  \\
\tau^{\geqslant -1}e^*\bfL_h\ar[r]   &\bfL_g \ar[r]  & \bfL'_e.  
}\end{align}
Combining diagrams \eqref{vir pull dia1}, \eqref{vir pull dia2}, \eqref{vir pull dia3}, we obtain diagram \eqref{eqn:triang_obst}with 
\[e:=\id\times ev_\Delta,\quad h:=f_{\mathrm{node}}\times\id_X,\quad g:=f_{\Delta}. \]
Notice that we have 
$$\sigma^*\bbL_g\cong \bbL_{g\circ \sigma}=\bbL_{f_{\bar{\Delta}}}, \quad  \sigma^*\bbL_e\cong \bbL_{e\circ \sigma}, \quad \sigma^*\bbE_{g}\cong \bbE_{f_{\bar{\Delta}}}, $$
where $\bbE_{f_{\bar{\Delta}}}$ is the symmetric obstruction theory used to define $\sqrt{f_{\bar{\Delta}}^!}$. Here the last isomorphism is due to the following commutative diagram 
of derived stacks
\begin{align*} 
\xymatrix{
&  \fB_1\times_{[\pt/G]^{n_1+1}} X^{n_1} \times X \times X^{n_2}  \times_{[\pt/G]^{n_2+1}}\fB_2 \ar[d]^{\Delta} \ar@/^10pc/[dd]^{\bar{\Delta}}\\
\bfQ M_{g_1,n_1+1}  \times \bfQ M_{g_2,n_2+1}  \ar[d]_{\sigma} \ar[r]^{\bf f_1\times \bf f_2 \quad \quad\quad \quad\quad\quad}  &  
\fB_1\times_{[\pt/G]^{n_1+1}} X^{n_1} \times X \times X \times X^{n_2}  \times_{[\pt/G]^{n_2+1}}\fB_2   \ar[d]^{\sigma }  \\
%Z(d\phi)/\!\!/G  \ar[r]^{\Delta \quad \quad } \ar[d]_{ }  & Z(d\phi)/\!\!/G  \times Z(d\phi)/\!\!/G \ar[d]_{ }    \\
\bfQ M_{g_1,n_1+1}  \times \bfQ M_{g_2,n_2+1} \ar[r]^{\bf f_1\times \bf f_2 \quad \quad\quad \quad\quad \quad} & 
\fB_1\times_{[\pt/G]^{n_1+1}} X^{n_1}  \times X \times X \times  X^{n_2}  \times_{[\pt/G]^{n_2+1}}\fB_2,  }
\end{align*}
where $\sigma$ is applied to $\bfQ M_{g_1,n_1+1}$ and $X^{n_1}\times X$. 

Consider the pullback of diagrams \eqref{vir pull dia1}, \eqref{vir pull dia3} by the map $\sigma$, we obtain the desired diagram \eqref{eqn:triang_obst} for maps $(\id_{\calY}\times ev_\Delta)\circ \sigma$, 
$(f_{\mathrm{node}}\times\id_X)$, $f_{\bar{\Delta}}$, therefore we are done. 
\end{proof}
\begin{lemma}\label{lem on ynode der cmp}
Let $\bfQ M_{\mathrm{node}}$ and $\bfQ M_{g_1,n_1+1}  \times_C \bfQ M_{g_2,n_2+1}$ be defined by diagrams \eqref{eqn:degeneration der}, \eqref{diag on diagonal der 2} respectively, where $C:=[(W\times^\bfL_{T^*W}W)^s/G]$.  Then there is a map of derived stacks
\begin{equation}\label{equ on map der r}\textbf{r}: \bfQ M_{g_1,n_1+1}  \times_C \bfQ M_{g_2,n_2+1}\to \bfQ M_{\mathrm{node}}  \end{equation}
whose classical truncation is an isomorphism. Moreover, the restriction of the cotangent complex of $\textbf{r}$ to the classical truncation is zero. 
\end{lemma}
\begin{proof}
Let $\mathcal{C}\to \mathfrak{M}_{g,n}$, $\mathcal{C}_i\to \mathfrak{M}_{g_i,n_i+1}$ ($i=1,2$) be the universal curves. 
Define $\mathcal{C}_{\mathrm{node}}$ to be the pullback of $\mathcal{C}$ via the gluing morphism $gl$ in \eqref{eqn:degeneration}.
Then we get the following diagram with the square being Cartesian
\[\xymatrix{
\mathcal{C}_1\times \mathfrak{M}_{g_2,n_2+1} \sqcup \mathfrak{M}_{g_1,n_1+1}\times \mathcal{C}_2  \ar[d]_{n} \ar[dr] & \\
\mathcal{C}_{\mathrm{node}} \ar[r] \ar[d] \ar@{}[dr]|{\Box} & \mathfrak{M}_{g_1,n_1+1}\times \mathfrak{M}_{g_2,n_2+1} \ar[d]^{gl}  \\
  \mathcal{C}  \ar[r]  &  \mathfrak{M}_{g,n},  } \]
where $n$ is the normalization of nodal curves.  
By viewing classical stacks as derived stacks, the square is also a homotopy pullback diagram as horizontal maps are flat. 

%and recall Eqn.~\eqref{def of mod stack of mapping}: $$\bMap_{g,n}([Y/H]):=\bMap_{\textbf{dSt}/\mathfrak{M}_{g,n}}(\mathcal{C},[Y/H]\times \mathfrak{M}_{g,n}).  $$
Recall that for a base Artin stack $S$, a stack $X$ flat and proper over $S$, and a
derived Artin stack $F$ which is locally of finite presentation over $S$,
by definition the derived mapping  stack $\bMap_{\mathbf{dSt}/S}(X, F)$ represents the sheaf
that sends any derived $S$-stack $T$ to 
the simplicial set $\Hom_{\mathbf{dSt}/T}(X_T,F_T)$ of morphisms of derived stacks over $T$, where $\bullet_T=\bullet\times^{\bfL}_ST$. 
In particular, base-change implies canonical isomorphism 
\begin{equation}\label{iso of bach}\bMap_{\mathbf{dSt}/S}(X, F)\times^{\bfL}_ST\cong \bMap_{\mathbf{dSt}/T}(X_{T},F_{T}). \end{equation}
Applying to the situation where $S=\mathfrak{M}_{g,n}$, $T=\mathfrak{M}_{g_1,n_1+1}\times \mathfrak{M}_{g_2,n_2+1}$, and $X=\calC$, 
for any derived Artin stack $Z$ over $\mathbb{C}$, 
we obtain the following homotopy pullback diagrams of derived stacks 
\begin{equation}\label{equ on mgn pb}\xymatrix{
\bMap_{\textbf{dSt}/\mathfrak{M}_{g_1,n_1+1}\times \mathfrak{M}_{g_2,n_2+1}}(\mathcal{C}_{\mathrm{node}},Z\times \mathfrak{M}_{g_1,n_1+1}\times \mathfrak{M}_{g_2,n_2+1})  \ar[r] \ar[d] \ar@{}[dr]|{\Box} & \mathfrak{M}_{g_1,n_1+1}\times \mathfrak{M}_{g_2,n_2+1} \ar[d]^{gl} \\
\bMap_{\textbf{dSt}/\mathfrak{M}_{g,n}}(\mathcal{C},Z\times \mathfrak{M}_{g,n})
\ar[r] & \mathfrak{M}_{g,n}. } \end{equation}
By definition, as the gluing of $\calC_1$ and $\calC_2$, $\calC_{\mathrm{node}}$ is the pushout 
\[\xymatrix{
 \mathfrak{M}_{g_1,n_1+1}\times \mathfrak{M}_{g_2,n_2+1}\ar[r]^{\quad  p_{n_1+1}\times \id}\ar[d]_{\id\times p_{n_2+1}} \ar@{}[dr]|{\Box} &\calC_1\times \mathfrak{M}_{g_2,n_2+1} \ar[d]\\
\mathfrak{M}_{g_1,n_1+1}\times \calC_2  \ar[r]&
\calC_{\mathrm{node}},
}\]
where $p_i$ denotes the $i$-th marked point. Consider also the homotopy pushout 
\[\xymatrix{
 \mathfrak{M}_{g_1,n_1+1}\times \mathfrak{M}_{g_2,n_2+1}\ar[r]^{\quad  p_{n_1+1}\times \id}\ar[d]_{\id\times p_{n_2+1}} \ar@{}[dr]|{\Box} &\calC_1\times \mathfrak{M}_{g_2,n_2+1} \ar[d]\\
\mathfrak{M}_{g_1,n_1+1}\times \calC_2  \ar[r]&
\calC^{\mathrm{der}}_{\mathrm{node}},
}\]
whose classical truncation recovers the previous diagram. 
 
%where $t_0(\calC^{\mathrm{der}}_{\mathrm{node}})=\calC^{}_{\mathrm{node}}$. 
For brevity, let $\mathfrak{M}_{1,2}:=\mathfrak{M}_{g_1,n_1+1}\times \mathfrak{M}_{g_2,n_2+1}$. Applying $\bMap_{\textbf{dSt}/\mathfrak{M}_{1,2}}(-,Z\times \mathfrak{M}_{1,2})$ to the above diagram, we obtain 
a homotopy pullback diagram
\begin{equation}\label{diag on der iso prop}\xymatrix{
\bMap_{\textbf{dSt}/\mathfrak{M}_{1,2}}(\mathcal{C}^{\mathrm{der}}_{\mathrm{node}},Z\times \mathfrak{M}_{1,2})  \ar[r] \ar[d]\ar@{}[dr]|{\Box} & \bMap_{\textbf{dSt}/\mathfrak{M}_{1,2}}(\calC_2\times 
\mathfrak{M}_{g_1,n_1+1} ,Z\times \mathfrak{M}_{1,2}) \ar[d]^{ }  \\
\bMap_{\textbf{dSt}/\mathfrak{M}_{1,2}}(\mathcal{C}_1\times \mathfrak{M}_{g_2,n_2+1} ,Z\times \mathfrak{M}_{1,2}) \ar[r]^{ } & 
\bMap_{\textbf{dSt}/\mathfrak{M}_{1,2}}(\mathfrak{M}_{1,2},Z\times \mathfrak{M}_{1,2}). } \end{equation}
By the base change property \eqref{iso of bach}, we have
$$\bMap_{\textbf{dSt}/\mathfrak{M}_{1,2}}(\mathcal{C}_1\times \mathfrak{M}_{g_2,n_2+1} ,Z\times \mathfrak{M}_{1,2})\cong \bMap_{\textbf{dSt}/\mathfrak{M}_{g_1,n_1+1}}(\mathcal{C}_1,Z\times \mathfrak{M}_{g_1,n_1+1})\times_{\mathfrak{M}_{g_1,n_1+1}}\mathfrak{M}_{1,2}, $$
$$\bMap_{\textbf{dSt}/\mathfrak{M}_{1,2}}(\mathcal{C}_2\times \mathfrak{M}_{g_1,n_1+1} ,Z\times \mathfrak{M}_{1,2})\cong \bMap_{\textbf{dSt}/\mathfrak{M}_{g_2,n_2+1}}(\mathcal{C}_2,Z\times \mathfrak{M}_{g_2,n_2+1})\times_{\mathfrak{M}_{g_2,n_2+1}}\mathfrak{M}_{1,2}, $$
$$\bMap_{\textbf{dSt}/\mathfrak{M}_{1,2}}(\mathfrak{M}_{1,2},Z\times \mathfrak{M}_{1,2})\cong Z\times \mathfrak{M}_{1,2}. $$
%where $ev:=ev_{n_1+1}\times ev_{n_2+1}$. 
%where the right vertical map is the diagonal immersion and the lower horizontal map is the universal evaluation map.  
Combining them with diagram \eqref{diag on der iso prop}, we obtain an isomorphism 
\begin{align}\label{equ on Mdernode}
&\quad\, \bMap_{\textbf{dSt}/\mathfrak{M}_{1,2}}(\mathcal{C}^{\mathrm{der}}_{\mathrm{node}},Z\times \mathfrak{M}_{1,2}) \\ \nonumber 
&\cong \bMap_{\textbf{dSt}/\mathfrak{M}_{g_1,n_1+1}}(\mathcal{C}_1,Z\times \mathfrak{M}_{g_1,n_1+1})\times_{Z} \bMap_{\textbf{dSt}/\mathfrak{M}_{g_2,n_2+1}}(\mathcal{C}_2,Z\times \mathfrak{M}_{g_2,n_2+1}). 
\end{align}
Via the inclusion $\mathcal{C}_{\mathrm{node}}=t_0(\mathcal{C}^{\mathrm{der}}_{\mathrm{node}})\to \mathcal{C}^{\mathrm{der}}_{\mathrm{node}}$, 
we obtain a map of derived stacks
\begin{equation}\label{equ map on Mdernode}\bMap_{\textbf{dSt}/\mathfrak{M}_{1,2}}(\mathcal{C}^{\mathrm{der}}_{\mathrm{node}},Z\times \mathfrak{M}_{1,2})\to \bMap_{\textbf{dSt}/\mathfrak{M}_{1,2}}(\mathcal{C}_{\mathrm{node}},Z\times \mathfrak{M}_{1,2}), \end{equation}
whose classical truncation is an isomorphism. 

To summarize, combining diagram \eqref{equ on mgn pb}, Eqns.~\eqref{equ on Mdernode}, \eqref{equ map on Mdernode} 
and using notation as Eqn.~\eqref{def of mod stack of mapping}, we obtain a map of derived stacks:
\[\xymatrix{
\bMap_{g_1,n_1+1}(Z)\times_Z \bMap_{g_2,n_2+1}(Z) \ar[d] & \\ 
\bMap_{\textbf{dSt}/\mathfrak{M}_{g_1,n_1+1}\times \mathfrak{M}_{g_2,n_2+1}}(\mathcal{C}^{}_{\mathrm{node}},Z\times \mathfrak{M}_{g_1,n_1+1}\times \mathfrak{M}_{g_2,n_2+1})  \ar[r] \ar[d] \ar@{}[dr]|{\Box} & \mathfrak{M}_{g_1,n_1+1}\times \mathfrak{M}_{g_2,n_2+1} \ar[d]^{gl} \\
\bMap_{g,n}(Z)\ar[r] & \mathfrak{M}_{g,n}.  } \]
Let $Y=W\times^\bfL_{T^*W}W$, $H_R=G\times \mathbb{C}^*$ and $Z=[Y/H_R]$. By a base change through diagram \eqref{fiber diag on mgn} and some diagram chasing, we 
obtain a map of derived stacks
\begin{align*}\xymatrix{
\bMap_{g_1,n_1+1}^{R_{\chi}=\omega_{\mathrm{log}}}([Y/H_R])\times_{[Y/G]} \bMap_{g_2,n_2+1}^{R_{\chi}=\omega_{\mathrm{log}}}([Y/H_R]) \ar[d]^{} & \\
\star  \ar[r]^{} \ar[d] \ar@{}[dr]|{\Box} &  \bMap_{g,n}^{R_{\chi}=\omega_{\mathrm{log}}}([Y/H_R]) \ar[d]^{}    \\
\fBun_{H_R,g_1,n_1+1}^{R_{\chi}=\omega_{\mathrm{log}}}\times_{[\pt/G]} \fBun_{ H_R,g_2,n_2+1}^{R_{\chi}=\omega_{\mathrm{log}}}
\ar[r]^{ } \ar[d] \ar@{}[dr]|{\Box} & 
\fBun_{H_R,g,n}^{R_{\chi}=\omega_{\mathrm{log}}} \ar[d]   \\
\mathfrak{M}_{g_1,n_1+1}  \times \mathfrak{M}_{g_2,n_2+1}   \ar[r]^{\quad\quad\quad gl} & \mathfrak{M}_{g,n}, }
\end{align*}
where squares are homotopy pullback diagrams and $\star$ contains $\bfQ M_{\mathrm{node}}$ as an open substack.
By restricting to the open locus where stability is imposed, we obtain the map \eqref{equ on map der r}. 
The statement about cotangent complex of $\textbf{r}$ is straightforward to check by 
 a direct calculation similar to that in the proof of Proposition~\ref{prop on gluing form}. 
%by using Proposition \ref{prop:ev_equi 2} and Assumption \ref{ass on Rchi}, we can replace $[(W\times^\bfL_{T^*W}W)/G]$ by $[(W\times^\bfL_{T^*W}W)^s/G]$.
\end{proof}

%\gufang{need to revise the following}
%\begin{remark}
%Recall that  $\beta,\beta_1,\beta_2\in \Hom_{\bbZ}(\Pic^{G\times F}(W),\bbZ)\cong\bbX(G)^\vee\times\bbX(F)^\vee$ is a lattice.  The condition that 
%$\beta=\beta_1+\beta_2$ is equivalent to
%$\overline\beta=\overline{\beta_1}+\overline{\beta_2}$ and $\beta^t=\beta^t_1+\beta^t_2$. that is, in the terminology of Definition~\ref{def:deg_twisting}, under gluing, both the reduced degree and the degree of twisting add up. 
%\end{remark} \yl{We want $\Hom_{\bbZ}(\Pic^{G\times F}(W),\bbZ)$ or $\Hom_{\bbZ}(\Pic^{G\times \mathbb{C}^*}(W),\bbZ)$? }

\section{Quasimap invariants }
In this section, we use virtual pullbacks introduced in the previous section to define quasimap
invariants and prove a gluing formula in the cohomological field theory.

\subsection{Definitions}\label{sect on def of qm inv}

In the rest of this section, we use notations in the following setting.
\begin{setting}\label{setting of use}
Let $W,G,\theta, F_0, \chi,\phi$ be as in Setting \ref{setting of glsm} and set 
$$X:=W/\!\!/_{\theta}G, \quad X_0:=W/_{\mathrm{aff}}G $$ 
to be the GIT and the affine quotient, so the natural map $\pi: X\to X_0$ is projective. 
Without causing confusion, let  
$$\phi: X\to \mathbb{C}$$ 
denote the descent (after quotient by $G$) regular function which is $F_0$-invariant and $\Crit(\phi)\subseteq X$ be the critical locus such that $\Crit(\phi)^{F_0}$ is proper.

%Without causing confusion, we use the same notation for the descent  of the potential function and critical locus as in Setting \ref{setting of glsm}.

Let $R:\bbC^*\to F$ be the $R$-charge as in Definition \ref {defi of R-charge} such that $\Ker R_\chi=\{1\}$. 
\end{setting}

\begin{definition}(\cite[Def.~3.2.2]{CiKM}) 
An element $\beta\in \Hom_{\bbZ}(\bbX(G),\bbZ)$ is said to be {\it effective} if it comes from a quasimap class to $W/\!\!/_{\theta}G$. 
All effective classes form a monoid (by considering possibly disconnected domain curves), denoted $\Eff(W,G,\theta)$. 

We denote the submonoid of effective classes in $\Crit(\phi)$ by $N_+(\Crit(\phi))$.
%Let $\Eff(W,H,\theta)\cong N_+(\Crit(\phi))\times \bbX(F_0)^\vee$.
\end{definition}
%Alternatively, this consists of $\beta$ with either $\beta=0$ or $\beta(\theta)>0$. 
%Notice that this semi-group is furthermore a module over $\bbX(F_0)$. 
%One has, by the condition that $\theta$ comes from a character of $G$, that $\Eff(W,H,\theta)\cong N_+(\Crit(\phi))\times \bbX(F_0)^\vee$.  
%Notice that $N_+(\Crit(\phi))$ is a cone in the lattice $\bbX(G)^\vee$. 
%\yl{Do we need $\Eff(W,H,\theta)$ above? No twisting in the notation below?  }

%\yl{We want $\Hom_{\bbZ}(\Pic^{G\times F}(W),\bbZ)$ or $\Hom_{\bbZ}(\Pic^{G\times \mathbb{C}^*}(W),\bbZ)$? }

\begin{definition}
%The Novikov ring is the monoid ring of $\Eff(W,H,\theta)$ over the group algebra $\bbQ[\bbX(F_0)^\vee]\cong H_{F_0}(\pt)$, completed at the augmentation ideal. 
We define the Novikov ring as 
$$
A_{*}^{F_0}(\pt)[\![z]\!]:=\left\{\sum_{\beta\in N_+(\Crit(\phi))}a_{\beta}z^{\beta}\,\,\Big{|}\,\, a_{\beta}\in A_{*}^{F_0}(\pt)\right\}.
$$
Similarly we also define $A_{*}^{F_0}\left(-\right)[\![z]\!]$ for any $(-)$ with $F_0$-action. 
%and more generally $A_*^{F_0^n}(Z/G^n)[\![z]\!]$ for any $Z\subset W^n$, an $H$-invariant closed subscheme  such that \eqref{equ on Z} holds.
\end{definition}
\begin{remark}
Since infinite sum is allowed in the above, this space does not have a ring structure. 
Nevertheless, for each given genus $g$ and number $n$ of marked points, any effective $\beta\in \Hom_{\bbZ}(\bbX(G),\bbZ)$ has the property that $\beta(\theta)$ is bounded below.
In what follows, we only consider infinite sums which are bounded in the negative direction and such elements are closed under multiplication. 
\end{remark}
%\begin{remark}
%If $R(\Ker R_{\chi})$ acts trivially on $(\Crit(\phi)/\!\!/G)$, there is a canonical isomorphism 
%\begin{equation*}A_*^{F_0\times R(\Ker R_{\chi})}(\Crit(\phi))_\bbQ\cong A_{*}^{F_0}(\Crit(\phi))_\bbQ, \end{equation*} 
%since $R(\Ker R_{\chi})$ is a finite abelian group and we work with $\mathbb{Q}$-coefficient. 
%In what follow, we work under the simplifying assumption $\Ker R_\chi=\{1\}$. Similar for any $H$-invariant closed subscheme $Z\subset W^n$  such that \eqref{equ on Z} holds.
%\end{remark}
When $2g-2+n>0$, we consider the composition of the forgetful map and the stablization map:
\[\fBun_{ H_R,g,n}^{R_\chi=\omega_{\mathrm{log}}}   \to \mathfrak{M}_{g,n}   \xrightarrow{st} \overline{M}_{g,n}, \]
which is flat, so is the base change 
 \[\fBun_{ H_R,g,n}^{R_\chi=\omega_{\mathrm{log}}}\times (Z^s/G^n) \to \overline{M}_{g,n}\times (Z^s/G^n),  \]
where $Z\subseteq W^n$ is a $H$-invariant closed subscheme satisfying condition \eqref{equ on Z}. 
Composing with the smooth map 
$$\fBun_{ H_R,g,n}^{R_\chi=\omega_{\mathrm{log}}}\times_{[\pt/G]^n} (Z^s/G^n)\to \fBun_{ H_R,g,n}^{R_\chi=\omega_{\mathrm{log}}}\times (Z^s/G^n), $$
we obtain a flat map 
\begin{equation}\label{equ on nu}\nu: \fBun_{ H_R,g,n}^{R_\chi=\omega_{\mathrm{log}}}\times_{[\pt/G]^n} (Z^s/G^n)\to \overline{M}_{g,n}\times (Z^s/G^n). \end{equation}
Recall the map  $f$ \eqref{equ on f qm2}, we then have
\[QM^{R_\chi=\omega_{\mathrm{log}}}_{g,n}(\Crit(\phi),\beta) \xrightarrow{f} 
\fBun_{H_R,g,n}^{R_{\chi}=\omega_{\mathrm{log}}}\times_{[\pt/G]^n} (Z^s/G^n)  \xrightarrow{\nu}  \overline{M}_{g,n}\times (Z^s/G^n).  \]
We define \textit{box (or exterior) products}
$$\boxtimes: A_*(\overline{M}_{g,n})\otimes A_{*}^{F_0}(Z^s/G^n)  \to A_{*}^{F_0}(\overline{M}_{g,n}\times (Z^s/G^n)), \quad (\alpha,\beta)\mapsto (\alpha \times \beta),  $$ 
$$\boxtimes_{i=1}^n: \otimes_{i=1}^nA_{*}^{F_0}(\Crit(\phi)) \to A_{*}^{F_0}(\Crit(\phi)^n), \quad (\gamma_1,\ldots, \gamma_n)\mapsto \gamma_1\times\cdots \times \gamma_n, $$
where $\times$ is the exterior product of \cite[\S 1.10]{Fu}. 
%where $\alpha\boxtimes 1$ (resp.~$1\boxtimes \beta$) denotes the flat pullback of $\alpha$ (resp.~$\beta$) to $A_{*}^{F_0}(\overline{M}_{g,n}\times X^n )$ (resp.~to $A_{*}^{F_0}(\overline{M}_{g,n}\times (Z^s/G^n))$) and $(\alpha \boxtimes 1)\cdot(1 \boxtimes \beta)$ denotes the refined intersection as in \cite[Def.~8.1.1]{Fu}\footnote{In the notation of \cite[Def.~8.1.1]{Fu}, we take $X=Y=Y'=\overline{M}_{g,n}\times X^n$ and $X'=\overline{M}_{g,n}\times (Z^s/G^n)$.}.

Fix $Z=Z(\boxplus^{n}\phi)$ in above and define the following: 
\begin{definition}\label{defi of Phi map}
When $2g-2+n>0$, we define the following map 
\begin{equation}\label{cohft map}\Phi_{g,n,\beta}:=p_*\circ \sqrt{f^!}\circ  \nu^*\circ \boxtimes: A_{*}(\overline{M}_{g,n})\otimes A_{*}^{F_0}\left(\frac{Z(\boxplus^{n}\phi)^s}{G^{n}}\right)\to  
A_*^{F_0}(\pt)_{loc},   \end{equation}
where $$p_{*}:A_{*}^{F_0}(QM^{R_\chi=\omega_{\mathrm{log}}}_{g,n}(\Crit(\phi),\beta))\to  A_*^{F_0}(\pt)_{loc}$$ 
is the localized pushforward map for the projection $p$, defined using 
Theorem \ref{thm on properness} and Eqn.~\eqref{pf for nonproper}. 
\end{definition}
%For $\gamma_i\in A_{*}^{F_0}(\Crit(\phi))$ with $1\leqslant i\leqslant n$, we define $$\boxtimes_{i=1}^n \gamma_i\in A_{*}^{F_0}(Z^s/G^n)$$  to be the box product in the ambient space $X^n$. 
 \begin{definition}\label{def of QC}
 The \textit{quasimap invariant with insertion} $\{\gamma_i\}_{i=1}^n$ in $A_{*}^{F_0}(\Crit(\phi))$ is
$$\big\langle \gamma_1,\ldots,\gamma_n \big\rangle_{g,\beta}:=\Phi_{g,n,\beta}\left([\overline{M}_{g,n}] \boxtimes (\boxtimes_{i=1}^n \gamma_i) \right)\in A_*^{F_0}(\pt)_{loc}.$$
More generally, $\boxtimes_{i=1}^n\gamma_i\in A_{*}^{F_0}(\Crit(\phi)^n)$ can be replaced by an arbitrary class $\gamma\in A_{*}^{F_0}\left(\frac{Z(\boxplus^{n}\phi)^s}{G^{n}}\right)$, and we simply write 
$$\big\langle \gamma \big\rangle_{g,\beta}:=\Phi_{g,n,\beta}\left([\overline{M}_{g,n}]\boxtimes \gamma\right)\in A_*^{F_0}(\pt)_{loc}, $$
or $\big\langle \gamma \big\rangle_{g,\beta,n}$ if $n$ is not clear from the context. 
\end{definition}

\subsection{Gluing formula}\label{subsec:CohFT_glue}

%Recall that $QM_{g,n}:=QM_{g,n}^{R_{\chi}=\omega_{\mathrm{log}}}(\Crit(\phi)/\!\!/G,\beta)$ is a separated DM stack (Theorem~\ref{prop:geom_properties}, Remark~\ref{rmk on ass on orbi}). 
In this section, we use properties of virtual pullbacks proved in \S \ref{sect on glu} to prove a \textit{gluing formula} for the map \eqref{cohft map} in the formulation of 
 cohomological field theory

As in \eqref{cohft map}, we can define a map (where $n=n_1+n_2$):
\begin{equation}\label{equ on prod phi}\Phi_{g_1,n_1+1,\beta_1}\otimes\Phi_{g_2,n_2+1,\beta_2}: 
A_{*}^{}(\overline{M}_{g_1,n_1+1}\times\overline{M}_{g_2,n_2+1})\otimes A_{*}^{F_0}\left(\frac{Z(\boxplus^{n+2}\phi)^s}{G^{n+2}}\right)\to  A_*^{F_0}(\pt)_{loc}, \end{equation}
$$(\alpha,\theta)\mapsto (p_1\times p_2)_*\sqrt{(f_1\times f_2)^!} (\nu_1\times \nu_2)^*(\alpha\boxtimes\theta), $$
where  
\begin{align}\label{equ on nu1timenu2}
\nu_1\times \nu_2&: \fBun_{ H_R,g_1,n_1+1}^{R_\chi=\omega_{\mathrm{log}}}\times_{[\pt/G]^{n_1+1}} \left(\frac{Z(\boxplus^{n+2}\phi)^s}{G^{n+2}}\right)
\times_{[\pt/G]^{n_2+1}}\fBun_{ H_R,g_2,n_2+1}^{R_\chi=\omega_{\mathrm{log}}}  \\  \nonumber 
&\to \overline{M}_{g_1,n_1+1}\times\overline{M}_{g_2,n_2+1}\times \left(\frac{Z(\boxplus^{n+2}\phi)^s}{G^{n+2}}\right) \end{align} 
is defined similarly as \eqref{equ on nu} and $p_1\times p_2: QM_{g_1,n_1+1}  \times QM_{g_2,n_2+1} \to \pt$ is the projection.
Here although the notation is in the product form, the map is not necessarily the tensor product of two maps in general.
%which by an abuse of notations is denoted by $\Phi_{g_1+1,n_1,\beta_1}\otimes\Phi_{g_2,n_2+1,\beta_2}$.
Let 
\begin{equation}\label{diagon class}\eta\in A_{*}^{F_0}\left(\frac{Z(\boxplus^{2}\phi)^s}{G^{2}}\right) \end{equation} 
be the class of the \textit{anti-diagonal} $\bar{\Delta}: X\to \frac{Z(\boxplus^{2}\phi)^s}{G^{2}}$ in \eqref{diag on Zboxphi}. 

For any $\gamma\in A_{*}^{F_0}\left(\frac{Z(\boxplus^{n}\phi)^s}{G^{n}}\right)$, we have its box (or exterior) product with $\eta$ (\cite[\S 1.10]{Fu}):
$$\gamma\boxtimes\eta\in A_{*}^{F_0}\left(\frac{Z(\boxplus^{n}\phi)^s}{G^{n}}\times \frac{Z(\boxplus^{2}\phi)^s}{G^{2}}\right),$$ 
which is also considered as an element in $A_{*}^{F_0}\left(\frac{Z(\boxplus^{n+2}\phi)^s}{G^{n+2}}\right)$ by the pushforward of inclusion. 
%that is, the refined intersection of $1\boxtimes \gamma\in A_{*}^{F_0}((W^s/G)^2\times Z^s/G^n)$ with $\eta\boxtimes1\in A_{*}^{F_0}(Z^s/G^2\times (W^s/G)^n)$.

For $n=n_1+n_2$, $g=g_1+g_2$, we have the \textit{gluing morphism}
$$\iota: \overline{M}_{g_1,n_1+1}  \times \overline{M}_{g_2,n_2+1} \to \overline{M}_{g,n}. $$ 
Note also that any class in $A_*^{F_0}\left(\frac{Z(\boxplus^{n_1}\phi)^s}{G^{n_1}}\times \frac{Z(\boxplus^{n_2}\phi)^s}{G^{n_2}}\right)$ can be considered 
as an element in $A_{*}^{F_0}\left(\frac{Z(\boxplus^{n}\phi)^s}{G^{n}}\right)$ with $n=n_1+n_2$ by the pushforward of inclusion.
\begin{theorem}\label{thm:CohFT_glue}
Let $\gamma \in \mathrm{Im}\left(A^{F_0}_*\left(\frac{Z(\boxplus^{n_1}\phi)^s}{G^{n_1}}\times \frac{Z(\boxplus^{n_2}\phi)^s}{G^{n_2}}\right)\to A^{F_0}_*\left(\frac{Z(\boxplus^{n}\phi)^s}{G^{n}}\right)\right)$ be in the image and 
$\alpha\in A_*(\overline{M}_{g_1,n_1+1}  \times \overline{M}_{g_2,n_2+1})$. Then  
\begin{equation}\label{glue formula new form} 
\Phi_{g,n,\beta}((\iota_*\alpha)\boxtimes\gamma)=\sum_{\beta_1+\beta_2=\beta}\left( \Phi_{g_1,n_1+1,\beta_1}\otimes \Phi_{g_2,n_2+1,\beta_2} \right)(\alpha\boxtimes (\gamma\boxtimes\eta)).\end{equation}
\end{theorem}
\begin{proof}
For a decomposition $\beta=\beta_1+\beta_2$, we introduce the following shorthands:
$$QM=QM^{R_{\chi}=\omega_{\mathrm{log}}}_{g,n}(\Crit(\phi),\beta), \quad QM_i=QM_i(\beta_i)=QM^{R_{\chi}=\omega_{\mathrm{log}}}_{g_i,n_i+1}(\Crit(\phi),\beta_i), \,\,\, i=1,2, $$
$$\fB_1:=\fBun_{H_R,g_1,n_1+1}^{R_{\chi}=\omega_{\mathrm{log}}}, \quad \fB_2:=\fBun_{H_R,g_2,n_2+1}^{R_{\chi}=\omega_{\mathrm{log}}}, \quad \fB:=\fBun_{H_R,g,n}^{R_{\chi}=\omega_{\mathrm{log}}}.$$
For any Deligne-Mumford stack $X$, we write the structure map  $X\to\pt$ as $p_{X}$. 

As in \cite[Proof of Prop.~6.22]{AGV}, we have the following commutative diagram with all squares being Cartesian (here we use notations as diagrams \eqref{eqn:degeneration0}, \eqref{eqn:degeneration}):
\begin{align}\label{big diag gl for} 
\xymatrix@C=-3pt{ 
& \pt  &  & \\ 
\coprod_{\beta_1+\beta_2=\beta}QM_{1}(\beta_1)  \times_{X} QM_{2}(\beta_2)  \ar@{}[drr]|{\Box} \ar[ru]^{\coprod_{\beta_1+\beta_2=\beta}p_{QM_1\times_X QM_2} \quad\quad\quad } \ar[rr]^{\quad\quad \quad \quad gl} \ar[d]_{f_{\mathrm{node}}}   &  & QM_{} \ar[d]^{f} \ar[lu]_{p_{QM}}  & \\
\calB:=\fB_1\times_{[\pt/G]} \fB_2 \times_{[\pt/G]^n} \frac{Z(\boxplus^{n}\phi)^s}{G^{n}}  \ar@/_7.5pc/[dddd]_{\nu_{12}}
\ar[rr]^{ \quad\quad\quad\quad gl} \ar@{}[drr]|{\Box}  \ar[d]_{s_{12}} &  &
\fB \times_{[\pt/G]^n} \frac{Z(\boxplus^{n}\phi)^s}{G^{n}}\ar[d]^{s}   \ar@/^2pc/[ddddr]^{\nu}  & \\
\fB_1\times_{[\pt/G]} \fB_2 \times \frac{Z(\boxplus^{n}\phi)^s}{G^{n}} \ar@{}[drr]|{\Box}   \ar[rr]^{\quad \quad  gl} \ar[d]_{\pi_{B12}}  & & 
\fB\times \frac{Z(\boxplus^{n}\phi)^s}{G^{n}} \ar[d]^{\pi_B}   & \\
\fB_1\times_{[\pt/G]} \fB_2 \ar@{}[drr]|{\Box}   \ar[rr]^{\quad \quad  gl} \ar[d]_{g_{12}}  & & 
\fB \ar[d]^{g} &   \\
 \mathfrak{M}_{g_1,n_1+1}  \times \mathfrak{M}_{g_2,n_2+1} \ar@/^1.2pc/[rr]_{gl}  \ar[r]^{\quad \quad i_{\mathfrak Q}}  \ar[dr]^{st_1\times st_2}  & \mathfrak Q  \ar@{}[dr]|{\Box}  \ar[r]^{j \quad } \ar[d]^{p}  & \mathfrak{M}_{g,n}  \ar[d]^{st}  & \\
\overline{M}_{g_1,n_1+1} \times \overline{M}_{g_2,n_2+1}\times \frac{Z(\boxplus^{n}\phi)^s}{G^{n}} & \overline{M}_{g_1,n_1+1} \times \overline{M}_{g_2,n_2+1} \ar[r]^{\quad\quad \quad \iota}  & \overline{M}_{g,n} & \overline{M}_{g,n}\times\frac{Z(\boxplus^{n}\phi)^s}{G^{n}}.}  
\end{align}
Then 
\begin{align}\label{equ on Phi alpha}
\Phi_{g,n,\beta}((\iota_*\alpha)\boxtimes\gamma)&:=p_{QM*}\circ \sqrt{f^!}\circ  \nu^*((\iota_*\alpha)\boxtimes\gamma)\\
&=p_{QM*}\circ \sqrt{f^!}\circ s^*((g^*st^*\iota_*\alpha)\boxtimes\gamma)  \nonumber \\
&=p_{QM*}\circ \sqrt{f^!}\circ s^*((g^*j_*p^*\alpha)\boxtimes\gamma) \nonumber \\
&=p_{QM*}\circ \sqrt{f^!}\circ s^*((g^*j_*i_{\mathfrak Q*}(st_1\times st_2)^*\alpha)\boxtimes\gamma) \nonumber\\
&=p_{QM*}\circ \sqrt{f^!}\circ s^*((g^*gl_*(st_1\times st_2)^*\alpha)\boxtimes\gamma)\nonumber \\
&=p_{QM*}\circ \sqrt{f^!}\circ s^*((gl_*g_{12}^*(st_1\times st_2)^*\alpha)\boxtimes\gamma)  \nonumber \\
&=p_{QM*}\circ \sqrt{f^!}\circ s^*gl_*((g_{12}^*(st_1\times st_2)^*\alpha)\boxtimes\gamma) \nonumber \\
&=p_{QM*}\circ \sqrt{f^!}\circ gl_*s_{12}^*((g_{12}^*(st_1\times st_2)^*\alpha)\boxtimes\gamma) \nonumber \\
&=p_{QM*}\circ gl_*\sqrt{f_{\mathrm{node}}^!}\nu_{12}^*(\alpha\boxtimes\gamma) \nonumber \\
&=\sum_{\beta_1+\beta_2=\beta}p_{QM_1\times_XQM_2*}\circ \sqrt{f_{\mathrm{node}}^!}\circ \nu_{12}^*(\alpha\boxtimes\gamma). \nonumber
\end{align}
%\begin{align}\label{equ on Phi alpha}
%\Phi_{g,n,\beta}(\iota_*\alpha)&=can \, ev^n_*  \sqrt{f^!}  \nu^* st^*\iota_*\alpha\boxtimes [X^n]\\ \nonumber 
%&=can\, ev^n_*  \sqrt{f^!}  \nu^*  j_* p^*\alpha\boxtimes [X^n]\\ \nonumber 
%&=can\, ev^n_*  \sqrt{f^!}  \nu^*  gl_*(st_1\times st_2)^*\alpha\boxtimes [X^n]\\ \nonumber 
%&=can\, ev^n_*  \sqrt{f^!}  gl_*(\nu_1\times \nu_2)^*  (st_1\times st_2)^*\alpha\boxtimes [X^n] \\ \nonumber 
%&=can\, ev^n_*  gl_*\sqrt{f_{\mathrm{node}}^!}  (\nu_1\times \nu_2)^*  (st_1\times st_2)^*\alpha\boxtimes [X^n] \\ \nonumber 
%&=can\, (ev^{n})_* \sqrt{(\mu_{\mathrm{node}}\times ev^n)^!}  (\nu_1\times \nu_2)^*  (st_1\times st_2)^*\alpha\boxtimes [X^n].
%\end{align}
Here we use
Proposition \ref{prop on gluing mor} in the 9th equality. 
We explain that the 4th equality
follows from \cite[Prop.~8]{B}. Indeed, {\it loc.~cit.}~states that $i_{\mathfrak{Q}}$ is proper, finite, and 
$$p^*[\overline{M}_{g_1,n_1+1}\times \overline{M}_{g_2,n_2+1}]= i_{\mathfrak{Q}*}[\mathfrak{M}_{g_1,n_1+1}  \times \mathfrak{M}_{g_2,n_2+1}]. $$
For $\alpha\in A_*(\overline{M}_{g_1,n_1+1}\times \overline{M}_{g_2,n_2+1})$,
by Poincar\'e duality, we can write 
$$\alpha=\xi\cap [\overline{M}_{g_1,n_1+1}  \times \overline{M}_{g_2,n_2+1}], $$ 
for some $\xi\in A^*(\overline{M}_{g_1,n_1+1} \times \overline{M}_{g_2,n_2+1})$, where
the cap product is defined on DM stacks by \cite[\S 5]{Vis}, following \cite[\S 17.2]{Fu}, and extended to Artin stacks by \cite[App.~C]{BS} based on \cite{Kre}. 
Then
\begin{align}
p^*\alpha&=p^*\left(\xi\cap [\overline{M}_{g_1,n_1+1}  \times \overline{M}_{g_2,n_2+1}]\right) \\ \nonumber 
&= \left(p^*\xi\cap p^*[\overline{M}_{g_1,n_1+1}  \times \overline{M}_{g_2,n_2+1}]\right) \\ \nonumber 
&=   \left(p^*\xi\cap i_{\mathfrak{Q}*}[\mathfrak{M}_{g_1,n_1+1}  \times \mathfrak{M}_{g_2,n_2+1}]\right)  \\ \nonumber 
&= i_{\mathfrak{Q}*}\left(i_{\mathfrak{Q}}^*p^*\xi\cap [\mathfrak{M}_{g_1,n_1+1}  \times \mathfrak{M}_{g_2,n_2+1}]\right) 
 \\ \nonumber 
&=i_{\mathfrak{Q}*}\left((st_1\times st_2)^*\xi\cap [\mathfrak{M}_{g_1,n_1+1}  \times \mathfrak{M}_{g_2,n_2+1}]\right) \\ \nonumber 
&=i_{\mathfrak{Q}*}(st_1\times st_2)^*\left(\xi\cap [\overline{M}_{g_1,n_1+1}  \times \overline{M}_{g_2,n_2+1}]\right) \\ \nonumber 
&=i_{\mathfrak{Q}*}  (st_1\times st_2)^*\alpha.
\end{align}
Recall diagram \eqref{diag def fnote bar}, we have 
\begin{equation}\label{diag copyd}
\xymatrix{
\mathcal{Y}\ar[d]_{\id_{\calY}\times ev_\Delta}\ar[drr]^{f_{\Delta} \quad }  &  & \overline{\mathcal{Y}} \ar[ll]_{\sigma} \ar[d]^{f_{\bar{\Delta}}} \\
\mathcal{Y}\times_{[\pt/G]} X\ar[d]^{p_\calY} \ar[rr]^{f_{\mathrm{node}}\times\id_X\quad} \ar@{}[drr]|{\Box} & & \mathcal{Z}=\mathcal{B}\times_{[\pt/G]} X \ar[d]_{p_\calB} \\
\mathcal{Y} \ar[rr]^{f_{\mathrm{node}} } \ar@/^1pc/[u]^{\id_{\calY}\times ev_\Delta} & & \mathcal{B},
}\end{equation}
where $\id_{\calY}\times ev_\Delta$ is a section of $p_\calY$ and hence 
\begin{equation}\label{eqn:section_pullback}
    (\id_{\calY}\times ev_\Delta)^!\circ p_\calY^*=\id_\calY^*.
\end{equation} 
Proposition \ref{prop on gluing form} gives  
\begin{equation}\label{prop on gluing form2}
\sigma^*\circ (\id_{\calY}\times ev_\Delta)^!\circ \sqrt{(f_{\mathrm{node}}\times\id_X)^!}=\sqrt{f_{\bar{\Delta}}^!}. \end{equation}
To sum up, we have 
\begin{align*} 
\Phi_{g,n,\beta}((\iota_*\alpha)\boxtimes\gamma)&\overset{\eqref{equ on Phi alpha}}{=}
\sum_{\beta_1+\beta_2=\beta} p_{QM_1\times_XQM_2*}\circ \sqrt{f_{\mathrm{node}}^!}\circ \nu_{12}^*(\alpha\boxtimes\gamma) \\
&\overset{\eqref{eqn:section_pullback}}{=}
\sum_{\beta_1+\beta_2=\beta} p_{QM_1\times_XQM_2*}\circ (\id_{\calY}\times ev_\Delta)^!\circ p_\calY^*\circ \sqrt{f_{\mathrm{node}}^!}\circ \nu_{12}^*(\alpha\boxtimes\gamma) \\
&\overset{\eqref{diag copyd}}{=} \sum_{\beta_1+\beta_2=\beta} p_{QM_1\times_XQM_2*}\circ (\id_{\calY}\times ev_\Delta)^!\circ  \sqrt{(f_{\mathrm{node}}\times\id_X)^!}\circ p_\calB^*\circ \nu_{12}^*(\alpha\boxtimes\gamma) \\
&\overset{\eqref{prop on gluing form2} }{=}\sum_{\beta_1+\beta_2=\beta}p_{QM_1\times_XQM_2*}\circ (\sigma^{-1})^*\circ\sqrt{f_{\bar{\Delta}}^!}\circ p_\calB^*\circ \nu_{12}^*(\alpha\boxtimes\gamma)\\
&\overset{\eqref{diag cpr Deltabar}}{=}\sum_{\beta_1+\beta_2=\beta}p_{QM_1\times_{\bar{\Delta}} QM_2*}\circ \sqrt{f_{\bar{\Delta}}^!}\circ p_\calB^*\circ \nu_{12}^*(\alpha\boxtimes\gamma)\\
&\overset{\eqref{diag on diagonal} }{=}\sum_{\beta_1+\beta_2=\beta}p_{QM_1\times QM_2*}\circ i_{\bar{\Delta}*}\circ \sqrt{f_{\bar{\Delta}}^!}\circ 
p_\calB^*\circ \nu_{12}^*(\alpha\boxtimes\gamma)\\
&\overset{\eqref{eqn:Delta}}{=}\sum_{\beta_1+\beta_2=\beta} p_{QM_1\times QM_2*}\circ \sqrt{(f_1\times f_2)^!} \circ \bar{\Delta}_*\circ p_\calB^*\circ \nu_{12}^*(\alpha\boxtimes\gamma) \\
&\overset{\mathrm{Lem}. \ref{Lem:diagonal}}{=}\sum_{\beta_1+\beta_2=\beta}p_{QM_1\times  QM_2*}\circ 
\sqrt{(f_1\times f_2)^!}\circ (\nu_{1}\times \nu_{2})^*(\alpha\boxtimes\gamma \boxtimes\eta) \\
&\overset{\eqref{equ on prod phi}}{=}\sum_{\beta_1+\beta_2=\beta}\left( \Phi_{g_1,n_1+1,\beta_1}\otimes \Phi_{g_2,n_2+1,\beta_2} \right)(\alpha\boxtimes (\gamma\boxtimes\eta)),
\end{align*}
where $(\nu_{1}\times \nu_{2})$ is defined in \eqref{equ on nu1timenu2}. 
\end{proof}
\begin{remark}
One can similarly show the genus reduction axiom in the cohomological field theory, 
which we leave to the reader to check details.  
%As we do not need it in this paper, 
%We refer to \cite{P} for a survey on recent development in cohomological field theories. 
\end{remark}
\begin{remark}
In general, without the condition on embedding $\Crit(\phi)\hookrightarrow Z(\phi)$ in Setting \ref{setting of glsm}, by Remark \ref{rmk on pullback without embed}, we simply replace $Z(\boxplus^{i}\phi)$ 
by $Z((\boxplus^{i}\phi)^r)$ for some large $r\geqslant1$ in the above theorem. 
\end{remark}
\begin{lemma}\label{Lem:diagonal} We have
$\bar{\Delta}_*\circ p_\calB^*\circ \nu_{12}^*(\alpha\boxtimes\gamma)=(\nu_{1}\times \nu_{2})^*(\alpha\boxtimes\gamma \boxtimes\eta). $
\end{lemma}
\begin{proof}
Recall notations in diagrams \eqref{diag on moduli rel to Bun}, \eqref{big diag gl for}, 
we have a Cartesian diagram 
\begin{align*}\xymatrix{
\mathcal{Z}=\mathcal{B}\times_{[\pt/G]} X \ar[r]^{\bar{\Delta}\quad\quad\quad\quad\quad\quad\quad\quad\quad\quad} \ar[d]^{s_1} \ar@{}[dr]|{\Box} &  \fB_1 \times_{[\pt/G]^{n_1+1}}\frac{Z(\boxplus^{n_1}\phi)^s}{G^{n_1}}\times \frac{Z(\boxplus^{n_2}\phi)^s}{G^{n_2}} \times \frac{Z(\boxplus^{2}\phi)^s}{G^{2}}  \times_{[\pt/G]^{n_2+1}}\fB_2 \ar[d]^{s}    \\
\fB_1 \times \frac{Z(\boxplus^{n_1}\phi)^s}{G^{n_1}}\times \frac{Z(\boxplus^{n_2}\phi)^s}{G^{n_2}} \times X \times \fB_2
\ar[r]^{\bar{\Delta}\quad \,\,}  & \fB_1 \times \frac{Z(\boxplus^{n_1}\phi)^s}{G^{n_1}}\times \frac{Z(\boxplus^{n_2}\phi)^s}{G^{n_2}} \times \frac{Z(\boxplus^{2}\phi)^s}{G^{2}} \times \fB_2,      }
\end{align*}
and a commutative diagram
\begin{align*}\xymatrix{
\mathcal{Z}=\mathcal{B}\times_{[\pt/G]} X \ar[r]^{p_\calB\quad\quad\quad\quad\quad\quad\quad} \ar[d]^{s_1}  &  
\mathcal{B}=\fB_1 \times_{[\pt/G]}\fB_2\times_{[\pt/G]^{n}} \frac{Z(\boxplus^{n_1}\phi)^s}{G^{n_1}}\times \frac{Z(\boxplus^{n_2}\phi)^s}{G^{n_2}}  \ar[d]^{s_{12}}    \\
\fB_1 \times \fB_2 \times  \frac{Z(\boxplus^{n_1}\phi)^s}{G^{n_1}}\times \frac{Z(\boxplus^{n_2}\phi)^s}{G^{n_2}} \times X  
\ar[d]^{s_2  }  & \fB_1 \times_{[\pt/G]} \fB_2\times \frac{Z(\boxplus^{n_1}\phi)^s}{G^{n_1}}\times \frac{Z(\boxplus^{n_2}\phi)^s}{G^{n_2}} \ar[dl]^{s_{3} }  \\
\fB_1\times \fB_2 \times \frac{Z(\boxplus^{n_1}\phi)^s}{G^{n_1}}\times \frac{Z(\boxplus^{n_2}\phi)^s}{G^{n_2}},  &       }
\end{align*}
where $s_2$ is the projection map and all maps in this diagram are smooth. 

Therefore for any $\theta\in A_*^{}(\fB_1\times \fB_2)$ and $\gamma\in A^{F_0}_*\left(\frac{Z(\boxplus^{n_1}\phi)^s}{G^{n_1}}\times \frac{Z(\boxplus^{n_2}\phi)^s}{G^{n_2}}\right)$, 
we have 
\begin{align}\label{equ on s123}
\bar{\Delta}_*p_\calB^*s_{12}^*s_{3}^*(\theta\boxtimes \gamma)&=\bar{\Delta}_*s_{1}^*s_{2}^*(\theta\boxtimes \gamma) \\ \nonumber
&=\bar{\Delta}_*s_{1}^*(\theta\boxtimes \gamma\boxtimes [X])\\  \nonumber 
&=s_{}^*(\theta\boxtimes \gamma\boxtimes \eta). 
\end{align}
In the notations of diagram \eqref{big diag gl for}, for any $\alpha\in A_*(\overline{M}_{g_1,n_1+1}  \times \overline{M}_{g_2,n_2+1})$, we have 
\begin{align}\label{equ on s1234}
\gamma_{12}^*(\alpha \boxtimes \gamma )&=s_{12}^*\left(\left(g_{12}^*(st_1\times st_2)^*\alpha\right)\boxtimes \gamma\right) \\ \nonumber
&=s_{12}^*s^*_3\left(\left((g_{1}\times g_2)^*(st_1\times st_2)^*\alpha\right)\boxtimes \gamma\right), 
\end{align}
where $g_1\times g_2: \fB_1\times \fB_2\to \mathfrak{M}_{g_1,n_1+1}  \times \mathfrak{M}_{g_2,n_2+1}$
is the product of forgetful maps. Note also 
\begin{equation}\label{equ on s12345}
(\nu_{1}\times \nu_{2})^*(\alpha\boxtimes\gamma \boxtimes\eta)
=s^*\left(\left((g_{1}\times g_2)^*(st_1\times st_2)^*\alpha\right)\boxtimes \gamma\boxtimes\eta\right). 
\end{equation}
Let $\theta=(g_{1}\times g_2)^*(st_1\times st_2)^*\alpha$ and combine with Equs.~\eqref{equ on s123}, \eqref{equ on s1234}, \eqref{equ on s12345}, we are done.
\end{proof}

\subsection{WDVV type equation}\label{sect on WDVV}
In this section, using the gluing formula \eqref{thm:CohFT_glue} proved in the previous section, we show a Witten-Dijkgraaf-Verlinde-Verlinde (WDVV) type equation
for the invariants defined in \eqref{equ on prod phi}. In the special cases discussed in \S \ref{subsec:crit2}, we show that 
such a WDVV type equation implies the associativity of the quantum product defined on critical cohomologies. 
%Motivated by the Witten-Dijkgraaf-Verlinde-Verlinde (WDVV) equation in Gromov-Witten theory (e.g.~\cite[\S 6.2, \S 6.3]{AGV}), which was used to show the associativity of the quantum product, we prove a WDVV type equation for the invariants \eqref{equ on prod phi}.

Let $n\in\bbN$.
Fix a collection of classes 
$$\delta_i\in A_*^{F_0}(Z(\phi)^s/G), \,\,\, 1\leqslant i\leqslant n, \quad \gamma_j\in A_*^{F_0}(Z(\phi)^s/G), \,\,\, 1\leqslant j\leqslant 4. $$
For any partition $A\sqcup B=\{1,2,\ldots,n\}$, we denote
$$\delta_A=\delta_{i_1}\boxtimes \cdots \boxtimes \delta_{i_m} \in A_*^{F_0}\left((Z(\phi)^s/G)^A\right), $$ 
where $A=\{i_1,\cdots, i_m\}$, and we similarly denote $\delta_B$. Let  
$$\widehat{A}:=A\sqcup \{n+1,n+2\}, \quad \widehat{B}:=B\sqcup \{n+3,n+4\}.  $$ 
%$$QM=QM^{R_{\chi}=\omega_{\mathrm{log}}}_{0,n+4}(\Crit(\phi),\beta), \quad QM_1=QM^{R_{\chi}=\omega_{\mathrm{log}}}_{0,\widehat{A}\sqcup \check{\bullet}}(\Crit(\phi),\beta_1), \quadQM_2=QM^{R_{\chi}=\omega_{\mathrm{log}}}_{0,\widehat{B}\sqcup \bullet}(\Crit(\phi),\beta_2), $$
%$$\fBun_1:=\fBun_{H_R,0,\widehat{A}\sqcup \check{\bullet}}^{R_{\chi}=\omega_{\mathrm{log}}}, \,\,\fBun_2:=\fBun_{H_R,0,\widehat{B}\sqcup \bullet}^{R_{\chi}=\omega_{\mathrm{log}}}, \,\, \fBun:=\fBun_{H_R,0,n+4}^{R_{\chi}=\omega_{\mathrm{log}}}.$$
As suggested by the notation, evaluation maps of  
$QM_1=QM^{R_{\chi}=\omega_{\mathrm{log}}}_{0,\widehat{A}\sqcup \check{\bullet}}(\Crit(\phi),\beta_1)$ are associated to marked points labelled by $\widehat{A}\sqcup \check{\bullet}$, and evaluation maps of 
$QM_2=QM^{R_{\chi}=\omega_{\mathrm{log}}}_{0,\widehat{B}\sqcup \bullet}(\Crit(\phi),\beta_2)$ are associated to marked points labelled by $\widehat{B}\sqcup \bullet$. 

Recall the class $\eta\in A_{*}^{F_0}\left(\frac{Z(\boxplus^{2}\phi)^s}{G^{2}}\right)$ of the anti-diagonal with $\frac{Z(\boxplus^{2}\phi)^s}{G^{2}}\subseteq X\times X$, where the two factors are associated to points $\check{\bullet}$ and $\bullet$ respectively. 
\begin{theorem}\label{thm:WDVV_invariants_form}
Notations as above, we have 
\begin{align*} 
&\quad \, \sum_{\beta_1+\beta_2=\beta}\sum_{A\sqcup B=\{1,2,\ldots,n\}}
\Phi_{0,|A|+3,\beta_1}\otimes\Phi_{0,|B|+3,\beta_2}\left([\overline{M}_{0,|A|+3}\times\overline{M}_{0,|B|+3}]\boxtimes\delta_A\boxtimes\gamma_1\boxtimes\gamma_2\boxtimes\eta\boxtimes \delta_B\boxtimes\gamma_3\boxtimes\gamma_4\right) \\
&=\sum_{\beta_1+\beta_2=\beta}\sum_{A\sqcup B=\{1,2,\ldots,n\}}\Phi_{0,|A|+3,\beta_1}\otimes\Phi_{0,|B|+3,\beta_2}\left([\overline{M}_{0,|A|+3}\times\overline{M}_{0,|B|+3}]\boxtimes\delta_A\boxtimes\gamma_1\boxtimes\gamma_3\boxtimes\eta\boxtimes \delta_B\boxtimes\gamma_2\boxtimes\gamma_4\right). 
\end{align*}
\end{theorem}
\begin{proof}
 By the gluing formula \eqref{glue formula new form}, the left-hand-side of the above becomes
\begin{align*}
&\sum_{\beta_1+\beta_2=\beta}\sum_{A\sqcup B=\{1,2,\ldots,n\}}
\Phi_{0,|A|+3,\beta_1}\otimes\Phi_{0,|B|+3,\beta_2}\left([\overline{M}_{0,|A|+3}\times\overline{M}_{0,|B|+3}]\boxtimes\delta_A\boxtimes\gamma_1\boxtimes\gamma_2\boxtimes\eta\boxtimes \delta_B\boxtimes\gamma_3\boxtimes\gamma_4 \right) \\&=\Phi_{0,n+4,\beta}\left(\iota_*[\overline{M}_{0,|A|+3}\times\overline{M}_{0,|B|+3}]\boxtimes \delta_A\boxtimes\gamma_1\boxtimes\gamma_2\boxtimes \delta_B\boxtimes\gamma_3\boxtimes\gamma_4\right).
\end{align*}
Similar formula holds for the right-hand-side. The assertion then follows directly from the commutativity of the box-product. 
\end{proof}

\subsection{Specialization for the zero-potential}
In this section, we show that our QM invariants (when $\phi=0$) specialize to the QM type invariants of smooth GIT quotients as defined in \cite{CiKM}\footnote{The slight difference is that there is no twist in the formulation of \cite{CiKM}.}.
 
Let $\phi=0$ in Setting \ref{setting of use} so $\Crit(\phi)=X=W/\!\!/G$. Recall the following 
maps of derived stacks (e.g.~\eqref{diag on map f}): 
\begin{equation*}   \xymatrix{ 
\bMap_{g,n}^{R_{\chi}=\omega_{\mathrm{log}}}([(W\times^\bfL_{T^*W}W)/H_R])\ar[r]^{\quad \quad \textbf{i}} & \bMap_{g,n}^{R_{\chi}=\omega_{\mathrm{log}}}([W/H_R]) \ar@/^6.5pc/[dd]^{\textbf{g}} \ar[d]^{\textbf{f}}  \\
& [W/H_R]^n\times_{[\pt/H_R]^n}\fBun_{H_R,g,n}^{R_{\chi}=\omega_{\mathrm{log}}}  \ar[d]^{\pi_B} \\
& \fBun_{H_R,g,n}^{R_\chi=\omega_{\mathrm{log}}}, }
\end{equation*}
%$$\textbf{g}: \bMap_{g,n}^{R_{\chi}=\omega_{\mathrm{log}}}([W/H_R])\xrightarrow{\textbf{f}} [W/H_R]^n\times_{[\pt/H_R]^n}\fBun_{H_R,g,n}^{R_{\chi}=\omega_{\mathrm{log}}} \xrightarrow{\pi_B} \fBun_{H_R,g,n}^{R_\chi=\omega_{\mathrm{log}}}, $$
with the induced maps of cotangent complexes 
\begin{equation}\label{map of ct cpx in phi=0}\textbf{f}^*\bbL_{\pi_B} \to \bbL_{\textbf{g}}\to \bbL_{\textbf{f}}, \quad  
\textbf{i}^*\bbL_{\textbf{f}}\to \bbL_{\textbf{f}\circ \textbf{i}}\to \bbL_{\textbf{i}}.  \end{equation}
As $\phi=0$, the classical truncation $i$ of $\textbf{i}$ is an isomorphism 
$$i: Map_{g,n}^{R_{\chi}=\omega_{\mathrm{log}}}([(W\times_{T^*W}W)/H_R])\cong Map_{g,n}^{R_{\chi}=\omega_{\mathrm{log}}}([W/H_R]). $$
The restriction of the relative cotangent complexes of $\textbf{f}$ and $\textbf{g}=\pi_B\circ \textbf{f}$ to the classical truncation defines
relative obstruction theories 
$$\varphi: \bbF \to \bbL_f, \quad \psi: \bbG \to \bbL_g, $$ 
where $f$ and $g$ are the classical truncations of $\textbf{f}$ and $\textbf{g}$, and we restrict to the open substack 
$$QM:=QM_{g,n}^{R_{\chi}=\omega_{\mathrm{log}}}(\Crit(\phi),\beta)\subset Map_{g,n}^{R_{\chi}=\omega_{\mathrm{log}}}([W/H_R])$$ of stable $R$-twisted quasimaps to $X$. 
We describe $\bbF$ and $\bbG$ explicitly as follows. 

Let $\pi:\mathcal{C}\to QM$ denote the universal curve with universal section $p_1,\ldots,p_n$, $\mathcal{P}\to \mathcal{C}$ be the universal $H_R$-bundle and $\calW:=\mathcal{P}\times_{H_R}W$. 
The log canonical bundle is 
$$\omega_{\mathrm{log}}=\omega_{\pi}(S), \quad \mathrm{where}\,\,\, S:=p_1+\cdots +p_n. $$
As argued in Theorem~\ref{prop:symm_ob}, we know 
\begin{equation*}
\bbF\cong  \dR\pi_*\left(\calW\boxtimes\oO(-S)\right)^\vee, \quad  \bbG\cong  \dR\pi_*\left(\calW\right)^\vee. \end{equation*} 
Restricting the first sequence in \eqref{map of ct cpx in phi=0} to the classical truncation gives a compatible diagram of relative perfect obstruction theories: 
\begin{equation*}    \xymatrix{
f^*\bbL_{\pi_B}\cong \dR\pi_*\left(\calW\boxtimes\oO_S \right)^\vee \ar[d] \ar[r]  & \bbG \ar[d] \ar[r]  &  \bbF \ar[d] & \\
f^*\bbL_{\pi_B}  \ar[r] & \bbL_{g}   \ar[r] & \bbL_{f}. }
\end{equation*}
By Manolache's virtual pullback \cite[Cor.~4.9]{Man}, we have 
$$g_{\psi}^!= f_{\varphi}^!\circ \pi_B^*,$$
where $\pi_B^*$ is the flat pullback. 
 
Applying the left hand side to $[\fBun_{H_R,g,n}^{R_\chi=\omega_{\mathrm{log}}}]$, we obtain the virtual class of $QM$ defined as in \cite[Prop.~4.4.1, \S 5.2]{CiKM}. 
Using the above equality, we obtain 
\begin{equation}\label{equ on two vir cla}
[QM]^{\mathrm{vir}}_{\varphi}=f_{\varphi}^!\left[[W/H_R]^n\times_{[\pt/H_R]^n}\fBun_{H_R,g,n}^{R_{\chi}=\omega_{\mathrm{log}}}\right]. \end{equation}
Recall Theorem \ref{prop:symm_ob}, the restriction of $\textbf{f}\circ \textbf{i}$ to the classical truncation gives an isotropic symmetric obstruction theory 
$\phi_f:\bbE_f\to \bbL_f$ with  
\begin{equation}\label{map in Ef}\bbE_f\cong \left(\dR\pi_*\left(\calW\boxtimes\oO(-S)\right) \to \dR\pi_*\left(\calW\boxtimes\oO(-S)\right)^\vee\right), \end{equation} 
and a virtual class:
\begin{equation}\label{equ on two vir cla2} [QM]^{\mathrm{vir}}_{\phi}:= \sqrt{f_{\phi}^!}\left[[W/H_R]^n\times_{[\pt/H_R]^n}\fBun_{H_R,g,n}^{R_{\chi}=\omega_{\mathrm{log}}}\right], 
\end{equation}
defined using Definition \ref{defi of qm vir class}.
In below we show these two virtual classes are the same. 
\begin{prop}\label{prop:red_smooth}
There is a map $\delta:\bbE_f\to \bbF$ such that $\phi_f=\varphi\circ\delta$, making $\bbF$ a maximal isotropic subcomplex in the sense of \cite[Def.~1.4]{Park}. 
Therefore for some choice of sign in \eqref{cho of sign}, we have  
$$\sqrt{f_{\phi}^!}=f^!_\varphi. $$
In particular, virtual class in \eqref{equ on two vir cla2} can recover the virtual class in \eqref{equ on two vir cla}. 
\end{prop}
\begin{proof}
The restriction of the second sequence in \eqref{map of ct cpx in phi=0} to the classical truncation gives a map 
\begin{equation*}
\xymatrix{ 
\bbF  \ar[r] \ar[d]_{\varphi}  & \bbE_f  \ar[d]^{\phi_f}   \ar[r] &  \bbF^\vee[2]  \ar[d] \\
 \bbL_f   \ar@{=}[r] & \bbL_f  \ar[r]  & \, 0. } 
\end{equation*}
As $\phi=0$, $W\times^\bfL_{T^*W}W=T^*[-1]W$ is the shifted cotangent bundle of $W$, and hence there is a zero section $W\to T^*[-1]W$ 
whose classical truncation is an isomorphism. This induces a map 
$$\textbf{j}: \bMap_{g,n}^{R_{\chi}=\omega_{\mathrm{log}}}([W/H_R])\to \bMap_{g,n}^{R_{\chi}=\omega_{\mathrm{log}}}([(W\times^\bfL_{T^*W}W)/H_R]) $$
whose composition with $\textbf{i}$ is the identity. Then we have a fiber sequence 
$$\textbf{j}^*\bbL_{\textbf{f}\circ \textbf{i}}\to \bbL_{\textbf{f}} \to \bbL_{\textbf{j}} $$
whose restriction to the classical truncation gives a commutative diagram 
\begin{equation*}    \xymatrix{ 
 \bbE_f \ar[r]^{\delta} \ar[d]_{\phi_f}  & \bbF  \ar[d]^{\varphi}   \\
 \bbL_f   \ar@{=}[r] & \bbL_f. } 
\end{equation*}
%Since the map in $\bbE_f$ \eqref{map in Ef} is given by the Hessian of $\phi=0$ which is zero. Therefore we have a map $\delta: \bbE_f\to \bbF$ given by 
%\begin{equation*}    \xymatrix{ \bbE_f \ar@{=}[r]  \ar[d]_{\delta} & \big( \dR\pi_*\left(\calW\boxtimes\oO(-S)\right) \ar[d] \ar[r]  & \dR\pi_*\left(\calW\boxtimes\oO(-S)\right)^\vee \big) \ar@{=}[d]  \\\bbF \ar@{=}[r] & \big(0 \ar[r] & \dR\pi_*\left(\calW\boxtimes\oO(-S)\right)^\vee \big). } \end{equation*}
That is, $\phi_f=\varphi\circ\delta$. It is easy to check $\bbF$ is a maximal isotropic subcomplex of $\bbE_f$.
Finally, the equality on virtual pullbacks follows from \cite[Prop.~1.18]{Park}. 
\end{proof}

\subsection{Dimensional reduction to symplectic quotients}\label{sec:dim_red}
In a forthcoming work, we will show that the quasimap  invariants defines in the present paper have dimensional reduction to quasimap invariants of symplectic quotients as defined by \cite{CiKM, Kim}.

Let $M$ be a symplectic vector space over $\bbC$ (also known as a quaternionic vector space) with a Hamiltonian action by an algebraic group $G$. 
Let $\fg$ be the Lie algebra of $G$ and $$\mu:M\to \fg^*$$ 
be the moment map. 
Define $W:=M\times \fg$ with the induced $G$-action. Let 
$$\phi:W\to \bbC, \quad (x,\xi)\in M\times \fg \mapsto \langle\mu(x),\xi\rangle, $$ 
where $\langle-,-\rangle$ denotes the pairing of dual vector spaces. 
%be the function sending $(x,\xi)\in M\times \fg$ to $\langle\mu(x),\xi\rangle$.
Note that 
$$d\phi=(d\phi_1,d\phi_2):M\times\fg\to M^*\times\fg^*, \,\,\, \mathrm{with}\,\,\,d\phi_2=\mu. $$
Hence $d\phi(x,\xi)=0$ implies $\mu(x)=0$. In particular, we have closed embeddings 
$$\Crit(\phi)\subseteq \mu^{-1}(0)\times\fg\subseteq Z(\phi). $$
And the critical locus is characterized as the zero locus of $d\phi_1|_{\mu^{-1}(0)\times\fg}$. 

The quotient stack $[\mu^{-1}(0)\times_G\fg]$ is a vector bundle over $[\mu^{-1}(0)/G]$ with fiber $\fg$, and 
$$[\Crit(\phi)/G]\subseteq \mu^{-1}(0)\times_G\fg$$ is a closed substack. Moreover, taking the stable locus of $\mu^{-1}(0)$, denoted by $\mu^{-1}(0)^{s}$, we obtain a vector bundle
$$\mu^{-1}(0)^s\times_G\fg\to \mu^{-1}(0)^{s}/G$$ over the symplectic reduction. 

Let $\bar F$ be a reductive group with a character $\bar\chi: \bar F\to \mathbb{C}^*$ acting on $M$ so that the symplectic form  $\Omega$ transforms under $\bar F$ as $\bar\chi$,~i.e.~$\Omega$ induces an $\bar F$-equivariant isomorphism $M\cong M^*\otimes\bar\chi$. 

Let $F=\bar F\times\bbC^*$, where $\bbC^*$ acts trivially on $M$
and 
\[\chi:F=\bar F\times\bbC^*\to \bbC^*, \quad \chi(f,t)=\bar\chi(f)\cdot t. \] 
%be given by $\bar\chi$ on $\bar F$ and identity on $\bbC^*$ \yl{what does this mean?}.
By definition, the moment map 
$$\mu: M\to \fg^*\otimes\bar\chi, $$ 
%\yl{I think the target is $\fg^*\otimes\bar\chi$}
is $F$-equivariant with $\bar F$ acting trivially on $\fg$ and $\bbC^*$ acting on $\fg$ by scaling.
In particular, the function 
$$\phi: W\to\bbC_\chi$$ 
is a $F$-equivariant map with $F$ acting on $\mathbb{C}$ by character $\chi$. 
Note that $F_0:=\Ker\chi$ preserves the function $\phi$, but does not preserve the symplectic structures on $M$ and its reduction.

The  \textit{quasimap invariants of symplectic quotients} as defined by \cite{CiKM, Kim} give a map 
\[\Phi^{\mathrm{symp}}_{g,n,\beta}: A_*(\overline{M}_{g,n})\otimes A_*^{F_0}(\mu^{-1}(0)^{s}/G)^{\otimes n}\to A_*^{F_0}(\pt)_{loc}.\]
We expect the following diagram 
\[\xymatrix{
A_*(\overline{M}_{g,n})\otimes A_*^{F_0}(\mu^{-1}(0)^{s}/G)^{\otimes n}\ar[r]^{\cong \quad} \ar[d]_{\Phi^{\mathrm{symp}}_{g,n,\beta}}&A_*(\overline{M}_{g,n})\otimes 
A_*^{F_0}(\mu^{-1}(0)^{s}\times_G\fg)^{\otimes n}\ar[d]^{\Phi_{g,n,\beta}}\\
A_*^{F_0}(\pt)_{loc}\ar@{=}[r]&A_*^{F_0}(\pt)_{loc}
}\]
to be commutative. Here the upper horizontal map is given by the smooth pullback of the projection of vector bundle $\mu^{-1}(0)^{s}\times_G\fg$ to the base $\mu^{-1}(0)^{s}/G$, and the right vertical map is given as \eqref{cohft map} (noticing that $\mu^{-1}(0)^{s}\times_G\fg\subseteq Z(\phi)^s/G$).

It is worth mentioning that there is an isomorphism \cite[Theorem~A1]{D}
$$H^{BM}_{F_0}(\mu^{-1}(0)^{s}\times_G\fg)\cong H_{F_0}(W/\!\!/G,\varphi_{\phi}), $$ 
where $H_{F_0}^{BM}$ denotes the (equivariant) Borel-Moore homology (Eqn.~\eqref{def of bm ho}), by abuse of notation $\phi$ denotes the descent function $W/\!\!/G\to \mathbb{C}$ and $\varphi_{\phi}$ denotes the vanishing cycle functor in Eqn.~\eqref{eqn:van}. 
We refer to the appendix for more discussions on the critical cohomology $H_{F_0}(W/\!\!/G,\varphi_{\phi})$ and its properties. 

By considering the $K$-theoretic version of what have been defined in \S \ref{sect on def of qm inv}, one will have dimensional reduction to  
the $K$-theoretic QM invariants of symplectic quotients which have been extensively studied (particularly on Nakajima quiver varieties)
by the Okounkov school (e.g.~\cite{AO, O, PSZ, KZ, KPSZ}).

\subsection{On quantum critical cohomology}\label{subsec:crit2} 
In this section, we discuss how our pullback map \eqref{eqn on virt pb} can be used to define a quantum critical cohomology 
in two cases. 

\subsubsection{Compact-type case and geometric phase}
We consider two special cases of our Setting~\ref{setting of glsm}. 

The first special case is referred to as the equivariantly compact-type case, which is motivated by the compact-type condition of \cite[Def.~4.1.4]{FJR2}. 
\begin{setting}\label{ass:compact_type}
Notations as in Setting \ref{setting of use} and we assume
$\phi|_{X^{F_0}}=0$. 
%and let $Z:=Z(\phi)$. 
\end{setting}
The assumption implies $X^{F_0}\subseteq Z(\phi)$.  
In particular, there is an element 
\begin{equation}\label{dist ele1}1:=\frac{[X^{F_0}]}{e^{F_0}(N_{X^{F_0}}X)}\in A_*^{F_0}(Z(\phi))_{loc}.\end{equation}
Recall the canonical map defined in Eqn.~\eqref{map can}, we have the following. 
\begin{prop}\label{prop on iso of crit and bm1}
In Setting \ref{ass:compact_type}, the canonical map induces an isomorphism:
\begin{equation}\label{iso of 3 cho}H^{BM}_{F_0}(Z(\phi))_{loc}\stackrel{\cong}{\to}H_{F_0}(X,\varphi_\phi)_{loc}, \end{equation}
and the natural inclusion map induces an isomorphism: 
\begin{equation}\label{iso of 3 cho2}H^{BM}_{F_0}(Z(\phi))_{loc}\stackrel{\cong}{\to}  H^{BM}_{F_0}(X)_{loc}. \end{equation}
\end{prop}
\begin{proof}
We first show \eqref{iso of 3 cho}. 
For any $F_0$-equivariant complex of sheaves $\calF$ on $X$, we denote the compactly supported cohomology 
$$H^*_{c,F_0}(X,\calF):=p_{X!}\calF. $$ 
Recall that Borel-Moore homology (resp.~critical cohomology) is the dual of the above cohomology when $\calF=\Q_X$ (resp.~$\calF=\varphi_\phi\Q_X$).
The Milnor triangle \eqref{eqn:Milnor} gives a long exact sequence 
\[\cdots\to H^i_{c,F_0}(X,\psi_\phi)\to H^i_{c,F_0}(X,\varphi_\phi)\to H^i_{c,F_0}(Z(\phi),\bbQ)\to\cdots.\]
If we can show $H^*_{c,F_0}(X,\psi_\phi)_{loc}=0$, the isomorphism \eqref{iso of 3 cho}
%$$H^{BM}_{F_0}(Z(\phi))_{loc}\stackrel{\cong}{\to} H_{F_0}(X,\varphi_\phi)_{loc}$$
would then follow from the same argument as  \cite[Lem.~4]{Brion}.

We use the commutation of hyperbolic restriction with nearby cycle. 
Let $j:X^*\to X$ be the open complement of $i:Z(\phi)\to X$, and let $\hat\pi:\widetilde{X^*}\to X^*$ be the 
$\bbZ$-cover obtained from pulling back the exponential map $\exp: \bbC\to \bbC^*$ along $\phi$. The composition $j\circ \hat\pi$ is denoted by $l$, 
and $\Psi_\phi:=i^*l_*l^*$, hence $\psi_\phi=i_*\Psi_\phi$ and we have 
%Setting~\ref{ass:compact_type} implies that $\rho$ factors as the composition of $h:X^{F_0}\to Z(\phi)$ with $i$. 
\begin{equation}\label{equ 0}H^*_{c,F_0}(X,\psi_\phi  ) \cong H^*_{c,F_0}\!\left(Z(\phi), \Psi_\phi \right). \end{equation}
First notice that for any space $X$ with zero-function, the non-vanishing locus is empty and hence the nearby cycle functor is the zero functor. In particular, 
\begin{equation}\label{equ 1}\Psi_{\phi|_{X^{F_0}}}=0. \end{equation}
We choose a one-parameter subgroup $\bbC^*\cong T\subseteq F_0$ with the same fixing locus 
$$X^T=X^{F_0}. $$ 
Let $A_X$ be the attracting set for $T$-action. 
We have the diagram
\[
X^T \xleftarrow{\;p_X\;} A_X \xrightarrow{\;j_X\;} X,
\]
and the hyperbolic restriction functor 
\[
p_{X*} j_X^! \colon D^b_{c,T}(X) \to D^b_{c,T}(X^T).
\]  
Similarly if we replace $X$ by $Z(\phi)$, we have $p_{Z}: A_{Z(\phi)}\to Z(\phi)^T$,  $j_Z: A_{Z(\phi)}\to Z(\phi)$ and similar 
hyperbolic restriction functor. 

By the localization theorem in equivariant cohomology (e.g.,~\cite[(5.3.3)]{Nak3}), we have a map
\begin{equation}\label{equ 2}
H^*_{c,F_0}(Z(\phi),\Psi_\phi) \to H^*_{c,F_0}\!\left(Z(\phi)^T, p_{Z*} j_Z^! \Psi_\phi\right),
\end{equation}
which becomes an isomorphism after taking  tensor with 
$\mathrm{Frac}(H^*_{F_0}(\mathrm{pt}))$.
As the nearby cycle functor commutes with the hyperbolic restriction (e.g., \cite[Prop.~5.4.1]{Nak3}), we have 
\begin{equation}\label{equ 3}
    p_{Z*} j_Z^! \, \Psi_{\phi} \cong \Psi_{\phi|_{X^T}} \, p_{X*} j_X^!,
\end{equation}
Combining \eqref{equ 0}--\eqref{equ 3}, we obtain $H^*_{c,F_0}(X,\psi_\phi)_{loc}=0$, and hence \eqref{iso of 3 cho}.

To show \eqref{iso of 3 cho2}, recall the following long exact sequence 
$$\cdots\to H^{BM}_{F_0}(Z(\phi))_{ } \to H^{BM}_{F_0}(X)_{ } \to H^{BM}_{F_0}(X\setminus Z(\phi))_{ } \to\cdots\,.  $$
Since $X^{F_0}\subseteq Z(\phi)$, so $X^{F_0}=Z(\phi)^{F_0}$. By localization, we know $H^{BM}_{F_0}(X\setminus Z(\phi))_{loc}=0$.  
\end{proof}
The second special case we consider is referred to as the geometric phase, which is motivated by the definition of geometric phase in \cite[Def.~1.4.5]{CFGKS}. 
\begin{setting}\label{ass:geo_phase}
Let  $M$ be a vector space with an $(H=G\times F)$-action, so that the $G$-action on the $\theta$-stable locus $M^s$ of $M$ is free.  Let $V\to M$ be an equivariant vector bundle  together with a section $s\in \Gamma(M,V)$ which is $G$-invariant and transforms under $F$ as character $\chi^{-1}: F\to \mathbb{C}^*$.  

Let $W$ be the total space of $V^\vee$ with the induced $H$-action and projection  $\pi:W\to M$. Define 
$$\phi: W\to \mathbb{C}, \quad \phi(v^{\vee})=\langle s\circ \pi(v^{\vee}),v^{\vee} \rangle. $$ 
Let $Z(s)\subseteq M$ be the zero locus of $s$. Assume furthermore that $Z(s)^s/G$ is smooth. 
Denote 
$$Z:=\pi^{-1}(Z(s)), \quad Z^s:=\pi^{-1}(Z(s)^s), \quad W^s:=\pi^{-1}(M^s).$$
\end{setting}
By an abuse of notations, we still write $X=W^s/G$ and $\phi: X\to \mathbb{C}$ for the descent function. 
There is an isomorphism:
\begin{equation}\label{equ on iso of dim re}H^{BM}_{F_0}(Z^s/G)\cong H_{F_0}(X,\varphi_{\phi}), \end{equation}
which goes in literature by the name dimensional reduction \cite[Thm.~A1]{D}.
Since $Z^s/G$ is the total space of a vector bundle over $Z(s)^s/G$, we obtain by the assumption in Setting \ref{ass:geo_phase} that $Z^s/G$ is smooth. 
Hence, there is a fundamental class 
\begin{equation}\label{dist ele2}1:=[Z^s/G]\in A_*^{F_0}(Z^s/G).\end{equation}
Recall the anti-diagonal class $\eta\in A_{*}^{F_0}\left(\frac{Z(\boxplus^{2}\phi)^s}{G^{2}}\right)$ defined in \eqref{diagon class}.
By an abuse of notations, we denote its image in the corresponding Borel-Moore homology (via cycle map) also by $\eta$. 

%Since $Z\subseteq Z(\phi)$, we know $$(Z^s/G)\times (Z^s/G)\subseteq \frac{Z(\phi)^s}{G}\times \frac{Z(\phi)^s}{G}\subseteq \frac{Z(\boxplus^{2}\phi)^s}{G^{2}}. $$ 
In below we show the class $\eta$ (after localization) sits in a smaller space. 
\begin{lemma}\label{lem:Casmir}
In Setting \ref{ass:compact_type}, we have 
$$\eta\in H^{BM}_{F_0}(Z(\phi))_{loc}\otimes H^{BM}_{F_0}(Z(\phi))_{loc}. $$ 
In  Setting \ref{ass:geo_phase}, we have 
$$\eta\in H^{BM}_{F_0}(Z^s/G)_{loc}\otimes H^{BM}_{F_0}( Z^s/G)_{loc}. $$
\end{lemma}
\begin{proof}
In Setting \ref{ass:compact_type}, $X^{F_0}=Z(\phi)^{F_0}$.
By the equivariant localization \cite[Thm.~6.2]{GKM}:
$$H^{BM}_{F_0}(Z(\phi))_{loc}\cong H^{BM}_{F_0}(X)_{loc}. $$
Therefore the claim obviously holds. 

In  Setting \ref{ass:geo_phase}, we have a commutative diagram 
\[
\xymatrix{
H^{BM}_{F_0}(X) \ar[r]^{\bar{\Delta}_* \quad\,\, } \ar[d]_{can}^{\cong} & H^{BM}_{F_0}\left(\frac{Z(\boxplus^{2}\phi)^s}{G^{2}}\right) \ar[d]^{can} & \\
H_{F_0}(X,\varphi_0) \ar[r]^{\bar{\Delta}_*\quad \,\,} & H^{BM}_{F_0}(X^2,\varphi_{\phi\boxplus\phi}) \ar[r]^{\mathrm{TS}}_{\cong} & H_{F_0}(X,\varphi_{\phi})^{\otimes 2},    
} \]
where $\mathrm{TS}$ is the Thom-Sabastiani isomorphism in \S \ref{sect on ts iso} and we refer to \S \ref{sect on functori} for the pushforward of critical cohomology. 

Note that $\eta=\bar{\Delta}_*[X]$. After localization and using the isomorphism \eqref{equ on iso of dim re}, we know 
\[\eta\in H^{BM}_{F_0}(Z^s/G)_{loc}\otimes  H^{BM}_{F_0}(Z^s/G)_{loc}. \qedhere \]  
\end{proof}
Using Lemma~\ref{lem:Casmir}, we can write  
\[\eta=\eta_i\boxtimes\eta^i\in H^{BM}_{F_0}(Z(\phi))_{loc}\otimes H^{BM}_{F_0}(Z(\phi))_{loc}, \]
\[\eta=\eta_i\boxtimes\eta^i\in H^{BM}_{F_0}(Z^s/G)_{loc}\otimes H^{BM}_{F_0}( Z^s/G)_{loc}, \]
in Setting \ref{ass:compact_type} and Setting \ref{ass:geo_phase} respectively. 
%In either Setting \ref{setting of glsm} or Setting \ref{ass:geo_phase}, 
Moreover, $QM_{g,n}^{R_{\chi}=\omega_{\mathrm{log}}}(\Crit(\phi),\beta)$ has a  \textit{virtual class}
\[[QM_{g,n}^{R_{\chi}=\omega_{\mathrm{log}}}(\Crit(\phi),\beta)]^{\mathrm{vir}}:=\sqrt{f^!}\circ  \nu^*\left([\overline{M}_{g,n}]\boxtimes 1^{\boxtimes n}\right)\in 
A_*^{F_0}\left(QM_{g,n}^{R_{\chi}=\omega_{\mathrm{log}}}(\Crit(\phi),\beta)\right), \]
where $1$ is given by Eqn.~\eqref{dist ele1} or \eqref{dist ele2}, $\nu^*$ is the flat pullback of \eqref{equ on nu} and $\sqrt{f^!}$ is defined in \eqref{eqn on virt pb}.
Recall the evaluation map \eqref{equ:ev_equi 2}: 
$$ev^n: QM_{g,n}^{R_{\chi}=\omega_{\mathrm{log}}}(\Crit(\phi),\beta)\to \Crit(\phi)^n, $$
which is proper at $F_0$-fixed locus, therefore we have a localized pushforward $ev^n_*$ (see \S \ref{sect on equi push}). 

We extend the definition of quasimap invariants (Definition \ref{def of QC}) to the following. 
%Notice that $H_{F_0}(X,\varphi)$ has a pairing $\langle-,-\rangle$.
\begin{definition}\label{def of QC2}
In either Setting \ref{ass:compact_type} or Setting \ref{ass:geo_phase}, let $\gamma_1,\ldots, \gamma_n \in H_{F_0}(X,\varphi_\phi)_{loc}$. 
\begin{enumerate}
\item The \textit{quasimap invariant} is 
$$\big\langle \gamma_1, \ldots, \gamma_n\big\rangle_{g,\beta,n}:=p_*\left(\left(\gamma_1\boxtimes\cdots\boxtimes\gamma_n\right)\cdot cl\left(ev^n_*[QM_{g,n}^{R_{\chi}=\omega_{\mathrm{log}}}(\Crit(\phi),\beta)]^{\mathrm{vir}}\right)\right),$$
Here $cl: A_*^{F_0}(-)_{loc}\to H^{BM}_{F_0}(-)_{loc}$ is the cycle map and $\cdot$ is the intersection product in $X^n$ with support on $\Crit(\phi)^n$, 
we use \eqref{iso of 3 cho}, \eqref{equ on iso of dim re} to identify $\gamma_i$'s as BM homology classes, 
and $p_*$ is the localized pushforward (Eqn.~\eqref{pf for nonproper bm}) from $\Crit(\phi)^n$ to a point. 
%with support on $(Z^s/G)^n$. 
    \item The \textit{quasimap class} is 
$$\big\langle \gamma_1, \ldots, \gamma_n,* \big\rangle_{g,\beta,n+1}:=\left\langle\gamma_1, \dots,\gamma_n,\eta_i\right\rangle_{g,\beta,n+1}\eta^i\in H_{F_0}(X,\varphi_\phi)_{loc}.$$
    \item The \textit{quantum product} of $\gamma_1, \gamma_2 \in H_{F_0}(X,\varphi_\phi)_{loc}$ is 
$$\gamma_1*\gamma_2:=\sum_{\beta\in N_+(\Crit(\phi))}\left\langle \gamma_1, \gamma_2,* \right\rangle_{0,\beta,3}\,z^\beta \in H_{F_0}(X,\varphi_\phi)_{loc}[\![z]\!]. $$
\end{enumerate}
\end{definition}
\begin{remark}\label{rmk:cycle_map}
Quasimap invariants defined in Definition \ref{def of QC2} are consistent with those in Definition \ref{def of QC} via the cycle map.  
They also satisfy the gluing formula and WDVV type equation as Theorems \ref{thm:CohFT_glue} and \ref{thm:WDVV_invariants_form}. 
%In fact, we have the following commutative diagram 
%\[\xymatrix{A_*^{F_0}(-)^{\otimes n}_{loc} \ar[r]^{\psi \quad } & A_*^{F_0}(QM_{g,n})_{loc} \ar[d]^{cl} \ar[r]^{ev_*^n \quad } & A_*^{F_0}(\Crit(\phi))_{loc} \ar[r]^{\,\,\, p_{*}} \ar[d]^{cl}^{} & A_*^{F_0}(\pt)_{loc} \ar[d]^{cl}  \\& H^{BM}_{F_0}(QM_{g,n})_{loc} \ar[r]^{ev_*^n \quad} & H^{BM}_{F_0}(\Crit(\phi))_{loc} \ar[r]^{\,\, \, p_{*}} & H^{BM}_{F_0}(\pt)_{loc},  } \]
%where $(-)$ is $Z(\phi)$ (resp.~$Z^s/G$) in Setting \ref{ass:compact_type} (resp.~in Setting \ref{ass:geo_phase}) and $\psi:=\sqrt{f^!}\circ  \nu^*\left([\overline{M}_{g,n}]\boxtimes -\right)$
\end{remark}
\begin{remark}\label{rem:const_degree}
All invariants defined above depend on the choice of $R$-charge, which has been used as an input in the definition of twisted quasimaps. Similar to the case of quiver varieties \cite[\S 4.3.12]{O}, we expect ``constant quasimaps" to have not-necessarily zero degree, which depends on the $R$-charge. Therefore, we do not expect the $\beta=0$ component of Definition \ref{def of QC}\,(3) to recover the classical product.  
\end{remark}
%\begin{remark}Due to the possible non-properness of the moduli stack ($F_0$-fixed locus is proper though, see Theorem \ref{thm on properness}), the quantum productis only defined on the \textit{localized} critical cohomology. \end{remark}
\begin{remark}
The above quantum product involves only three pointed QM invariants and is the analogy of ``small quantum product" in the theory of quantum cohomology. 
One can also define the ``big quantum product"  using genus zero invariants with more than three points. The WDVV type equations proved below 
will enable us to define Dubrovin type quantum connections exactly as before (see~e.g.~\cite[\S 4]{KM}, \cite[\S 9]{RT}).  
\end{remark}

\subsubsection{WDVV for the quantum product}
Next we show the \textit{associativity} of the quantum product as defined in Definition \ref{def of QC2}. 
%We follow a similar approach in the well-known Gromov-Witten theory (e.g.~\cite[\S 6.2, \S 6.3]{AGV}), namely, proving a more general equation known as the Witten-Dijkgraaf-Verlinde-Verlinde (WDVV)  equation.

Fix a collection of cohomology classes $\delta_i\in H_{F_0}(X,\varphi_{\phi})$ with $i=1,\dots,n$ and $\gamma_j\in H_{F_0}(X,\varphi_{\phi})$ with $j=1,2,3$. 
For any partition $A\sqcup B=\{1,2,\ldots,n\}$, we denote
$$\delta_A=\delta_{i_1}\boxtimes \cdots \boxtimes \delta_{i_m} \in H_{F_0}\left(X^A,\varphi_{\boxplus^A \phi}\right), $$ 
where $A=\{i_1,\cdots, i_m\}$ subjects to the ordering condition $i_1<\cdots<i_m$, and similarly denote $\delta_B$. 
We introduce signs $(-1)^{\epsilon_1(A)}$, $(-1)^{\epsilon_2(A)}$ by
$$(\gamma_1\wedge\gamma_2\wedge \gamma_3)\wedge (\delta_1\wedge \cdots \wedge \delta_n)=(-1)^{\epsilon_1(A)}\,(\gamma_1\wedge\gamma_2\wedge \delta_A)\wedge (\gamma_3\wedge \delta_B),   $$
$$(\gamma_1\wedge\gamma_2\wedge \gamma_3)\wedge (\delta_1\wedge \cdots \wedge \delta_n)=(-1)^{\epsilon_2(A)}\,(\gamma_1\wedge\gamma_3\wedge \delta_A)\wedge (\gamma_2\wedge \delta_B).   $$
\begin{theorem}\label{thm on wdvv}
Notations as above, for any $\beta\in N_+(\Crit(\phi))$, we have 
\begin{align*}
&\quad \, \sum_{\beta_1+\beta_2=\beta}\sum_{A\sqcup B=\{1,2,\ldots,n\}}(-1)^{\epsilon_1(A)}\big\langle \big\langle \gamma_1, \gamma_2,\delta_A,* \big\rangle_{0,\beta_1},\gamma_3,\delta_B,* \big\rangle_{0,\beta_2} \\
&=\sum_{\beta_1+\beta_2=\beta}\sum_{A\sqcup B=\{1,2,\ldots,n\}}(-1)^{\epsilon_2(A)}\big\langle \big\langle \gamma_1, \gamma_3,\delta_A,* \big\rangle_{0,\beta_1},\gamma_2,\delta_B,* \big\rangle_{0,\beta_2}.
\end{align*}
\end{theorem}
As a corollary, by setting $A=B=\emptyset$, we get the associativity of the quantum product. 
\begin{corollary}\label{cor on ass of prod}
The quantum product in Definition \ref{def of QC2} is associative, i.e. for any $\gamma_1,\gamma_2, \gamma_3$, 
$$(\gamma_1*\gamma_2)*\gamma_3=\gamma_1*(\gamma_2*\gamma_3). $$ 
\end{corollary}
\begin{proof}[Proof of Theorem~\ref{thm on wdvv}]
Notice that by definition 
\[\big\langle \big\langle \gamma_1, \gamma_2,\delta_A,* \big\rangle_{0,\beta_1},\gamma_3,\delta_B,* \big\rangle_{0,\beta_2}= \big\langle \gamma_1, \gamma_2,\delta_A,\eta_i\big\rangle_{0,\beta_1}\big\langle \eta^i,\gamma_3,\delta_B,\eta^j \big\rangle_{0,\beta_2}\eta_j.\]
And for any $\delta$ we have 
\begin{align*}&\quad \,\big\langle \gamma_1, \gamma_2,\delta_A,\eta_i\big\rangle_{0,\beta_1}\big\langle \eta^i,\gamma_3,\delta_B,\delta \big\rangle_{0,\beta_2} \\
&=\Phi_{0,|A|+3,\beta_1}\otimes\Phi_{0,|B|+3,\beta_2}([\overline{M}_{0,|A|+3}\times\overline{M}_{0,|B|+3}]\boxtimes\delta_A\boxtimes\gamma_1\boxtimes\gamma_2\boxtimes\eta\boxtimes \delta_B\boxtimes\gamma_3\boxtimes\delta ) .\end{align*}
Taking $|A|=|B|=\emptyset$, the assertion now follows from an analogy of Theorem~\ref{thm:WDVV_invariants_form} as explained in Remark \ref{rmk:cycle_map}.
\end{proof}

\subsection{Towards quantum cohomology for $(-1)$-shifted symplectic derived stacks}\label{sect on to qh of -1}
In future investigations, we expect to study a more general theory of quantum critical cohomology.

Let $\calX$ be an oriented $(-1)$-shifted symplectic derived Artin stack over $\bbC$ (e.g.~\cite[Def.~3.6]{BBBJ}). There is a perverse sheaf $\calP_\calX$ on $\calX$ \cite[Thm.~1.3]{BBBJ} (see also \cite{KL2} for the moduli scheme case). 
When $\calX$ has a torus $F$-action so that the shifted symplectic form transforms under a character $\chi: F\to \bbC^*$ (see Definition \ref{def of trans}), then $\calP_\calX$ is equivariant under $F_0:=\Ker(\chi)$-action. 
Its hypercohomology 
$$\calH:=H_{c,F_0}(\calX,\calP_\calX)^\vee$$ 
is a generalization of the critical cohomology $H_{F_0}(X,\varphi)$ in Appendix~\ref{sect on bm ho}. 

Let $R:\bbC^*_R\to F$ be an $R$-charge, the 
stack $Map_{g,n}^{R_{\chi}=\omega_{\mathrm{log}}}(\calX/\bbC^*_R)$ is well-defined in the same way as in \S \ref{subsec:recol_qas}. 
With an appropriate stability condition, there is a substack of ``quasimaps": 
$$QM_{g,n}^{R_{\chi}=\omega_{\mathrm{log}}}(\calX/\bbC^*_R)\subseteq Map_{g,n}^{R_{\chi}=\omega_{\mathrm{log}}}(\calX/\bbC^*_R), $$
which is expected to yield a map
\[\Phi_{g,n,\beta}^{\mathrm{top}}:H^{BM}(\overline{M}_{g,n})\otimes \calH^{\otimes n}\to H^{BM}_{F_0}(\pt)_{loc} \] where $\beta\in \Hom(\Pic(\calX),\bbZ)$,
 and provisional also an  algebraic version
\[\Phi_{g,n,\beta}^{\mathrm{alg}}:A_*(\overline{M}_{g,n})\otimes \mathcal{C}^{\otimes n}\to A_*^{F_0}(\pt)_{loc} \] along the lines of the preset paper,
where $\mathcal{C}$ is the Chow group of certain stack associated with $\calX$.

%such that $\Phi_{g,n,\beta}^{\mathrm{alg}}$ and $\Phi_{g,n,\beta}^{\mathrm{top}}$ are related. 
In the case when $\calX$ comes from the setting of gauged linear sigma models, i.e., it is a global derived critical locus, one can take 
$\mathcal{C}$ to be the equivariant Chow group of the critical locus (or zero locus) as in diagram \eqref{intro comm diag}, and $\Phi_{g,n,\beta}^{\mathrm{alg}}$ and $\Phi_{g,n,\beta}^{\mathrm{top}}$ are related by the commutative diagram in \textit{loc}.~\textit{cit}. 

In the case when $\calX=\bbT^*[-1]\calM$ is the $(-1)$-shifted cotangent bundle of a quasi-smooth derived Artin stack $\calM$, 
there is an analogue of the dimensional reduction isomorphism \cite{kinjo}:
\[H^{BM}_{F_0}(\calM)\xrightarrow{\cong} \calH, \]
and we may take $\mathcal{C}=A_*^{F_0}(\calM)$. 
With some care, the method developed in this paper is expected to define both the provisional maps $\Phi_{g,n,\beta}^{\mathrm{alg}}$ and $\Phi_{g,n,\beta}^{\mathrm{top}}$   
as in Definitions \ref{defi of Phi map}, \ref{def of QC2}, which satisfies certain compatibility conditions. 
The details will appear in a forthcoming work.
%similar to Remark~\ref{rmk:cycle_map}.  

%\subsection{Borel-Moore homology}\label{sect on bm coho}
%\subsubsection{Borel-Moore homology}

\section{Variants of quasimaps and applications}
In the previous sections, we defined virtual counts of quasimaps from arbitrary prestable curves to the critical locus. 
Following  works of the Okounkov school \cite[\S 6]{O}, \cite[\S 2.2, \S 2.5]{PSZ}, \cite{KZ, KPSZ} which are based on \cite[\S 7.2]{CiKM}, \cite[\S 7.2]{CiK1}, 
one can consider a variant of the above quasimaps by labelling a distinguished component of the genus 0 curves and putting relative marked points on them. 
In this section, we use such a variant to define analogues quasimap counts. Our discussions are kept sketchy as most constructions are similar as before. 

\subsection{Quasimaps with parametrized components and relative points}\label{sect on qm with rel pts}

Notations as in Setting \ref{setting of glsm}, we concentrate on the genus 0 case and label the \textit{distinguished component} by $$D\cong \mathbb{P}^1, $$ 
with \textit{relative points} on it, which  
are distinct smooth points $p_1,\ldots,p_n\in D$. 
We fix a principal $\mathbb{C}^*$-bundle $P_0$ on $D$ and an $R$-\textit{charge} $R: \mathbb{C}^*\to F$ (Definition \ref{defi of R-charge}) with a fixed isomorphism
\begin{equation}\label{equ on p0}P_0\times_{\mathbb{C}^*}R_{\chi}=\omega_{D,\log}, \quad \mathrm{where}\,\,\, R_{\chi}:=\chi\circ R. \end{equation}
The induced $F$-bundle $P_{0,F}$ is defined by  
\begin{equation}\label{equ on p0F}P_{0,F}=P_0\times_{\mathbb{C}^*}R. \end{equation}
The isomorphism \eqref{equ on p0} then induces an isomorphism 
\begin{equation}\label{equ on p0F2}\varkappa: P_{0,F}\times_F\chi=\omega_{D,\log}.\end{equation}
Note that when $R_\chi$ is a non-trivial map, $P_0$ is \textit{determined} by the $R$-charge as any $\mathbb{C}^*$-bundle on a rational curve $D$ is determined by its degree. 
%and consider QM with one parametrized component with relative points. 
\begin{definition}\label{def of para qm 2}
A \textit{stable genus 0}, $D$-\textit{parametrized $R$-twisted quasimap} to $\Crit(\phi)/\!\!/ G$ \textit{relative to} $p_1,\ldots, p_n$ is given by the data
$$\left(C,p_1',\ldots,p_n',\pi,P,u \right), $$
where 
\begin{itemize} 
\item $(C,p_1',\ldots,p_n')$ is a prestable genus $0$, $n$-pointed curve with a regular map $\pi: C\to D$, 
\item $P_G$ is a principal $G$-bundle on $C$. 
%\item $P$ is a principal $H_R=G\times \mathbb{C}^*$-bundle on $C$ with an isomorphism $\varkappa:P/G\times_{ \bbC^*}R_\chi\cong\omega_{\mathrm{log},C}$, 
\item $u$ is a section of the vector bundle 
$$(P_G\times_C\pi^*P_{0,F})\times_{G\times F} W\to C, $$
\end{itemize}
whose image lies in $(P_G\times_C\pi^*P_{0,F})\times_{G\times F} \Crit(\phi)$,
subject to the conditions: 
\begin{enumerate}
\item $\pi(p_i')=p_i$ for all $i$.
\item There is a distinguished component $C_0$ of $C$ such that $\pi$ restricts to an isomorphism 
$\pi|_{C_0}: C_0\cong D$ and $\pi(C\setminus C_0)$ is zero dimensional (possibly empty). 
\item  There is a finite (possibly empty) set $B\subset C$ of points such that $u(C\setminus B)$ is contained in the stable locus 
$(P_G\times_C\pi^*P_{0,F})\times_{G\times F} \Crit(\phi)^s$.
\item  The set $B$ is disjoint from all nodes and markings on $C$. 
\item  $\omega_{\widetilde{C}}\left(\sum p_i+\sum q_j\right)\otimes L_{\theta}^{\epsilon}$ is ample for every rational number $\epsilon>0$, 
where $L_{\theta}:=P_G\times_{G} \mathbb{C}_{\theta}$, $\widetilde{C}$ is the closure of $C\setminus C_0$, $p_i$ are markings 
on $\widetilde{C}$ and $q_j$ are nodes of $\widetilde{C}\cap C_0$.  
%\item the class of the map is $\beta$. 
\end{enumerate}
The \textit{class} $\beta$ of such a quasimap is given by the degree of the principal $G$-bundle $P_G$. 
\end{definition}
\begin{remark}
By stability, there should be a marked point in the last component of every rational tail attached to the distinguished component $C_0$ of $C$. 
As all points in the same rational tail are contracted to a point in $D$, by the condition  $\pi(p_i')=p_i$, we know there can not be other marked points in 
the same rational tail (see \cite[Figure 1]{PSZ} for an example of its shape). 

Therefore, all components of $C$ (other than $C_0$) have \textit{exactly two special points} (marked points or nodes), and 
$\pi^*P_{0,F}\times_{F}\chi\cong \omega_{C,\mathrm{log}}$ is automatically satisfied. Pullback of the isomorphism \eqref{equ on p0F2} provides a \textit{preferred choice} of such an isomorphism. 
% use pp. 118 of \cite{ACGH}
%\yl{Do we need to take care of the choice of this iso?}
\end{remark}
\begin{remark}\label{rmk on n=0}
When $n=0$, by stability, $C=C_0\cong D=\mathbb{P}^1$ in above. 
Definition \ref{def of para qm 2} dramatically simplifies.  
In \S \ref{sect vertex func on hilb}, \S \ref{sect on more vertx}, we will study in detail the so-called \textit{vertex function} (also known as \textit{hemispherical partition function}) defined by such quasimaps with $\infty\notin B$.
\end{remark}
We denote 
\begin{equation}\label{qm par cpn}QM=QM^{R_\chi=\omega_{\log}}_{\mathrm{rel},p_1,\ldots,p_n}(\Crit(\phi)/\!\!/G,\beta,D)  \end{equation}
%or simply $QM^{\overrightarrow{R},\beta}_{\mathrm{relative},p_1,\ldots,p_n}(D)$ 
to be the \textit{moduli stack} of stable genus $0$, $D$-parametrized $R$-twisted quasimaps to $\Crit(\phi)/\!\!/G$ with relative points $p_1,\ldots,p_n\in D$ in class $\beta$ 
as in Definition \ref{def of para qm 2}. It is a closed substack of the similar moduli stack of quasimaps to $W/\!\!/G$ considered in 
\cite[Def.~3,~Thm.~8]{PSZ}. In particular, it is DM of finite type. Similar to Theorem~\ref{thm on properness}, if the $F_0$-fixed locus in the affine quotient $(\Crit(\phi)/_{\mathrm{aff}}G)^{F_0}$ is finite, then the $F_0$-fixed locus
 $(QM)^{F_0}$ is proper. Note that properness in this setting holds \textit{without} assuming $\Ker(R_\chi)=1$ or using balanced twisted maps, since the data of principal $F$-bundle is fixed, and hence so is the $r$-Spin structures occurring in the theory of balanced twisted maps.  
 
Forgetting maps and principal bundles gives a morphism
$$QM^{R_\chi=\omega_{\log}}_{\mathrm{rel},p_1,\ldots,p_n}(\Crit(\phi)/\!\!/G,\beta,D)\to U_{p_1,\ldots,p_n} $$
to the stack $U_{p_1,\ldots,p_n} $ of underlying $n$-pointed trees of rational curves with one parametrized component $C_0\cong D$ and relative points $p_1,\ldots,p_n\in D$. 

In fact, denote $\widetilde{D[n]}$ to be the Fulton-MacPherson stack of (not necessarily stable) $n$-pointed trees of rational curves with one parametrized component $C_0\cong D$,
%\yl{The underlying curve has no stability requirement, so the last component could have no marked point?} 
which is a smooth Artin stack locally of finite type over $\mathbb{C}$ \cite[\S 7.2]{CiK1}. 
Let $U\subset \widetilde{D[n]}$ be the open substack where every component of the rational curve (other than $D$) has at least two special points. 
Define $U_{p_1,\ldots,p_n}$ by the following Cartesian diagram  
\begin{equation*}\begin{xymatrix}{
U_{p_1,\ldots,p_n} \ar[d]\ar[r]& U \ar[d]^{\pi}\\
\{(p_1,p_2,\ldots,p_n)\} \ar@{^{(}->}[r]^{ }  &D^n, }
\end{xymatrix}\end{equation*}
%\begin{equation*}\begin{xymatrix}{ \widetilde{D[n]}_{p_1,\ldots,p_n} \ar[d]\ar[r]&  \widetilde{D[n]} \ar[d]^{\pi}\\ \{(p_1,p_2,\ldots,p_n)\} \ar@{^{(}->}[r]^{ }  &D^n. }
%\end{xymatrix}\end{equation*}
where $\pi$ sends the $n$-pointed trees of rational curves to the image of marked points under the contraction to $D$. Away from the big diagonal of $D^n$ (i.e. the locus of $n$-distinct points on $D$), the map $\pi$ is a smooth morphism. 
Therefore we know $U_{p_1,\ldots,p_n} $ is also smooth.

\subsection{Corresponding quasimap invariants}\label{sect on para qm class}

Let $\mathcal{C}$ be the universal curve over $S:=U_{p_1,\ldots,p_n}$. Similar to \eqref{fiber diag on mgn}, we consider
the mapping stacks (relative to $S$):
%\yl{The notation below is not good, consult \eqref{def of mod stack of mapping}}
\[Map_{S}^{\chi=\omega_{\mathrm{log}}}(\mathcal{C},[\Crit(\phi)/H]\times S)\to 
\fBun_{H}^{\chi=\omega_{\mathrm{log}}}(\mathcal{C}/S) \to \fBun_{F}^{\chi=\omega_{\mathrm{log}}}(\mathcal{C}/S), \]
where $\fBun_{\bullet}^{\chi=\omega_{\mathrm{log}}}(\mathcal{C}/S):=Map_{S}^{\chi=\omega_{\mathrm{log}}}(\mathcal{C},[\pt/\bullet]\times S)$ for $\bullet=H=G\times F$ or $F$. 

Pulling back the pair $(P_{0,F},\varkappa)$ given in \eqref{equ on p0F}, \eqref{equ on p0F2} from $D$ to $C\in S$ via the map in Definition \ref{def of para qm 2}
defines a section $S\to\fBun_{F}^{\chi=\omega_{\mathrm{log}}}(\mathcal{C}/S)$. 
%\yl{This is not precise as one also needs to fix a $G$-bundle} 
The base-change along this section defines/gives the following pullback diagrams 
\begin{equation*}
\begin{xymatrix}{Map_{S}^{\chi=\omega_{\mathrm{log}}}(\mathcal{C},[\Crit(\phi)/H]\times S)_D \ar[d]_{ } \ar[r] \ar@{}[dr]|{\Box} & \fBun_{G}(\mathcal{C}/S) \ar[d]  \ar[r] \ar@{}[dr]|{\Box} & S \ar[d]\\
Map_{S}^{\chi=\omega_{\mathrm{log}}}(\mathcal{C},[\Crit(\phi)/H]\times S) \ar[r]^{ } &  \fBun_{H}^{\chi=\omega_{\mathrm{log}}}(\mathcal{C}/S) \ar[r] & \fBun_{F}^{\chi=\omega_{\mathrm{log}}}(\mathcal{C}/S),  }\end{xymatrix} \end{equation*} 
where $QM^{R_\chi=\omega_{\log}}_{\mathrm{rel},p_1,\ldots,p_n}(\Crit(\phi)/\!\!/G,\beta,D)$ is an open substack of 
$Map_{S}^{\chi=\omega_{\mathrm{log}}}(\mathcal{C},[\Crit(\phi)/H]\times S)_D$ determined by the stability conditions in Definition~\ref{def of para qm 2}. 
Similar to Theorem \ref{thm:sympl_marked}, the map 
\[f:Map_{S}^{\chi=\omega_{\mathrm{log}}}(\mathcal{C},[\Crit(\phi)/H]\times S)_D \to   \fBun_{G}(\mathcal{C}/S)\times_{[\pt/G]^n} [W/G]^n \] 
has a derived enhancement with a (relative) $(-2)$-\textit{shifted symplectic structure}. The construction of \S \ref{subsec:obst} (as in Definition~\ref{defi of qm vir class}) then defines a 
\textit{virtual pullback}
%\yl{discuss whether can remove $\Ker(R_\chi)=1$}
\begin{equation}\label{equ on sq pb}\sqrt{f^!}:A_*^{F_0}\left(\fBun_{G}(\mathcal{C}/S)\times_{[\pt/G]^n}\frac{Z(\boxplus^{n}\phi)^s}{G^{n}}\right)\to 
A_*^{F_0}\left(QM^{R_\chi=\omega_{\log}}_{\mathrm{rel},p_1,\ldots,p_n}(\Crit(\phi)/\!\!/G,\beta,D)\right).  \end{equation}
%where $Z\subseteq W^n$ is any $H$-invariant closed subscheme  such that \eqref{equ on Z} holds. 
As in Definition \ref{defi of Phi map}, we can define 
\begin{equation}
\Phi^{R,\beta,D}_{p_1,\ldots,p_n}: =p_{QM*}\circ\sqrt{f^!}\circ \nu^*([\fBun_{G}(\mathcal{C}/S)]\boxtimes -): A_{*}^{F_0}\left(\frac{Z(\boxplus^{n}\phi)^s}{G^{n}}\right) \to A_*^{F_0}(\pt)_{loc},  \nonumber \end{equation}
where 
$$\nu: \fBun_{G}(\mathcal{C}/S)\times_{[\pt/G]^n}\left(\frac{Z(\boxplus^{n}\phi)^s}{G^{n}}\right)\to \fBun_{G}(\mathcal{C}/S)\times \left(\frac{Z(\boxplus^{n}\phi)^s}{G^{n}}\right) $$
is a smooth map. 

More generally, one can put \textit{insertions} in above: let $\calP\to \calC$ be the universal $(G\times F)$-bundle on $\calC$, 
for $\tau\in K_{G\times F}(\pt)$, we form 
$$\calP\times_{G\times F}\tau\in K_F(\calC), $$
where the $F$-action is induced from the quasimap stack $QM$. 
One can restrict this class to the distinguished component $QM^{ }\times C_0=QM^{ }\times\bbP^1$ 
and also $QM^{ }\times Q \inj QM^{ }\times\bbP^1$ for a finite number of distinct points $Q$ in $\bbP^1$. 
\begin{definition}
Notations as above, we define
\begin{align}\label{def of PHI R}
\Phi^{R,\beta,D}_{p_1,\ldots,p_n}(\{\tau_i\}_{i=1}^{|Q|},Q)&:=p_{QM*}\circ\left(\prod_{x_i\in Q}e^{F_0}\left((\calP\times_{G\times F}\tau_i)|_{QM^{ }\times \{x_i\}}\right)\cap\right) \circ \sqrt{f^!}\circ \nu^*([\fBun_{G}(\mathcal{C}/S)\boxtimes -]) \\ \nonumber
&\quad \, :A_{*}^{F_0}\left(\frac{Z(\boxplus^{n}\phi)^s}{G^{n}}\right) \to A_*^{F_0}(\pt)_{loc}.
\end{align}
\end{definition}
%\yl{when $Q$ is $\mathbb{C}^*_q$-fixed and replace $e^{F_0}$ by $e^{T_0}$, the above keeps the same?}
%Its dual gives a map \begin{equation}I^{R,\beta}_{D,p_1,\ldots,p_n}: \left(H_{F_0}(X,\varphi_\phi)_{loc}^\vee\right)^{\otimes n}\to   H^*_{F_0}(\pt)_{loc}. \nonumber \end{equation}
%\yl{1. Smoothness of $ \widetilde{D[n]}$ when $D$ is reducible ? 2. As we need to fix an $R$-twist on each component of $D$, this generalizes the setting in previous sections, be careful whether obs theory etc works. Whether we have smooth base as $\fBun_{H_R}^{R_\chi}$ in before?}
%\gufang{The base might not be $ \widetilde{D[n]}$ but rather a "relative version" of it?}
The above construction can be generalized to the case when 
$$D=D_1\cup D_2\cup\cdots \cup D_d $$
is a chain of rational curves ($D_i\cong \mathbb{P}^1$) with relative points $p_1,\ldots,p_n\in D$ and $Q$ is a finite number of distinct smooth points in $D$. 
One fixes a principal $\mathbb{C}^*$-bundle $P_0$ and an $R$-charge $R_i: \mathbb{C}^*\to F$ (Definition \ref{defi of R-charge}) on each component $D_i$ 
such that \eqref{equ on p0} and \eqref{equ on p0F} hold on each $D_i$. Then one defines 
\begin{equation}
\Phi^{(R_1,\ldots,R_d),\beta,D}_{p_1,\ldots,p_n}(\{\tau_i\}_{i=1}^{|Q|},Q): A_{*}^{F_0}\left(\frac{Z(\boxplus^{n}\phi)^s}{G^{n}}\right) \to A_*^{F_0}(\pt)_{loc}\nonumber \end{equation}
exactly as Eqn.~\eqref{def of PHI R}.

If each $p_i$ is fixed by a torus $K$ action on $D$, and each $\tau_i$ has corresponding equivariance, the above map can be defined on $(F_0\times K)$-equivariant Chow groups.

%If each $p_i$ is $\bbC^*_q$-fixed (for the $\bbC^*_q$-action on $\mathbb{P}^1$), and each $\tau_i$ is a $(T_0:=F_0\times\bbC^*_q)$-equivariant class, the above map can be defined on $T_0$-equivariant Chow groups. 

\subsection{Degeneration and gluing formulae}\label{sect on dege for}
When $D\cong \mathbb{P}^1$ degenerates to a union $D_1\cup_{p} D_2$ of two smooth rational curves gluing at $p$ such that 
$Q\subset D $ is identified with $Q'\subset D_1\cup_{p} D_2$ (where $Q'\cap\{p\}=\emptyset$), and  
$[D]=\beta$, $[D_i]=\beta_i$ ($i=1,2$), 
%\yl{Need to say how $R$-charges and curve classes behave under degenerations. }
one has a \textit{degeneration formula}: 
\begin{equation}\label{deg form}\Phi^{R_1\cdot R_2,\beta,D}_{p_1,\ldots,p_n}(\{\tau_i\}_{i=1}^{|Q|},Q)=\Phi^{(R_1,R_2),\beta_1+\beta_2,D_1\cup_{p} D_2}_{p_1,\ldots,p_n}(\{\tau_i\}_{i=1}^{|Q|},Q'). \end{equation}
where $R_1\cdot R_2$ is defined using the multiplication in $F$. %\yl{how to understand why we take product $R_1\cdot R_2$?}
%\yl{degeneration formula gives constrains on invs}
%\yl{When the marked points are either $0$ or $\infty$, we want the degeneration of $\bbP^1$ to $\bbP^1\cup_p\bbP^1$ $\bbC^*_q$-equivariant, the tautological line bundle $\tau_0$ should be $\bbC^*$-equivariant and flat in the family?}
Using a diagrammatic notation as \cite[Eqn.~(23)]{PSZ}, it is represented as 
\begin{center}
	\begin{tikzpicture}
	\draw [ultra thick] (-2.2,0) -- (1.2,0);
	\node at (1.8,0) {$=$};
	\draw [ultra thick] (4+0.5,0) to [out=0,in=240] (4.7+0.5,0.4);
	\draw [ultra thick] (5+0.5,0) to [out=180,in=300] (4.3+0.5,0.4);
	\draw [ultra thick] (5+0.5,0) -- (7+0.5,0);
	\draw [ultra thick] (2+0.5,0) -- (4+0.5,0);
	\node at (-2+0.3,-0.3) {$p_1$};
	\node at (-0.9+0.3,-0.3) {$\dots$};
	\node at (0.2+0.3,-0.3) {$p_n$};
	\node at (2.4+0.5,-0.3) {$p_1$};
	\node at (3+0.5,-0.3) {$\dots$};
	\node at (3.6+0.5,-0.3) {$p_s$};
	\node at (5.2+0.5,-0.3) {$p_{s+1}$};
	\node at (6+0.5,-0.3) {$\dots$};
	\node at (6.6+0.5,-0.3) {$p_{n}$};
	\node at (4.55+0.5,0.45) {$p$};
	\end{tikzpicture}
\end{center}
Let $p_1,\ldots,p_s\in D_1$, $p_{s+1},\ldots,p_n\in D_2$. We break the rational curve $D_1\cup_{p} D_2$ into $D_1$ and $D_2$ with relative points 
$p_1,\ldots,p_s,p$ and $p_{s+1},\ldots,p_n,p$ respectively. 
%Let $\eta^{ij}T_i\boxtimes T_j$ be the ``diagonal class" given in \eqref{diago}.
We then have a \textit{gluing formula} as \eqref{glue formula new form}:
\begin{align}\label{glue formula conv form for parametrized}
&\quad \quad \,\, \Phi^{(R_1,R_2),\beta,D_1\cup_{p} D_2}_{p_1,\ldots,p_n}(\{\tau_i\}_{i=1}^{|Q|},Q)(-) \\ \nonumber 
&=\sum_{\beta_1+\beta_2=\beta}\left(\Phi^{R_1,\beta_1,D_1}_{p_1,\ldots,p_s,p}(\{\tau_i\}_{x_i\in Q\cap D_1},Q\cap D_1)\otimes \Phi^{R_2,\beta_2,D_2}_{p_{s+1},\ldots,p_n,p}(\{\tau_i\}_{x_i\in Q\cap D_2},Q\cap D_2) \right)(-\boxtimes\eta),
\end{align}
where the right-hand-side 
\begin{equation}\label{tens prd}\Phi^{R_1,\beta_1,D_1}_{p_1,\ldots,p_s,p}\otimes \Phi^{R_2,\beta_2,D_2}_{p_{s+1},\ldots,p_n,p}: A_{*}^{F_0}\left(\frac{Z(\boxplus^{n+2}\phi)^s}{G^{n+2}}\right) \to A_*^{F_0}(\pt)_{loc} \end{equation}
is defined similarly as \eqref{equ on prod phi}.
%$$I^{(R_1,R_2),\beta}_{D_1\cup_{p} D_2,p_1,\ldots,p_n}(\gamma_1,\ldots, \gamma_n)=\sum_{\beta_1+\beta_2=\beta}\eta^{ij}I^{R_1,\beta_1}_{D_1,p_1,\ldots,p_s,p}(\gamma_1,\ldots, \gamma_{s},T_i)\, I^{R_2,\beta_2}_{D_2,p_{s+1},\ldots,p_n,p}(\gamma_{s+1},\ldots, \gamma_{n},T_j).$$
The only difference between \eqref{glue formula conv form for parametrized} and \eqref{glue formula new form} is that here we parametrize component $D$ and relative points, so 
the nodal point $p$ can not be deformed.
Using a diagrammatic notation as \cite[Eqn.~(25)]{PSZ}, the formula can be represented as follows \\

\begin{center}
	\label{glform}
	\begin{tikzpicture}
	\draw [ultra thick] (4,0) to [out=0,in=240] (4.7,0.4);
	\draw [ultra thick] (5,0) to [out=180,in=300] (4.3,0.4);
	\draw [ultra thick] (5,0) -- (7,0);
	\draw [ultra thick] (2,0) -- (4,0);
	\node at (2.4,-0.3) {$p_1$};
	\node at (3,-0.3) {$\ldots$};
	\node at (3.6,-0.3) {$p_s$};
	\node at (5.2,-0.3) {$p_{s+1}$};
	\node at (6,-0.3) {$\dots$};
	\node at (6.6,-0.3) {$p_{n}$};
	\node at (4.55,0.45) {$p$};
	\node at (7.8,0) {$=$};
	\draw [ultra thick] (3+6.6-1,0) -- (4+7.6-1,0);
	\draw [ultra thick] (4+7.5-1,0.2) to [out=330,in=30] (4+7.5-1,-0.2);
	\node at (4+8.05-1,0.5) {$\boxtimes\,\eta$};
	\draw [ultra thick] (4+8.5-1,0) -- (5+9.5-0.8,0);
	\draw [ultra thick] (4+3.1+5.5-1,0.2) to [out=210,in=150] (4+3.1+5.5-1,-0.2);
	\node at (10-1,-0.3) {$p_1$}; 
	\node at (10.6-1,-0.3) {$\dots$};
	\node at (11.2-1,-0.3) {$p_s$}; 
	\node at (13.1-1,-0.3) {$p_{s+1}$}; 
	\node at (13.8-1,-0.3) {$\dots$};
	\node at (14.4-1,-0.3) {$p_n$};
	\end{tikzpicture}
\end{center}
%A detailed proof of \eqref{glue formula conv form for parametrized} and \eqref{deg form} should follow the same way as the proof of Theorem \ref{thm:CohFT_glue} and the cases when the target is $W/\!\!/G$ \cite[\S 6.5]{O} \cite[Eqns.~(23), (24)]{PSZ}, which will be addressed elsewhere. We treat these two formulae as a blackbox and use them in the next section for applications. 
The proof of  \eqref{glue formula conv form for parametrized} and \eqref{deg form} follows the same way as the proof 
of Theorem \ref{thm:CohFT_glue} and as well as in the symplectic case \cite[\S 6.5]{O} \cite[Eqns.~(23), (24)]{PSZ}. We leave the details to interested readers.

\subsection{Quasimaps invariants with parametrized components, relative and smooth points}\label{subsec:QDE}

\subsubsection{Generalities}
For distinct smooth points $p_1,\ldots,p_n,q_1,\ldots, q_m\in D$, one can consider the open substack 
\begin{equation}\label{qm par cpn with sm pts}QM^{R_\chi=\omega_{\log}}_{\begin{subarray}{c}\mathrm{rel},p_1,\ldots,p_n \\  \mathrm{sm},q_1,\ldots,q_m \end{subarray}}(\Crit(\phi)/\!\!/G,\beta,D)\subseteq QM^{R_\chi=\omega_{\log}}_{\mathrm{rel},p_1,\ldots,p_n}(\Crit(\phi)/\!\!/G,\beta,D), \end{equation}
consisting of quasimaps such that $\pi^{-1}(q_1),\ldots,\pi^{-1}(q_m)\in C_0$ are \textit{away from} the base locus $B$ (where $\pi$ is as in Definition \ref{def of para qm 2}).  
Then we have \textit{evaluation maps}\footnote{As \cite[pp.~80, \S 6.4.9]{O}, here $ev_{p_i}$ are evaluation maps at relative points $p_i'$ in Definition \ref{def of para qm 2}. 
As any principal $\mathbb{Z}_r$-bundle on a rational curve is trivial, the target of the evaluation map 
is the GIT quotient without the finite group automorphism as in Proposition \ref{prop:ev_equi}. Maps $ev_{q_j}$ are evaluations at $\pi^{-1}(q_j)\in C_0$.} 
%\yl{$q_i$ not contributing to $\omega_{\log}$, do we need canonical trivialization of $\omega_{\log}$ at $q_i$ as in Proposition \ref{prop:ev_equi}?}
$$ev_{p_i}, \, ev_{q_j}: QM^{R_\chi=\omega_{\log}}_{\begin{subarray}{c}\mathrm{rel},p_1,\ldots,p_n \\  \mathrm{sm},q_1,\ldots,q_m \end{subarray}}(\Crit(\phi)/\!\!/G,\beta,D)
\to \Crit(\phi)/\!\!/G.  $$
As in \eqref{def of PHI R}, we have in this setting:  
\begin{align}\label{def of PHI R sm}
\Phi^{R,\beta,D}_{\begin{subarray}{c}p_1,\ldots,p_n \\  q_1,\ldots,q_m \end{subarray}}(\{\tau_i\}_{i=1}^{|Q|},Q)&:=\prod_{j=1}^mev_{q_j*}\circ\left(\prod_{x_i\in Q}e^{F_0}\left((\calP\times_{G\times F}\tau_i)|_{QM^{ }\times \{x_i\}}\right)\cap\right) \circ \sqrt{f^!}\circ \nu^*([\fBun_{G}(\mathcal{C}/S)]\boxtimes -) \\ \nonumber 
&\quad \, :A_{*}^{F_0}\left(\frac{Z(\boxplus^{n}\phi)^s}{G^{n}}\right) \to A_*^{F_0}(\Crit(\phi)/\!\!/G)_{loc}^{\otimes m}.
\end{align}

\subsubsection{Some distinguished quasimap invariants}\label{subsec:qconn}
Let $D=\bbP^1$ and $\bbC^*_q$ be the 1-dimensional torus acting on $\bbP^1$, i.e. in the homogenous coordinate $[x,y]\in\bbP^1$, $q\in\bbC^*_q$ acts via 
$$q[x,y]=[q^{-1}x,qy].$$ 
In below, we use the following notations 
$$\hbar:=-c_1(q), \quad T:=F\times\bbC^*_q, \quad T_0:=F_0\times\bbC^*_q.  $$ 
Write $0=[0,1]$ and $\infty=[1,0]$. If $p_i$ and $q_j$ are all either $0$ or $\infty$, then $\bbC^*_q$ and $T$ act on 
$$QM:=QM^{R_\chi=\omega_{\log}}_{\begin{subarray}{c}\mathrm{rel},p_1,\ldots,p_n \\  \mathrm{sm},q_1,\ldots,q_m \end{subarray}}(\Crit(\phi)/\!\!/G,\beta,D), $$
and $\calP\times_{G\times F}\alpha$ also has $\bbC^*_q$-equivariant structure obtained via the $R$-charge $R:\bbC^*_q\to F$. 
In this case, the map \eqref{def of PHI R sm} can be defined equivariantly with respect to the $T_0$-action. 
%\yl{write some explanations for this sentence.}

%Let $q_1,\dots,q_m\in D$ be smooth points disjoint from $p_i$'s. We then have the open subset \[QM^{\beta}_{\mathrm{rel},p_1,\ldots,p_n, \mathrm{sm}q_1,\dots,q_m}(D)\subseteq QM^{\beta}_{\mathrm{relative},p_1,\ldots,p_n}(D)\] consisting of quasimaps whose base points are disjoint from $q_i$'s.

As in \cite[\S 2.4, \S 2.7]{PSZ}, we also introduce a few invariants needed in the below discussions.  
%Hence, the degeneration formula \eqref{deg form} should have the form \yl{Put the following two formulae as black box and move to before?}
%\begin{align} \label{eqn:degen} & \quad \, can\circ ev^2_*\,e^{T_0}(\tau_0)\cap [QM^{ }_{\mathrm{rel},0,\infty}(\Crit(\phi)/\!\!/G,\beta,\bbP^1)]^{\mathrm{vir}} \\ \nonumber &=can\circ ev^2_*\,e^{T_0}(\tau_0)\cap[QM^{ }_{\mathrm{rel},0,\infty}(\Crit(\phi)/\!\!/G,\beta,\bbP^1\cup_p\bbP^1)]^{\mathrm{vir}}\in H_{F_0}(X,\varphi_\phi)^{\otimes 2}.   \end{align}
%Correspondingly, the gluing formula \eqref{glue formula conv form for parametrized} should have the form 
%\begin{align}\label{eqn:gluing_rel} &\quad \quad can\circ ev^2_*\left(e^T(\tau_0)\cap[QM^{ }_{\mathrm{rel},0,\infty}(\Crit(\phi)/\!\!/G,\beta,\bbP^1\cup_p\bbP^1)]^{\mathrm{vir}}\right) \\ \nonumber &= \sum_{\beta_1+\beta_2=\beta}\langle-|-\rangle\,\circ\, can\circ ev^2_*\,e^T(\tau_0)\cap[QM^{ }_{\mathrm{rel},0,\infty}(\Crit(\phi)/\!\!/G,\beta_1,\bbP^1)]^{\mathrm{vir}}\\ \nonumber&\quad\quad\quad\quad  \otimes can\circ ev^2_*\,[QM^{ }_{\mathrm{rel},0,\infty}(\Crit(\phi)/\!\!/G,\beta_2,\bbP^1)]^{\mathrm{vir}}   \in H_{F_0}(X,\varphi_\phi)^{\otimes 2} \end{align}
\begin{definition}\label{defi of vertex fun}
(1) When $n=0$, $m=1$ and $q_1=\infty\in\bbP^1$ and $Q=\{0\}$ in \eqref{def of PHI R sm}, we define  
\begin{align*}V^{\tau}(z)&:=\sum_{\beta}z^\beta  
\Phi^{R,\beta,\bbP^1}_{\begin{subarray}{c}\emptyset \\  \infty \end{subarray}}(\{\tau\},\{0\}) \in A_{*}^{T_0}(\Crit(\phi)/\!\!/G)_{loc}[\![z]\!]. \end{align*}
This is called the \textit{vertex function} (or \textit{hemispherical partition function}) with descendent insertion $\tau$.

When $\tau=1$, we simply write 
$$V(z):=V^{\tau=1}(z), $$
which is an analogue of Givental's $I$-function \cite{Giv} (see also \cite[\S 7.2.1]{O}).

(2) When $n=1$, $m=0$ and $p_1=\infty\in\bbP^1$  and $Q=\{0\}$ in \eqref{def of PHI R sm},  
we define  
\[\hat{V}^{\tau}(z):=\sum_{\beta}z^\beta  \Phi^{R,\beta,\bbP^1}_{\begin{subarray}{c} \infty \\ \emptyset \end{subarray}}(\{\tau\},\{0\}) : A_{*}^{T_0}(Z(\phi)^s/G)\to A_{*}^{T_0}(\pt)_{loc}[\![z]\!].\]
This is called the \textit{topologically twisted vertex function}\footnote{One call it
``topologically twisted" as domain curves of relative quasimaps can develop rational tails attaching to the distinguished component $C_0=\mathbb{P}^1$ and we integrate over all such possible configurations.} with descendent insertion $\tau$. 

(3) We denote the limit 
  $$\hat\tau(z):=\lim_{\hbar=0}\hat{V}^{\tau}(z): A_{*}^{T_0}(Z(\phi)^s/G)\to A_{*}^{T_0}(\pt)_{loc}[\![z]\!]. $$
  If $\tau=1$, we simply write
\[\hat{1}(z):=\lim_{\hbar=0}\hat{V}^{\tau=1}(z): A_{*}^{T_0}(Z(\phi)^s/G)\to A_{*}^{T_0}(\pt)_{loc}[\![z]\!]. \] 
Notice that since $m=0$ (i.e.~no smooth points labelled), the pushforward can be defined using only $F_0$-localization (as the $F_0$-fixed locus is already proper), hence $\hat{V}^{\tau}$ is a polynomial in $\hbar$ and the above limit is well-defined.

(4) When $n=m=1$, let $p_1=0\in\bbP^1$ and $q_1=\infty\in\bbP^1$   and $Q=\emptyset$ in \eqref{def of PHI R sm}, we denote 
\[J(z):=\sum_{\beta}z^\beta \Phi^{R,\beta,\bbP^1}_{\begin{subarray}{c}0 \\  \infty \end{subarray}} : 
A_{*}^{T_0}(Z(\phi)^s/G)\to A_{*}^{T_0}(\Crit(\phi)/\!\!/G)_{loc}[\![z]\!]. \]
We sometimes view $J(z)$ as an operator 
$$A_{*}^{T_0}(\Crit(\phi)/\!\!/G)_{loc}[\![z]\!]\to A_{*}^{T_0}(\Crit(\phi)/\!\!/G)_{loc}[\![z]\!]$$
by using the embedding $\Crit(\phi)/\!\!/G\to Z(\phi)^s/G$. 
\end{definition}

\begin{definition}\label{defi of qcm}
Let $n=2$, $m=0$ with $p_1=0$ and $p_2=\infty$ in \eqref{qm par cpn with sm pts}. We take the $R$-charge to be trivial, which is possible as $\omega_{\mathrm{log},D}\cong\calO$ in this case.
For each $\alpha\in \bbX(G)$, we define 
\begin{align*}
\bfM^\alpha(z)&:=\sum_{\beta}z^\beta \Phi^{R=1,\beta,\bbP^1}_{\begin{subarray}{c}0,\infty\\  \emptyset \end{subarray}}(\{\alpha\},\{0\}):
A_*^{T_0}\left(\frac{Z(\boxplus^{2}\phi)^s}{G^{2}}\right)\to A_*^{T_0}(\pt)_{loc}[\![z]\!]. 
\end{align*}
%where we use the equivariance in $T_0=F_0\times\bbC^*_q$.  
%\begin{align*}\bfM^\alpha(z)&:=\sum_{\beta}z^\beta can\circ ev_*^2\left(c_1^{\bbC^*_q}(\pi_{QM*}(\calL_\alpha\otimes \pi^*(\calO_\infty)))\cap \left[QM_{\mathrm{rel},0,\infty}(\Crit(\phi)/\!\!/G,\beta,\bbP^1)\right]^{\mathrm{vir}}\right)  \\ & \quad \,\, \in H_{F_0}\left(X^{2},\varphi_{\boxplus^2\phi}\right)[\![z]\!],  \end{align*}
%where $c_1(L_\alpha)\cap(-)$ is the Euler class operator \S \ref{sect on eu cla op} on the factor labelled by $0$($L_{\alpha}$ is the line bundle on $X=W/\!\!/_{}G$ associated with $\alpha\in \bbX(G)$).  %$\pi: \calC\to \mathbb{P}^1$ is the universal projection map of $\pi$ in Definition \ref{def of para qm 2}, 
%where the virtual class is $T_0$-equivariant andwe omit $R_\chi=\omega_{\log}$ in the above notation of QM stack. 
\end{definition}
%As in \eqref{equ on prod phi}, we can form  $$\bfM^\alpha(z)\otimes J(z): A_*^{T_0}\left(\frac{Z(\boxtimes^{3}\phi)^s}{G^{3}}\right)\to A_{*}^{T_0}(\Crit(\phi)/\!\!/G)_{loc}[\![z]\!]$$
Fix a basis $\alpha_1,\ldots,\alpha_N$ of the free part of the character group $\bbX(G)$. 
For a curve class $\beta\in \Hom_{\mathbb{Z}}(\bbX(G), \bbZ)$, we write $\beta_i=\beta(\alpha_i)\in \bbZ$. Then we can write  
\begin{equation}\label{equ on zi}z^\beta=z_1^{\beta_1}\cdots z_N^{\beta_N}. \end{equation}
Since the definition of the $J$-operator involves a smooth marked point, $T_0=(F_0\times\bbC^*_q)$-equivariant localization is needed 
to define invariants ($F_0$-fixed locus is not necessarily proper). The degeneration and gluing formulae 
hold for this $T_0$-equivariant case to prove the following result, where our argument follows closely \cite[Thm.~8.1.16]{O}.
\begin{theorem}\label{thm:fund_sol}
%Assume \eqref{deg form} and \eqref{glue formula conv form for parametrized} hold for $T_0$-equivariant setting. Then f
For any $1\leqslant i\leqslant N$, we have 
\[\frac{\partial J(z)}{\partial z_i}(-)=-\frac{1}{2\hbar z_i}\bfM^{\alpha_i}(z)\otimes J(z)(-\boxtimes\eta)+\frac{1}{2\hbar z_i}\left( c_1(L_{\alpha_i})\right)\circ J(z)(-),\]
%\yl{does $-\boxtimes\eta$ sit in $A_*^{T_0}\left(\frac{Z(\boxplus^{2}\phi)^s}{G^{2}}\right)\otimes A_*^{T_0}\left(\frac{Z(\phi)^s}{G^{}}\right)$?}
where the tensor is defined as \eqref{tens prd} and the class $\eta$ is applied to the factor of $\bfM^{\alpha_i}(z)$ labelled by $\infty$ and 
the factor of $J(z)$ labelled by $0$, $L_{\alpha}$ is the line bundle on $X=W/\!\!/_{}G$ associated with $\alpha\in \bbX(G)$. 
%\yl{The $z_i$ in the denominator looks strange? In Konstevich-Manin we did not have that. Another difference is their vector bundle is tangent bdl of coho, but ours is crit coho times Kahler space determined by $\bbX(G)$, so fiber has not relation with base?}\gufang{regarding your sequence question: this is the difference between small quantum cohomolgy vs big quantum cohomology. Regarding the denominator, it can be absorbed into the connection matrix. I like the way we write it since it acknowledges the fact that it has an order 1 pole.} \yl{Can we still view this a flat section?}\gufang{Correted. It's a gauge transform.}
\end{theorem}
\begin{proof}
Let $\calC$ be the universal curve with universal map $[u]:\calC\to [W/H]$. 
For any $\alpha\in\bbX(G)$, extend it trivially to a character $\alpha\in\bbX(H)$, write $L_\alpha$ for the corresponding line bundle on $[W/H]$ and
$\calL_\alpha:=[u]^*L_\alpha$ for the pullback. 
Let $i: QM\to \calC$ be the section of $\pi_{QM}: \calC\to QM$ corresponding to the marked point $0\in C_0$ which is mapped to $0 \in \mathbb{P}^1$ under contraction map $\pi$ in Definition \ref{def of para qm 2}. Then 
\begin{align*}
\left(\calP\times_{G\times F}\alpha\right)|_{QM\times \{0\}}
&=\pi_{QM*}i_*(i^*\calL_\alpha) \\
&=\pi_{QM*}(\calL_\alpha\otimes i_*\oO_{QM}) \\
&=\pi_{QM*}(\calL_\alpha\otimes \pi^*(\calO_0)).
\end{align*}
Therefore we have 
$$\bfM^\alpha(z)=\sum_{\beta}z^\beta 
\Phi^{R=1,\beta,\bbP^1}_{\begin{subarray}{c}0,\infty\\  \emptyset \end{subarray}}\left(\{c_1^{T_0}\pi_{QM*}(\calL_\alpha\otimes \pi^*(\calO_0))\},\{0\}\right),  $$
here compared to \eqref{def of PHI R sm}, we write explicitly the insertion inside the bracket for convenience . 

Recall the definition of curve class 
$$\beta\in \Hom_{\mathbb{Z}}(\bbX(G), \bbZ),\quad \beta(\xi):=\deg_C(P_G\times_G \bbC_{\xi}). $$
As in \cite[Eqn.~(8.1.4)]{O}, we have 
\begin{align}\label{equ on betaalpha}
\beta(\alpha) 
=\frac{(-1)}{2\hbar}c_1^{T_0}\pi_{QM*}(\calL_\alpha\otimes \pi^*(\calO_0-\calO_\infty))\in \bbZ. 
\end{align}
In fact, we have (see e.g.~Eqns.~\eqref{equ str of O0}, \eqref{eqn:infty}):
$$\calO_{\mathbb{P}^1}(-\{0\})\cong \calO_{\mathbb{P}^1}(-1)\,q^{-1}, \quad \calO_{\mathbb{P}^1}(-\{\infty\})\cong \calO_{\mathbb{P}^1}(-1)\,q. $$
Hence we obtain (recall $\hbar:=-c^{T_0}_1(q)$): 
\begin{align*}
c_1^{T_0}\pi_{QM*}(\calL_\alpha\otimes \pi^*(\calO_0-\calO_\infty))
&=c_1^{T_0}\left(\pi_{QM*}(\calL_\alpha\otimes \pi^*\calO_{\mathbb{P}^1}(-1))\otimes (q-q^{-1})\right) \\
&=-2\hbar\cdot \rk\left(\pi_{QM*}(\calL_\alpha\otimes \pi^*\calO_{\mathbb{P}^1}(-1))\right) \\
&=-2\hbar\cdot \chi(C,P\times_{G}\bbC_{\alpha}\otimes \pi^*\calO_{\mathbb{P}^1}(-1)) \\
&=-2\hbar\cdot \chi(\mathbb{P}^1,\pi_*\left(P\times_{G}\bbC_{\alpha}\right)\otimes \calO_{\mathbb{P}^1}(-1)) \\
&=-2\hbar\cdot \deg_{\mathbb{P}^1}\left(\pi_*\left(P\times_{G}\bbC_{\alpha}\right)\right) \\
&=-2\hbar\cdot \deg_{C}\left(P\times_{G}\bbC_{\alpha}\right) \\
&=-2\hbar\cdot \beta(\alpha), 
\end{align*}
where the third equality is by the base change to a $\mathbb{C}$-point $\left(C,p_1',\ldots,p_n',\pi,P,u \right)$ of $QM$ and 
the sixth equality is by Riemann-Roch and adjunction formula $\chi(C,-)=\chi(\mathbb{P}^1,\pi_*(-))$. 
%For abbreviation, in below we write  $$QM^{R_\chi=\omega_{\log}}_{\begin{subarray}{c}\bullet \end{subarray}}=QM^{R_\chi=\omega_{\log}}_{\begin{subarray}{c}\bullet \end{subarray}}(\Crit(\phi)/\!\!/G,\beta,\bbP^1)$$

By definition, we have 
\[\frac{\partial J(z)}{\partial z_i}=\frac{1}{z_i}\sum_{\beta_1,\ldots,\beta_N} \beta(\alpha_i)\,z_1^{\beta_1}\cdots z_N^{\beta_N}\Phi^{R,\beta,\bbP^1}_{\begin{subarray}{c}0 \\  \infty \end{subarray}}.\]
Combining with Eqn.~\eqref{equ on betaalpha}, we have 
\begin{align*}
-2\hbar\frac{\partial J(z)}{\partial z_i}&=
%\frac{1}{z_i}\sum_\beta z^{\beta}\Phi^{R,\beta,\bbP^1}_{\begin{subarray}{c}0 \\  \infty \end{subarray}}\left(\{c_1^{T_0}\pi_{QM*}(\calL_{\alpha_i}\otimes \pi^*(\calO_0)), c_1^{T_0}\pi_{QM*}(\calL_{\alpha_i}\otimes \pi^*(-\calO_\infty))\},\{0,\infty\}\right) \\ 
\frac{1}{z_i}\sum_{\beta} z^{\beta}
\Phi^{R,\beta,\bbP^1}_{\begin{subarray}{c}0 \\  \infty \end{subarray}}\left(\{c_1^{T_0}\pi_{QM*}(\calL_{\alpha_i}\otimes \pi^*(\calO_0)))\},\{0\}\right) \\
& -\frac{1}{z_i}\sum_{\beta} z^{\beta}\Phi^{R,\beta,\bbP^1}_{\begin{subarray}{c}0 \\  \infty \end{subarray}}\left(\{ c_1^{T_0}\pi_{QM*}(\calL_{\alpha_i}\otimes \pi^*(\calO_\infty))\},\{\infty\}\right).
\end{align*} 
We simplify the two terms in the right hand side separately below. 

The degeneration and gluing formulae \eqref{deg form}, \eqref{glue formula conv form for parametrized} imply that the first term is equal to
%\yl{The gluing formula when $\beta=\beta+0$ seems problematic.}\gufang{We are talking about $1:H_{T_0}(X,\varphi_{\phi})\to H_{T_0}(X,\varphi_{\phi})$, that is, the pairing itself.Let me think about the $\beta=\beta+0$ issue.}
\begin{align*}&\frac{1}{z_i} \Bigg(\sum_{\beta}z^\beta \Phi^{R=1,\beta,\bbP^1}_{\begin{subarray}{c}0,\infty\\  \emptyset \end{subarray}}\left(\{c_1^{T_0}\pi_{QM*}(\calL_{\alpha_i}\otimes \pi^*(\calO_0))\},\{0\}\right) \otimes\sum_\beta z^\beta \Phi^{R,\beta,\bbP^1}_{\begin{subarray}{c}0 \\  \infty \end{subarray}}\Bigg)(-\boxtimes \eta).
\end{align*}
Let $i_{\infty}: QM\to \calC$ be the section of $\pi_{QM}: \calC\to QM$ corresponding to the marked point at $\infty$.  
Then
\begin{align*}
ev_\infty^*L_\alpha=i_{\infty}^*\calL_\alpha=\pi_{QM*}i_{\infty*}i_{\infty}^*\calL_\alpha
=\pi_{QM*}(\calL_\alpha\otimes i_{\infty*}\oO_{QM})
=\pi_{QM*}(\calL_\alpha\otimes \pi^*(\calO_\infty)),
\end{align*}
where we use the smoothness condition at $\infty$ in the last equality. 
Therefore by projection formula, the second term becomes
\[-\frac{1}{z_i}\sum_\beta z^\beta  \left(c_1(L_{\alpha_i})\cap-\right)\circ 
\Phi^{R,\beta,\bbP^1}_{\begin{subarray}{c}0 \\  \infty \end{subarray}}.\]
Combining the above, we obtain
the statement.
\end{proof}

\subsection{Quantum connections and quantum multiplications from parametrized quasimaps}

\subsubsection{Quantum connections}\label{sect on qc}
In the special cases of Setting~\ref{ass:compact_type} or Setting \ref{ass:geo_phase}, 
similar to Definition~\ref{def of QC2}(1), the operator $\Phi^{R,\beta,D}_{\begin{subarray}{c}p_1,\ldots,p_n \\  q_1,\ldots,q_m \end{subarray}}$ can be defined on critical cohomology
\[H_{F_0}(X,\varphi)^{\otimes n}\to H_{F_0}(X,\varphi)^{\otimes m}_{loc}.\]
%We remark that when $m=0$, the evaluation map is proper (Remark \ref{rmk:prop when no twist}) and $ev_*^n$ is the usual proper pushforward. 
By using the analogue of Definition~\ref{def of QC2}(2),
$\bfM^\alpha(z)$ becomes an \textit{operator}:
$$\bfM^\alpha(z)(-\boxtimes\eta^i)\,\eta_i: H_{T_0}(X,\varphi_{\phi})_{}[\![z]\!]\to H_{T_0}(X,\varphi_{\phi})_{loc }[\![z]\!]. $$ 
Here we recall that $F_0$-equivariance can be enhanced to $T_0$-equivariance as special points used to define $\bfM^\alpha(z)$ are $T_0$-invariant. 

%\yl{I rewrote the insertion in above definition, it is the same as before (see the proof of Thm.~\ref{thm:fund_sol}). See whether you agree. }\gufang{The insertion is now changed to 0.}
Next we introduce the \textit{quantum connection}. 
Consider the trivial bundle:
$$H_{T_0}(X,\varphi_{\phi})\times \mathbb{C}[\![z_1^{ },z_2^{ },\ldots,z_N^{ }]\!]\to \mathbb{C}[\![z_1^{ },z_2^{ },\ldots,z_N^{ }]\!], $$
where $z_i$'s are as in \eqref{equ on zi}. 

One can define a ``quantum connection" on this bundle: 
\begin{equation}\label{quan conn}\nabla=d-\sum_{i=1}^N\frac{1}{2\hbar z_i}\bfM^{\alpha_i}(z)(-\boxtimes\eta^i)\,\eta_i\,dz_i. \end{equation}
Using Theorem~\ref{thm:fund_sol}, the $J$-function in Definition \ref{defi of vertex fun} 
is then a gauge transformation\footnote{This means $J$ satisfies 
$dJ=JA-BJ$ which is the equation usual gauge transformations satisfy. We know there are examples where $J(0)=0$, meaning it is not invertible in the formal power series. 
However it may be possible that $J$ has convergence property on the K\"ahler moduli space such that it is invertible for generic $z$. }, 
which transforms $\nabla$ 
to a connection of the form 
$$d-\sum_{i=1}^N\frac{c_1(L_{\alpha_i})\cap(-)}{2\hbar z_i} dz_i, $$
whose flatness is then obvious. 

\subsubsection{Quantum multiplications}

Using the analogue of Definition~\ref{def of QC2}(3), we obtain a \textit{quantum multiplication}
$$\star: H_{F_0}(X,\varphi_{\phi})^{\otimes 2}\to H_{F_0}(X,\varphi_{\phi})_{loc}[\![z]\!], $$ 
and hence an \textit{operator} 
\begin{equation}\label{star prod2}\gamma\star(-)\, : H_{F_0}(X,\varphi_{\phi})\to H_{F_0}(X,\varphi_{\phi})_{loc}[\![z]\!],\end{equation}
for any $\gamma\in H_{F_0}(X,\varphi_{\phi})$.
It is an interesting question to explore the relation between the multiplication $\star$ here and the quantum product $*$ in Definition \ref{def of QC2}. 
We note that even for the special choice of $\phi$ discussed in \S \ref{sec:dim_red}, this comparison already appears to be unclear at the moment \cite[Footnote 1]{KZ}, \cite[\S 1.2]{KPSZ}.

\subsubsection{Bethe ansatz}\label{sect on ansztz}
Similar to \S \ref{subsec:crit2}, via a topological version of the invariants on critical cohomology,  a quantum multiplication is expected to exist without the assumption in Settings \ref{ass:compact_type} or  \ref{ass:geo_phase}. For each character $\alpha\in \bbX(G)$, the topological version of $\hat{\alpha}(z): H_{F_0}(X,\varphi_{\phi})\to H_{F_0}(\pt)_{loc}[\![z]\!]$ in Definition \ref{defi of vertex fun}\,(3) defines a element in $H_{F_0}(X,\varphi_{\phi})_{loc}[\![z]\!]$ using Verdier duality\,\footnote{Here Verdier duality works as $\varphi_{\phi}$ is supported on the critical locus which is $F_0$-equivariantly proper.}, 
the quantum multiplication $\hat{\alpha}(z)\star(-)$ by which is given as in \eqref{star prod2}. 

The following is an analogue of \cite[Prop.~9]{AO}, \cite[Thm.~17]{PSZ}, which are partially developed from physics consideration of Nekrasov and Shatashvili \cite{NS}. It states that 
\textit{eigenvalues} of the above multiplication can be computed from the \textit{saddle point equation} of the integrant when we write the \textit{vertex function} $V^\alpha(z)$ (Definition \ref{defi of vertex fun}\,(1)) as a contour integral.  
%\yl{In the below Cor, do we care about the explicit form of the integrant? In \cite[Thm.~17]{PSZ}, when $\hbar \to 0$, one needs divergence with leading term by the saddle point approximation}
%\gufang{I made the assumption that " limit $\lim_{\hbar=0}$ is approximated by the saddle points of the integrant"}

Let $\fg$ be the Lie algebra of a complex $n$-dimensional reductive group $G$ and $(s_1,\dots,s_n)$ be the coordinates of $\fg$.
\begin{ansatz}\label{ansatz of saddle pt}
Assume $V^\alpha(z)$ has a formula in terms of an integral of a meromorphic $n$-form on an $n$-cycle in $\fg$,  which in the limit $\lim_{\hbar=0}$ is approximated by the saddle points of the integrant,  then the eigenvalue of $\hat{\alpha}(z)\star(-)$ is given by the symmetric polynomial 
$\alpha(s_1,\dots,s_n)$ with $(s_1,\dots,s_n)$ lies in the saddle locus of the integrant.
\end{ansatz}
In \S \ref{sect on mbi}, \S \ref{sect on spe}, \S \ref{sect on sapq for nn}, \S \ref{sect on more bethe equ}, we will compute vertex functions and corresponding 
saddle point equations for several quivers with potentials.

\section{Vertex functions and Bethe equations for $\Hilb^n(\bbC^3)$}\label{sect vertex func on hilb}
In this section, we compute the vertex function of Hilbert schemes $\Hilb^n(\bbC^3)$ of points on $\bbC^3$ (ref.~Example \ref{key exam}). 
We express it in terms of a contour integral and determine the ``saddle point equations" of the integrant. This gives the ``Bethe equations" for the 
representation of the $(-1)$-shifted affine Yangian as studied by \cite{RSYZ2}.

\subsection{Cohomology on $\bbP^1$}\label{sect on coho of p1}
We first recall some basics on equivariant cohomology of $\bbP^1$. 
Following \S \ref{subsec:qconn}, we write 
$$\bbP^1=\mathrm{Proj}\, \bbC[x,y],$$ 
with a $\bbC^*_q$-action so that weight of $x$ is $-1$ and weight of $y$ is $1$. The point $\infty$ is  $[1:0]$ in the homogeneous coordinates, and $0$ is  $[0:1]$. In particular, the local coordinate function around $\infty$ is $y/x$ which has weight $2$ and the tangent space at $\infty$ has weight $-2$. The line bundle $\calO(1)$ has the space of global sections being 
$$H^0(\bbP^1,\calO(1))=\bbC^2, $$ 
with $\bbC^*_q$-eigenbasis given by $x$ and $y$. 
In particular, $y:\calO\to \calO(1)\,q^{-1}$ is up to scalar the only $\bbC^*_q$-equivariant section that is non-vanishing at the point $0$. 
Equivalently, endow the ideal sheaf $\calO(-\{\infty\})\subseteq \calO$ with the induced equivariant structure, we have 
\begin{equation}\label{equ str of O0}\calO(-1)\,q\cong \calO(-\{\infty\}), \end{equation}
with the isomorphism induced by $y$.  Similarly, we have
\begin{equation}
\label{eqn:infty}
    \calO(-1)\,q^{-1}\cong \calO(-\{0\}),
\end{equation}
which is induced by $x:\calO\to \calO(\{0\})$, up to scalar the unique $\bbC^*_q$-equivariant map non-vanishing at $\infty$.
With the above convention, we have
\begin{align*}
H^*(\bbP^1,\calO(d)\,q^{d})=
\left\{\begin{array}{rcl}1+q^{2}+q^{4}+\cdots+q^{2d}, \quad \quad   & \mathrm{if} \,\,d\geqslant  0, \,\,\, \\
  & \\ 
  0,  \quad\quad\quad\quad\quad\quad    &  \mathrm{if} \,\,d=-1, \\
   & \\ 
 -(q^{-2}+q^{-4}+\cdots +q^{2d+2}), \quad \,& \mathrm{if} \,\,d\leqslant -2.
\end{array} \right. 
\end{align*}
Let $c_1(q)=-\hbar\in H^*_{\bbC^*_q}(\pt)$. For any character $a$ of a torus $T$ containing $\bbC^*_q$ as subtorus, let $u:=c_1(a)\in H^*_T(\pt)$.  We express the equivariant Euler class of 
$$\chi_{\bbP^1}(\calO(d)\,q^{d}a)=H^*(\bbP^1,\calO(d)\,q^{d}a)$$ 
in terms of $\Gamma$-functions in below. 
For this purpose, we define 
$$\Gamma_{2\hbar}(z):=\Gamma(z/2\hbar), $$ 
which is a meromorphic function of $z\in \bbC$ with only simple poles at $z=-d(2\hbar)$
with $d\in\bbN$. The quasi-periodiciy of $\Gamma$-function then yields \begin{equation}
    \Gamma_{2\hbar}(z+2\hbar)=\frac{z}{2\hbar}\Gamma_{2\hbar}(z),
\end{equation}
or equivalently 
\begin{eqnarray*}
\Gamma_{2\hbar}(z)=\frac{2\hbar}{z}\Gamma_{2\hbar}(z+2\hbar), \quad \Gamma_{2\hbar}(z)=\frac{(z-2\hbar)}{2\hbar}\Gamma_{2\hbar}(z-2\hbar), 
\end{eqnarray*}
where all the equalities are as meromorphic functions. We also have the special value
\[\Gamma_{2\hbar}(2\hbar)=\Gamma(1)=1.\]
Then by quasi-periodicity, we have the following equality of meromorphic functions in $u$:
\begin{equation}\label{eqn:gamma_section}
    e^T(\chi_{\bbP^1}(\calO(d)q^{d}a)-a)=(2\hbar)^{d}\frac{\Gamma_{2\hbar}(u)}{\Gamma_{2\hbar}(u-2d\hbar)}, \quad \forall\,\, d\in \mathbb{Z}.
\end{equation}
This is an analogue of the Pochhammer symbol. 
Using quasi-periodicity, we evaluate the residue
\begin{equation}\label{eqn:res_gamma}
\Res_{u=-2d\hbar}\Gamma_{2\hbar}(u), \quad \mathrm{for} \,\,d\geqslant  0,\end{equation}
to be
\begin{align}\label{eqn:residue0}
\lim_{u=-2d\hbar}(u+2d\hbar)\Gamma_{2\hbar}(u)&=\lim_{u=-2d\hbar}(u+2d\hbar)\frac{(2\hbar)^{d+1}}{(u)(u+2\hbar)\cdots(u+2d\hbar)}\Gamma_{2\hbar}(u+2d\hbar+2\hbar)\\ \nonumber
&=\frac{\Gamma_{2\hbar}(2\hbar)(2\hbar)^{d+1}}{(-2d\hbar)(-2d\hbar+2\hbar)\cdots(-2\hbar)}=\frac{(-1)^d(2\hbar)}{d!} \\ \nonumber 
&=\frac{(2\hbar)^{d+1}}{e^T(\chi_{\bbP^1}(\calO(d)q^{d})-1)}. 
\end{align}
Combining it with Eqn.~\eqref{eqn:gamma_section}, for any $d\in\bbN$, we also have
\begin{equation}\label{eqn:residue}
\Res_{u=-2d\hbar}\Gamma_{2\hbar}(u)=2\hbar\frac{\Gamma_{2\hbar}(-2d\hbar)}{\Gamma_{2\hbar}(0)}, \end{equation}
where although both the numerator and denominator on the right hand side are taking values at poles, the ratio is well-defined.
Hence we make the convention that
\begin{equation}\label{eqn:residue2}\frac{e^T(1)}{e^T(\chi_{\bbP^1}(\calO(d)q^{d})}=\frac{1}{e^T(\chi_{\bbP^1}(\calO(d)q^{d})-1)}, \end{equation}
although $e^T(1)$ is zero, we keep this factor for convenience and write the above, which is equal to \eqref{eqn:res_gamma}, so that the right hand side of \eqref{eqn:residue} makes sense.

%which is equal to $\frac{1}{e^T(\chi_{\bbP^1}(\calO(d)q^{-d})-1)}$ by the discussion above.
%\begin{YC}The above equality looks wrong?  \end{YC}\gufang{Corrected.}

\subsection{Vertex functions with insertions}

Work in the setting of Example \ref{key exam}. 
%Fix $P_0$ to be the principal $\mathbb{C}^*$-bundle obtained from $\oO_{\mathbb{P}^1}(1)$ by removing the zero section, and also 
Fix an $R$-\textit{charge}
\begin{equation}\label{equ on choice of R charge}R: \mathbb{C}^*\to F=(\mathbb{C}^*)^3, \quad t\mapsto (t^{-\sigma_1},t^{-\sigma_2},t^{-\sigma_3}), 
\,\,\, \mathrm{with}\,\,\, \sigma_i\in\mathbb{Z}. \end{equation}
Then we have 
$$R_\chi=\chi\circ R:\, \mathbb{C}^*\to \mathbb{C}^*, \quad t\mapsto t^{-\sigma_1-\sigma_2-\sigma_3}. $$
%\yl{do we need $R_{\chi}$ is injective? Then it can not be $\pm 2$?}
Recall Definition \ref{def of para qm 2} and consider the moduli stack \eqref{qm par cpn} (with $n=0$)
$$QM^R_d(\bbP^1,\Hilb^n(\mathbb{C}^3)):=QM^{R_\chi=\omega_{\log}}(\Hilb^n(\mathbb{C}^3),d,\bbP^1)  $$
of stable genus $0$, $\bbP^1$-parametrized, $R$-twisted quasimaps to $\Hilb^n(\mathbb{C}^3)$ in class 
$$d\in \Hom_{\mathbb{Z}}(\bbX(\GL_n),\mathbb{Z})\cong \mathbb{Z}. $$ 
%\yl{the class is a map $\bbX(G)\to \bbZ$, why can be written as an integer? }
The action of $F$ on $W$ (ref.~Example \ref{key exam}) induces an action on $QM^R_d(\bbP^1,\Hilb^n(\mathbb{C}^3))$. Moreover the action of $\bbC^*_q$ on $\bbP^1$ induces an action on $QM^R_d(\bbP^1,\Hilb^n(\mathbb{C}^3))$ which commutes with the action of $F$. Therefore the moduli space has an action given 
by the products
\begin{equation}\label{equ torus T} T:=F\times\bbC^*_q, \quad T_0:=F_0\times\bbC^*_q.  \end{equation}
%\yl{Relation between $T_0$ and CY torus on $X$? }
\begin{prop}\label{iso of qm and pt}
There is a $T$-equivariant isomorphism   
$$QM^R_d(\bbP^1,\Hilb^n(\mathbb{C}^3))\xrightarrow{\cong} P_{n+d}(X,n), $$
to the moduli space of Pandharipande-Thomas (PT) stable pairs $(F,s)$ on $X=\mathrm{Tot}_{\bbP^1}(\calL_1\oplus\calL_2\oplus \calL_3 )$  
with $[F]=n[\bbP^1]$ and $\chi(F)=n+d$ \cite{PT, CMT2}. Here 
$\calL_i=\calO_{\bbP^1}(-\sigma_i)$ satisfies $\calL_1\otimes \calL_2\otimes \calL_3\cong \omega_{\bbP^1}$. 
\end{prop}
\begin{proof}
This is similar to \cite[Exer.~4.3.22]{O} and \cite[Thm.~1.9]{Dia}.
Let $P_F$ denote the principal $F$-bundle obtained by removing the zero section of each summand in $\oO_{\bbP^1}(-\sigma_1)\oplus \oO_{\bbP^1}(-\sigma_2)\oplus \oO_{\bbP^1}(-\sigma_3)$. By Definition \ref{def of para qm 2} and Remark \ref{rmk on n=0}, a $\mathbb{C}$-point of 
$QM^R_d(\bbP^1,\Hilb^n(\mathbb{C}^3))$ is given by 
a principal $G=\mathrm{GL_n}$-bundle $P_G$ on $C=\bbP^1$ 
and a section $u$ of the vector bundle 
\begin{align*}
(P_G\times_CP_F)\times_{G\times F}W&=(P_G\times_CP_F)\times_{G\times F}V\oplus \left((P_G\times_CP_F)\times_{G\times F}\calE nd\,V\right)^{\oplus 3} \\
&=P_G\times_{G}V\oplus \left(P_G\times_{G}\calE nd\,V\right)\otimes \bigoplus_{i=1}^3(P_{F}\times_{F}\mathbb{C}_{-\sigma_i}) \\
&=\calV \oplus \calE nd\,\calV\otimes \calL_1\oplus \calE nd\,\calV\otimes \calL_2\oplus \calE nd\,\calV\otimes \calL_3, \quad \mathrm{where} \,\,\calV:=P_G\times_{G}V, 
\end{align*}
such that outside a finite (possibly empty) set $B\subset C$ of points, $u(C\setminus B)$ is contained in the stable locus 
$(P_G\times_CP_F)\times_{G\times F}\Crit(\phi)^s$.
Recall the setting of Example \ref{key exam}, we know this is equivalent to a section $s\in H^0(C,\calV)$ and commuting homomorphisms  
$\phi_i: \calV\to \calV\otimes \calL_i$ such that on $C\setminus B$, the morphisms $s$ and $\phi_i$'s generate $\calV$. 
Interpreting $\phi_i$'s as Higgs fields, the above is equivalent to a pure one dimensional sheaf $F$ on $X$ with a section $s: \oO_X\to F$
such that $\mathrm{Coker}(s)$ is zero dimensional. The above construction works in families and gives an isomorphism 
of two moduli spaces which is obviously $T$-equivariant. 
\end{proof}
\begin{remark}
To do wall-crossing, besides using the $\epsilon$-stability on quasimaps (Remark \ref{rmk on e-stab}), one can also consider the $Z_t$-stability \cite{CT1, CT3, CT4} on Calabi-Yau 4-folds which generalizes $\PT$-stability. 
\end{remark}
%Let $QM^\sigma_d(\bbP^1,\Hilb^n(\mathbb{C}^3))$ be the moduli space of $\sigma$-twisted, degree-$d$ quasimaps with fixed domain $\bbP^1$. 
\begin{definition}
An $R$-twisted quasimap in $QM^R_d(\bbP^1,\Hilb^n(\mathbb{C}^3))$ is \textit{smooth} at $\infty\in \mathbb{P}^1$ if $\infty \notin B$ for the base locus $B$ in
Definition \ref{def of para qm 2}.
Denote the open subscheme of such $R$-twisted quasimaps by 
\[QM^{R}_{d,\mathrm{sm},\infty}:=QM^R_d(\bbP^1,\Hilb^n(\mathbb{C}^3))_{\mathrm{sm},\infty}\subseteq QM^R_d(\bbP^1,\Hilb^n(\mathbb{C}^3)). \] 
\end{definition}
%For simplicity, we denote this open subspace by $QM^{R}_{d,\mathrm{sm},\infty}$. 
Observe that the open subscheme $QM^{R}_{d,\mathrm{sm},\infty}$ is invariant under the action of $T$ \eqref{equ torus T} and 
there is a well-defined evaluation map 
$$ev_\infty: QM^{R}_{d,\mathrm{sm},\infty}\to \Hilb^n(\mathbb{C}^3),$$ 
which is $T$-equivariant (here $\bbC^*_q$ acts trivially on the target).
Although this map is not proper, it is $T_0$ (and in particular $T$)-equivalently proper,~i.e.~it is a proper map on the $T_0$ (and in particular $T$)-fixed locus, 
because the $T$-fixed locus coincides with the $T_0$-fixed locus which is a 
finite number of reduced points:
\begin{prop}\label{prop on red points}\emph{(\cite[Prop.~2.6]{CK2}, Proposition \ref{iso of qm and pt})}
$$\left(QM^R_d(\bbP^1,\Hilb^n(\mathbb{C}^3))\right)^T=\left(QM^R_d(\bbP^1,\Hilb^n(\mathbb{C}^3))\right)^{T_0}$$
are finite number of reduced points. Therefore 
$$\left(QM^{R}_{d,\mathrm{sm},\infty}\right)^T=\left(QM^{R}_{d,\mathrm{sm},\infty}\right)^{T_0}$$
are also finite number of reduced points.
\end{prop}
\iffalse
We can then define the equivariant pushforward map by torus localization\footnote{It is consistent with Definition \ref{defi of localized push} by virtual localization formula 
\cite[Thm.~7.1]{OT}.}: 
\begin{equation}ev_{\infty*}: A_*^{T_0}(QM^{R}_{d,\mathrm{sm},\infty})\to A_*^{T_0}(\Hilb^n(\mathbb{C}^3))_{loc}, \nonumber \end{equation}
$$\gamma \mapsto  \sum_{(\calV,u)\in (QM^{R}_{d,\mathrm{sm},\infty})^{T}}ev_{\infty*}\left(\frac{\gamma|_{(\calV,u)}}{\sqrt{e^{T_0}(\bbT^{\mathrm{vir}}_{(\calV,u)} 
QM^{R}_{d,\mathrm{sm},\infty})}}\right), $$
where $\bbT^{\mathrm{vir}}QM^{R}_{d,\mathrm{sm},\infty}$ is the virtual tangent complex given by the restriction of the tangent complex of the derived enhancement of 
$QM^{R}_{d,\mathrm{sm},\infty}$ to the classical truncation. \yl{above should be replaced by below}
\fi
We can then define the capping with virtual class by the torus localization \cite[Thm.~7.1]{OT}:
%\footnote{It is consistent with Definition \ref{defi of localized push} by virtual localization formula \cite[Thm.~7.1]{OT}.}: 
\begin{equation}\label{equ on evstar}[QM^{R}_{d,\mathrm{sm},\infty}]^{\mathrm{vir}}\cap (-): A^*_{T_0}(QM^{R}_{d,\mathrm{sm},\infty})\to A_*^{T_0}(\Hilb^n(\mathbb{C}^3))_{loc}, \end{equation}
$$\gamma \mapsto  \sum_{(\calV,u)\in (QM^{R}_{d,\mathrm{sm},\infty})^{T}}ev_{\infty*}\left(\frac{\gamma|_{(\calV,u)}}{ 
\sqrt{(-1)^{\mathrm{vdim}/2}\cdot e^{T_0}(\bbT^{\mathrm{vir}}_{(\calV,u)} 
QM^{R}_{d,\mathrm{sm},\infty})}}\right), $$
where $\bbT^{\mathrm{vir}}QM^{R}_{d,\mathrm{sm},\infty}$ is the virtual tangent complex given by the restriction of the tangent complex of the derived enhancement of 
$QM^{R}_{d,\mathrm{sm},\infty}$ to the classical truncation, and $\mathrm{vdim}$ denotes its rank. 

Let $\calV$ be the tautological bundle on $QM^{R}_{d,\mathrm{sm},\infty}\times\bbP^1$, which is $T$-equivariant. 
The inclusion $\{0\}\inj\bbP^1$ induces an embedding
$$\iota:QM^{R}_{d,\mathrm{sm},\infty}\times\{0\}\inj QM^{R}_{d,\mathrm{sm},\infty}\times\bbP^1. $$ 
Pullback along $\iota$ defines $\iota^*\calV$ on $QM^{R}_{d,\mathrm{sm},\infty}$.
More generally, for any $\tau\in K_{G\times T}(\pt)$, we have $\tau(\calV)\in K_{T}(QM^{R}_{d,\mathrm{sm},\infty}\times\bbP^1)$ 
and also $\tau(\iota^*\calV)\in K_{T}(QM^{R}_{d,\mathrm{sm},\infty})$ defined similarly. More specifically, write 
$$\tau=\sum_{i} t_1^{i_1}t_2^{i_2}t_3^{i_3}s_{\lambda_i}, $$
where $s_{\lambda_i}$ are irreducible representations of $\GL_n$. Let $\mathrm{Fr}(\calV)$  be the framed bundle of $\calV$, then 
$$\tau(\calV)=\sum_{i} t_1^{i_1}t_2^{i_2}t_3^{i_3}\cdot \mathrm{Fr}(\calV)\times_{\GL_n} s_{\lambda_i}. $$
Parallel to the tautological insertions/descendent insertions in $\DT_4$ setting, e.g.,~\cite{CK1, CT2, CT3}, there are vertex functions with descendent and tautological insertions.
\begin{definition}\label{def on midtau}
With  \eqref{equ on evstar}, we define 
$$|\,\tau\rangle^d:=[QM^{R}_{d,\mathrm{sm},\infty}]^{\mathrm{vir}}\cap \left(e^{T_0}(\tau(\iota^*\calV))\right)\in A_*^{T_0}(\Hilb^n(\mathbb{C}^3))_{loc}, $$
where $d$ denotes the degree of quasimaps. 

The \textit{vertex function with descendent insertion} $\tau$ at $0\in\bbP^1$ is defined as
\[|\,\tau\rangle(z)=\sum_{d\in \mathbb{Z}} |\,\tau\rangle^dz^d\in A_*^{T_0}(\Hilb^n(\mathbb{C}^3))_{loc}[\![z]\!]. \]
Similarly, the \textit{vertex function with tautological descendent insertion} $\tau$ at $0\in\bbP^1$ is defined as
\[|\,\tau\rangle(z,m)=\sum_d\, [QM^{R}_{d,\mathrm{sm},\infty}]^{\mathrm{vir}}\cap\left(e^{T_0\times \mathbb{C}^*}(\chi_{\bbP^1}(\calV)\otimes e^m)\cdot e^{T_0}(\tau(\iota^*\calV))\right)\,z^d\in 
A_*^{T_0}(\Hilb^n(\mathbb{C}^3))_{loc}[m][\![z]\!],\] 
where $\mathbb{C}^*$ acts trivially on moduli spaces and $e^m$ denotes a trivial line bundle such that $c_1^{\mathbb{C}^*}(e^m)=m$. 
%\yl{I revise the target to be BM homology instead of critical homology.}
\end{definition}
%\begin{remark}By the canonical map \eqref{map can2}, we can also map the above vertex functions to functions with coefficients in critical cohomology $H^*_{T_0}(W/\!\!/ G,\varphi_\phi)_{loc}$. \end{remark}
\begin{remark}
One can also consider $K$-theoretic vertex functions which recover the above one by cohomological limit (e.g.~\cite[\S 0.4]{CKM1}).
\end{remark}
%Here $F$ is the field of fractions of $H^*_T(\pt)$.
We expand $|\,\tau\rangle^d$  under the torus fixed points\footnote{For Calabi-Yau subtorus $F_0\subseteq F$, we have 
$\Hilb^n(\mathbb{C}^3)^F=\Hilb^n(\mathbb{C}^3)^{F_0}$ as schemes \cite[Lem.~4.1]{BF2}.} $\lambda\in \Hilb^n(\mathbb{C}^3)^{F_0}$, labelled by 3d Young diagrams (i.e.~plane partitions) of $n$-boxes (i.e.~size $n$).
Define $QM^{R}_{d,\infty=\lambda}$ by the Cartesian diagram 
%\yl{The target is $\Hilb^n(\mathbb{C}^3)/R(\Ker R_\chi)?$}
\begin{equation}\label{diag defin qm inf lamb}\xymatrix{
QM^{R}_{d,\infty=\lambda}\ar[r]\ar[d] \ar@{}[dr]|{\Box} &\{\lambda\} \ar[d]
\\
QM^{R}_{d,\mathrm{sm},\infty}\ar[r]^{ev_\infty}&\Hilb^n(\mathbb{C}^3). 
}\end{equation}
Replace $(QM^{R}_{d,\mathrm{sm},\infty})^{T}$ in \eqref{equ on evstar} and Definition \ref{def on midtau} by the subset $(QM^{R}_{d,\infty=\lambda})^{T}$, one defines 
$$|\,\tau\rangle^d_\lambda\in A_*^{T_0}(\pt)_{loc}, $$ 
which obviously satisfies 
$$|\,\tau\rangle^d=\sum_{\lambda} |\,\tau\rangle^d_\lambda, $$
where we use localization formula \eqref{equ on toru loc} to identify 
$$A_*^{T_0}(\Hilb^n(\mathbb{C}^3))_{loc}\cong \bigoplus_{\lambda\in \Hilb^n(\mathbb{C}^3)^{F_0}}A_*^{T_0}(\pt)_{loc}. $$
\begin{definition}\label{def of mid tau lamb}
We write 
\begin{equation*}|\,\tau\rangle_\lambda(z)=\sum_d |\,\tau\rangle^d_\lambda\,z^d\in  A_*^{T_0}(\pt)_{loc}[\![z]\!],  \end{equation*}
and similarly define $|\,\tau\rangle_\lambda(z,m)\in A_*^{T_0}(\pt)_{loc}[m][\![z]\!]$ based on $ev(\infty)=\lambda \in \Hilb^n(\mathbb{C}^3)^{F_0}$. 
\end{definition}

\subsection{Computations of vertex functions}\label{sec:vertex_hilb}
Now we fix a 3d Young diagram $\lambda$ of size $|\lambda|=n$. We write each $\square\in \lambda$ as $(i_1,i_2,i_3)$ with $i_j\in\bbN$.
\begin{lemma}\label{lem on t fix pt}
$(QM^{R}_{d,\infty=\lambda})^T$ consists of pairs $(\calV, u)$ such that 
$$\calV=\bigoplus_{(i_1,i_2,i_3)\in \lambda}\calL_1^{-i_1}\calL_2^{-i_2}\calL_3^{-i_3}\calO(z_{i_1,i_2,i_3})\,q^{z_{i_1,i_2,i_3}},
$$ 
where $z_{i_1,i_2,i_3}\in\bbN$ subject to the condition 
\begin{equation}\label{equ on ziii}z_{i_1,i_2,i_3}\geqslant z_{i_1-1,i_2,i_3},\, z_{i_1,i_2-1,i_3},\, z_{i_1,i_2,i_3-1}, \end{equation} 
and 
$$u\in H^0\left(\mathbb{P}^1,\calV\oplus \bigoplus_{i=1,2,3}\calE nd(\calV)\otimes\calL_i\right)^T$$ 
is the $T$-equivariant section given by the canonical maps   
$$\oO\to \oO(z_{0,0,0}\cdot\{0\}), \quad \oO(z_{i_1-1,i_2,i_3}\cdot\{0\})\to \oO(z_{i_1,i_2,i_3}\cdot\{0\}), $$
$$ \oO(z_{i_1,i_2-1,i_3}\cdot\{0\})\to \oO(z_{i_1,i_2,i_3}\cdot\{0\}),\quad \oO(z_{i_1,i_2,i_3-1}\cdot\{0\})\to \oO(z_{i_1,i_2,i_3}\cdot\{0\}). $$
\end{lemma}
\begin{proof}
Under the isomorphism in Proposition \ref{iso of qm and pt}, it follows from similar analysis as \cite[\S 5.2]{CMT2} which we recall as follows.
Note that $(QM^{R}_{d,\infty=\lambda})^T$ consists of pairs $(\calV, u)$, where 
$$\calV=\bigoplus_{\square\in \lambda}\calL_\square$$ 
with each $\calL_\square$ a $T$-equivariant line bundle on $\bbP^1$, and $u$ is a $T$-equivariant section of 
$$\calW=\calV\oplus \bigoplus_{i=1,2,3}\calE nd(\calV)\otimes\calL_i$$ 
which is smooth at  $\infty$. 
The latter is equivalent to the following two conditions. 
\begin{enumerate}
    \item $s:\calO_{\bbP^1}\to \calL_{(0,0,0)}$ is a $T$-equivariant section non-vanishing at $\infty\in\bbP^1$,
    \item for each $\square=(i_1,i_2,i_3)\in\lambda$, the maps  
    $$\calL_{(i_1-1,i_2,i_3)}\otimes\calL_1^{-1}\to\calL_{i_1,i_2,i_3}, \,\, \calL_{(i_1,i_2-1,i_3)}\otimes\calL_2^{-1}\to\calL_{i_1,i_2,i_3}, \,\, \calL_{(i_1,i_2,i_3-1)}\otimes\calL_3^{-1}\to\calL_{i_1,i_2,i_3}$$ 
 are all $T$-equivariant and non-vanishing at $\infty\in\bbP^1$. \end{enumerate}
An $\bbC^*_q$-equivariant section of a line bundle exists and non-vanishing at $\infty\in\bbP^1$ only if the line bundle is $\calO(d)\,q^{d}$ for some $d\in\bbN$ and such section is unique up to scalars by \eqref{eqn:infty}. 
Keeping this in mind, the two conditions above then implies that 
\begin{eqnarray*}
\calL_{(i_1,i_2,i_3)}=\calL_1^{-i_1}\calL_2^{-i_2}\calL_3^{-i_3}\calO(z_{i_1,i_2,i_3})\,q^{z_{i_1,i_2,i_3}},
\end{eqnarray*}
where $z_{i_1,i_2,i_3}\in\bbN$ for each $(i_1,i_2,i_3)\in\lambda$ subject to the condition 
$$z_{i_1,i_2,i_3}\geqslant  z_{i_1-1,i_2,i_3}, \,\, z_{i_1,i_2-1,i_3}, \,\, z_{i_1,i_2,i_3-1},  $$ 
for any $(i_1,i_2,i_3)\in\lambda$. 
%One easily recognizes that the above data describes torus fixed PT stable pairs $(F,s)$ on non-compact Calabi-Yau 4-fold 
%$$\mathrm{Tot}_{\mathbb{P}^1}(\calL_1\oplus \calL_2\oplus \calL_3)\stackrel{\pi}{\to} \mathbb{P}^1, $$
%such that $\coker(s)$ lies in the fiber $\pi^{-1}(0)$ over $0\in \mathbb{P}^1$. 
\end{proof}
To determine \eqref{equ on evstar}, we need to compute:  
\begin{lemma}\label{lem on vir tange of qm p1}
For any $T$-fixed point $(\calV,u)\in (QM^{R}_{d,\infty=\lambda})^{T}\subset (QM^{R}_{d,\mathrm{sm},\infty})^{T}$, we have 
\begin{equation}\label{equ on sqr virtang}\sqrt{(-1)^{\mathrm{vdim}/2}\cdot e^{T_0}(\bbT^{\mathrm{vir}}_{(\calV,u)}QM^{R}_{d,\mathrm{sm},\infty})}=\frac{e^{T_0}(\chi_{\bbP^1}(\calV))e^{T_0}(\chi_{\bbP^1}(\bigoplus_{i=1}^3\calE nd(\calV)\otimes\calL_i))}{e^{T_0}(\chi_{\bbP^1}(\calE nd(\calV)))}, \end{equation}
for certain choice of sign in the square root.
%\yl{Be careful about sign. }
\end{lemma}
\begin{proof}
Recall that $QM^{R}_{d,\mathrm{sm},\infty}$ is an open subscheme of $QM^{R_\chi=\omega_{\mathrm{log}}}_{}(\Hilb^n(\mathbb{C}^3),d, \bbP^1)$ 
whose virtual class, on one hand, is constructed by the pullback map \eqref{equ on sq pb} (take fundamental class of $\fBun$ as domain since there is no marked point), 
on the other hand, can be computed by virtual localization (noticing that by Proposition \ref{prop on red points}, it is then reduced to calculate the LHS of \eqref{equ on sqr virtang}). 
 
More specifically, relative to 
$$\fBun_{H_R}^{R_{\chi}=\omega_{\mathrm{log}}}(\bbP^1)\cong \fBun_{G}(\bbP^1), $$
the symmetric obstruction theory is given by Eqn.~\eqref{equ on repe1} whose restriction to the closed point $(\calV,u)\in (QM^{R}_{d,\infty=\lambda})^{T}$
(in $K$-theory) is
\begin{equation}\label{equ on tan cpx at a pt}\dR \Gamma(\calW)+\dR \Gamma(\calW)^\vee, \end{equation}
where $\calW$ is the $W$-bundle given in  the proof of Proposition \ref{iso of qm and pt}: 
\begin{equation}\label{equ on tan cpx at a pt2}\calW=\calV \oplus \calE nd\,\calV\otimes \calL_1\oplus \calE nd\,\calV\otimes \calL_2\oplus \calE nd\,\calV\otimes \calL_3.  \end{equation}
The tangent complex of $\fBun_{G}(\bbP^1)$ at point $\calV$ (in $K$-theory) is 
\begin{equation}\label{equ on tan cpx at a pt3}-\dR\Gamma(\calE nd\,\calV). \end{equation}
Therefore, we have 
\begin{align*}
\sqrt{(-1)^{\mathrm{vdim}/2}\cdot e^{T_0}(\bbT^{\mathrm{vir}}_{(\calV,u)}QM^{R}_{d,\mathrm{sm},\infty})}&=
\sqrt{(-1)^{\mathrm{\mathrm{rk}(\dR \Gamma(\calW))}}\cdot e^{T_0}(\dR \Gamma(\calW)+\dR \Gamma(\calW)^\vee)}\cdot e^{T_0}(-\dR\Gamma(\calE nd\,\calV)) \\
&=e^{T_0}(\dR \Gamma(\calW))\cdot e^{T_0}(-\dR\Gamma(\calE nd\,\calV)).
\end{align*}
By plugging in \eqref{equ on tan cpx at a pt2} and a direct calculation, we are done. 
%Taking the half $\dR \Gamma(\calW)$ of \eqref{equ on tan cpx at a pt}, plugging in \eqref{equ on tan cpx at a pt2}, and adding \eqref{equ on tan cpx at a pt3}, we get 
%$$\sqrt{\bbT^{\mathrm{vir}}_{(\calV,u)}QM^{R}_{d,\mathrm{sm},\infty}}=\dR \Gamma(\calV)+\sum_{i=1}^3\dR \Gamma(\calE nd\,\calV\otimes \calL_i)
%-\dR \Gamma(\calE nd\,\calV)$$
%in $K$-theory. Taking the Euler class finishes the proof.  
\end{proof}
Let $t_i$ ($i=1,2,3$) be the torus weights of $F$ and $\hbar_i=c_1^F(t_i)$. 
For $\square=(i_1,i_2,i_3)$, we introduce the following notations 
\[\langle\sigma,\square\rangle=
i_1\sigma_1+i_2\sigma_2+i_3\sigma_3,\quad d_\square=z_\square+\langle\sigma,\square\rangle, \quad \chi_\square=t_1^{-i_1}t_2^{-i_2}t_3^{-i_3},\quad \hbar_{\square}=-\sum_{j=1}^3i_j\hbar_j.  \]
Then   
\begin{eqnarray*}
\calL_1^{-i_1}\calL_2^{-i_2}\calL_3^{-i_3}\calO(z_{i_1,i_2,i_3})\,q^{z_{\square}}
=\calO(z_{\square}+\langle \sigma,\square\rangle)\,\chi_\square \,q^{z_\square}
=\calO(d_\square)\,\chi_\square\, q^{d_\square-\langle \sigma,\square\rangle}.
\end{eqnarray*}
Using Lemma \ref{lem on vir tange of qm p1}, we obtain the following \textit{explicit calculations} of vertex functions.
\begin{prop}\label{prop:hypergeom}
Notations as above, we have  
\begin{align*}
|\,\tau\rangle_\lambda(z)&=\sum_{(z_\square)_{\square\in\lambda}}z^{\sum_{\square\in\lambda}d_\square}\frac{e^{T_0}(\tau(\calV)|_{\{0\}})e^{T_0}(\chi_{\bbP^1}(\calE nd(\calV)))}{e^{T_0}(\chi_{\bbP^1}(\calV))e^{T_0}(\chi_{\bbP^1}(\bigoplus_{i=1}^3\calE nd(\calV)\otimes\calL_i))}\\
&=\sum_{(z_\square)_{\square\in\lambda}}z^{\sum_{\square\in\lambda}d_\square} \frac{e^{T_0}(\tau(\sum_{\square\in\lambda}\calL_\square)|_{\{0\}})e^{T_0}(\chi_{\bbP^1}(\sum_{\square,\square'\in\lambda}\calL_{\square'}^{-1}\otimes\calL_\square))}{e^{T_0}(\chi_{\bbP^1}(\sum_{\square\in\lambda}\calL_\square))e^{T_0}(\chi_{\bbP^1}(\bigoplus_{i=1}^3\sum_{\square,\square'\in\lambda}\calL_{\square'}^{-1}\otimes\calL_\square\otimes\calL_i))} \\
&=\sum_{(z_\square)_{\square\in\lambda}}\left(\frac{z}{2\hbar}\right)^{\sum_{\square\in\lambda}d_\square}(2\hbar)^{-|\lambda|}\tau(\hbar_\square-\langle\sigma,\square\rangle \hbar+2d_\square \hbar)\,\cdot \prod_{\square\in\lambda}\frac{\Gamma_{2\hbar}(\hbar_\square+\langle\sigma,\square\rangle \hbar-(d_\square)2\hbar)}{\Gamma_{2\hbar}(\hbar_\square+\langle\sigma,\square\rangle \hbar+2\hbar)}  \\
& \quad \times \frac{\prod_{i=1}^3\prod_{\square,\square'\in\lambda}\frac{\Gamma_{2\hbar}(\hbar_\square-\hbar_{\square'}+\langle\sigma,\square\rangle \hbar-\langle\sigma,\square'\rangle \hbar-(d_\square-d_{\square'})2\hbar+\hbar_i+\sigma_i \hbar)}{\Gamma_{2\hbar}(\hbar_\square-\hbar_{\square'}+\langle\sigma,\square\rangle \hbar-\langle\sigma,\square'\rangle \hbar+\hbar_i-\sigma_i \hbar+2\hbar)}}{\prod_{\square,\square'\in\lambda}
\frac{\Gamma_{2\hbar}(\hbar_\square-\hbar_{\square'}+\langle\sigma,\square\rangle \hbar-\langle\sigma,\square'\rangle \hbar-(d_\square-d_{\square'})2\hbar)}{\Gamma_{2\hbar}(\hbar_\square-\hbar_{\square'}+\langle\sigma,\square\rangle \hbar-\langle\sigma,\square'\rangle \hbar+2\hbar)}. }
\end{align*}  
%\gufang{The plus $\sigma_i\hbar$ on the numerator (showing up with the d's) is important for the contour integral to work. The denominator is a constant (independent on d's) and hence does not matter.}
Similarly, 
\begin{align*}
|\,\tau\rangle_\lambda(z,m)&=\sum_{(z_\square)_{\square\in\lambda}}z^{\sum_{\square\in\lambda}d_\square}\frac{e^{T_0\times \mathbb{C}^*}(\chi_{\bbP^1}(\calV)\otimes e^m)\,e^{T_0}(\tau(\calV)|_{\{0\}})\,e^{T_0}(\chi_{\bbP^1}(\calE nd(\calV)))}{e^{T_0}(\chi_{\bbP^1}(\calV))\,e^{T_0}(\chi_{\bbP^1}(\bigoplus_{i=1}^3\calE nd(\calV)\otimes\calL_i))}\\
&=\sum_{(z_\square)_{\square\in\lambda}}z^{\sum_{\square\in\lambda}d_\square}\tau(\hbar_\square-\langle\sigma,\square\rangle \hbar+2d_\square \hbar)\cdot \frac{\prod_{\square\in\lambda}\frac{\Gamma_{2\hbar}(\hbar_\square+\langle\sigma,\square\rangle \hbar-(d_\square)2\hbar)}{\Gamma_{2\hbar}(\hbar_\square+\langle\sigma,\square\rangle \hbar+2\hbar)}}{\prod_{\square\in\lambda}\frac{\Gamma_{2\hbar}(\hbar_\square+\langle\sigma,\square\rangle \hbar-(d_\square)2\hbar+m)}{\Gamma_{2\hbar}(\hbar_\square+\langle\sigma,\square\rangle \hbar+2\hbar+m)}}  \\
&\times \frac{\prod_{i=1}^3\prod_{\square,\square'\in\lambda}\frac{\Gamma_{2\hbar}(\hbar_\square-\hbar_{\square'}+\langle\sigma,\square\rangle \hbar-\langle\sigma,\square'\rangle \hbar-(d_\square-d_{\square'})2\hbar+\hbar_i+\sigma_i \hbar)}{\Gamma_{2\hbar}(\hbar_\square-\hbar_{\square'}+\langle\sigma,\square\rangle \hbar-\langle\sigma,\square'\rangle \hbar+\hbar_i-\sigma_i \hbar+2\hbar)}}{\prod_{\square,\square'\in\lambda}
\frac{\Gamma_{2\hbar}(\hbar_\square-\hbar_{\square'}+\langle\sigma,\square\rangle \hbar-\langle\sigma,\square'\rangle \hbar-(d_\square-d_{\square'})2\hbar)}{\Gamma_{2\hbar}(\hbar_\square-\hbar_{\square'}+\langle\sigma,\square\rangle \hbar-\langle\sigma,\square'\rangle \hbar+2\hbar)}.  }
\end{align*}
\end{prop}

\begin{remark}\label{rmk:convention_gamma}
The factor $\Gamma_{2\hbar}(\hbar_\square+\langle\sigma,\square\rangle \hbar-(d_\square)2\hbar)$ for $\square=(0,0,0)$ is $\Gamma_{2\hbar}(-d_{(0,0,0)}2\hbar)$ with $d_{(0,0,0)}\in\bbN$, hence undefined. Nevertheless, by \eqref{eqn:residue0},~\eqref{eqn:residue},~\eqref{eqn:residue2}, we understand it as 
\begin{align*}\Gamma_{2\hbar}(-d_{(0,0,0)}2\hbar)&=\frac{(-1)^{d_{(0,0,0)}}}{d_{(0,0,0)}!}\cdot\Gamma_{2\hbar}(0)=\frac{(-1)^{d_{(0,0,0)}}}{d_{(0,0,0)}!}\cdot \frac{1}{e^{T_0}(1)}, 
\end{align*}
where $e^{T_0}(1)=0$.
Similarly, in the formula of $|\,\tau\rangle_\lambda(z)$ and $|\,\tau\rangle_\lambda(z,m)$, many factors are undefined because they have poles. 
To make sense of the expressions in Proposition~\ref{prop:hypergeom}, we interpreter all such ratios as above. By Proposition \ref{prop on red points}, we know all the $e^{T_0}(1)$-factors in the denominator are cancelled by some $e^{T_0}(1)$-factors in the numerator,
therefore the expressions are well-defined. 
\end{remark}

%We leave the analysis of convergence to future investigations. 
In what follows, we use Ansatz~\ref{ansatz of saddle pt} to explore potential representation theory behind.
\subsection{Contour integral}\label{sect on mbi}

%\yl{is it clear what Mellin-Barnes integral means?}
As in \cite[\S 1.1.6]{AO}, \cite[Prop.~4.1]{PSZ}, one can  use \textit{Cauchy residue formula} to write the generating series in Proposition~\ref{prop:hypergeom} in terms of a contour integral. 
%For Mellin-Barnes integral method, we refer to \cite{AAR} for some reference. 

Define the following  
$$A_1^{\lambda}:=e^F(T^{\mathrm{vir}}_\lambda\Hilb^n(\bbC^3))=\frac{\prod_{\square\in \lambda}\hbar_\square\prod_{s=1}^3\prod_{\square,\square'\in\lambda}(\hbar_\square-\hbar_{\square'}+\hbar_s)}{\prod_{\square,\square'\in\lambda}(\hbar_\square-\hbar_{\square'})}, $$
which is well-defined and non-zero by \cite[Lem.~4.1]{BF2}, and 
$$\bar{A}_1^{\lambda}:=\frac{\prod_{\square\in \lambda}\hbar_\square\prod_{s=1}^3\prod_{\square,\square'\in\lambda}(\hbar_\square-\hbar_{\square'}+\hbar_s)}{\prod_{\square\neq\square'\in\lambda}(\hbar_\square-\hbar_{\square'})}, $$ 
\[A_2^\lambda:=
\frac{\prod_{\square,\square'\in\lambda}\Gamma_{2\hbar}(\hbar_\square-\hbar_{\square'}+\langle\sigma,\square\rangle \hbar-\langle\sigma,\square'\rangle \hbar+2\hbar)}
{\prod_{\square\in\lambda}\Gamma_{2\hbar}(\hbar_\square+\langle\sigma,\square\rangle \hbar+2\hbar)\prod_{i=1}^3\prod_{\square,\square'\in\lambda}\Gamma_{2\hbar}(\hbar_\square-\hbar_{\square'}+\langle\sigma,\square\rangle \hbar-\langle\sigma,\square'\rangle \hbar+\hbar_i-\sigma_i \hbar+2\hbar)}, \]
$$A_\lambda:=A_2^\lambda\times \frac{\bar{A}_1^{\lambda}}{A_1^{\lambda}}, $$
$$A_\lambda(m):=A_2^\lambda\times \frac{\bar{A}_1^{\lambda}}{A_1^{\lambda}}\times\prod_{\square\in\lambda}\frac{\Gamma_{2\hbar}(\hbar_\square+\langle\sigma,\square\rangle \hbar+m+2\hbar)}{2\hbar}. $$ 

\begin{prop}\label{prop:MB_int}
Notations as above,  we have 
\begin{eqnarray*} 
|\,\tau\rangle_\lambda(z,m)
=\int_{C}A_\lambda(m) \prod_{\square\in\lambda}\tau(2\hbar_\square-s_\square) z^{\frac{s_\square-\hbar_{\square}-\langle\sigma,\square\rangle\hbar}{(-2\hbar)}} \frac{\Gamma_{2\hbar}(s_\square)}{\Gamma_{2\hbar}(s_\square+m)}\prod_{\square,\square'\in\lambda}\frac{\prod_{i=1}^3\Gamma_{2\hbar}(s_\square-s_{\square'}+\hbar_i+\sigma_i\hbar)}{\Gamma_{2\hbar}(s_\square-s_{\square'})}\prod_{\square\in\lambda}ds_\square,
\end{eqnarray*}
%\yl{where are $ds_i$'s in above integral?}
which is independent of $(s_\square)_{\square}$.
And $C$ is a real $n$-cycle determined by the properties
\begin{enumerate}
    \item in the $s_{(0,0,0)}$-plane, it encloses $s_{(0,0,0)}=-d(2\hbar)$ for any $d\in\bbN$;
    \item inductively, in the $s_{(i_1,i_2,i_3)}$-plane, it encloses 
    $$s_{(i_1+1,i_2,i_3)}-s_{(i_1,i_2,i_3)}+\sigma_1\hbar+\hbar_1=-d(2\hbar), \quad\forall\,\,\, d\in\bbN, $$  
    $$s_{(i_1,i_2+1,i_3)}-s_{(i_1,i_2,i_3)}+\sigma_2\hbar+\hbar_2=-d(2\hbar), \quad\forall\,\,\,d\in\bbN, $$  
    $$s_{(i_1,i_2,i_3+1)}-s_{(i_1,i_2,i_3)}+\sigma_3\hbar+\hbar_3=-d(2\hbar), \quad\forall\,\,\,d\in\bbN. $$
    \end{enumerate}
Similarly, we have  
\begin{eqnarray*}
|\,\tau\rangle_\lambda(z)=\int_{C}A_\lambda \prod_{\square\in\lambda}\tau(2\hbar_\square-s_\square) (z/2\hbar)^{\frac{s_\square-\hbar_{\square}-\langle\sigma,\square\rangle\hbar}{2\hbar}} \Gamma_{2\hbar}(s_\square)\prod_{\square,\square'\in\lambda}\frac{\prod_{i=1}^3\Gamma_{2\hbar}(s_\square-s_{\square'}+\hbar_i+\sigma_i\hbar)}{\Gamma_{2\hbar}(s_\square-s_{\square'})}\prod_{\square\in\lambda}ds_\square,
\end{eqnarray*}
which is independent of $(s_\square)_{\square}$, and $C$ is the same as above. 
\end{prop}
\begin{proof}
We evaluate the integral via iterated residues. For this purpose, we fix a linear order $s_i$ of the variables $s_\square$ in such a way that 
the orders of $s_{i_1+1,i_2,i_3}$, $s_{i_1,i_2+1,i_3}$, $s_{i_1,i_2,i_3+1}$ are all bigger than the order of $s_{i_1,i_2,i_3}$. 
%$s_{i+1}$ corresponds to $(i_1,i_2,i_3)$ with two of the coordinates equal to those corresponding to $s_i$, and the other coordinate is 1 more than that of $s_i$. 
Such an ordering always exists but might not be unique. 
In other words, we evaluate  
$$\Res_{s_n}\cdots\Res_{s_1}A_\lambda(m)  \prod_{\square\in\lambda}\tau(2\hbar_\square-s_\square)\left(z^{\frac{s_\square}{2\hbar}} \frac{\Gamma_{2\hbar}(s_\square)}{\Gamma_{2\hbar}(s_\square+m)}\right)\prod_{\square,\square'\in\lambda}\frac{\prod_{i=1}^3\Gamma_{2\hbar}(s_\square-s_{\square'}+\hbar_i+\sigma_i\hbar)}{\Gamma_{2\hbar}(s_\square-s_{\square'})},  $$
where the residue of $s_\square$ is taken at $s_\square=-d_\square 2\hbar+\hbar_\square+\langle\sigma,\square\rangle\, \hbar$. 

Now evaluate the iterated residue inductively. For example, the initial step is to evaluate $s_{(0,0,0)}$,  which is always the first in the above-mentioned order. 
The factor $\Gamma_{2\hbar}(s_{(0,0,0)})$ has a pole at $s_{(0,0,0)}=-d_{(0,0,0)}2\hbar$, the residue of which is 
$$(2\hbar)\frac{\Gamma_{2\hbar}(-d_{(0,0,0)}2\hbar)}{\Gamma_{2\hbar}(0)}, $$ 
where the ratio is understood the same way as in Remark~\ref{rmk:convention_gamma} hence well-defined and non-zero. In particular, $\Gamma_{2\hbar}(0)=\frac{(2\hbar)\Gamma_{2\hbar}(2\hbar)}{(0)}$ the factor $(0)$ is the corresponding factor from $A^\lambda_1$, and the $\Gamma_{2\hbar}(2\hbar)$-factor is the corresponding factor in $A_2^\lambda$. Similar for the inductive process. The iterated residue then is given by the formula stated above. 
\end{proof}
 
\subsection{Saddle point equations}\label{sect on spe}

\begin{prop}\label{prop:saddle}
At $\hbar\to 0$, critical points of the integrant of $|\,\tau\rangle_\lambda(z,m)$ in Proposition~\ref{prop:MB_int} are determined by the equation
\[z=\frac{s_i+m}{s_i}\prod_{s=1}^3\prod_{j\neq i}
\frac{s_i-s_j-\hbar_s}{s_i-s_j+\hbar_s}, \]
for any $i=1,\ldots,|\lambda|$.

Make substitution $\overline{z}=z2\hbar$ in $|\,\tau\rangle_\lambda(z)$. At $\hbar\to 0$, critical points of the integrant of $|\,\tau\rangle_\lambda(z)$ 
in Proposition~\ref{prop:MB_int} are determined by the equation
\begin{equation}\label{eqn:Bethe1shift}
\overline{z}=\frac{1}{s_i}\prod_{s=1}^3\prod_{j\neq i}
\frac{s_i-s_j-\hbar_s}{s_i-s_j+\hbar_s}, \end{equation}
for any $i=1,\ldots,|\lambda|$.
\end{prop}
\begin{proof}
We prove the first statement, as the second is proven in a similar way.
Recall Stirling's approximation formula. For $x$ contained in a bounded region, as $\hbar\to 0$, we have 
\[\ln\Gamma_{2\hbar}(x)=(1/2\hbar)(x(\ln(x)-\ln(2\hbar)-1)+o(\hbar)).\]
Keeping this in mind, apply $s_{\square}(\frac{\partial}{\partial s_{\square}})\ln(-)$
to 
\[z^{\frac{s_\square}{2\hbar}} \frac{\Gamma_{2\hbar}(s_\square)}{\Gamma_{2\hbar}(s_\square+m)}\prod_{\square,\square'\in\lambda}\frac{\prod_{i=1}^3\Gamma_{2\hbar}(s_\square-s_{\square'}+\hbar_i+\sigma_i\hbar)}{\Gamma_{2\hbar}(s_\square-s_{\square'})},\]
we obtain 
\begin{align*}& \quad \, \frac{s_{\square}}{2\hbar}(\ln(z)-\ln(s_{\square}+m)+\ln(s_{\square})-\sum_{\square'\neq\square}\ln(s_{\square}-s_{\square'})+\sum_{\square'\neq\square}\ln(s_{\square'}-s_{\square})\\
&-\sum_{s=1}^3\sum_{\square'\neq\square}\ln(s_{\square'}-s_{\square}+\hbar_s)+\sum_{s=1}^3\sum_{\square'\neq\square}\ln(s_{\square}-s_{\square'}+\hbar_s)+o(\hbar)).\end{align*}
Setting it to be zero, taking limit $\hbar\to 0$, and exponentiating, we obtain the desired equation. 
\end{proof}
%\begin{remark}By considering the $K$-theoretic version of the above vertex functions, one expect the $\hbar_s$ saddle point equation. \end{remark}

\subsection{Bethe equations}\label{subsec:Bethe}
%\gufang{this section is now complete}
We recall the \textit{Bethe equation} of the \textit{Fock space representation} of the affine Yangian of $\fg\fl_1$ as written in \cite[Eqn.~(6.1)]{FJMMa}: 
%\[qp^{-1}=\frac{a_i-\hbar_2-u}{a_i-u}\prod_{s=1}^3\prod_{j\neq i}\frac{a_i-a_j+\hbar_s}{a_i-a_j-\hbar_s}\hbox{,  } \quad i=1,\dots,n.\]
%Here we write the functions additively and focus on the case when $k=1$ in loc.~cit.. 
%Notice that with the substitution $s_i=a_i-u$ and $z=qp^{-1}$, this is a special case of Proposition~\ref{prop:saddle} with $m=-\hbar_2$.
\[q^{-1}p=\frac{a_i-u}{a_i-\hbar_2-u}\prod_{s=1}^3\prod_{j\neq i}\frac{a_i-a_j-\hbar_s}{a_i-a_j+\hbar_s}\hbox{,  } \quad i=1,\dots,n.\]
Here we write the functions additively and focus on the case when $k=1$ in loc.~cit.. 
With the substitution $s_i=a_i-u-\hbar_2$, $z=q^{-1}p$, this is a special case of Proposition~\ref{prop:saddle} with $m=\hbar_2$.

It is known from \cite{RSYZ2} (see also  \cite{LY} for related study from physical point of view) that 
$$\bigoplus_nH^{\mathrm{crit}}_{F_0}(\Hilb^n(\mathbb{C}^3))$$ is a representation of the \textit{$(-1)$-shifted affine Yangian} $Y_{-1}(\widehat{\fg\fl_1})$. In particular, the Borel subalgebra action 
is constructed from the general \textit{cohomological Hall algebra} framework \cite{KS}. Different shifts of Yangians associated to the same Lie algebra have isomorphic Borel subalgebras. 

Now we give evidence that Eqn.~\eqref{eqn:Bethe1shift} is related to the Bethe equation of 
$Y_{-1}(\widehat{\fg\fl_1})$. 
Notice that the result of \cite[Cor.~5.7]{FJMMb} gives an algorithm of calculating the Bethe equation from the $q$-characters of the representation. Recall that
a Drinfeld fraction of a representation of the Borel subalgebra which lies in a certain category $\calO$  is a collection of rational functions of the form 
$$\prod_{i}(z-a_i)\prod_j(z-b_j)$$ 
with one rational function for each simple root. The degree of each rational function in such a collection agrees with the shift of the Yangian when the 
action of the Borel algebra extends to the action of a shifted Yangian \cite{HZ}. In the case the Lie algebra is $\widehat{\fg\fl_1}$, such a Drinfeld fraction is one single rational function $\psi$. The $q$-character of such a representation can be written in the form \cite[Eqn.~(4.30)]{FJMMb}: 
\[\chi_q=\bfm(\psi)(1+\sum_im_i)\chi_0, \]
where $\bfm(\psi)$ is determined by $\psi$ in an explicit way (which we omit here), and in turn determines a factor in the Bethe equation where $z$ is replaced by the variable $s_i$. The factor $\chi_0$ is not used in the algorithm. The factor $(1+\sum_im_i)$ determines a factor in the Bethe equation (the formula of which again we omit),  although we expect it to be independent of the shift of the Yangian. 

Our result in \eqref{eqn:Bethe1shift} is expected to be related to  the Bethe equation for $Y_{-1}(\widehat{\fg\fl_1})$-representations.  Although the Bethe Ansatz for such representations has not been studied, the algorithm \cite[Cor.~5.7]{FJMMb} can be formally applied.  In particular, 
the factor $1/s_i$ in \eqref{eqn:Bethe1shift} agrees with the Drinfeld fraction for representations of $Y_{-1}(\widehat{\fg\fl_1})$.

%\yl{Mention how quantum product structure interact with the COHA structure.}
We conclude this section with the context of the \textit{Bethe equations} obtained from \textit{quasimaps to quivers with potentials}. 
As has been mentioned, this is largely motivated by works of the Okounkov school \cite{AO,O,PSZ}. It is well-known that cohomology (resp.~$K$-theory) of Nakajima quiver varieties carry the structure of representations of the Yangians (resp.~quantum loop algebras) \cite{Nak, Var, MO}. Such representations also carry the structure of \textit{integrable systems}, known as the Casimir connection and the Knizhnik-Zamolodchikov connection (resp.~their $q$-analogues). These structures of integrable systems are realized geometrically as the \textit{quantum connections} and \textit{shift operators}. 

Nevertheless, there are large classes of representations of the Yangians which can \textit{not} be realized as cohomology of Nakajima quiver varieties. Indeed, if the Lie algebra is \textit{non-simply-laced}, the Yangians are constructed from quivers with potentials which do not reduce to symplectic quotients \cite{YZ}. For simply-laced Lie algebras, quivers with potentials are necessary to construct the 
\textit{higher spin representations}  \cite{BZ}. Moreover, \cite{RSYZ2} indicates that the construction of cohomological Hall algebras of more general quivers with potentials provides a generalized notion of Yangians, examples of which coming from toric local Calabi-Yau 3-folds are expected to recover shifted affine super Yangians. 
Therefore, it is natural to expect that \textit{quivers with potentials} provide a more general framework for geometric construction of \textit{quantum groups}, whose associated \textit{integrable systems} are expected to come from \textit{quasimaps to quivers with potentials}. 

More precisely, in the example of $\Hilb^n(\bbC^3)$, the quantum group in question is the $(-1)$-shifted affine Yangian which has a triangular decomposition 
$$Y_{-1}(\widehat{\fg\fl_1})=Y^{+}\otimes Y^0\otimes Y^{-}, $$ 
which acts on $\oplus_nH^{\mathrm{crit}}_{F_0}(\Hilb^n(\mathbb{C}^3))$. The algebra structure on $Y^{+}$ as well as its action are constructed via the usual framework of cohomological Hall algebra \cite{KS}. The algebra $Y^0$ is commutative, whose action is realized as cup product by tautological classes on 
$\oplus_n H^{\mathrm{crit}}_{F_0}(\Hilb^n(\mathbb{C}^3))$. The coproduct of $Y_{-1}(\widehat{\fg\fl_1})$ is expected to come from a stable envelope construction. The braiding on the module category is an $R$-matrix, which is expected to relate to some $\bfS$-operator. The quantum connections and the $\bfS$-operator form a commuting system.

\section{More examples of Bethe equations}\label{sect on more vertx}
Following the same strategy as above, one can compute vertex functions and saddle point equations for  
other quivers with potentials. In below, we give a brief overview for two more examples, one is the quiver with potential that 
describes perverse coherent systems on $\calO_{\bbP^1}(-1,-1)$ as 
studied by Nagao-Nakajima \cite{NN}, the other one defines the higher $\fs\fl_2$-spin chains in the lattice model.

\subsection{Perverse coherent systems on $\calO_{\bbP^1}(-1,-1)$}\label{sect vertex func on perv}
%We also refer to \cite{CT4} for a discussion of perverse coherent systems on $\calO_{\bbP^1}(-1,-1)\times \mathbb{C}$. 

\subsubsection{The target}
%We recall some terminology and results of \cite{NN}.
Let $m\in \mathbb{Z}_{>0}$ and consider the following quiver with potential (where $V_i$ denotes a complex vector space of dimension $v_i$ for $i=0,1$ in below):
$$
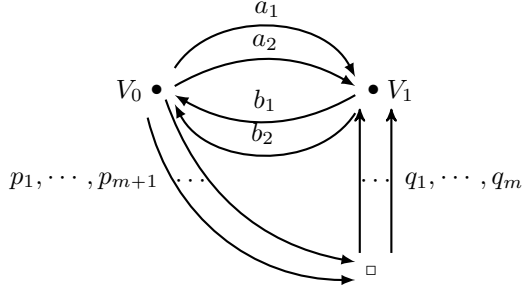
\begin{tikzpicture}[scale=0.6]
        	\node at (-2.4, 0) {$\bullet$};   	\node at (2.4, 0) {$\bullet$};  
		\draw[-latex,  bend left=30, thick] (-2, 0.1) to node[above]{$a_2$} (2, 0.1);
		\draw[-latex,  bend left=60, thick] (-2, 0.5) to node[above]{$a_1$} (2, 0.3);
	\draw[-latex,  bend left=30, thick] (2, -0.1) to node[above]{$b_1$} (-2, -0.1);
		\draw[-latex,  bend left=60, thick] (2, -0.5) to node[above]{$b_2$} (-2, -0.3);
				%\path (2.2, 0) edge [loop right, min distance=2cm, thick, bend right=40 ] node {} (2.2, 0);
				%\path (-2.2, 0) edge [loop left, min distance=2cm, thick, bend left=40 ] node {} (-2.2, 0);
					\node at (-3, 0) {$V_0$};\node at (3, 0) {$V_1$};
\node at (2.2, -4) {\, $\square$};  
 % \draw[->, thick] (-2.9, -3.8) -- (-2.9, -0.4) ;
\draw[->, thick] (2.8, -3.6) -- (2.8, -0.4) ;
 \node at (2.5, -2) {$\cdots$};  
 \draw[->, thick] (2.1, -3.6) -- (2.1, -0.4) ;
  \node at (4.41, -2) {$q_1, \cdots, q_m$};  
 
    \draw[-latex,  bend right=30, thick] (-2.2, -0.2) to node[above]{} (2, -3.8) ;   
     \node at (-1.6, -2) {$\cdots$};   
      \node at (-4, -2) {$p_1, \cdots, p_{m+1}$};  
     \draw[-latex,  bend right=38, thick] (-2.6, -0.6) to node[above]{} (2, -4.2) ; 
       \node at (0, -5) {The quiver $\tilde{Q}^+_{m}$ with potential $\phi_m$ given by}; 
       \node at (0, -5.8){$a_1b_1a_2b_2-a_1b_2a_2b_1+p_1b_1q_1+p_2(b_1q_2-b_2q_1)+\cdots+p_m(b_1 q_m-b_2q_{m-1})-p_{m+1}b_2q_m.$};
    \end{tikzpicture}
   $$   
Let $\fM^{(\tilde{Q}^+_{m},\phi_m)}_{\zeta_{\mathrm{cyclic}}}(v_0,v_1)$ be the corresponding moduli stack of \textit{cyclic stable framed representations}\footnote{One can also consider the quiver with potential 
$(\tilde{Q}^-_{m},\phi_m)$ as in \cite[\S 4.3 \& Fig.~10]{NN}, the whole section extends to this setting.} (ref.~\cite[\S 4.3 \& Fig.~9]{NN}\footnote{Here we use the labelling in the arxiv version of \cite{NN}.}). 
It is a projective scheme which parametrizes stable perverse coherent systems in stability chambers between PT chamber and the empty chamber in 
\cite[Fig.~1]{NN} (e.g.~it recovers the moduli space 
of PT stable pairs on $\calO_{\bbP^1}(-1,-1)$ when $m\to \infty$).

%\begin{remark} When $m\to \infty$, $\fM^{(\tilde{Q}^+_{m},\phi_m)}_{\zeta_{\mathrm{cyclic}}}(v_0,v_1)$ recovers the moduli space of PT stable pairs on $\calO_{\bbP^1}(-1,-1)$.  \end{remark}
   
We define an acton of $F=(\bbC^*)^3$ on $\fM^{(\tilde{Q}^+_{m},\phi_m)}_{\zeta_{\mathrm{cyclic}}}(v_0,v_1)$ as follows: for $(t_1,t_2,t_3)\in (\bbC^*)^3$, it acts trivially on $b_1$, scales $b_2$ by $t_3$, scales $a_1$ by $t_1$, scales $a_2$ by $t_2$, scales $q_i$ by $t_3^{i-1}$, and scales $p_i$ by $t_1t_2 t_3^{2-i}$.
It is straightforward to check that the torus weight of the potential $\phi_m$ is $t_1t_2t_3$. 
In particular, the \textit{Calabi-Yau subtorus} $F_0=\{(t_1,t_2,t_3)\in F\,|\,t_1t_2 t_3=1\}$ preserves $\phi_m$.  

\subsubsection{Torus fixed representations} 

%\yl{I do some revise on how to define pyramid partitions (avoid mentioning empty room configuration), see whether you agree}
\iffalse
Consider combinatorial arrangements as in the figure below:
\begin{equation*}
\begin{tikzpicture}[scale=0.3]
\draw[fill=black!40!white] (-2, 0) circle (1cm);
 \draw[fill=black!40!white] (0, 0) circle (1cm);
 \draw[fill=black!40!white] (2, 0) circle (1cm);
  
 \draw[fill=white] (-1, 0) circle (1cm); 
 \draw[fill=white] (1, 0) circle (1cm);
    
\draw[fill=black!40!white] (1, 1) circle (1cm);
\draw[fill=black!40!white] (1, -1) circle (1cm);
\draw[fill=black!40!white] (-1, -1) circle (1cm);
\draw[fill=black!40!white] (-1, 1) circle (1cm);

   \draw[fill=white] (0, 1) circle (1cm); \draw[fill=white] (0, -1) circle (1cm); 
   
   \draw[fill=black!40!white] (0, 0) circle (1cm);
   \draw[fill=black!40!white] (0, -2) circle (1cm);
      \draw[fill=black!40!white] (0, 2) circle (1cm);    
    \end{tikzpicture}
\end{equation*}
\fi
%\begin{definition}(\cite[\S 2.4]{Sze}, \cite[\S 4.5]{NN})\label{pyramid partition}
We refer to \cite[\S 2.4]{Sze}, \cite[\S 4.5]{NN} for the definition of 
\textit{a finite type pyramid partition of length} $m$, which is a finite subset $\lambda$ of a combinatorial arrangement of stones (with $m$ stones in top) such that for every stone
in $\lambda$, the stones directly above it are also in $\lambda$.
%\end{definition}
\begin{example}
The following are examples of 
pyramid partitions (ref.~\cite[Figure~12]{NN}):

%The following are the pictures for some finite type
%pyramid partitions with length 3 and length 4 (ref.~\cite[Figure~12]{NN}):
\begin{equation*}
\begin{tikzpicture}[scale=0.32]
\draw[fill=black!40!white] (-2, 0) circle (1cm);
 \draw[fill=black!40!white] (0, 0) circle (1cm);
 \draw[fill=black!40!white] (2, 0) circle (1cm);
  
 \draw[fill=white] (-1, 0) circle (1cm); 
 \draw[fill=white] (1, 0) circle (1cm);
    
\draw[fill=black!40!white] (1, 1) circle (1cm);
\draw[fill=black!40!white] (1, -1) circle (1cm);
\draw[fill=black!40!white] (-1, -1) circle (1cm);
\draw[fill=black!40!white] (-1, 1) circle (1cm);

   \draw[fill=white] (0, 1) circle (1cm); \draw[fill=white] (0, -1) circle (1cm); 
   
   \draw[fill=black!40!white] (0, 0) circle (1cm);
   \draw[fill=black!40!white] (0, -2) circle (1cm);
      \draw[fill=black!40!white] (0, 2) circle (1cm);    
      \node at (0, -5) {A finite type };
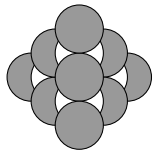
        \node at (0, -6.2) {pyramid partition with length 3};
\end{tikzpicture}
\,\ \,\ \,\ \,\ 
\begin{tikzpicture}[scale=0.32]
      \draw[fill=black!40!white] (-3, 0) circle (1cm);
   \draw[fill=black!40!white] (-1, 0) circle (1cm);
      \draw[fill=black!40!white] (1, 0) circle (1cm);
   \draw[fill=black!40!white] (3, 0) circle (1cm);

   \draw[fill=white] (0, 0) circle (1cm); 
   \draw[fill=white] (-2, 0) circle (1cm); 
   \draw[fill=white] (2, 0) circle (1cm); 

\draw[fill=black!40!white] (0, -1) circle (1cm);
\draw[fill=black!40!white] (-2, -1) circle (1cm);
\draw[fill=black!40!white] (2, -1) circle (1cm);
\draw[fill=black!40!white] (0, 1) circle (1cm);
\draw[fill=black!40!white] (-2, 1) circle (1cm);
\draw[fill=black!40!white] (2, 1) circle (1cm);

\draw[fill=white] (1, 1) circle (1cm);
\draw[fill=white] (1, -1) circle (1cm);
\draw[fill=white] (-1, -1) circle (1cm);
\draw[fill=white] (-1, 1) circle (1cm);

\draw[fill=black!40!white] (-1, 0) circle (1cm);
\draw[fill=black!40!white] (-1, -2) circle (1cm);
\draw[fill=black!40!white] (-1, 2) circle (1cm);
\draw[fill=black!40!white] (1, 0) circle (1cm);
\draw[fill=black!40!white] (1, -2) circle (1cm);
\draw[fill=black!40!white] (1, 2) circle (1cm);

   \draw[fill=white] (0, 0) circle (1cm); 
   \draw[fill=white] (0, -2) circle (1cm); 
   \draw[fill=white] (0, 2) circle (1cm); 
   
      \draw[fill=black!40!white] (0, -3) circle (1cm);
   \draw[fill=black!40!white] (0, -1) circle (1cm);
      \draw[fill=black!40!white] (0, 1) circle (1cm);
   \draw[fill=black!40!white] (0, 3) circle (1cm);
   \node at (0, -5) {A finite type };
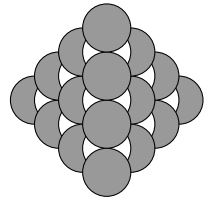
     \node at (0, -6.2) {pyramid partition with length 4};
\end{tikzpicture}
\end{equation*}
\end{example}
\begin{remark}
In general, finite type pyramid partitions with length $m$ consists of\,: $1\times m$ black stones on
the first layer, $1\times (m-1)$ white stones on the second layer,  $2 \times (m-1)$ black stones on the
third, $2 \times (m-2)$ white stones on the fourth, and so on until we reach $m\times 1$ black stones. 
\end{remark}
The following result classifies torus fixed cyclic stable framed representations of $(\tilde{Q}^+_{m},\phi_m)$ in terms of pyramid partitions. 
\begin{prop}\emph{(\cite[Prop.~4.14]{NN})}
We have 
$$\fM^{(\tilde{Q}^+_{m},\phi_m)}_{\zeta_{\mathrm{cyclic}}}(v_0,v_1)^F=\fM^{(\tilde{Q}^+_{m},\phi_m)}_{\zeta_{\mathrm{cyclic}}}(v_0,v_1)^{F_0}, $$ 
which is a finite number of reduced points and parameterized by finite type pyramid partitions of length $m$ with $v_0$ white stones and $v_1$ black stones.
\end{prop}

%Write $\fM^{\zeta_m}=\sqcup_{(v_0,v_1)\in\bbN^2}\fM^{\zeta_m}(v_0,v_1)$. 
At each $F$-fixed point $\lambda$, the tautological bundles $V_0$ and $V_1$ have basis labelled respectively by the white and black stones in $\lambda$. Each basis element spans a $F$-weight space, with the $F$-weight determined by the position of the stone. 
\begin{example}\label{ex:weights}
For example, by our conventions on the torus action, the weights of the black stones on the top layer of a finite type pyramid partition are
\[
1;\, t_3;\, t_3^2;\, \ldots;\, t_3^{m-1}. 
\]
The weights of the black stones on the layer 3 are 
\[
t_1t_3,\, t_2t_3\,;\,  t_1t_3^2,\, t_2t_3^2\,; \, \ldots; t_1t_3^{m-1},\, t_2t_3^{m-1}. 
\]
The weights of the black stones on the layer 5 are
\[
t_1^2t_3^2, \, t_1t_2 t_3^2,\, t_2^2t_3^2\,;\, t_1^2t_3^3,\, t_1t_2 t_3^3\,, t_2^2t_3^3\,;\, \ldots,\, t_1^2t_3^{m-1},\, t_2^2t_3^{m-1}. 
\]
The last one is the $2(m-1)+1=(2m-1)$-th layer, where the weights of the black stones  are
\[
t_1^{m-1}t_3^{m-1},\, t_1^{m-2}t_2t_3^{m-1},\, t_1^{m-3}t_2^2t_3^{m-1},\,\ldots,\, t_2^{m-1}t_3^{m-1}. \]
\end{example}

\subsubsection{Torus fixed quasimaps}

% Fix $P_0$ to be the principal $\mathbb{C}^*$-bundle obtained from $\oO_{\mathbb{P}^1}(1)$ by removing the zero section, and 
 Fix an $R$-twist as \eqref{equ on choice of R charge}, where $\sigma_i\in \bbZ$ ($i=1,2,3$) such that 
$$-\sigma_1-\sigma_2-\sigma_3=-2. $$ 
%Let $\calL_{i}=\calO(-\sigma_i)$ be line bundles on $\bbP^1$.
As in \S \ref{sect on coho of p1}, we define $\bbC^*_q$ action on $\bbP^1$ and   
$$T=F\times \bbC^*_q, \quad T_0=F_0\times \bbC^*_q. $$ 
%which acts on the total space $X=\mathrm{Tot}_{\bbP^1}(\calL_1\oplus\calL_2\oplus\calL_3)$. 
For $\lambda\in \fM^{(\tilde{Q}^+_{m},\phi_m)}_{\zeta_{\mathrm{cyclic}}}(v_0,v_1)^F$, as in the previous section, we define the moduli stack 
$$QM^R_{d}(\bbP^1,\fM^{(\tilde{Q}^+_{m},\phi_m)}_{\zeta_{\mathrm{cyclic}}}(v_0,v_1))_{\infty=\lambda}$$
of $R$-twisted quasimaps to $\fM^{(\tilde{Q}^+_{m},\phi_m)}_{\zeta_{\mathrm{cyclic}}}(v_0,v_1)$ which are smooth at $\infty$ and whose image 
under evaluation map $ev_{\infty}$ is $\lambda$.
Here $d=(d_0,d_1)\in\bbZ^2$ denotes the degree. 
As $F$ acts on the target and $\bbC^*_q$ scales the $\bbP^1$, the moduli stack has a natural $T$-action.
\begin{notation}
For a finite type pyramid partition, we write $\square$ for a stone in it regardless of its color,  $\circ$ for a white stone, and $\bullet$ for a black stone.

We use the following terminology in order to describe the relative position of one stone with respect to another one:
\[\begin{tikzpicture}[scale=0.52]
       \node at (0, 1.5) {front};
       \draw[-latex, thick] (0, 0.2) to  (0, 1);      
       \draw[-latex, thick] (0, -0.2) to  (0, -1);
       \node at (0, -1.5) {back }; \node at (0, -2.3) {($t_3$)};
        \draw[-latex, thick] (-0.2, 0) to  (-1, 0);
        \node at (-1.9, 0) {left};
         \node at (-1.9, -0.8) {($t_1$)};
         \draw[-latex, thick] (0.2, 0) to  (1, 0);
        \node at (1.9, 0) {right};
          \node at (1.9, -0.8) {($t_2$)};
       \end{tikzpicture} \]
Furthermore, in this  terminology the word \textit{above} will mean {\it up with respect to the paper surface}, and \textit{below} will mean {\it down with respect to the paper surface}. 
\end{notation}
%\yl{How about $T_0$-fixed locus?}
\begin{prop} \label{lem:fixed_qm_pt}
Let $\lambda$ be a finite type pyramid partition of length $m$ with $v_0$ white stones and $v_1$ black stones. 
For each $\circ$, we denote the black stone in front of (and above) it by $\bullet_{\mathrm{front}}$, the black stone at the back of (and above) it by 
$\bullet_{\mathrm{back}}$. For each $\bullet$, we denote the white stone on the left of (and above) it by $\circ_{\mathrm{left}}$, the white stone on the right of (and above) it by $\circ_{\mathrm{right}}$. 

Then the $T_0$-fixed points of $QM^R_{d}(\bbP^1,\fM^{(\tilde{Q}^+_{m},\phi_m)}_{\zeta_{\mathrm{cyclic}}}(v_0,v_1))_{\infty=\lambda}$ are finite and labelled by the following data: for each stone $\square$, we associate a number $d_\square\in\bbZ$. The collection $(d_\square)_{\square\in\lambda}$ is subject to the following conditions:\begin{align}\label{eqn:fixed}
d_{\bullet}\geqslant  (i-1)\sigma_3 \hbox{ if  }\bullet\hbox{ is the $i$-th stone on the first layer,}\\\notag
d_{\circ}\geqslant  d_{\bullet_{\mathrm{front}}}+\sigma_3 \hbox{ for each }\circ, \\\notag
d_{\circ}\geqslant  d_{\bullet_{\mathrm{back}}} \hbox{ for each }\circ, \\\notag
d_{\bullet}\geqslant  d_{\circ_{\mathrm{left}}}+\sigma_2 \hbox{ for each }\bullet, \\\notag
d_{\bullet}\geqslant  d_{\circ_{\mathrm{right}}}+\sigma_1 \hbox{ for each }\bullet,
\end{align}
with $\sum_{\circ\in\lambda} d_\circ=d_0$ and $\sum_{\bullet\in\lambda}d_\bullet=d_1$.
\end{prop}
\begin{proof}
This is similar to Lemma \ref{lem on t fix pt}. On a fixed point, we have  
$$\calV_0=\oplus\,\calL_{\circ}, \quad \calV_1=\oplus\,\calL_{\bullet}$$ 
with each $\calL_\square$ a $T_0$-equivariant line bundle on $\bbP^1$ and all the maps are $T_0$-equivariant.
For example, the condition $d_{\circ}\geqslant  d_{\bullet_{\mathrm{front}}}+\sigma_3$  for each $\circ$ comes from the fact that the map 
$$\calL_{d_{\bullet_{\mathrm{front}}}}\otimes\calL_3^{-1}\to \calL_\circ$$ 
of vector bundles on $\bbP^1$ is  $\bbC^*_q$-equivariant and non-zero at $\infty\in\bbP^1$.
\end{proof}

\subsubsection{Vertex functions and saddle point equations}\label{sect on sapq for nn}

As in Definition \ref{def of mid tau lamb}, one can define the vertex function without insertions: 
$$|\emptyset\rangle_\lambda(z)\in A_*^{T_0}(\pt)_{loc}[\![z]\!], $$
and can explicitly compute it as Proposition \ref{prop:hypergeom}. As the expression is very complicated and not so suggestive, we do not present it here.
By the Cauchy integral method as in Proposition \ref{prop:MB_int}, 
we can write the vertex function as a contour integral, with variables $s^0_i$ ($i=1,\dots,v_0$) and $s^1_j$ ($j=1,\dots,v_1$).
As in Proposition~\ref{prop:saddle}, we obtain the following \textit{saddle point equations}: 
\[\overline z_0={\prod_{i=1}^m(s_j^0+(1-i)\hbar_3)}\frac{\prod_{i=1}^{v_1}(s_j^0-s_i^1-\hbar_3)(s_j^0-s_i^1)}{\prod_{i=1}^{v_1}(s_j^0-s_i^1+\hbar_1)(s_j^0-s_i^1+\hbar_2)}\,\,\hbox{ for each }\,\, j=1,\ldots,v_0,\]
\[\overline z_1=\frac{1}{\prod_{i=1}^{m+1}(s_j^1+(1-i)\hbar_3)}\frac{\prod_{i=1}^{v_0}(s_j^1-s_i^0-\hbar_1)(s_j^1-s_i^0-\hbar_2)}{\prod_{i=1}^{v_0}(s_j^1-s_i^0+\hbar_3)(s_j^1-s_i^0)}\,\,\hbox{ for each }\,\, j=1,\ldots,v_1,\]
where $\hbar_k:=c_1^F(t_k)$ $(k=1,2,3)$ are the equivariant parameters.  

Based on  calculations, it is expected from \cite{RSYZ2} that the direct sum
$$\bigoplus_{(v_0,v_1)}H^{\mathrm{crit}}_{F_0}(\fM^{(\tilde{Q}^+_{m},\phi_m)}_{\zeta_{\mathrm{cyclic}}}(v_0,v_1))$$ 
of critical cohomologies  carries the structure of a representation of the shifted super affine Yangian of $\fg\fl(1|1)$. Therefore, it is natural to expect the above equations to be related to  
\textit{Bethe equations} of the \textit{shifted super affine Yangian} of $\fg\fl(1|1)$, 
which to our knowledge has not been worked out from representation theoretic point of view.

\subsection{Higher $\fs\fl_2$-spin chains}\label{sect on more bethe equ}
%\gufang{this part is now complete}
%In this section, we give more examples of Bethe equations calculated using quasimaps to quivers with potentials. 
%Let us consider the following example that recovers  well-studied equations from Bethe Ansatz of higher $\fs\fl_2$-spin chains. 
For any $k\in \mathbb{Z}_{>0}$, consider the following quiver with potential   
\begin{equation}
\xymatrix{
 \square_{} \ar@/^0.5pc/[r]^{Q} & \circ  \ar@/^0.5pc/[l]^{\overline{Q}}  \ar@(dr,ur)_{l} &  \quad \quad  \phi=\tr(l^kQ\overline{Q}). }
\nonumber 
\end{equation}
They are related to non-simply-laced Yangians and higher spin representations of simply-laced Yangians, which originate from physical literature including \cite{NS2,Ce,CD}.
 
Fix the dimension vector to be $N$ at the square node and $n$ at the circular node. Let 
$$G:=\GL_n, \quad F:=(\bbC^*)^N\times(\bbC^*)^2, $$ 
where $(\bbC^*)^N\subseteq\GL_N$ is the maximal torus with coordinates $(e^{a_1},\dots,e^{a_N})$ and the action of $(t_1,t_2)\in (\bbC^*)^2$ is given by scaling the arrow $Q$ by $t_1$ and the arrow $l$ by $t_2$. 
%The coordinates of $(\bbC^*)^N$ are written as $(a_1,\dots,a_N)$.
Define the $G$-character 
$$\theta:={\det}^{-1}: G\to \bbC^*, $$ 
which gives the usual cyclic stability. Define the $F$-character
$$\chi:F\to \bbC^*, \quad (e^{a_1},\dots,e^{a_N},t_1,t_2)\mapsto t_1t_2^k. $$
Then the Calabi-Yau subtorus is 
$$F_0=\left\{(e^{a_1},\dots,e^{a_N},t_1,t_2)\,|\, t_1t_2^k=1\right\}. $$ 
By definition, 
$$W=\Hom(\bbC^N,\bbC^n)\times\Hom(\bbC^n,\bbC^N)\times \Hom(\bbC^n,\bbC^n), $$ 
and the torus $F_0$-fixed points of $\Crit(\phi)/\!\!/G$ are labelled by the following set
\[\left\{\lambda=(k_1,\dots,k_N)\in\bbN^N\,\, \Big{|} \,\, 0\leqslant k_i\leqslant k,\,\, \sum_{i=1}^Nk_i=n \right\}.\]
In what follows it is convenient to consider $\lambda$ as an $N$-tuple of 1-dimensional Young diagrams with length no more than $k$ 
and $n$ many boxes in total. The position of each box $\circ\in\lambda$ is determined by a pair $(i,h)$ called its {\it coordinates}, where $i=1,\dots,N$ says the box lies in the $i$-th Young diagram, and $h\in\bbN$ says this is the $h$-th box in this 1-dimensional Young diagram.
Note that the cardinality of this set is equal to the dimension of $n$-th weight space in the $\fs\fl_2$-representation $\big(\Sym^k(\bbC^2)\big)^{\otimes N}$. Indeed, the action of the Yangian on the cohomology has been constructed by Bykov and Zinn-Justin \cite{BZ}.

Consider the moduli stack $QM^R_d(\bbP^1,\Crit(\phi)/\!\!/G)_{\infty=\lambda}$ of quasimaps 
%which are regular at $\infty$ and whose image under evaluation map $ev_{\infty}$ is $\lambda$ 
as \eqref{diag defin qm inf lamb} with $R$-charge:  
$$R: \bbC^*\to F, \quad t\mapsto (t^{\alpha_1},\ldots, t^{\alpha_N},t^{-\sigma_1},t^{-\sigma_2})$$ 
%the twist given by $$\calL_1=\calO(-\sigma_1)t_1, \quad \calL_2=\calO(-\sigma_2)t_2 $$ 
subject to the condition $\sigma_1+k\sigma_2=2$.
%as well as $\calL_{a_i}=\calO(\alpha_i)a_i$ for $i=1,\dots,N$.
%The same analysis as Lemma~\ref{lem:fixed_qm_pt} gives the following. 
\begin{prop}
The $(F_0\times\bbC^*_q)$-fixed points on $QM^R_d(\bbP^1,\Crit(\phi)/\!\!/G)_{\infty=\lambda}$ are labelled by tuples 
\[\Big\{(z_\circ)_{\circ\in\lambda}\, \big| \, \hbox{subject to }(1) \hbox{ and } (2)\Big\}.\]
Here for each $\circ\in\lambda$ with coordinates $(i,h)$, we write $\langle\circ ,\sigma\rangle:=h\sigma_2+\sigma_1-\alpha_i$, 
and $d_\circ=z_\circ+\langle\circ ,\sigma\rangle$. Then, the tuple $(z_\circ)_{\circ\in\lambda}$ is such that each $z_\circ\in\bbN$, and 
\begin{enumerate}
    \item[(1)] for each $i=1,\dots, N$, the sequence $(z_{i,h})_{h\in \bbN}$ form a 2-dimensional Young diagram;
    \item[(2)] $\sum_{\circ\in\lambda}d_\circ=d$.
\end{enumerate}
\end{prop}
We omit the proof, which is similar to that of Proposition~\ref{lem:fixed_qm_pt}.
The same calculation as in the proof of Proposition~\ref{prop:saddle} gives the following constraint of eigenvalues of quantum multiplication. 
Write $\hbar:=c_1^T(t_2)$, then we have the following equations for variables $s_i$'s: 
\[\prod_{j=1}^N\frac{s_i-a_j}{a_j-s_i+k\hbar}=z\prod_{j=1,j\neq i}^n\frac{s_i-s_j+\hbar}{s_i-s_j-\hbar}, \quad i=1,\dots,n. \]
By specializing $a_1=\cdots=a_N$ and $\hbar$ to certain values, 
these equations agree with the well-known Bethe equations of \textit{higher} $\fs\fl_2$-\textit{spin chains} in the lattice model e.g.~\cite[Eqn.~(54)]{Ba}, \cite{KRS}.
%which can be found in \cite{BZ} in the framework of $R$-matrices of $\fs\fl_2$-Yangians.
%\gufang{Ask Paul for a reference} 

When $k=1$, by dimensional reduction \cite[Thm.~A.1]{D}, we have an isomorphism \[H^*_{F_0}(X,\varphi_\phi)\cong H^*_{F_0}(T^*\Gr(n,N)). \]  
Following \S \ref{sec:dim_red}, the quasimap invariants of the quiver with potential above should recover those in \cite{PSZ}\footnote{To be more precise, \cite{PSZ} considers $K$-theoretic invariants. The 
$K$-theoretic version of our quasimap invariants will have dimensional reduction to theirs (see \S \ref{sec:dim_red}). Alternatively one can recover their cohomological limit from our (cohomological) invariants. }.
In particular, we see our Bethe equations recover theirs \cite[Thm.~2]{PSZ}\footnote{In loc.~cit., they use $K$-theory instead cohomology theory, 
so $\hbar$ is multiplied instead of being summed.} when $k=1$. 
%We leave the systematic investigation of such phenomenon in a forthcoming work. 

In the framework discussed in \S \ref{subsec:Bethe}, the above is an example of reproducing higher spin Bethe equations using quivers with potentials. 
We further expect that the same method can be applied to recover Bethe equations for \textit{non-simply-laced Yangian representations} as studied in \cite{YZ} and 
\textit{higher spin analogy} of $\fs\fl_n$-spin chains considered in \cite{KPSZ}.

\appendix 
\section{}
In this section, we recall the notions of Borel-Moore homology, vanishing cycle functor, critical cohomology and their basic properties. 
Some standard references are \cite{Fu, Iv, KaSc}. We also refer to \cite{KS} for the construction of cohomological Hall algebra (COHA) structures on critical cohomologies.

\subsection{Equivariantly localized pushforward}\label{sect on equi push}
When a map is proper, one has its pushforward in Chow groups. 
It is useful to extend the definition to the equivariantly proper setting. Let 
\begin{equation*} f:X\to X' \end{equation*}
be a $F_0$-equivariant map between Deligne-Mumford stacks, where $F_0$ is a torus. Let $j: Y\inj X$ be a $F_0$-invariant closed substack, by \cite[Thm.~5.3.5]{Kre}, there is
an isomorphism 
\begin{equation}\label{equ on toru loc}
i_{Y*}: A_*^{F_0}(Y^{F_0})_{loc}\stackrel{\cong}{\to} A_{*}^{F_0}(Y)_{loc}.\end{equation}
%\yl{I add the properness assumption below, see whether it is OK. }
Here for any $A_{*}^{F_0}(\pt)$-module $M$, we write its localization 
$$M_{loc}:=M\otimes_{A_{*}^{F_0}(\pt)}A_{*}^{F_0}(\pt)_{loc}, $$ 
where $A_{*}^{F_0}(\pt)_{loc}$ is the field of fractions of $A_{*}^{F_0}(\pt)$.
\begin{definition}\label{defi of localized push}
Assume $Y^{F_0}$ is proper over $X'$. We define an \textit{equivariantly localized pushforward}:
\begin{equation}\label{pf for nonproper}f_{*}: A_{*}^{F_0}(Y)_{loc} \xrightarrow{f_{Y^{F_0}*}\circ (i_{Y*})^{-1}} A_{*}^{F_0}(X')_{loc}, \end{equation}
%\yl{Notation $f_*$ is bad. Maybe use $f_{Y*}$}
where $f_{Y^{F_0}}:=f|_{Y^{F_0}}: Y^{F_0}\to X'$ is a proper map and $f_{Y^{F_0}*}$ is the usual pushforward. 
\end{definition}
One can similarly define an \textit{equivariantly localized pushforward} 
\begin{equation}\label{pf for nonproper bm}f_{*}: H^{BM}_{F_0}(Y)_{loc} \to H^{BM}_{F_0}(X')_{loc}, \end{equation}
for Borel-Moore homology (introduced in Eqn.~\eqref{def of bm ho}) by using the isomorphism 
$$i_{Y*}: H^{BM}_{F_0}(Y^{F_0})_{loc}\stackrel{\cong}{\to} H^{BM}_{F_0}(Y)_{loc} $$
due to \cite[Thm.~6.2]{GKM}. 

\subsection{Borel-Moore homology, vanishing cycle and critical cohomology} \label{sect on bm ho}
Let $D_c^b(X)$ be the bounded derived category of constructible sheaves of $\bbQ$-vector spaces on a complex algebraic variety $X$, and $\mathbb{D}_X$ be the Verdier duality functor for $D_c^b(X)$. 

If $X$ is smooth of dimension $d$, then $$\bbD_X(-)=(-)^\vee[2d]. $$ 
We also refer to the sheaf $\bbD_X(\bbQ_X)$  as the \textit{dualizing sheaf}, and use the shorthand 
$$\bbD_X=\bbD_X(\bbQ_X). $$ 
In particular, $\mathbb{D}_{\pt}$ is the vector space dual. We write the structure morphism of a complex algebraic variety as  $p_X: X\to \pt$.
 Then the \textit{Verdier dual} of the \textit{compactly supported cohomology} is
\begin{equation}\label{def of bm ho}H_{c}^*(X)^\vee:= \mathbb{D}_{\pt}(p_{X!} \Q_{X})=p_{X*}\mathbb{D}_X,\end{equation}
which is the \textit{Borel-Moore homology} $H^{BM}(X)$ of $X$ in the usual sense (e.g.~\cite[\S IX]{Iv}). 
We refer to \cite[\S 2.6]{KaSc}, \cite[\S C.2]{HTT} for basic properties of six functor formalism used in this paper. 

There is a cycle map \cite[\S 19]{Fu}:
\[cl: A_*(X)\to H^{BM}(X),\] 
which is a graded group homomorphism. We refer to \cite[\S 3.2]{KV} for a more general theory of Borel-Moore homology on stacks. 
%though it is not used anywhere in this paper. 
%\subsection{Vanishing cycle and critical cohomology}\label{sect on crit hom}

%\yl{Does this section generalize to the case when $X$ is a DM stacks? }

We will also consider the critical cohomology in the following setup. Let $X$ be a complex manifold and
$$\phi:X\to \bbC$$ 
be a regular function, referred to as the potential function. 
We define the \textit{functor of vanishing cycles}: 
$$\varphi_\phi: D_c^b(X)\to D_c^b(Z(\phi)), $$
\begin{equation}
    \label{eqn:van}
\varphi_\phi(F):=\dR \Gamma_{\mathrm{Re}(\phi)\geqslant  0}(F)|_{Z(\phi)}, \quad \mathrm{where} \,\, Z(\phi):=\phi^{-1}(0). 
\end{equation}
Here we use an equivalent definition due to \cite[Ex.~VIII 13]{KaSc}. 

By Remark \ref{rmk on crit emb to zero}, without loss of generality, we may assume for some $r\geqslant 1$, there is an embedding 
$$\Crit(\phi)\hookrightarrow Z(\phi^r),$$
and note that the underlying topological spaces of $Z(\phi^r)$ and $Z(\phi)$ are the same, so are their Borel-Moore homology. 
Denote the embedding of the zero locus by 
$$i:Z(\phi)\to X. $$
Recall the \textit{Milnor triangle} (also known as the canonical triangle):
\begin{equation}
    \label{eqn:Milnor}
    \psi_\phi\to\varphi_\phi\to i^*,
\end{equation}
which is a distinguished triangle of functors to $D_c^b(Z(\phi))$.
We are primarily interested in the complex $\varphi_\phi\Q_X$, which is supported on the critical locus of $f$. If $X$ is smooth, then  $\varphi_\phi\bbD_X$ is also supported on the critical locus. Without causing confusion, we also consider both $\varphi_\phi\Q_X$ and $\varphi_\phi\bbD_X$ as objects in $D_c^b(X)$.

The {\it critical cohomology}\footnote{One can say it is more appropriate to call it critical homology as we take the dual of a cohomology. 
Here we follow the convention from the literature and call it critical cohomology.} of $(X,\phi)$ is defined to be 
\[H_{c}(X,\varphi_\phi\Q_X)^\vee=\bbD_{\pt} p_{X!}\varphi_\phi\Q_X=p_{X*}\varphi_\phi\bbD_X.\] 
%\yl{do we know $\varphi_\phi\Q_X$ is related to some IC sheaves?}
For simplicity, we also denote this by $H(X,\varphi_\phi)$.
More generally, for any $A\in D^b_c(X)$, we denote 
$$H(X,A):=\bbD_{\pt}p_{X!}A. $$
The Milnor triangle \eqref{eqn:Milnor} gives a natural transformation $\varphi_\phi\to i^*$.
Using the description of $\varphi_\phi$ in Eqn.~\eqref{eqn:van}, this natural transformation is induced by 
\[\bfR\Gamma_{\mathrm{Re}(\phi)\geqslant  0}\to\id.\]
In particular, applying $\bbD_{\pt}p_{Z!}$, we obtain a \textit{canonical map} from BM homology to critical cohomology 
\begin{equation}\label{map can}can: H^{BM}(Z(\phi))=\bbD_{\pt}p_{Z!}\bbQ_Z=\bbD_{\pt}p_{Z!}i^*\bbQ_X\to \bbD_{\pt}p_{Z!}\varphi_\phi\bbQ_X=H(X,\varphi_\phi).\end{equation}
Without causing confusion, for any closed subscheme $V \subseteq Z(\phi^r)$ (with $r\geqslant 1$), we also denote the composition of $can$ with $H^{BM}(V)\to H^{BM}(Z(\phi^r))\cong H^{BM}(Z(\phi))$ as 
\begin{equation}\label{map can2}can: H^{BM}(V)\to H(X,\varphi_\phi). \end{equation}

\iffalse
\subsection{Gysin pullback}\label{sect on gys pul}
%\subsubsection{Gysin pullback}
Let $f:X\to Y$ be a map of complex manifolds, $W$ be a complex analytic space with an analytic map to $Y$ fitting into the following 
Cartesian diagram 
\begin{align*}\xymatrix{
Z \ar[r]^{g}\ar[d]_{h} \ar@{}[dr]|{\Box} & W\ar[d]^{k} \\
X \ar[r]^{f}& Y. }
\end{align*}
We can define a Gysin pullback map 
\begin{equation}\label{gysin map}f^!: H^{BM}(W)\to H^{BM}(Z), \end{equation}
given by the following procedure: by adjunction, we have
$$f_!\bbQ_X\to \bbQ_Y[-2\dim f], \quad \mathrm{where} \,\, \dim f:=\dim X-\dim Y.$$
Applying $k^*$ to it and using $g_!h^* \cong k^*f_! $ (ref.~\cite[Prop.~2.5.11]{KaSc}), we get 
\begin{align*}
g_!\Q_{Z}\cong g_!h^*\Q_{X}\cong k^*f_!\Q_{X}\to k^*\bbQ_Y[-2\dim f]\cong \Q_{W}[-2\dim f].
\end{align*}
Finally, applying $\mathbb{D}_{\pt}\circ p_{W!}$ to it, we are done. 
One can verify that this definition coincides with the Gysin pullback defined in \cite[\S 2.4]{KL} and commutes with proper pushforward. 
\fi

\subsection{Functoriality}\label{sect on functori}

The functor $\varphi$ is natural in the sense that if $f:X\to Y$ is a map of complex manifolds and $\phi:Y\to \bbC$ a regular function, 
then there is a natural transformation
\begin{equation}\label{eqn:phi_*}
    \varphi_\phi f_*\to f_*\varphi_{\phi\circ f},
\end{equation} and hence by duality, a transformation
\begin{equation}\label{eqn:phi_!}f_!\varphi_{\phi\circ f}\to \varphi_\phi f_!.\end{equation} 
Both of them agree and become natural isomorphisms when $f$ is proper (ref.~\cite[Ex.~VIII 15]{KaSc}).

As in the case of Borel-Moore homology, critical cohomology has \textit{functoriality} under \textit{pullbacks} and \textit{proper pushforwards} 
assuming potential functions are compatible. More precisely, let $f:X\to Y$ be a map of complex manifolds, and $\phi:Y\to \bbC$ a regular function.
Usual adjunction gives morphisms in $D_c^b(Y)$: 
$$\bbQ_Y\to f_*\bbQ_X, \quad f_!\bbQ_X\to \bbQ_Y[-2\dim f], \quad \mathrm{where} \,\, \dim f:=\dim X-\dim Y. $$
Applying $\bbD_{\pt}p_{Y!}\varphi_\phi$ to the latter, and composing with Eqn.~\eqref{eqn:phi_!}, we obtain 
\[\bbD_{\pt}p_{Y!}\varphi_\phi\bbQ_Y\to \bbD_{\pt}p_{Y!}\varphi_\phi f_!\bbQ_X[-2\dim f]\to \bbD_{\pt}p_{Y!}f_!\varphi_{\phi\circ f}\bbQ_X[-2\dim f],\]
which we denote by 
$$f^*: H(Y,\varphi_\phi)\to H_{*+2\dim f}(X,\varphi_{\phi\circ f}). $$
While the homological degree is useful in general, in the present paper we sometimes omit it for simplicity. 
Similarly, assuming $f$ is proper, applying $\bbD_{\pt}p_{Y!}\varphi_\phi$ to $\bbQ_Y\to f_*\bbQ_X$, using Eqn.~\eqref{eqn:phi_*} and the fact that $f_!=f_*$ for a proper map, 
we obtain the proper pushforward
\[f_*: H(X,\varphi_{\phi\circ f})=\bbD_{\pt}p_{Y!}f_!\varphi_{\phi\circ f}\bbQ_X\to
\bbD_{\pt}p_{Y!}\varphi_\phi f_*\bbQ_X\to \bbD_{\pt}p_{Y!}\varphi_\phi\bbQ_Y= H(Y,\varphi_\phi). \]

\subsection{Thom-Sebastiani isomorphism}\label{sect on ts iso}

%\subsubsection{Thom-Sebastiani isomorphism}

Given complex manifolds $X, Y$ with regular functions $\phi:X\to \bbC$,  $\phi': Y\to \bbC$, one can define
$\phi\boxplus \phi':X\times Y\to \bbC$ as the sum of the two pullback functions. 
Denote 
$$i_{Z(\phi)}\times i_{Z(\phi')}:=j\circ k: Z(\phi)\times Z(\phi')\stackrel{k}{\hookrightarrow}  Z(\phi\boxplus \phi') \stackrel{j}{\hookrightarrow}  X\times Y $$ 
to be the natural inclusions. 
There exists an isomorphism of functors from $D_c^b(X\times Y)$:
\begin{equation}\label{TS iso}k^*\varphi_{\phi\boxplus \phi'}\stackrel{\mathrm{TS}}{\cong} \varphi_\phi\boxtimes\varphi_{\phi'},\end{equation}
called \textit{Thom-Sebastiani isomorphism} (e.g. \cite{Mas}). It is easy to see that this is \textit{compatible} with the natural morphism in Milnor triangle \eqref{eqn:Milnor}, i.e. 
the following is commutative
\begin{align}\label{compatible of TS and Milnor}\xymatrix{
k^*\varphi_{\phi\boxplus \phi'} \ar[r]^{\stackrel{\mathrm{TS}}{\cong}\,\,}\ar[d]& \varphi_\phi\boxtimes\varphi_{\phi'} \ar[d] \\
k^*j^* \ar[r]^{= \quad \,\, }&i_{Z(\phi)}^*\boxtimes i_{Z(\phi')}^*. }
\end{align}
Indeed, by \cite[Lemma~1.2]{Mas} the natural transform $\mathrm{TS}$ comes from 
\[\bfR\Gamma_{\mathrm{Re}\phi\geqslant  0\times \mathrm{Re}\phi'\geqslant  0}\to \bfR\Gamma_{\mathrm{Re}\phi\boxplus\phi'\geqslant  0},\]
which in turn commutes with $\bfR\Gamma_{\mathrm{Re}\phi\geqslant  0\times \mathrm{Re}\phi'\geqslant  0}\to\id$ and $\bfR\Gamma_{\mathrm{Re}\phi\boxplus\phi'\geqslant  0}\to \id$, hence implying \eqref{compatible of TS and Milnor}.

\subsection{Equivariance}\label{sect on equi}

If $X$ carries a $F_0$-action, where $F_0$ is a complex linear algebraic group, we denote $H_{c,F_0}^*(X)^\vee$ to be the Verdier dual  to the corresponding equivariant compactly supported cohomology of $X$. More generally, we can consider cohomology valued in an equivariant sheaf (see \cite{GKM}). 
%\yl{I did not find the notion of equivariant sheaf in \cite{GKM}?}
For any equivariant complex of constructible sheaves $A$ on $X$, we define 
$$H_{c,F_0}(X,A)^\vee:=\mathbb{D}_{\pt}p_!A. $$
We denote this by $H_{F_0}(X,A)$ for simplicity (when $A=\mathbb{Q}$, we simply write it as $H^{BM}_{F_0}(X)$). This is a module over $H^*_{F_0}(\pt)$, 
the ring of conjugation invariant functions on $\frak{f}_0^*:=Lie(F_0)^\ast$. 
%Let $G$ be a complex linear algebraic group.

Suppose $X$ is a smooth complex algebraic variety and endowed with a $F_0$-invariant regular function 
$$\phi: X\to \mathbb{C}. $$ 
As in \cite[\S 2.4]{D}, we assume every $x\in X$ is contained in a $F_0$-invariant open affine neighborhood. The vanishing cycles functor $\varphi_\phi$
 applied to any $F_0$-equivariant complex of sheaves on $X$, results in an equivariant complex of sheaves. All the discussions above carries to the equivariant setting.
 Notice here that the function $\phi$ has to be $F_0$-invariant for the vanishing cycle functor to be well-defined in the equivariant setting. 
 With this definition of equivariant (co)homology, one has the cycle map 
 $$cl: A_*([X/F_0])\to H^{BM}_{F_0}(X), $$ 
 where we follow e.g.~\cite[Def.~A.2.(2)]{Park} to define the left hand side via Totaro construction \cite{To}.

\providecommand{\bysame}{\leavevmode\hbox to3em{\hrulefill}\thinspace}
\providecommand{\MR}{\relax\ifhmode\unskip\space\fi MR }
\providecommand{\MRhref}[2]{%
 \href{http://www.ams.org/mathscinet-getitem?mr=#1}{#2}}
\providecommand{\href}[2]{#2}

\end{document}